\documentclass[11pt, reqno]{amsart} % AMS article style
\usepackage{amssymb,amscd,amsfonts,amsbsy}
\usepackage{latexsym}
\usepackage{exscale}
\usepackage{amsmath,amsthm,amsfonts}
\usepackage{mathrsfs}
\usepackage{xcolor} % Enables the creations of colors.
\usepackage[colorlinks=true, linkcolor=blue, citecolor=red, urlcolor=red, backref=page]{hyperref} % Add coloured links
\usepackage{esint} % For \fint symbol
\usepackage{stmaryrd}
\usepackage{pifont}
\usepackage{enumitem}
\usepackage[utf8]{inputenc}

\calclayout

\allowdisplaybreaks%[3] % Allow equation break between pages

\newtheorem{theorem}{Theorem}[section]

\newtheorem{lemma}[theorem]{Lemma}
 
\newtheorem{claim}[theorem]{Claim}

\theoremstyle{definition}
\newtheorem{definition}[theorem]{Definition}
\newtheorem{remark}[theorem]{Remark}
\newtheorem{example}[theorem]{Example}

\numberwithin{equation}{section} %numbers equations within sections.

\def\A{\mathcal{A}}
\def\C{\mathfrak{C}}
\def\N{\mathbb{N}}
\def\D{\mathcal{D}}
\def\B{\mathscr{B}}
\def\bB{\mathbb{B}}
\def\L{\mathscr{L}}
\def\Ln{\mathscr{L}^n}
\def\M{\mathcal{M}}
\def\S{\mathcal{S}}
\def\T{\mathcal{T}}
\def\V{\mathcal{V}}
\def\F{\mathcal{F}}
\def\G{\mathcal{G}}
\def\H{\mathcal{H}}
\def\K{\mathcal{K}}
\def\Z{\mathbb{Z}}
\def\R{\mathbb{R}}
\def\Rn{\mathbb{R}^n}

\def\diam{\operatorname{diam}} 

\def\supp{\operatorname{supp}}
\def\BMO{\operatorname{BMO}}
\def\VMO{\operatorname{VMO}}
\def\CMO{\operatorname{CMO}}

\def\loc{\operatorname{loc}}
\def\osc{\operatorname{osc}}
\def\Re{\operatorname{Re}}

\def\b{\mathbf{b}}
\def\w{\omega}
\def\f{\mathfrak{f}}
\def\m{\mathfrak{m}}

\DeclareMathOperator*{\esssup}{ess\,sup}
\DeclareMathOperator*{\essinf}{ess\,inf}
\renewcommand{\emptyset}{\text{\textup{\O}}}
\newcommand{\tinyemptyset}{\text{\tiny \textup{\O}}}

\begin{document}

\author[M. Cao]{Mingming Cao}
\address{Mingming Cao\\
Instituto de Ciencias Matem\'aticas CSIC-UAM-UC3M-UCM\\
Con\-se\-jo Superior de Investigaciones Cient{\'\i}ficas\\
C/ Nicol\'as Cabrera, 13-15\\
E-28049 Ma\-drid, Spain} \email{mingming.cao@icmat.es}

\author[G. Iba\~{n}ez-Firnkorn]{Gonzalo Iba\~{n}ez-Firnkorn}
\address{Gonzalo~Iba\~{n}ez-Firnkorn\\ 
INMABB, Universidad Nacional del Sur (UNS)-CONICET, Bah\'{i}a Blanca, Argentina. 
Departamento de Matem\'{a}tica, Universidad Nacional del Sur (UNS), Bah\'{i}a Blanca, Argentina} 
\email{gonzalo.ibanez@uns.edu.ar}

\author[I.P. Rivera-R\'{i}os]{Israel P. Rivera-R\'{i}os}
\address{Israel P. Rivera-R\'{i}os\\ 
Departamento de An\'{a}lisis Matem\'{a}tico, Estad\'{i}stica e Investigaci\'{o}n Operativa
y Matem\'{a}tica Aplicada. Facultad de Ciencias. Universidad de M\'{a}laga
\\ M\'{a}laga, Spain. Departamento de Matem\'{a}tica. Universidad Nacional del Sur \\ 
Bah\'{i}a Blanca, Argentina}
\email{israelpriverarios@uma.es}

\author[Q. Xue]{Qingying Xue}
\address{Qingying Xue\\ 
School of Mathematical Sciences\\
Beijing Normal University\\
Beijing 100875\\
People's Republic of China}\email{qyxue@bnu.edu.cn}

\author[K. Yabuta]{K\^{o}z\^{o} Yabuta}
\address{K\^{o}z\^{o} Yabuta\\
Research Center for Mathematics and Data Science\\
Kwansei Gakuin University\\
Gakuen 2-1, Sanda 669-1337\\
Japan}\email{kyabuta3@kwansei.ac.jp}

%\thanks{}

\date{\today}


%42B20 Singular and oscillatory integrals (Calderon-Zygmund, etc.)
%42B25 Maximal functions, Littlewood-Paley theory
%42B35 Function spaces arising in harmonic analysis


%%%%%%%%%%%%%%%%%%%%% ABSTRACT ABSTRACT ABSTRACT %%%%%%%%%%%%%%%%%%%%%%
\begin{abstract}
In recent years, dyadic analysis has attracted a lot of attention due to the $A_2$ conjecture. It has been well understood that in the Euclidean setting, Calder\'{o}n-Zygmund operators can be pointwise controlled by a finite number of dyadic operators with a very simple structure, which leads to some significant weak and strong type inequalities. Similar results hold for Hardy-Littlewood maximal operators and Littlewood-Paley square operators. These owe to good dyadic structure of Euclidean spaces. Therefore, it is natural to wonder whether we could work in general measure spaces and find a universal framework to include these operators. In this paper, we develop a comprehensive weighted theory for a class of Banach-valued multilinear bounded oscillation operators on measure spaces, which merges multilinear Calder\'{o}n-Zygmund operators with a quantity of operators beyond the multilinear Calder\'{o}n-Zygmund theory. We prove that such multilinear operators and corresponding commutators are locally pointwise dominated by two sparse dyadic operators, respectively. We also establish three kinds of typical estimates: local exponential decay estimates, mixed weak type estimates, and sharp weighted norm inequalities. Beyond that, based on Rubio de Francia extrapolation for abstract multilinear compact operators, we obtain weighted compactness for commutators of specific multilinear operators on spaces of homogeneous type. A compact extrapolation allows us to get weighted estimates in the full range of exponents, while weighted interpolation for multilinear compact operators is crucial to the compact extrapolation. These are due to a weighted Fr\'{e}chet-Kolmogorov theorem in the quasi-Banach range, which gives a characterization of relative compactness of subsets in weighted Lebesgue spaces. As applications, we illustrate multilinear bounded oscillation operators with examples including multilinear Hardy-Littlewood maximal operators on measure spaces, multilinear $\omega$-Calder\'{o}n-Zygmund operators on spaces of homogeneous type, multilinear Littlewood-Paley square operators, multilinear Fourier integral operators, higher order Calder\'{o}n commutators, maximally modulated multilinear singular integrals, and $q$-variation of $\omega$-Calder\'{o}n-Zygmund operators. 
\end{abstract}
%%%%%%%%%%%%%%%%%%%%%% ABSTRACT ABSTRACT ABSTRACT %%%%%%%%%%%%%%%%%%%%%

\title[Multilinear bounded oscillation operators]
{A class of multilinear bounded oscillation operators on measure spaces and applications}  

\maketitle

\tableofcontents

%%%%%%%%%%%%%%%%%%%%%%%% SECTION SECTION SECTION %%%%%%%%%%%%%%%%%%%%%%
%%%%%%%%%%%%%%%%%%%%%%%% SECTION SECTION SECTION %%%%%%%%%%%%%%%%%%%%%%
\section{Introduction}
\subsection{Motivation}
The purpose of this paper is to develop a comprehensive weighted theory for a class of multilinear bounded oscillation operators on measure spaces (cf. Definition \ref{def:MBO}), which covers multilinear Calder\'{o}n-Zygmund operators and a number of operators beyond the multilinear Calder\'{o}n-Zygmund theory. The main reason why we study it comes from three aspects: 
{\bf (1)} Dyadic analysis has rapidly developed recently. This owes to the $A_2$ conjecture for general Calder\'{o}n-Zygmund operators solved by Hyt\"{o}nen \cite{Hyt}. Since then, it has drawn much attention to controlling operators by simple dyadic operators, which improves numerous explored results and helps to investigate some unexplored fields. This is the case for many operators such as Bochner-Riesz multipliers \cite{BBL}, singular non-integral operators \cite{BFP}, rough operators \cite{CCDO}, the discrete cubic Hilbert transform \cite{CKL}, and oscillatory integrals \cite{LS}. 
{\bf (2)} Some operators beyond the multilinear Calder\'{o}n-Zygmund theory developed in \cite{GT1, GT2, LOPTT} enjoy the same properties as Calder\'{o}n-Zygmund operators, for example, multilinear singular integrals with non-smooth kernels \cite{DGY}, multilinear Fourier multipliers \cite{HWXY, LiS}, multilinear pseudo-differential operators \cite{CXY}, and multilinear square operators \cite{ChXY, SXY, SiXY, XPY, XY}. 
{\bf (3)} The oscillation of an operator determines its behavior including pointwise bounds and weighted norm inequalities. A classical tool to measure oscillation is the sharp maximal function of Fefferman and Stein
\begin{align}
M^{\#}_{\lambda} f(x) 
:= \sup_{Q \ni x} \inf_{c \in \R} \bigg(\fint_Q |f(x)-c|^{\lambda} \, dx \bigg)^{\frac{1}{\lambda}}, \quad 0<\lambda<1, 
\end{align}
where the supremum is taken over all cubes $Q$ in $\Rn$ containing $x$. A typical estimate \cite[Theorem 3.2]{LOPTT} for multilinear Calder\'{o}n-Zygmund operators $T$ is that 
\begin{align}\label{MTM}
M_{\lambda}^{\#}(T(\vec{f}))(x) 
\lesssim \mathcal{M}(\vec{f})(x), \quad x \in \Rn, \quad 0<\lambda<1/m, 
\end{align}
where $\mathcal{M}$ is the multilinear Hardy-Littlewood maximal operator. This inequality encodes a common property for operators in \cite{CXY, ChXY, LiS, SXY, SiXY, XPY, XY}. Beyond that, Lerner \cite{Ler10} introduced the method of local mean oscillation to refine estimates as  \eqref{MTM}, where the local mean oscillation of $f$ on a cube $Q$ is defined by
\begin{align}
\omega_{\lambda}(f; Q) 
:= \inf_{c \in \R} ((f-c) \mathbf{1}_Q)^*(\lambda |Q|), \quad 0<\lambda <1. 
\end{align}
Lerner's formula may be thought of as a variation of the Calder\'{o}n-Zygmund decomposition of $f-m_f(Q_0)$, replacing the mean by the median. But unlike the Calder\'{o}n-Zygmund decomposition done at one scale, one has to estimate the local mean oscillation of $f$ at all scales. What's more, carrying out this approach requires the good dyadic structure of underlying spaces.  
See \cite{BH, CMP12, DLP, Ler10, Ler11, LN} for applications of local mean oscillation. Along this direction, the authors in \cite{LO} considered the oscillation of an operator $T$ 
\begin{align}
\mathcal{M}^{\#}_{T, \alpha} f(x) 
:= \sup_{Q \ni x} \sup_{x', \, x'' \in Q} 
|T(f \mathbf{1}_{\Rn \setminus \alpha Q})(x') - T(f \mathbf{1}_{\Rn \setminus \alpha Q})(x'')|, 
\end{align}
and determined nearly minimal assumptions on $T$ for which it admits a sparse domination.

The main contributions of this article are the following. 
\begin{itemize}%[leftmargin=1cm, labelwidth=1cm, itemsep=0.1cm, topsep=0.2cm]
\item Our general framework unifies the study of multilinear Hardy-Littlewood maximal operators, multilinear Calder\'{o}n-Zygmund operators, multilinear Littlewood-Paley square operators, and operators beyond the multilinear Calder\'{o}n-Zygmund theory. The latter contain multilinear Fourier integral operators, higher order Calder\'{o}n commutators, and maximally modulated multilinear singular integrals on measure spaces. Interestingly, it incorporates multilinear Hardy-Littlewood maximal operators into singular integrals.

\item Our formulation is given by Banach-valued multilinear operators on measure spaces. It allows us to work in spaces with non-doubling measures and investigate vector-valued operators, for example, multilinear Littlewood-Paley square operators (cf. Section \ref{sec:LP}), maximally modulated multilinear singular integrals (cf. Section \ref{sec:mod}), and $q$-variation of $\omega$-Calder\'{o}n-Zygmund operators (cf. Section \ref{sec:var}).

\item The assumption \eqref{list:T-reg} is given by a product-type bound instead of a multilinear one so that we can include  multilinear operators with non-smooth kernels, for example, higher order Calder\'{o}n commutators (cf. Section \ref{sec:Calderon}), which are not multilinear Calder\'{o}n-Zygmund operators.

\item The first main result, Theorem \ref{thm:sparse}, gives a pointwise sparse domination for multilinear bounded oscillation operators and their commutators. The greatest difficulty in the proof is that unlike the Euclidean case, measure spaces do not own any dyadic structure. Instead we use stopping time argument in \cite{Kar} to construct generation balls on which pointwise inequalities hold. It also leads that sparse domination is local and two sparse families are needed, which is quite different from known results in \cite{CXY, CY, Ler16, LOR}. Although Theorem \ref{thm:sparse} is local, we can conclude desired inequalities using a limiting argument and uniform bounds for dyadic operators.

\item The second main result, Theorem \ref{thm:local}, establishes local exponential decay estimates for multilinear bounded oscillation operators and local sub-exponential decay estimates for commutators. They accurately reflect the extent that an operator is locally controlled by certain maximal operator, which improves the corresponding good-lambda inequalities. The approach of local mean oscillations in \cite{CXY, OPR} is invalid because it needs the good dyadic structure of underlying spaces. Our proof is based on extrapolation method and local Coifman-Fefferman inequalities (cf. Lemma \ref{lem:Ub}).

\item Theorem \ref{thm:weak} presents a mixed weak type inequality, which parallels to Sawyer's conjecture for Calder\'{o}n-Zygmund operators \cite{LOP, LOPi} and for pseudo-differential operators \cite{CXY}, and improves the classical endpoint weighted inequality (cf. Section \ref{sec:app}). To prove it, we utilize endpoint extrapolation techniques from \cite{CMP} and Coifman-Fefferman inequalities (cf. Lemma \ref{lem:TM}). Additionally, the latter needs an $A_{\infty}$ extrapolation on general measure spaces (cf. \cite[Theorem 3.34]{CMM}).

\item Theorems \ref{thm:T}--\ref{thm:T-Besi} obtain quantitative and sharp weighted norm inequalities for multilinear bounded oscillation operators and their commutators. Back to Euclidean spaces or spaces of homogeneous type, without using extrapolation, we recover the sharp estimates in \cite[Corollary 4.4]{Nie} and extend \cite[Theorem 1.4]{LMS} with Banach exponents to the quasi-Banach range. As seen in \cite{BMMST}, Rubio de Francia extrapolation provides an effective way to directly obtain weighted inequalities for commutators from weighted estimates for operators, but it requires sharp reverse H\"{o}lder and John-Neirenberg inequalities, which are not necessarily true in the current scenario. Instead of that we establish quantitative weighted estimates for multilinear dyadic operators (cf. Section \ref{sec:sharp}).

\item By means of weighted estimates above, we further achieve weighted compactness for commutators of specific multilinear operators on spaces of homogeneous type. We first establish Rubio de Francia extrapolation for abstract multilinear compact operators (cf. Theorems \ref{thm:Ap}), which allows one to extrapolate compactness of $T$ from just one space to the full range of weighted spaces, whenever $m$-linear or $m$-sublinear operators $T$ are bounded on weighted Lebesgue spaces. In terms of commutators of $m$-linear or $m$-linearizable operators $T$, we obtain a similar compact multilinear extrapolation (cf. Theorem \ref{thm:Tb}), which asserts that to get weighted compactness of commutators $[T, \b]_{\alpha}$ in full range, it suffices to prove weighted bounds for $T$ on one space and unweighted compactness of $[T, \b]_{\alpha}$ on another (or the same) space. Theorems \ref{thm:Ap}--\ref{thm:Tb} are vastly significant and generalized since they are successfully given for $m$-sublinear operators in the full range of $p$ and $\vec{p}=(p_1, \ldots, p_m)$, which enables us to extend special estimates in the Banach space setting to the quasi-Banach range, while Theorem \ref{thm:Tb} reduces weighted compactness to the unweighted case.  For example, in the bilinear case, \cite{BO, BT, BuC, Hu17} only obtained compactness for $\frac1p=\frac{1}{p_1} + \frac{1}{p_2}$ with $1<p, p_1, p_2 < \infty$, and \cite{BDMT} proved weighted compactness from $L^{p_1}(\Rn, w_1) \times L^{p_2}(\Rn, w_2)$ to $L^p(\Rn, w)$ for $1<p,p_1,p_2<\infty$ and $(w_1, w_2) \in A_p \times A_p \subsetneq A_{p_1, p_2}$. Certainly, these results are incomplete since the restriction on weights or exponents are unnatural. Our extrapolation Theorems \ref{thm:Ap}--\ref{thm:Tb} will overcome these problems.

\item To prove Theorem \ref{thm:Ap}, we present weighted interpolation for multilinear compact operators (cf. Theorem \ref{thm:WMIP-4}) on spaces of homogeneous type, which improves known bilinear compact interpolation in Banach spaces (for example, \cite{CFM, FM}). Furthermore, the target space $L^{p_i}(\Sigma, w_i)$ are viewed as an interpolation space of another two different weighted spaces $L^{r_i}(\Sigma, u_i)$ and $L^{s_i}(\Sigma, v_i)$ such that $\vec{u}=(u_1, \ldots, u_m) \in A_{\vec{r}}$ and $\vec{v}=(v_1, \ldots, v_m) \in A_{\vec{s}}$, where $\vec{v}$ is the weight appearing in hypotheses so that $T$ is compact from $L^{s_1}(\Sigma, v_1) \times \cdots \times L^{s_m}(\Sigma, v_m)$ to $L^s(\Sigma, v)$. By extrapolation for multilinear Muckenhoupt weights (cf. Theorem \ref{thm:extraAp}), $T$ is bounded from $L^{r_1}(\Sigma, u_1) \times \cdots \times L^{r_m}(\Sigma, u_m)$ to $L^r(\Sigma, u)$. Then, the desired compactness from $L^{p_1}(\Sigma, w_1) \times \cdots \times L^{p_m}(\Sigma, w_m)$ to $L^p(\Sigma, w)$ will follow from these two results and multilinear compact interpolation Theorem \ref{thm:WMIP-4}.  Having shown Theorem \ref{thm:Ap}, the proof of Theorem \ref{thm:Tb} relies on extrapolation from operators to commutators (cf. Theorem \ref{thm:TTb}).

\item In terms of weighted compact interpolation, Theorem \ref{thm:WMIP-4} just requires $w_0, v_0 \in A_{\infty}$. For its proof, we build a criterion for relative compactness in $L^p(\Sigma, w)$ with $p \in (0, \infty)$ by means of a weighted Fr\'{e}chet-Kolmogorov theorem (cf. Theorem \ref{thm:FKhs-2}). To verify the criterion above, we utilize weighted multilinear interpolation theorems in both Lebesgue spaces (cf. \cite[Theorem 3.1]{COY}) and mixed-norm Lebesgue spaces (cf. \cite[Theorem 3.5]{COY}). In addition, when studying compactness, in contrary to Theorem \ref{thm:Tb} suitable to $m$-linear and $m$-linearizable operators, Theorem \ref{thm:FKhs-2} provides a direct way and can be applied to arbitrary operators or commutators (cf. Section \ref{sec:mod}).

\item In spaces of homogeneous type $(\Sigma, \rho, \mu)$, compactness becomes more subtle than that in the Euclidean space. For example, one can not ensure that every closed and bounded set in $(\Sigma, \rho, \mu)$ is compact, which holds if and only if in $(\Sigma, \rho, \mu)$ is complete. Another obstacle is the lack of the continuity of $x \mapsto \mu(B(x, r))$. These two problems lead that we can not obtain the equi-continuity of $\{f_{B(x, r)}\}_{f \in \K}$ on the closure of a metric ball, which is crucial to establish a criterion for relative compactness.  

\item For applications, as announced above, we present many remarkable examples in Section \ref{sec:app}, where plenty of known results are improved and extended. In Sections \ref{sec:CZO}--\ref{sec:var}, we establish both weighted bounded and weighted compactness in the natural and full range. Our theory can be applied to singular integrals on uniformly rectifiable domain, for example, the harmonic double layer potential (cf. Example \ref{ex:UR}),  which is a key tool to investigate the solvability of PDE on unbounded domains and the regularity of various domains (cf. \cite{MMMMM}). Moreover, considering multilinear Littlewood-Paley operators and square operators associated with Fourier multipliers, we use a minimal regularity assumption, the $L^r$-H\"{o}rmander condition, so that our results cover and refine all estimates for square operators in \cite{CY, ChXY, SXY, SiXY, XPY, XY}. Finally, having presented multilinear bounded oscillation operators, we can use them to define maximally modulated multilinear singular integrals. It greatly  enlarges the theory of maximally modulated Calder\'{o}n-Zygmund operators in the linear case \cite{GMS}. A celebrated example in this direction is the (lacunary/polynomial) Carleson operator. 

\end{itemize}

%%%%%%%%%%%%%%%%%%%%% SUBSECTION SUBSECTION SUBSECTION %%%%%%%%%%%%%%%%%%
%%%%%%%%%%%%%%%%%%%%% SUBSECTION SUBSECTION SUBSECTION %%%%%%%%%%%%%%%%%%
\subsection{The hypotheses}  
To set the stage and formulate our main theorems, we first recall the concept of ball-basis for measurable spaces introduced in \cite{Kar}, which gives the basic geometric properties and is a foundation for our analysis.

%%%%%%%%%%%%%%%%%%%% DEFINITION DEFINITION DEFINITION %%%%%%%%%%%%%%%%%%%%%
\begin{definition}\label{def:basis}
Let $(\Sigma, \mu)$ be a measure space. A family of measurable sets $\B$ in $\Sigma$ is said to be {\tt a ball-basis} if it satisfies the following conditions:
\begin{list}{\rm (\theenumi)}{\usecounter{enumi}\leftmargin=1.2cm \labelwidth=1cm \itemsep=0.2cm \topsep=.2cm \renewcommand{\theenumi}{B\arabic{enumi}}}

\item\label{list:B1} $\B$ is a basis, namely, $0<\mu(B)<\infty$ for any $B \in \B$.

\item\label{list:B2} For all points $x, y \in \Sigma$, there exists some $B \in \B$ containing both $x$ and $y$. 

\item\label{list:B3} For any $\varepsilon>0$ and a measurable set $E$, there exists a (finite or infinite) sequence $\{B_k\} \subset \B$ such that $\mu(E \Delta \bigcup_k B_k) < \varepsilon$. 

\item\label{list:B4} For any $B \in \B$, there exists $B^* \in \B$ such that $\mu(B^*) \le \C_0 \mu(B)$, 
\begin{align*}
\bigcup_{B' \in \B: B' \cap B \neq \emptyset \atop \mu(B') \le 2\mu(B)} B' \subset B^*, 
\quad\text{ and }\quad 
A \subset B \Longrightarrow A^* \subset B^*, 
\end{align*}
where $\C_0 \ge 1$ is a universal constant.
\end{list}
\end{definition}
%%%%%%%%%%%%%%%%%%%% DEFINITION DEFINITION DEFINITION %%%%%%%%%%%%%%%%%%%%%

\begin{example}\label{example}
It is not hard to verify that each of the following is a ball-basis. 
\begin{itemize} 
\item Let $\B$ be the collection of all balls/cubes in $\Rn$. Then $\B$ is a ball-basis of $(\Rn, \Ln)$, where here and elsewhere $\Ln$ stands for the Lebesgue measure in $\Rn$. Indeed, it suffices to take $\C_0=\kappa^n$ with $\kappa \ge 1+2^{1+\frac1n}$, and $B^*=\kappa B$ for each $B \in \B$.

\item If $\B$ is a dyadic lattice in $\Rn$, then $\B$ is a ball-basis of $(\Rn, dx)$. A dyadic lattice $\mathscr{D}$ in $\Rn$ is any collection of cubes satisfying 

\begin{itemize}
\item for each $Q \in \mathscr{D}$, $\D(Q) \subset \mathscr{D}$;

\item for all $Q', Q'' \in \mathscr{D}$, there exists $Q \in \mathscr{D}$ such that $Q', Q'' \in  \D(Q)$;

\item for every compact set $K \subset \Rn$, there exists a cube $Q \in \mathscr{D}$ containing $K$.
\end{itemize} 
Here $\D(Q)$ denotes the collection of all dyadic descendants of $Q$. 

\item Let $\B$ be the family of metric balls in spaces of homogeneous type $(\Sigma, \rho, \mu)$. If $\B$ satisfies the density condition, then it is a ball-basis of $(\Sigma, \rho, \mu)$ (cf. Section \ref{sec:CZO}). 

\item Let $(\Sigma, \mu)$ be a measure space with a martingale basis $\B$. The latter means   
\begin{itemize}
\item $\B:=\bigcup_{j \in \mathbb{Z}} \B_j$ generates the $\sigma$-algebra of all measurable sets. 
\item Each $\B_j$ is a finite or countable partition of $\Sigma$.
\item For each $j$ and $B \in \B_j$, the set $B$ is the union of sets from $\B_{j+1}$.
\item For any $x, y \in \Sigma$, there is a set $B \in \B$ such that $x, y \in B$.
\end{itemize}
Given $B \in \B$, let $\mathscr{F}(B) \in \B$ be the minimal element such that $B \subsetneq \mathscr{F}(B)$. Define $\mathscr{F}^1 := \mathscr{F}$, $\mathscr{F}^{k+1} := \mathscr{F}(\mathscr{F}^k)$ for each $k \ge 1$, and 
\begin{align*}
B^*=\begin{cases}
B &\text{ if } \mu(\mathscr{F}(B)) > 2 \mu(B), 
\\
\mathscr{F}^k(B) & \text{ if } \mu(\mathscr{F}^k(B)) \leq 2 \mu(B) < \mu(\mathscr{F}^{k+1}(B)). 
\end{cases}
\end{align*}
Then it is not hard to verify that $\B$ is a ball-basis of $(\Sigma, \mu)$ with the constant $\C_0=2$. In this scenario,  $\mu$ can be a non-doubling measure.
\end{itemize} 

However, the following collections are not ball-bases. 
\begin{itemize}
\item The collection of all standard dyadic cubes in $\Rn$ does not form a ball-basis of $(\Rn, \Ln)$ since it does not satisfy the property \eqref{list:B2}. 

\item The family of all rectangles in $\Rn$ with sides parallel to the coordinate axes is not a ball-basis of $(\Rn, \Ln)$. Indeed, considering the case $n=2$, one can pick $B=[0, 1) \times [0, 1)$ and 
\begin{align*}
B_k :=[0, 2^{k+1}) \times [0, 2^{-k}), \quad k=0, 1, \ldots.
\end{align*} 
Then $B_k \cap B \neq \emptyset$ and $|B_k|=2|B|$, $k=0, 1, \ldots,$ but there is no rectangle containing the union of all $B_k$.

\item Let $\B$ be the collection of Zygmund rectangles in $(\R^3, \L^3)$ whose sides are parallel to the coordinate axes and have lengths $s$, $t$ and $st$ with $s, t > 0$. Then $\B$ is not a ball-basis of $(\Rn, \L^3)$. The reason is similar to the above since it suffices to choose $B=[0, 1)^3$ and 
\begin{align*}
B_k :=[0, 2^{k+\frac12}) \times [0, 2^{-k}) \times [0, 2^{\frac12}), \quad k=0, 1, \ldots.
\end{align*} 

\end{itemize} 
\end{example}

To introduce multilinear bounded oscillation operators, we need to fix some notation. Denote $\vec{f} = (f_1, \ldots, f_m)$ and $\vec{f} \mathbf{1}_E = (f_1 \mathbf{1}_E, \ldots, f_m \mathbf{1}_E)$ for any measurable set $E$. Given $r \in [1, \infty)$ and $B \in \B$, we set 
\begin{align*}
\lfloor f \rfloor_{B, r} := \sup_{B' \in \B: B' \supset B} \langle f \rangle_{B', r}, 
\quad\text{ where }\quad 
\langle f \rangle_{B, r} := \bigg(\fint_B |f|^r \, d\mu \bigg)^{\frac1r}.  
\end{align*} 
When $r=1$, we always denote $\lfloor f \rfloor_B := \lfloor f \rfloor_{B, r}$ and $\langle f \rangle_B :=\langle f \rangle_{B, r} $. 

%%%%%%%%%%%%%%%%%%%% DEFINITION DEFINITION DEFINITION %%%%%%%%%%%%%%%%%%%%%
\begin{definition}
Let $\mathscr{M}(\Sigma, \mu)$ be the set of all real-valued measurable functions on measure space $(\Sigma, \mu)$, and let $\mathscr{S}(\Sigma, \mu)$ be an appropriate linear subspace of $\mathscr{M}(\Sigma, \mu)$. Given an operator $T: \mathscr{S}(\Sigma, \mu) \times \cdots \times \mathscr{S}(\Sigma, \mu) \to \mathscr{M}(\Sigma, \mu)$, we say that $T$ is {\tt $m$-linear} if for all $f_i, g_i \in \mathscr{S}(\Sigma, \mu)$, $i=1, \ldots, m$, and for all $\lambda \in \R$,  
\begin{align*}
&T(f_1, \ldots, \lambda f_i, \ldots, f_m)
=\lambda \, T(f_1, \ldots, f_i, \ldots, f_m), 
\\
&T(\ldots, f_i+ g_i, \ldots) 
= T(\ldots, f_i, \ldots)  + T(\ldots, g_i, \ldots).  
\end{align*}
Similarly, $T$ is called an {\tt $m$-sublinear} operator if 
\begin{align*}
&|T(f_1, \ldots, \lambda f_i, \ldots, f_m)| 
=|\lambda| |T(f_1, \ldots, f_i, \ldots, f_m)|, 
\\
&|T(\ldots, f_i+ g_i, \ldots)| 
\le |T(\ldots, f_i, \ldots)|  + |T(\ldots, g_i, \ldots)|. 
\end{align*}
An $m$-sublinear operator $T$ is called {\tt $m$-linearizable} if there exist a Banach space $\bB$ and a $\bB$-valued $m$-linear operator $\mathcal{T}$ such that $T(\vec{f}) (x) = \|\T(\vec{f}) (x)\|_{\bB}$. 
\end{definition}
%%%%%%%%%%%%%%%%%%%% DEFINITION DEFINITION DEFINITION %%%%%%%%%%%%%%%%%%%%%

%%%%%%%%%%%%%%%%%%%% DEFINITION DEFINITION DEFINITION %%%%%%%%%%%%%%%%%%%%%
\begin{definition}\label{def:MBO}
Let $(\Sigma, \mu)$ be a measure space with a ball-basis $\B$. Given a Banach space $\bB$, we say that an operator $T$ is {\tt a $\bB$-valued multilinear bounded oscillation operator} with respect to $\B$ and $r \in [1, \infty)$ if there exist constants $\C_1(T), \C_2(T) \in (0, \infty)$ and a family of $\bB$-valued $m$-linear operators $\T$ with $T(\vec{f})(x) = \|\T(\vec{f})(x)\|_{\bB}$ for every $x \in \Sigma$, or real-valued $m$-linear or $m$-sublinear operators $\T:=\{T\}$ in the case $\bB=\R$, such that for all $\vec{f} \in L^r(\Sigma, \mu)^m$,  

\begin{list}{\rm (\theenumi)}{\usecounter{enumi}\leftmargin=1.2cm \labelwidth=1cm \itemsep=0.2cm \topsep=.2cm \renewcommand{\theenumi}{T\arabic{enumi}}}

\item\label{list:T-size} For every $B_0 \in \B$ with $B_0^* \subsetneq \Sigma$, there exists $B \in \B$ with $B \supsetneq B_0$ so that 
\begin{align*}
\sup_{x \in B_0} \|\T(\vec{f} {\bf 1}_{B^*})(x) - \T(\vec{f} {\bf 1}_{B_0^*})(x)\|_{\bB} 
\le \C_1(T) \prod_{i=1}^m \langle f_i \rangle_{B^*, r}. 
\end{align*}

\item\label{list:T-reg} For every $B \in \B$,  
\begin{align*}
\sup_{x, x' \in B} \big\| \big(\T(\vec{f}) - \T(\vec{f} \mathbf{1}_{B^*}) \big)(x) 
- &\big(\T(\vec{f}) - \T(\vec{f} \mathbf{1}_{B^*}) \big)(x') \big\|_{\bB} 
\le \C_2(T) \prod_{i=1}^m \lfloor f_i \rfloor_{B, r}. 
\end{align*}
\end{list}
When $T$ is a real-valued operator, that is, $\bB=\R$, we will drop $\bB$-valued and $\|\cdot\|_{\bB}$.
\end{definition}
%%%%%%%%%%%%%%%%%%%% DEFINITION DEFINITION DEFINITION %%%%%%%%%%%%%%%%%%%%%

We should mention that the one-dimensional bounded oscillation operators were introduced by Karagulyan \cite{Kar}, but we address some inaccuracies there and improve his formula to the Banach-valued case so that one can work for more kinds of operators. This abstract formalism was also used to establish good-lambda inequalities for bounded oscillation operators in \cite{Kar21}. 

Let $(\Sigma, \mu)$ be a measure space with a basis $\B$. Given $1 \le r < \infty$, we define multilinear Hardy-Littlewood maximal operators as 
\begin{align}\label{def:Mr}
\mathcal{M}_{\B, r}(\vec{f})(x) 
&:= \sup_{B \in \B: x \in B} \prod_{i=1}^m \langle f_i \rangle_{B, r}, 
\\
\M_{\B, r}^{\otimes}(\vec{f})(x) 
&:= \prod_{i=1}^m M_{\B, r}f_i(x), 
\end{align}
if $x \in \Sigma_\B:=\bigcup_{B \in \B} B$, and $\mathcal{M}_{\B, r}(\vec{f})(x) = \mathcal{M}_{\B, r}^{\otimes}(\vec{f})(x) = 0$ otherwise. When $r=1$, we simply write $\M_{\B} := \M_{\B, r}$. In the case $m=1$, we denote $M_{\B, r} := \M_{\B, r}$. 

A collection $\S \subset \B$ is said to be {\tt $\eta$-sparse}, $\eta \in (0, 1)$, if for any $B \in \S$ there is a subset $E_B \subset B$ such that $\mu(E_B) \ge \eta\mu(B)$ and the family $\{E_B: B \in \S\}$ is pairwise disjoint. Given a sparse family $\S$, $r_1, \ldots, r_m \in [1, \infty)$, and two disjoint subsets $\tau_1, \tau_2 \subset \{1, \ldots, m\}$, we define multilinear operators and commutator as 
\begin{align*}
\A_{\S, \vec{r}} (\vec{f})(x)
&:= \sum_{B \in \S} \prod_{i=1}^m \langle f_i \rangle_{B, r_i} \mathbf{1}_{B}(x), 
\\
\A_{\S, \vec{r}}^{\b, \tau_1, \tau_2}(\vec{f})(x) 
&:= \sum_{B \in \S} \lambda_{B, \vec{r}}^{\b, \tau_1, \tau_2}(\vec{f})(x) \, \mathbf{1}_{B}(x),
\end{align*}
where 
\begin{align*}
\lambda_{B, \vec{r}}^{\b, \tau_1, \tau_2}(\vec{f})(x) 
&:= \prod_{i \in \tau_1} |b_i(x) - b_{i, B}| \langle f_i \rangle_{B, r_i}
\\
&\qquad\quad\times \prod_{j \in \tau_2} \langle (b_j - b_{j, B}) f_j \rangle_{B, r_j} 
\times \prod_{k \not\in \tau_1 \cup \tau_2} \langle f_k \rangle_{B, r_k}.
\end{align*}
If $\vec{r}=(r, \ldots, r)$, we write $\A_{\S, r} := \A_{\S, \vec{r}}$ and $\A_{\S, r}^{\b, \tau_1, \tau_2} := \A_{\S, \vec{r}}^{\b, \tau_1, \tau_2}$. When $\vec{r}=(1, \ldots, 1)$ we will drop the subscript $\vec{r}$.

Let us next define the general commutators. Let $T$ be an operator from $\mathscr{X}_1  \times \cdots \times \mathscr{X}_m$ into $\mathscr{Y}$, where $\mathscr{X}_1, \ldots, \mathscr{X}_m$ are some normed spaces and and $\mathscr{Y}$ is a quasi-normed space. Given $\vec{f} := (f_1, \ldots, f_m) \in \mathscr{X}_1 \times \cdots \times \mathscr{X}_m$, $\b=(b_1, \ldots, b_m)$ of measurable functions, and $k \in \N$, we define, whenever it makes sense, the $k$-th order commutator of $T$ in the $i$-th entry of $T$ as 
\begin{align*}
[T, \b]_{k e_i} (\vec{f})(x) 
:= T(f_1,\ldots, (b_i(x)-b_i)^k f_i, \ldots, f_m)(x), 
\end{align*}
where $e_i$ is the basis of $\Rn$ with the $i$-th component being $1$ and other components being $0$. Then, for a multi-index $\alpha = (\alpha_1, \ldots, \alpha_m) \in \N^m$, we denote 
\begin{align*}
[T, \b]_{\alpha}:= [\cdots[[T, \b]_{\alpha_1 e_1}, \b]_{\alpha_2 e_2} \cdots, \b]_{\alpha_m e_m}.
\end{align*}
When $\b=(b, \ldots, b)$, we write $[T, b]_{\alpha} := [T, \b]_{\alpha}$. In particular, if $T$ is an $m$-linear operator with a kernel representation of the form 
\begin{equation*}
T(\vec{f})(x) 
:= \int_{\Sigma^m} K(x, \vec{y}) f_1(y_1) \cdots f_m(y_m) \, d\mu(\vec{y}), 
\end{equation*}
where $d\mu(\vec{y}) := d\mu(y_1) \cdots d\mu(y_m)$, then one can write $[T, \b]_{\alpha}$ as 
\begin{equation*}
[T, \b]_{\alpha}(\vec{f})(x) 
:= \int_{\Sigma^m} \prod_{i=1}^m (b_i(x)-b_i(y_i))^{\alpha_i} K(x, \vec{y}) \prod_{j=1}^m f_j(y_j) \, d\mu(\vec{y}). 
\end{equation*}
Additionally, if $T$ is $m$-linearizable, that is, $T(\vec{f}) (x) = \|\T(\vec{f}) (x)\|_{\bB}$, then for all $\alpha \in \N^m$, 
\begin{align*}
[T, \b]_{\alpha}(\vec{f})(x) 
= \|[\T, \b]_{\alpha}(\vec{f})(x) \|_{\bB}. 
\end{align*}

%%%%%%%%%%%%%%%%%%%%% SUBSECTION SUBSECTION SUBSECTION %%%%%%%%%%%%%%%%%%
%%%%%%%%%%%%%%%%%%%%% SUBSECTION SUBSECTION SUBSECTION %%%%%%%%%%%%%%%%%%
\subsection{Main results} 
Now we are ready to state the main theorems of this article. The first one is a pointwise sparse domination as follows. 

%%%%%%%%%%%%%%%%%%%%%%% THEOREM THEOREM THEOREM %%%%%%%%%%%%%%%%%%%%%
\begin{theorem}\label{thm:sparse}
Let $(\Sigma, \mu)$ be a measure space with a ball-basis $\B$. Let $\b=(b_1, \ldots, b_m)$ of measurable functions and $\alpha \in \{0, 1\}^m$. If $T$ is a $\bB$-valued multilinear bounded oscillation operator with respect to  $\B$ and $r \in [1, \infty)$ such that $T$ is bounded from $L^r(\Sigma, \mu) \times \cdots \times L^r(\Sigma, \mu)$ to $L^{\frac{r}{m}, \infty}(\Sigma, \mu)$,  then for any $B \in \B$ and $f_i \in L^r(\Sigma, \mu)$ with $\|f_i\|_{L^r(B, \mu)}^r \ge \frac12 \|f_i\|_{L^r(\Sigma, \mu)}^r$, $i=1, \ldots, m$, there exist two $\frac{1}{2 \C_0^3}$-sparse families $\S_1, \S_2 \subset \B$ such that for a.e. $x \in B$,
\begin{align}
\label{Tsparse} |T(\vec{f})(x)| 
& \lesssim \C(T) \, \big[\A_{\S_1, r}(\vec{f})(x) + \A_{\S_2, r}(\vec{f})(x) \big], 
\\
\label{Tbsparse} |[T, \b]_{\alpha}(\vec{f})(x)| 
&\lesssim \C(T) \sum_{\tau_1 \uplus \tau_2=\tau} \big[\A_{\S_1, r}^{\b, \tau_1, \tau_2}(\vec{f})(x) 
+ \A_{\S_2, r}^{\b, \tau_1, \tau_2}(\vec{f})(x) \big],  
\end{align}
where $\tau = \tau(\alpha) :=\{i: \alpha_i \neq 0\}$ and $\C(T) := \C_1(T) + \C_2(T) + \|T\|_{L^r(\mu) \times \cdots \times L^r(\mu) \to L^{\frac{r}{m}, \infty}(\mu)}$.
\end{theorem}
%%%%%%%%%%%%%%%%%%%%%%% THEOREM THEOREM THEOREM %%%%%%%%%%%%%%%%%%%%%

The second result is the local decay estimates for bounded oscillation operators and commutators. 

%%%%%%%%%%%%%%%%%%%%%%% THEOREM THEOREM THEOREM %%%%%%%%%%%%%%%%%%%%%
\begin{theorem}\label{thm:local}
Let $(\Sigma, \mu)$ be a measure space with a ball-basis $\B$. If $T$ is a $\bB$-valued multilinear bounded oscillation operator with respect to  $\B$ and $r \in [1, \infty)$ such that $T$ is bounded from $L^r(\Sigma, \mu) \times \cdots \times L^r(\Sigma, \mu)$ to $L^{\frac{r}{m}, \infty}(\Sigma, \mu)$, then there exists $\gamma>0$ such that for all $B \in \B$ and $f_i \in L^{\infty}_c(\Sigma, \mu)$ with $\supp(f_i) \subset B$, $1 \leq i \leq m$, 
\begin{align}
\mu\big(\big\{x \in B:  |T(\vec{f})(x)| > t \, \mathcal{M}_{\B, r} (\vec{f})(x)  \big\}\big) 
& \lesssim e^{- \gamma t}  \mu(B),  
\end{align}
for all $t>0$. If in addition $A_{\infty, \B}$ satisfies the sharp reverse H\"{o}lder property, then for any $\alpha \in \{0, 1\}^m$, 
\begin{align}
\mu \big(\big\{x \in B:  |[T, \b]_{\alpha}(\vec{f})(x)| > t \, \mathcal{M}_{\B, r}(\vec{f^*})(x)  \big\}\big)
\lesssim e^{-(\frac{\gamma t}{\|\b\|_{\tau}})^{\frac{1}{|\tau|+1}}} \mu(B),  
\end{align}
for all $t>0$, where $\tau=\{i: \alpha_i \neq 0\}$, $\|\b\|_{\tau} = \prod_{i \in \tau} \|b_i\|_{\osc_{\exp L}}$, and 
$f_i^*=M_{\B}^{\lfloor r \rfloor}(|f_i|^r)^{\frac1r}$, $i=1, \ldots, m$. Here and elsewhere, $\lfloor r \rfloor$ denotes the minimal integer no less than $r$. 
\end{theorem}
%%%%%%%%%%%%%%%%%%%%%%% THEOREM THEOREM THEOREM %%%%%%%%%%%%%%%%%%%%%

In the endpoint case, we establish the weighted mixed weak type inequality as follows. See Section \ref{sec:Mucken} for some insight about Muckenhoupt weights. 

%%%%%%%%%%%%%%%%%%%%%%% THEOREM THEOREM THEOREM %%%%%%%%%%%%%%%%%%%%%
\begin{theorem}\label{thm:weak}
Let $(\Sigma, \mu)$ be a measure space with a ball-basis $\B$ so that $A_{1, \B} \subset \bigcup_{s>1} RH_{s, \B}$. Let $T$ be a $\bB$-valued multilinear bounded oscillation operator with respect to  $\B$ and $r \in [1, \infty)$ such that $T$ is bounded from $L^r(\Sigma, \mu) \times \cdots \times L^r(\Sigma, \mu)$ to $L^{\frac{r}{m}, \infty}(\Sigma, \mu)$. If $w \in A_{1, \B}$ and $v^{\frac{r}{m}} \in A_{\infty, \B}$, 
%$\vec{w} \in A_{\vec{1}, \B}$ and $w v^{\frac1m} \in A_{\infty, \B}$, or $\vec{w} \in A_{1, \B} \times \cdots \times A_{1, \B}$ and $v \in A_{\infty, \B}$, 
then 
\begin{align}
\bigg\|\frac{T(\vec{f})}{v}\bigg\|_{L^{\frac{r}{m},\infty}(\Sigma, \, w v^{\frac{r}{m}})}
&\lesssim \bigg\|\frac{\mathcal{M}_{\B, r}(\vec{f})}{v}\bigg\|_{L^{\frac{r}{m},\infty}(\Sigma, \, w v^{\frac{r}{m}})}.
\end{align}
\end{theorem}
%%%%%%%%%%%%%%%%%%%%%%% THEOREM THEOREM THEOREM %%%%%%%%%%%%%%%%%%%%%

In order to present the quantitative weighted inequalities, given $\vec{w} \in A_{\vec{p}/\vec{r}}$ with $1 \le r_i < p_i<\infty$, $i=1, \ldots, m$, and a subset $\tau \subset \{1, \ldots, m\}$, we define 
\begin{align*}
\mathcal{N}_1(\vec{r}, \vec{p}, \vec{w}) := 
\begin{cases}
\prod_{i=1}^m \|M_{\B, \sigma_i}\|_{L^{\frac{p_i}{r_i}}(\Sigma, \sigma_i)}^{\frac{1}{r_i}} & \text{if } p \leq 1,  
\\
\|M_{\B, w}\|_{L^{p'}(\Sigma, w)}\prod_{i=1}^m \|M_{\B, \sigma_i}\|_{L^{\frac{p_i}{r_i}}(\Sigma, \sigma_i)}^{\frac{1}{r_i}}   & \text{if } p>1,
\end{cases}
\end{align*}
and 
\begin{align*}
\mathcal{N}_2^{\tau}(\vec{r}, \vec{p}, \vec{w}) := 
\begin{cases}
\prod_{j \in \tau} \|M_{\B, \sigma_j}\|_{L^{\frac{p_j}{r_j s_j}}(\Sigma, \sigma_i)}^{\frac{1}{r_j s_j}} & \text{if } p \leq 1,  
\\
\|M_{\B, w}\|_{L^{\frac{p'}{s}}(\Sigma, w)}^{\frac1s}  
\prod_{j \in \tau} \|M_{\B, \sigma_j}\|_{L^{\frac{p_j}{r_j s_j}}(\Sigma, \sigma_i)}^{\frac{1}{r_j s_j}}  & \text{if } p>1,
\end{cases}
\end{align*}
where $\frac1p=\sum_{i=1}^m \frac{1}{p_i}$, $w=\prod_{i=1}^m w_i^{\frac{p}{p_i}}$, and $\sigma_i=w_i^{\frac{r_i}{r_i-p_i}}$, $i=1, \ldots, m$. In the definition of $\mathcal{N}_2^{\tau}(\vec{r}, \vec{p}, \vec{w})$, it naturally requires $1<s<p'$ and $1<s_j<p_j/r_j$. If $\vec{r}=(r, \ldots, r)$, we simply write $\mathcal{N}_1(r, \vec{p}, \vec{w}) := \mathcal{N}_1(\vec{r}, \vec{p}, \vec{w})$ and $\mathcal{N}_2^{\tau}(r, \vec{p}, \vec{w}) := \mathcal{N}_2^{\tau}(\vec{r}, \vec{p}, \vec{w})$.

%%%%%%%%%%%%%%%%%%%%%%% THEOREM THEOREM THEOREM %%%%%%%%%%%%%%%%%%%%%
\begin{theorem}\label{thm:T}
Let $(\Sigma, \mu)$ be a measure space with a ball-basis $\B$. If $T$ is a $\bB$-valued multilinear bounded oscillation operator with respect to  $\B$ and $r \in [1, \infty)$ such that $T$ is bounded from $L^r(\Sigma, \mu) \times \cdots \times L^r(\Sigma, \mu)$ to $L^{\frac{r}{m}, \infty}(\Sigma, \mu)$, then for all $\vec{p}=(p_1, \ldots, p_m)$ with $r<p_1, \ldots, p_m<\infty$ and for all $\vec{w} \in A_{\vec{p}/r, \B}$, 
\begin{align}\label{eq:Npw}
\|T&\|_{L^{p_1}(\Sigma, w_1) \times \cdots \times L^{p_m}(\Sigma, w_m) \to L^p(\Sigma, w)} 
\lesssim \C(T) \, \mathcal{N}_1(r, \vec{p}, \vec{w}) 
[\vec{w}]_{A_{\vec{p}/r}}^{\max\limits_{1 \leq i \leq m} \{p, (\frac{p_i}{r})' \}}, 
\end{align}
where $\frac1{p}=\sum_{i=1}^m \frac1{p_i}$, $w=\prod_{i=1}^m w_i^{\frac{p}{p_i}}$, and $\C(T) := \C_1(T) + \C_2(T) + \|T\|_{L^r(\mu) \times \cdots \times L^r(\mu) \to L^{\frac{r}{m}, \infty}(\mu)}$. 

If in addition $A_{\infty, \B}$ satisfies the sharp reverse H\"{o}lder property, then for the same exponents $\vec{p}$ and weights  $\vec{w}$, and for all $\alpha \in \{0, 1\}^m$, 
\begin{multline}
\|[T, \b]_{\alpha}\|_{L^{p_1}(\Sigma, w_1) \times \cdots \times L^{p_m}(\Sigma, w_m) \rightarrow L^p(\Sigma, w)} 
\\ 
\lesssim \C(T) \, \mathcal{N}_1(r, \vec{p}, \vec{w})  
\sum_{\tau_2 \subset \tau} \mathcal{N}_2^{\tau_2}(r, \vec{p}, \vec{w}) \|\b\|_{\tau} 
[\vec{w}]_{A_{\vec{p}/r}}^{(|\tau| +1) \max\limits_{1 \leq i \leq m} \{p, (\frac{p_i}{r})' \}}, 
\end{multline}
where $\tau=\tau(\alpha) := \{i: \alpha_i \neq 0\}$ and $\|\b\|_{\tau} = \prod_{i \in \tau} \|b_i\|_{\osc_{\exp L}}$. 
\end{theorem}
%%%%%%%%%%%%%%%%%%%%%%% THEOREM THEOREM THEOREM %%%%%%%%%%%%%%%%%%%%%

We would like to improve Theorem \ref{thm:T}  in order to provide sharp weighted estimates. Observe that in absence of any further assumption on the measure space $(\Sigma, \mu)$, it is hard to obtain quantitative bounds of the weighted maximal functions appearing in $\mathcal{N}(r, \vec{p}, \vec{w})$ in \eqref{eq:Npw}. To achieve this target, we consider a geometric condition. We say that a basis $\B$ satisfies {\tt the Besicovitch condition} with a constant $N_0 \in \N_+$ if for any collection $\B' \subset \B$ one can find a subscollection $\B'' \subset \B'$ such that 
\begin{align*}
\bigcup_{B \in \B'} B = \bigcup_{B \in \B''} B 
\quad\text{ and }\quad 
\sum_{B \in \B''} \mathbf{1}_B(x) \le N_0, \quad x \in \Sigma. 
\end{align*}

%%%%%%%%%%%%%%%%%%%%%%% THEOREM THEOREM THEOREM %%%%%%%%%%%%%%%%%%%%%
\begin{theorem}\label{thm:T-Besi}
Let $(\Sigma, \mu)$ be a measure space with a ball-basis $\B$ satisfying the Besicovitch condition. If $T$ is a $\bB$-valued multilinear bounded oscillation operator with respect to $\B$ and $r \in [1, \infty)$ such that $T$ is bounded from $L^r(\Sigma, \mu) \times \cdots \times L^r(\Sigma, \mu)$ to $L^{\frac{r}{m}, \infty}(\Sigma, \mu)$, then for all $\vec{p}=(p_1, \ldots, p_m)$ with $r<p_1, \ldots, p_m<\infty$ and for all $\vec{w} \in A_{\vec{p}/r, \B}$, 
\begin{align}\label{T-Besi-sharp}
\|T\|_{L^{p_1}(\Sigma, w_1) \times \cdots \times L^{p_m}(\Sigma, w_m) \to L^p(\Sigma, w)} 
\lesssim \C(T) \, [\vec{w}]_{A_{\vec{p}/r, \B}}^{\max\limits_{1 \le i \le m}\{p, (\frac{p_i}{r})'\}}, 
\end{align}
where $\frac1{p}=\sum_{i=1}^m \frac1{p_i}$, $w=\prod_{i=1}^m w_i^{\frac{p}{p_i}}$, and $\C(T) := \C_1(T) + \C_2(T) + \|T\|_{L^r(\mu) \times \cdots \times L^r(\mu) \to L^{\frac{r}{m}, \infty}(\mu)}$. 

If in addition $A_{\infty, \B}$ satisfies the sharp reverse H\"{o}lder property, then for the same exponents $\vec{p}$ and weights  $\vec{w}$, and for all $\alpha \in \{0, 1\}^m$, 
\begin{align}\label{Tb-Besi-sharp}
&\|[T, \b]_{\alpha}\|_{L^{p_1}(\Sigma, w_1) \times \cdots \times L^{p_m}(\Sigma, w_m) \rightarrow L^p(\Sigma, w)} 
\lesssim \C(T)  \|\b\|_{\tau} 
[\vec{w}]_{A_{\vec{p}/r}}^{(|\tau| +1) \max\limits_{1 \leq i \leq m} \{p, (\frac{p_i}{r})' \}}, 
\end{align}
where $\tau=\tau(\alpha)=\{i: \alpha_i \neq 0\}$ and $\|\b\|_{\tau} = \prod_{i \in \tau} \|b_i\|_{\osc_{\exp L}}$. 
\end{theorem}
%%%%%%%%%%%%%%%%%%%%%%% THEOREM THEOREM THEOREM %%%%%%%%%%%%%%%%%%%%%

The weighted bound in \eqref{T-Besi-sharp} is {\tt sharp} for the multilinear $\omega$-Calder\'{o}n-Zygmund operators in  Euclidean case or in spaces of homogeneous type (cf. Section \ref{sec:CZO}). Additionally, it was shown by the first and last authors \cite{CY} that the estimates \eqref{T-Besi-sharp} and \eqref{Tb-Besi-sharp} are optimal for multilinear Littlewood-Paley operators with kernels satisfying the $L^r$-H\"{o}rmander condition.

On the other hand, relying on Theorem \ref{thm:T}, we are able to establish the weighted compactness for some specific multilinear operators on spaces of homogeneous type (cf. Sections \ref{sec:CZO}--\ref{sec:mod}). In the multilinear setting, it is significant to get results in the quasi-Banach range. To attain the aim, we present the Rubio de Francia extrapolation for multilinear compact operators as follows.  
%%%%%%%%%%%%%%%%%%%%%% THEOREM THEOREM THEOREM %%%%%%%%%%%%%%%%%%%%%%
\begin{theorem}\label{thm:Ap}
Let $(\Sigma, \rho, \mu)$ be a complete space of homogeneous type with a metrically continuous measure $\mu$. Assume that $T$ is an $m$-linear or $m$-linearizable operator satisfying 
\begin{list}{\rm (\theenumi)}{\usecounter{enumi}\leftmargin=1.2cm \labelwidth=1cm \itemsep=0.2cm \topsep=.2cm \renewcommand{\theenumi}{\roman{enumi}}}

\item\label{list-1} there exists $\vec{r}=(r_1, \ldots, r_m)$ with $1 < r_1, \ldots, r_m <\infty$ such that for all $\vec{u}=(u_1, \ldots, u_m) \in A_{\vec{r}, \B_{\rho}}$, 
\begin{align}\label{eq:Ap-1}
T \text{ is bounded from $L^{r_1}(\Sigma, u_1) \times \cdots \times L^{r_m}(\Sigma, u_m)$ to $L^r(\Sigma, u)$}, 
\end{align}
where $\frac1r=\sum_{i=1}^m \frac{1}{r_i}$ and $u=\prod_{i=1}^m u_i^{\frac{r}{r_i}}$. 

\item\label{list-2} there exist $\vec{s}=(s_1, \ldots, s_m)$ with $1 < s_1, \ldots, s_m <\infty$ and $\vec{v}=(v_1, \ldots, v_m) \in A_{\vec{s}, \B_{\rho}}$ such that 
\begin{align}\label{eq:Ap-2}
T \text{ is compact from $L^{s_1}(\Sigma, v_1) \times \cdots \times L^{s_m}(\Sigma, v_m)$ to $L^s(\Sigma, v)$},   
\end{align} 
where $\frac1s=\sum_{i=1}^m \frac{1}{s_i}$ and $v=\prod_{i=1}^m v_i^{\frac{s}{s_i}}$. 
\end{list} 
Then for all $\vec{p}=(p_1, \dots, p_m)$ with $1 < p_1, \ldots, p_m <\infty$ and for all $\vec{w}=(w_1, \ldots, w_m) \in A_{\vec{p}, \B_{\rho}}$, 
\begin{equation}
T \text{ is compact from $L^{p_1}(\Sigma, w_1) \times \cdots \times L^{p_m}(\Sigma, w_m)$ to $L^p(\Sigma, w)$,}
\end{equation} 
where $\frac1p=\sum_{i=1}^m \frac{1}{p_i}$ and $w=\prod_{i=1}^m w_i^{\frac{p}{p_i}}$.  
\end{theorem}
%%%%%%%%%%%%%%%%%%%%%% THEOREM THEOREM THEOREM %%%%%%%%%%%%%%%%%%%%%%

We also prove the following extrapolation theorem in order to obtain weighted compactness for commutators. 
%%%%%%%%%%%%%%%%%%%%%%% THEOREM THEOREM THEOREM %%%%%%%%%%%%%%%%%%%%%
\begin{theorem}\label{thm:Tb}
Let $(\Sigma, \rho, \mu)$ be a complete space of homogeneous type with a metrically continuous measure $\mu$. Let $\alpha \in \N^m$ be a multi-index and $\b=(b_1,\ldots, b_m) \in \BMO_{\B_{\rho}}^m$. Assume that $T$ is an $m$-linear or $m$-linearizable operator satisfying 
 
\begin{list}{\rm (\theenumi)}{\usecounter{enumi}\leftmargin=1cm \labelwidth=1cm \itemsep=0.2cm \topsep=.2cm \renewcommand{\theenumi}{\roman{enumi}}}

\item\label{list:Tb1} there exists $\vec{r}=(r_1, \ldots, r_m)$ with $1 < r_1, \ldots, r_m <\infty$ such that for all $\vec{u}=(u_1, \ldots, u_m) \in A_{\vec{r}, \B_{\rho}}$, 
\begin{align}\label{eq:bT-1}
T \text{ is bounded from $L^{r_1}(\Sigma, u_1) \times \cdots \times L^{r_m}(\Sigma, u_m)$ to $L^r(\Sigma, u)$}, 
\end{align}
where $\frac1r=\sum_{i=1}^m \frac{1}{r_i}$ and $u=\prod_{i=1}^m u_i^{\frac{r}{r_i}}$. 

\item\label{list:Tb2} there exists $\vec{s}=(s_1, \ldots, s_m)$ with $1 < s_1, \ldots, s_m <\infty$ such that 
\begin{align}\label{eq:bT-2}
[T, \b]_{\alpha} \text{ is compact from $L^{s_1}(\Sigma, \mu) \times \cdots \times L^{s_m}(\Sigma, \mu)$ to $L^s(\Sigma, \mu)$}, 
\end{align} 
where $\frac1s =\sum_{i=1}^m \frac{1}{s_i}$. 
\end{list} 
Then for all $\vec{p}=(p_1, \dots, p_m)$ with $1 < p_1, \ldots, p_m <\infty$ and for all $\vec{w}=(w_1, \ldots, w_m) \in A_{\vec{p}, \B_{\rho}}$, 
\begin{align}
[T, \b]_{\alpha} \text{ is compact from $L^{p_1}(\Sigma, w_1) \times \cdots \times L^{p_m}(\Sigma, w_m)$ to $L^p(\Sigma, w)$} 
\end{align}
where $\frac1p=\sum_{i=1}^m \frac{1}{p_i}$ and $w=\prod_{i=1}^m w_i^{\frac{p}{p_i}}$.  
\end{theorem}
%%%%%%%%%%%%%%%%%%%%%%% THEOREM THEOREM THEOREM %%%%%%%%%%%%%%%%%%%%%

A key point to show Theorems \ref{thm:Ap}--\ref{thm:Tb} is the following weighted interpolation for compact operators with quasi-Banach exponents. 

%%%%%%%%%%%%%%%%%%%%%%%% THEOREM THEOREM THEOREM %%%%%%%%%%%%%%%%%%%%
\begin{theorem}\label{thm:WMIP-4}
Let $(\Sigma, \rho, \mu)$ be a complete space of homogeneous type with a metrically continuous measure $\mu$. Let $0< p_0, q_0<\infty$, $1 \leq p_i, q_i \leq \infty$, $i=1,\ldots,m$, and let $w_0, v_0 \in A_{\infty, \B_{\rho}}$ and $w_i, v_i$ be weights on $(\Sigma, \rho, \mu)$. If an $m$-linear or $m$-linearizable operator $T$ satisfies  
\begin{align}
\label{eq:WMIP-41} &T \text{ is bounded from $L^{p_1}(\Sigma, w_1) \times \cdots \times L^{p_m}(\Sigma, w_m)$ to $L^{p_0}(\Sigma, w_0)$}, 
\\
\label{eq:WMIP-42} &T \text{ is compact from $L^{q_1}(\Sigma, v_1) \times \cdots \times L^{q_m}(\Sigma, v_m)$ to $L^{q_0}(\Sigma, v_0)$}, 
\end{align}
then $T$ can be extended to a compact operator from $L^{s_1}(\Sigma, u_1) \times \cdots \times L^{s_m}(\Sigma, u_m)$ to $L^{s_0}(\Sigma, u_0)$ for all exponents satisfying $0<\theta<1$,
\begin{equation}\label{eq:exp}
\begin{aligned}
\frac{1}{s_i}=\frac{1-\theta}{p_i}+\frac{\theta}{q_i},  
\quad\text{and}\quad u_i^{\frac{1}{s_i}}=w_i^{\frac{1-\theta}{p_i}} v_i^{\frac{\theta}{q_i}},\quad i=0,1, \ldots, m.
\end{aligned}
\end{equation}
\end{theorem}
%%%%%%%%%%%%%%%%%%%%%%%% THEOREM THEOREM THEOREM %%%%%%%%%%%%%%%%%%%%

To demonstrate Theorem \ref{thm:WMIP-4}, we establish weighted Fr\'{e}chet-Kolmogorov theorems below, which give characterizations of relative compactness of subsets in weighted Lebesgue spaces. 

%%%%%%%%%%%%%%%%%%%%%% THEOREM THEOREM THEOREM %%%%%%%%%%%%%%%%%%%%%%
\begin{theorem}\label{thm:FKhs-1}
Let $(\Sigma, \rho, \mu)$ be a complete space of homogeneous type with a metrically continuous measure $\mu$. Let $x_0 \in \Sigma$, $p \in (1, \infty)$, and $w \in A_{p, \B_{\rho}}$. Then a subset $\K \subset L^p(\Sigma, w)$ is relatively compact in $L^p(\Sigma, w)$ if and only if the following are satisfied: 
\begin{list}{\rm (\theenumi)}{\usecounter{enumi}\leftmargin=1cm \labelwidth=1cm \itemsep=0.1cm \topsep=.2cm \renewcommand{\theenumi}{\alph{enumi}}}

\item\label{list:FK1} $\sup\limits_{f \in \K} \|f\|_{L^p(\Sigma, w)}<\infty$, 

\item\label{list:FK2} $\lim\limits_{A \to \infty}
\sup\limits_{f \in \K} \|f \mathbf{1}_{\Sigma \setminus B(x_0, A)}\|_{L^p(\Sigma, w)}=0$, 

\item\label{list:FK3} $\lim\limits_{r \to 0} \sup\limits_{f \in \K}\|f-f_{B(\cdot,r)}\|_{L^p(\Sigma, w)}=0$.  

\end{list}
\end{theorem}
%%%%%%%%%%%%%%%%%%%%%% THEOREM THEOREM THEOREM %%%%%%%%%%%%%%%%%%%%%%

We extend Theorem \ref{thm:FKhs-1} to the case $0<p \le 1$ as follows.  
%%%%%%%%%%%%%%%%%%%%%% THEOREM THEOREM THEOREM %%%%%%%%%%%%%%%%%%%%%%
\begin{theorem}\label{thm:FKhs-2}
Let $(\Sigma, \rho, \mu)$ be a complete space of homogeneous type with a metrically continuous measure $\mu$. Let $x_0 \in \Sigma$, $0<p<\infty$, and $w \in A_{p_0, \B_{\rho}}$ for some $p_0 \in (1, \infty)$. Then a subset $\K \subseteq L^p(\Sigma, w)$ is relatively compact if and only if the following are satisfied:
\begin{list}{\rm (\theenumi)}{\usecounter{enumi}\leftmargin=1cm \labelwidth=1cm \itemsep=0.1cm \topsep=.2cm \renewcommand{\theenumi}{\roman{enumi}}}

\item\label{list:FKhs-1} $\sup\limits_{f \in \K} \|f\|_{L^p(\Sigma, w)} < \infty$, 

\item\label{list:FKhs-2} $\lim\limits_{A \to \infty} \sup\limits_{f \in \K}\|f \mathbf{1}_{\Sigma \setminus B(x_0, A)}\|_{L^p(\Sigma, w)}=0$, 

\item\label{list:FKhs-3} ${\displaystyle \lim\limits_{r \to 0} \sup\limits_{f \in \K} 
\int_{\Sigma} \bigg(\fint_{B(x, r)} |f(x)-f(y)|^{\frac{p}{p_0}} d\mu(y) \bigg)^{p_0}w(x)d\mu(x)=0}$. 
\end{list}
\end{theorem}
%%%%%%%%%%%%%%%%%%%%%% THEOREM THEOREM THEOREM %%%%%%%%%%%%%%%%%%%%%%

In Sections \ref{sec:CZO}--\ref{sec:mod}, we will show the powerfulness of Theorems \ref{thm:Tb}--\ref{thm:FKhs-2}, which can be used to establish weighted compactness for commutators of $m$-linear or $m$-linearizable operators.

%%%%%%%%%%%%%%%%%%%%% SUBSECTION SUBSECTION SUBSECTION %%%%%%%%%%%%%%%%%%
%%%%%%%%%%%%%%%%%%%%% SUBSECTION SUBSECTION SUBSECTION %%%%%%%%%%%%%%%%%%
\subsection{Historical background} 
The local exponential decay estimates concern the following quantity: 
\begin{align}\label{def:local}
\psi_t(\mathbf{T}, \mathbf{M}) 
:= \sup_{Q} \sup_{\substack{f \in L_c^{\infty}(\Rn) \\ \supp(f) \subset Q}} 
|Q|^{-1} |\{x \in Q: |\mathbf{T}f(x)| > t\, \mathbf{M}f(x)\}|,
\end{align} 
where $\mathbf{T}$ is a singular operator and $\mathbf{M}$ is an appropriate maximal operator. In general, for $\psi_t(\mathbf{T}, \mathbf{M})$, one would like to obtain an exponential/sub-exponential decay with respect with $t$. Local decay estimates accurately reflect the extent that an operator is locally controlled by certain maximal operator, which improves the corresponding good-lambda inequality:
\begin{align}\label{eq:good}
|\{x \in Q: \mathbf{T}f(x) > 2\lambda, \mathbf{M}f(x) \leq \gamma \lambda \}| 
\lesssim e^{-c/\gamma} |Q|.
\end{align} 
We mention that Buckley \cite{B93} used \eqref{eq:good} to pursue the sharp dependence on weights norm for the maximal singular integrals. 

This kind of estimates in \eqref{def:local} first appeared in the work of Karagulyan \cite{Kar02} and was further studied by Ortiz-Caraballo, P\'{e}rez, and Rela \cite{OPR} using a different approach. This problem originated in the control of singular integral operators by certain maximal operators. For example, the Calder\'{o}n-Zygmund operator is controlled by the Hardy-Littlewood maximal operator \cite{CF}; the fractional integral is controlled by the fractional maximal operator \cite{MW}; the Littlewood-Paley square function is controlled by a non-tangential maximal operator \cite{GW}; and the maximal operator is controlled by the sharp maximal operator \cite{FS}. Note that these controls are just norm inequalities as follows:
\begin{equation}\label{eq:Lp}
\|\mathbf{T} f\|_{L^p(\Rn, \, w)} \lesssim \|\mathbf{M}f\|_{L^p(\Rn, \, w)}
\end{equation}
for any $w \in A_{\infty}$ and $0<p<\infty$. Unfortunately, the inequality \eqref{eq:Lp} can not provide us enough information to measure the size of $\mathbf{T}$ and $\mathbf{M}$. At this point, it will bring advantages of local decay estimates \eqref{def:local} into play.

The mixed weak type estimate for an operator $T$ means the following inequality: 
\begin{align}\label{eq:mwsaw}
\bigg\|\frac{T(fv)}{v}\bigg\|_{L^{1,\infty}(\Rn, \, uv)} 
\lesssim \|f\|_{L^1(\Rn, \, uv)},   
\end{align}
where $u$ and $v$ are weights. The estimate \eqref{eq:mwsaw} contains the usual endpoint weighted inequality if one takes $u \in A_1$ and $v \equiv 1$. Such estimates originated in the work of Muckenhoupt and Wheeden \cite{MW}, where the estimate \eqref{eq:mwsaw} with $v^{-1} \in A_1$ and $uv \equiv 1$ was established for the Hardy-Littlewood maximal operator $M$ and the Hilbert transform $H$ on the real line $\R$. Shortly afterward it was improved by Sawyer \cite{Saw} to the case $u, v \in A_1$ but only for $M$ on $\R$. Simultaneously, Sawyer conjectured that \eqref{eq:mwsaw} is true for both $H$ and $M$ in $\Rn$. An affirmative answer to both conjectures was given by Cruz-Uribe, Martell, and P\'{e}rez \cite{CMP} using extrapolation arguments.  Recently, by means of some delicate techniques, Li, Ombrosi, and P\'{e}rez \cite{LOP} extended Sawyer's conjecture for $M$ to the setting of $u \in A_1$ and $v \in A_{\infty}$. Additionally, some multilinear extensions of \eqref{eq:mwsaw} were given in \cite{LOPi} for the multilinear Hardy-Littlewood maximal operators
\begin{align}\label{HLS}
\bigg\|\frac{\mathcal{M}(\vec{f})}{v}\bigg\|_{L^{\frac1m, \infty}(\Rn, \, wv^{\frac1m})} 
\lesssim \prod_{i=1}^m \|f_i\|_{L^1(\Rn, \, w_i)},   
\end{align}
where $\vec{w} \in A_{\vec{1}}$ and $wv^{\frac1m} \in A_{\infty}$ or $\vec{w} \in A_1 \times \cdots \times A_1$ and $v \in A_{\infty}$, in \cite{CXY} for the multilinear pseudo-differential operators, and in \cite{PR} for the multilinear maximal operators in Lorentz spaces.

The sharp weighted norm inequality for an operator $T$ gives an estimate of the form 
\begin{align}\label{Tsharp}
\|T\|_{L^p(\Rn, w) \to L^p(\Rn, w)} 
\leq C_{n, p, T} \, [w]_{A_p}^{\alpha_p(T)}, 
\end{align}
for all $p \in (1, \infty)$ and $w \in A_p$, where the positive constant $C_{n, p, T}$ depends only on $n$, $p$, and $T$, and the exponent $\alpha_p(T)$ is optimal such that \eqref{Tsharp} holds. This kind of estimates gives the exact rate of growth of the weights norm. The first result  was established by Buckley \cite{B93} for the Hardy-Littlewood maximal operator $M$, who proved that \eqref{Tsharp} holds with the best possible exponent $\alpha_p(M) = \frac{1}{p-1}$. The estimate \eqref{Tsharp} for singular integrals attracted a lot of attention due to certain important applications to PDE. For example, obtaining the sharp weighted estimate for the Ahlfors-Beurling operator $B$ with $\alpha_2(B)=1$, Petermichl and Volberg \cite{PV} first settled a long-standing regularity problem for the solution of Beltrami equation on the plane in the critical case. This invites the question that whether \eqref{Tsharp} with $\alpha_2(T)=1$ holds for general Calder\'{o}n-Zygmund operators $T$, which is known as the $A_2$ conjecture. Together with extrapolation with sharp bounds \cite{DGPP}, it immediately implies \eqref{Tsharp} with $\alpha_p(T)=\max\{1, \frac{1}{p-1}\}$.

Hyt\"{o}nen \cite{Hyt} solved the $A_2$ conjecture by proving that an arbitrary Calder\'{o}n-Zygmund operator can be   represented as an average of dyadic shifts over random dyadic systems. Subsequently, Lacey \cite{Lac17} and Lerner \cite{Ler16} independently established a sparse domination for Calder\'{o}n-Zygmund operators to present a shorter proof of the $A_2$ conjecture. Since then, many significant publications came to enrich the literature in this area. For example, pointwise sparse dominations were shown for commutators of Calder\'{o}n-Zygmund operators \cite{LOR},  for bounded oscillation operators \cite{Kar}, and for an arbitrary family of functions \cite{LLO}, the multilinear Calder\'{o}n-Zygmund operators \cite{DHL}, and the multilinear pseudo-differential operators \cite{CXY}.

%%%%%%%%%%%%%%%%%%%%% SUBSECTION SUBSECTION SUBSECTION %%%%%%%%%%%%%%%%%%
%%%%%%%%%%%%%%%%%%%%% SUBSECTION SUBSECTION SUBSECTION %%%%%%%%%%%%%%%%%%
\subsection{Structure of the paper} 
The rest of the paper is organized as follows. In Section \ref{sec:app}, we present some examples of multilinear bounded oscillation operators, and then include a number of applications, which illustrate the utility of our main results. Section \ref{sec:pre} contains some preliminaries including the geometry of measure spaces and the properties of Muckenhoupt weights on measure spaces. Section \ref{sec:BO} is devoted to establishing fundamental estimates for multilinear bounded oscillation operators, which will be used below. After that, in Section \ref{sec:sparse}, we obtain a pointwise sparse domination for multilinear bounded oscillation operators. Using extrapolation techniques, we prove the local exponential decay estimates and mixed weak type estimates in Sections \ref{sec:local} and \ref{sec:weak}, respectively. In Section \ref{sec:sharp}, we obtain  sharp weighted norm inequalities for multilinear sparse operators and then for multilinear bounded oscillation operators and their commutators. Finally, in Section \ref{sec:compact}, we present the proof of extrapolation Theorems \ref{thm:Ap}--\ref{thm:Tb}, which is based on weighted interpolation for multilinear compact operators (Theorem \ref{thm:WMIP-4}), while showing the latter needs weighted Fr\'{e}chet-Kolmogorov Theorems \ref{thm:FKhs-1}--\ref{thm:FKhs-2}.

%%%%%%%%%%%%%%%%%%%%%%%% SECTION SECTION SECTION %%%%%%%%%%%%%%%%%%%%%%
%%%%%%%%%%%%%%%%%%%%%%%% SECTION SECTION SECTION %%%%%%%%%%%%%%%%%%%%%%
\section{Applications}\label{sec:app}
In this section, we present applications of results obtained in the preceding section. We will see that the multilinear bounded oscillation operators formulated in Definition \ref{def:MBO} contain multilinear Hardy-Littlewood maximal operators on measure spaces,  multilinear $\omega$-Calder\'{o}n-Zygmund operators on spaces of homogeneous type, multilinear Littlewood-Paley square operators, multilinear Fourier integral operators, higher order Calder\'{o}n commutators, maximally modulated multilinear singular integrals, and $q$-variation of $\omega$-Calder\'{o}n-Zygmund operators. Furthermore, we will establish weighted boundedness and compactness for them.

%%%%%%%%%%%%%%%%%%%%% SUBSECTION SUBSECTION SUBSECTION %%%%%%%%%%%%%%%%%%
%%%%%%%%%%%%%%%%%%%%% SUBSECTION SUBSECTION SUBSECTION %%%%%%%%%%%%%%%%%%
\subsection{Multilinear Hardy-Littlewood maximal operators}\label{sec:maximal}
Let us recall the definition of $\mathcal{M}_{\B, r}$ in \eqref{def:Mr}.

%%%%%%%%%%%%%%%%%%%%% THEOREM THEOREM THEOREM %%%%%%%%%%%%%%%%%%%%%%%
\begin{theorem}\label{thm:MBO}
Let $1 \le r < \infty$. If $(\Sigma, \mu)$ be a measure space with a ball-basis $\B$, then $\mathcal{M}_{\B, r}$ is a multilinear bounded oscillation operator with respect to $\B$ and $r$. 
\end{theorem}
%%%%%%%%%%%%%%%%%%%%% THEOREM THEOREM THEOREM %%%%%%%%%%%%%%%%%%%%%%%

%%%%%%%%%%%%%%%%%%%%%%%% PROOF PROOF PROOF %%%%%%%%%%%%%%%%%%%%%%%%%
\begin{proof}
Fix $B_0 \in \B$, $x \in B_0$, and set 
\begin{align*}
b := \inf_{B \in \B'} \mu(B),
\quad\text{where } 
\B' := \{B \in \B: B \cap B_0 \neq \emptyset, \mu(B)>\mu(B_0)\}.
\end{align*}
There exists a ball $B_1 \in \B'$ such that $b \leq \mu(B_1) <2b$. Picking $B := B_1^*$, we use the property \eqref{list:B4} to see that 
\begin{align}\label{BBB-1}
B_0 \subsetneq B_1^*=B.
\end{align}
Observe that for any function $f_i\in L^r(\Sigma, \mu)$ and for any $0<\varepsilon< \prod_{i=1}^m \langle f_i \rangle_{B^*, r}$, there exists $B_2 \in \B$ containing $x$ so that 
\begin{align}\label{Maxdelta}
\mathcal{M}_{\B, r}(\vec{f}\mathbf{1}_{B^*})(x) 
\leq \varepsilon + \prod_{i=1}^m \langle f_i  \mathbf{1}_{B^*} \rangle_{B_2, r}. 
\end{align}
If $\mu(B_2) \leq \mu(B_0)$, then the property \eqref{list:B4} and \eqref{BBB-1} imply 
\begin{align}\label{BBB-2}
B_2 \subset B_0^* \subset B^*. 
\end{align}
Hence, 
\begin{equation*}
\mathcal{M}_{\B, r}(\vec{f}\mathbf{1}_{B_0^*})(x)
\ge \prod_{i=1}^m \langle f_i \mathbf{1}_{B_0^*} \rangle_{B_2, r}
=\prod_{i=1}^m \langle f_i \rangle_{B_2, r}, 
\end{equation*}
which along with \eqref{Maxdelta} and \eqref{BBB-2} immediately implies 
\begin{align*}
&|\mathcal{M}_{\B, r}(\vec{f}\mathbf{1}_{B^*})(x) - \mathcal{M}_{\B, r}(\vec{f}\mathbf{1}_{B_0^*})(x)| 
\\
&=\mathcal{M}_{\B, r}(\vec{f}\mathbf{1}_{B^*})(x) - \mathcal{M}_{\B, r}(\vec{f}\mathbf{1}_{B_0^*})(x)
\\
&\leq \varepsilon + \prod_{i=1}^m \langle f_i  \mathbf{1}_{B^*} \rangle_{B_2, r}
- \prod_{i=1}^m \langle f_i \rangle_{B_2, r}
= \varepsilon
\leq \prod_{i=1}^m \langle f_i \rangle_{B^*, r}. 
\end{align*}
If $\mu(B_2) > \mu(B_0)$, then by definition, $B_2 \in \B'$. Thus, $\mu(B_2) \ge b$ and 
\begin{align}\label{BBB-3}
\mu(B^*) 
=\mu(B_1^{**})
\le \C_0^2 \mu(B_1) 
\le 2 \C_0^2 b 
\le 2\C_0^2 \mu(B_2). 
\end{align}
Consequently, we conclude from \eqref{Maxdelta} and \eqref{BBB-3} that 
\begin{align*}
&|\mathcal{M}_{\B, r}(\vec{f}\mathbf{1}_{B^*})(x) - \mathcal{M}_{\B, r}(\vec{f}\mathbf{1}_{B_0^*})(x)| 
\\
&\leq \mathcal{M}_{\B, r}(\vec{f}\mathbf{1}_{B^*})(x) 
\le \varepsilon + \prod_{i=1}^m \langle f_i  \mathbf{1}_{B^*} \rangle_{B_2, r}
\\ 
&\le \prod_{i=1}^m \langle f_i \rangle_{B^*, r} 
+ \prod_{i=1}^m \bigg(\frac{\mu(B^*)}{\mu(B_2)} \fint_{B^*} |f_i|^r \, d\mu \bigg)^{\frac1r}
\lesssim \prod_{i=1}^m \langle f_i \rangle_{B^*, r}. 
\end{align*}
This justifies the condition \eqref{list:T-size}.

To proceed, fix $B \in \B$, $x, x' \in B$, and nonzero functions $f_i \in L^r(\Sigma, \mu)$, $i=1, \ldots, m$. By definition, 
\begin{align}\label{M>M2}
\mathcal{M}_{\B, r}(\vec{f})(x) 
\leq \prod_{i=1}^m \langle f_i \rangle_{A, r} + \prod_{i=1}^m \lfloor f_i \rfloor_{B, r}, 
\end{align}
for some $A \in \B$ containing $x$. If $\mu(A) \leq \mu(B)$, then the property \eqref{list:B4} gives $A\subset B^*$, hence,  
\begin{align}\label{M>M3}
\mathcal{M}_{\B, r}(\vec{f}{\bf 1}_{B^*})(x) 
\ge \prod_{i=1}^m \langle f_i \rangle_{A, r}. 
\end{align}
Gathering \eqref{M>M2} and \eqref{M>M3}, we obtain 
\begin{align*}
&\big|\mathcal{M}_{\B, r}(\vec{f})(x) - \mathcal{M}_{\B, r}(\vec{f}{\bf 1}_{B^*})(x) \big|
\\
&=\mathcal{M}_{\B, r}(\vec{f})(x) - \mathcal{M}_{\B, r}(\vec{f}{\bf 1}_{B^*})(x) 
\leq \prod_{i=1}^m \lfloor f_i \rfloor_{B, r}.
\end{align*}
If $\mu(A)>\mu(B)$, then $B \subset A^*$ and 
\begin{align}\label{M>M4}
\prod_{i=1}^m \langle f_i \rangle_{A, r} 
\lesssim \prod_{i=1}^m \langle f_i \rangle_{A^*, r}
\le \prod_{i=1}^m \lfloor f_i \rfloor_{B, r}. 
\end{align}
We then invoke \eqref{M>M2} and \eqref{M>M4} to deduce 
\begin{align*}
\big|\mathcal{M}_{\B, r}(\vec{f})(x) - \mathcal{M}_{\B, r}(\vec{f}{\bf 1}_{B^*})(x)\big|
\leq \mathcal{M}_{\B, r}(\vec{f})(x)
\\
\leq \prod_{i=1}^m \langle f_i \rangle_{A, r} + \prod_{i=1}^m \lfloor f_i \rfloor_{B, r} 
\lesssim \prod_{i=1}^m \lfloor f_i \rfloor_{B, r}. 
\end{align*}
Hence, we have proved that for any $x \in B$, 
\begin{align*}
\big|\mathcal{M}_{\B, r}(\vec{f})(x) - \mathcal{M}_{\B, r}(\vec{f}{\bf 1}_{B^*})(x)\big|
\lesssim \prod_{i=1}^m \lfloor f_i \rfloor_{B, r}, 
\end{align*}
which immediately implies 
\begin{align*}
\big| \big(\mathcal{M}_{\B, r}(\vec{f}) - \mathcal{M}_{\B, r}(\vec{f}{\bf 1}_{B^*}) \big)(x) 
- \big(\mathcal{M}_{\B, r}(\vec{f}) - &\mathcal{M}_{\B, r}(\vec{f}{\bf 1}_{B^*}) \big)(x') \big|
\lesssim \prod_{i=1}^m \lfloor f_i \rfloor_{B, r}. 
\end{align*}
This shows the condition \eqref{list:T-reg}. Therefore, $\mathcal{M}_{\B, r}$ is a multilinear bounded oscillation operator with respect to $\mathfrak{B}$ and the exponent $r$. 
\end{proof}
%%%%%%%%%%%%%%%%%%%%%%%% END END END PROOF %%%%%%%%%%%%%%%%%%%%%%%%%

%%%%%%%%%%%%%%%%%%%%%%%%% LEMMA LEMMA LEMMA %%%%%%%%%%%%%%%%%%%%%%%
\begin{lemma}\label{lem:M}
Let $(\Sigma, \mu)$ be a measure space with a ball-basis $\B$. Then for any $r \ge 1$, 
\begin{align}
\label{Mr-1} & \|M_{\B, r}\|_{L^r(\Sigma, \mu) \to L^{r, \infty}(\Sigma, \mu)} \le \C_0^{\frac1r}, 
\\
\label{Mr-2} & \|M_{\B, r}\|_{L^p(\Sigma, \mu) \to L^p(\Sigma, \mu)} \le \C_0^{\frac1p},  \quad r<p \le \infty, 
\\
\label{Mr-33} & \|\M^{\otimes}_{\B, r}\|_{L^r(\Sigma, \mu) \times \cdots \times L^r(\Sigma, \mu) \to L^{\frac{r}{m}, \infty}(\Sigma, \mu)} 
\le (\C_0 \, m)^{\frac{m}{r}}, 
\\ 
\label{Mr-3} & \|\M_{\B, r}\|_{L^r(\Sigma, \mu) \times \cdots \times L^r(\Sigma, \mu) \to L^{\frac{r}{m}, \infty}(\Sigma, \mu)} 
\le (\C_0 \, m)^{\frac{m}{r}}. 
\end{align}
\end{lemma} 
%%%%%%%%%%%%%%%%%%%%%%%%% LEMMA LEMMA LEMMA %%%%%%%%%%%%%%%%%%%%%%%

%%%%%%%%%%%%%%%%%%%%%%%%%% PROOF PROOF PROOF %%%%%%%%%%%%%%%%%%%%%%%
\begin{proof}
The inequality \eqref{Mr-1} was shown in \cite[Theorem 4.1]{Kar}. Note that 
\begin{align}\label{Minfty}
\|M_{\B, r}\|_{L^{\infty}(\Sigma, \mu) \to L^{\infty}(\Sigma, \mu)} \le 1. 
\end{align}
Then interpolation between \eqref{Mr-1} and \eqref{Minfty} gives \eqref{Mr-2} as desired. To show \eqref{Mr-33} and  \eqref{Mr-3}, let us recall H\"{o}lder's inequality for weak spaces in \cite[p. 16]{Gra-1}: 
\begin{align}\label{Hol-weak}
\|f_1 \cdots f_m\|_{L^{p, \infty}(\Sigma, \mu)} 
\le p^{-\frac1p} \prod_{i=1}^m p_i^{\frac{1}{p_i}} \|f_i\|_{L^{p_i, \infty}(\Sigma, \mu)}, 
\end{align}
for all $\frac1p=\sum_{i=1}^m \frac{1}{p_i}$ with $0<p_1, \ldots, p_m<\infty$. Thus, it follows from \eqref{Hol-weak} and \eqref{Mr-1} that 
\begin{align*}
\|\M_{\B, r}(\vec{f})\|_{L^{\frac{r}{m}, \infty}(\Sigma, \mu)} 
\le \|\M^{\otimes}_{\B, r}(\vec{f})\|_{L^{\frac{r}{m}, \infty}(\Sigma, \mu)} 
= \bigg\|\prod_{i=1}^m M_{\B, r} f_i \bigg\|_{L^{\frac{r}{m}, \infty}(\Sigma, \mu)} 
\\
\le \Big(\frac{r}{m}\Big)^{-\frac{m}{r}} \prod_{i=1}^m r^{\frac1r} \|M_{\B, r}f_i\|_{L^{r, \infty}(\Sigma, \mu)}
\le m^{\frac{m}{r}} \prod_{i=1}^m \C_0^{\frac1r} \|f_i\|_{L^r(\Sigma, \mu)}. 
\end{align*}
The proof is complete. 
\end{proof}
%%%%%%%%%%%%%%%%%%%%%%%%%% END END END PROOF %%%%%%%%%%%%%%%%%%%%%%%

Invoking Theorem \ref{thm:MBO}, Lemma \ref{lem:M}, and Theorems \ref{thm:T}--\ref{thm:T-Besi} applied to $\mathcal{M}_{\B, r}$, we obtain the quantitative weighted norm inequalities as follows. 

\begin{theorem}\label{Mapp}
If  $(\Sigma, \mu)$ be a measure space with a ball-basis $\B$. Let $1 \le r <\infty$, $\vec{p}=(p_1, \ldots, p_m)$ with $r<p_1, \ldots, p_m < \infty$, and $\vec{w} \in A_{\vec{p}/r, \B}$. Then 
\begin{equation}
\|\mathcal{M}_{\B, r}\|_{L^{p_1}(\Sigma, w_1) \times \cdots \times L^{p_m}(\Sigma, w_m) \to L^p(\Sigma, w)} 
\lesssim \mathcal{N}_1(r, \vec{p}, \vec{w}) [\vec{w}]_{A_{\vec{p}, \B}}^{\max\limits_{1 \le i \le m}\{p, (\frac{p_i}{r})'\}}, 
\end{equation}
where $\frac1{p}=\sum_{i=1}^m \frac1{p_i}$ and $w=\prod_{i=1}^m w_i^{\frac{p}{p_i}}$. If in addition $\B$ satisfies the Besicovitch condition, then  
\begin{equation}
\|\mathcal{M}_{\B, r}\|_{L^{p_1}(\Sigma, w_1) \times \cdots \times L^{p_m}(\Sigma, w_m) \to L^p(\Sigma, w)} 
\lesssim [\vec{w}]_{A_{\vec{p}, \B}}^{\max\limits_{1 \le i \le m}\{p, (\frac{p_i}{r})'\}}. 
\end{equation}
\end{theorem}

Note that Theorem \ref{Mapp} does not give the optimal weighted estimate as in Euclidean spaces \cite[Theorem 1.2]{LMS}. In the current setting with $r=1$, the sharpness occurs only when $p \leq \max\{p'_1, \ldots, p'_m\}$.

%%%%%%%%%%%%%%%%%%%%% SUBSECTION SUBSECTION SUBSECTION %%%%%%%%%%%%%%%%%%
%%%%%%%%%%%%%%%%%%%%% SUBSECTION SUBSECTION SUBSECTION %%%%%%%%%%%%%%%%%%
\subsection{Multilinear $\omega$-Calder\'{o}n-Zygmund operators}\label{sec:CZO}
A {\tt metric} on a set $\Sigma$ is a function $\rho: \Sigma \times  \Sigma \to [0, \infty)$ satisfying the following conditions: 

\begin{enumerate}
\item $\rho(x, y)=\rho(y, x) \ge 0$ for every $x,y\in \Sigma$;

\item $\rho(x,y)=0$ if and only if $x=y$; 

\item $\rho(x,y) \le \rho(x,z) + \rho(z,y)$ for every $x, y, z \in \Sigma$. 
\end{enumerate} 

A set $\Sigma$ endowed with a metric $\rho$ is said to be a metric space $(\Sigma, \rho)$. We then define the ball $B(x, r)$ in $(\Sigma, \rho)$ with center $x$ and radius $r$ as  
\[
B(x,r) := \{y\in \Sigma: \rho(x,y)<r\}, \quad x \in \Sigma, \quad r>0.
\]
Note that the metric defines a topology for which balls form a base, and balls are open sets in this topology. Given a ball $B \subset \Sigma$ we shall denote by $c(B)$ and $r(B)$ respectively its center and its radius. Given $\lambda>0$, we define $\lambda B=B(c(B), \lambda r(B))$. 

 Let $\B_{\rho}$ be the family of all balls in the metric space $(\Sigma, \rho)$. We define also an enlarged family of balls $\B^*_{\rho}$ as follows: $\B^*_{\rho} := \B_{\rho}$ if $\mu(\Sigma)=\infty$, $\B^*_{\rho} := \B_{\rho} \cup\{\Sigma\}$ otherwise. 

We say that $(\Sigma, \rho, \mu)$ is {\tt a space of homogeneous type} if a nonnegative Borel measure $\mu$ on the metric space $(\Sigma, \rho)$ is {\tt doubling}: 
\begin{align}\label{Cu}
0 < \mu(2B) \le C_{\mu} \, \mu(B) \quad\text{for every } B \in \B_{\rho}.
\end{align}
We always let $C_{\mu}$ be the smallest constant for which \eqref{Cu} holds, then the number $D_{\mu} := \log_2 C_{\mu}$ is called the doubling order of $\mu$. By \eqref{Cu}, we have 
\begin{align}\label{Du}
\mu(\lambda B) \le (2\lambda)^{D_{\mu}} \mu(B), \quad\forall \lambda>1, \, B \in \B_{\rho}.
\end{align}
A nonnegative Borel measure $\mu$ on the metric space $(\Sigma, \rho)$ is said to satisfy the {\tt reverse doubling condition} if there exist $D'_{\mu} \in (0, \infty)$ and $C'_{\mu} \in (0, 1]$ such that, for all $x \in \Sigma$, $0 < r < 2\diam(\Sigma)$ and $1 \le \lambda < 2\diam(\Sigma)/r$, 
\begin{align}\label{RD}
\mu(B(x, \lambda r)) \ge C'_{\mu} \, \lambda^{D'_{\mu}}\, \mu(B(x, r)).
\end{align}
The measure $\mu$ is said to be {\tt metrically continuous} if for all $x \in \Sigma$ and $r>0$, 
\begin{align*}
\lim_{y \to x} \mu(B(x, r) \Delta B(y, r)) = 0,
\end{align*}
where $A \Delta B$ stands for the symmetric difference, i.e. $A \Delta B =(A \setminus B) \cup (B \setminus A)$. It was shown in \cite{GG} that given $x \in \Sigma$, the continuity of $r \to \mu(B(x, r))$ implies $\mu(\partial B(x, r))=0$ for each $r>0$, which further gives that $\lim_{y \to x} \mu(B(x, r) \Delta B(y, r)) = 0$. 

For any $\alpha \in (0, 1]$, let $\mathscr{C}^{\alpha}(\Sigma, \mu)$ be the set of all functions $f : \Sigma \to \mathbb{C}$ such that 
\begin{align*}
\|f\|_{\mathscr{C}^{\alpha}(\Sigma, \mu)} 
:= \|f\|_{L^{\infty}(\Sigma, \mu)} + \sup_{x , y \in \Sigma \atop x \neq y} \frac{|f(x) - f(y)|}{\rho(x, y)^{\alpha}}. 
\end{align*}
Define the space 
\begin{align}\label{def:Cba}
\mathscr{C}^{\alpha}_b(\Sigma, \mu) 
:= \{f \in \mathscr{C}^{\alpha}(\Sigma, \mu) : f \text{ has bounded support}\}, 
\end{align}
with the norm $\|\cdot\|_{\mathscr{C}^{\alpha}(\Sigma, \mu)}$. Note that $\mathscr{C}^{\beta}_b(\Sigma, \mu) \subset \mathscr{C}^{\alpha}_b(\Sigma, \mu) \subset L^{\infty}(\Sigma, \mu) \subset \BMO_{\B_{\rho}}$ for any $0< \alpha < \beta \le 1$, and $\mathscr{C}_b^{\alpha}(\Sigma, \mu)$ is dense in $L^p(\Sigma, \mu)$ for any $\alpha \in (0, 1]$ and $p \in [1, \infty)$ (cf. \cite[Corollary 2.11]{HMY}). Recall the definition of $\BMO_{\B}$ in \eqref{def:BMO}. Define $\CMO_{\B_{\rho}}$ as the $\BMO_{\B_{\rho}}$-closure of $\bigcup_{0 < \alpha \le 1} \mathscr{C}_b^{\alpha}(\Sigma, \mu)$.

Let $\omega: [0, \infty) \to [0, \infty)$ be a modulus of continuity, which means that $\omega$ is increasing, subadditive, and $\omega(0)=0$. We say that a modulus of continuity $\omega$ satisfies {\tt the Dini condition} (or, $\omega \in \text{Dini}$) if it verifies 
\begin{align*}
\|\omega\|_{{\rm Dini}} 
:= \int_{0}^{1}\omega(t) \frac{dt}{t}<\infty.
\end{align*}

\begin{definition}\label{def:wCZO}
Let $(\Sigma, \rho, \mu)$ be a space of homogeneous type and $\omega$ be a modulus of continuity. We say that a function $K:\Sigma^{m+1} \setminus \{x=y_1=\cdots=y_m\} \to \mathbb{C}$ is {\tt an $m$-linear $\omega$-Calder\'{o}n-Zygmund kernel}, if there exists a constant $C_K>0$ such that 
\begin{align}
\label{eq:size} |K(x, \vec{y})| 
&\le \frac{C_K}{\big(\sum_{i=1}^m \mu(B(x, \rho(x, y_i)))\big)^m},
\\
\label{eq:smooth-1} |K(x, \vec{y}) - K(x', \vec{y})| 
&\le \frac{C_K \, \w \big(\frac{\rho(x, x')}{\max\limits_{1 \le i \le m} \rho(x, y_i)}\big)}{\big(\sum_{i=1}^m \mu(B(x, \rho(x, y_i)))\big)^m},
\end{align}
whenever $\rho(x, x') \le \frac12 \max\limits_{1 \le i \le m} \rho(x, y_i)$, and for each $i=1, \ldots, m$, 
\begin{align}\label{eq:smooth-2} 
|K(x, \vec{y}) - K(x, \vec{y}')| 
\le \frac{C_K \, \w \big(\frac{\rho(y_i, y'_i)}{\max\limits_{1 \le i \le m} \rho(x, y_i)}\big)}{\big(\sum_{i=1}^m \mu(B(x, \rho(x, y_i)))\big)^m},
\end{align}
where $\vec{y}'=(y_1, \ldots, y'_i, \ldots, y_m)$, whenever $\rho(y_i, y'_i) \le \frac12 \max\limits_{1 \le i \le m} \rho(x, y_i)$. When $\w(t)=t^{\delta}$ with $\delta \in (0, 1]$, $K$ is called an $m$-linear standard Calder\'{o}n-Zygmund kernel.

An $m$-linear operator $T: \mathscr{C}^{\alpha}_b(\Sigma, \mu) \times \cdots \times \mathscr{C}^{\alpha}_b(\Sigma, \mu) \to \mathscr{C}^{\alpha}_b(\Sigma, \mu)'$ is called {\tt an $\omega$-Calder\'{o}n-Zygmund operator} if there exists an $m$-linear $\omega$-Calder\'{o}n-Zygmund kernel $K$ such that 
\begin{align*}
T(\vec{f})(x) =\int_{\Sigma^m} K(x, \vec{y}) f_1(y_1)\cdots f_m(y_m) \, d\mu(\vec{y}),  
\end{align*}
whenever $x \not\in \bigcap_{i=1}^m \supp(f_i)$ and $\vec{f}=(f_1,\ldots,f_m) \in \mathscr{C}^{\alpha}_b(\Sigma, \mu) \times \cdots \times \mathscr{C}^{\alpha}_b(\Sigma, \mu)$, and 
\begin{align}\label{T:assumption}
\text{$T$ is bounded from $L^{q_1}(\Sigma, \mu) \times \cdots \times L^{q_m}(\Sigma, \mu)$ to $L^q(\Sigma, \mu)$} 
\end{align}
for some $\frac1q=\sum_{i=1}^m \frac{1}{q_i}$ with $1<q_1, \ldots, q_m < \infty$. Here and elsewhere, given $\vec{y}=(y_1, \ldots, y_m)$, we simply denote $d\mu(\vec{y}) := d\mu(y_1) \cdots d\mu(y_m)$. 
\end{definition}

The main result of this subsection can be formulated as follows. 
\begin{theorem}\label{thm:CZOBO}
Let $(\Sigma, \rho, \mu)$ be a space of homogeneous type such that $\B_{\rho}$ satisfies the density condition. If $T$ is an $m$-linear $\omega$-Calder\'{o}n-Zygmund operator with $\omega \in \mathrm{Dini}$, then it is a multilinear bounded oscillation operator with respect to the ball-basis $\B^*_{\rho}$ and the exponent $r=1$, with constants $\C_1(T) \lesssim C_K$ and $\C_2(T) \lesssim \|\omega\|_{{\rm Dini}}$. 
\end{theorem}

\begin{proof}
We first note that $\B^*_{\rho}$ is a ball-basis and satisfies the doubling property, which were shown in \cite[Theorem 7.1]{Kar}.  Besides, for any $B=B(x_0, r_0) \in \B_{\rho}$, 
\begin{align}\label{Bhull}
\text{$B^*=B(x_0, R_0)$ with $2r_0 \le R_0 \le \infty$.} 
\end{align}
Let $T$ be an $m$-linear $\omega$-Calder\'{o}n-Zygmund operator with $\omega \in \mathrm{Dini}$.  Next, let us prove that $T$ verifies the conditions \eqref{list:T-size} and  \eqref{list:T-reg}. Take an arbitrary ball $B_0=B(x_0, r_0) \in \B_{\rho}$ with $B_0^{*} \subsetneq \Sigma$. By \eqref{Bhull}, we see that $B_0^*=B(x_0, R_0)$ with $R_0 \ge 2r_0$. Set 
\begin{align*}
B := B(x_0, 2R), 
\quad\text{where}\quad 
R := \sup\{r \ge R_0:\, B(x_0, r)=B(x_0, R_0)\}. 
\end{align*}
Since $B_0^*=B(x_0, R_0) \subsetneq \Sigma$, we have $R<\infty$ and 
\begin{align}\label{BB}
B_0^*=B(x_0, R_0)=B(x_0, R) \subsetneq B(x_0, 2R)=B.
\end{align}
Then it follows from the size condition \eqref{eq:size} that for any $x \in B_0$ and $\vec{y} \in (B^*)^m \setminus (B_0^*)^m$, 
\begin{align}\label{KB}
|K(x, \vec{y})| 
\lesssim \frac{C_K}{\max\limits_{1\le i \le m} \mu(B(x, \rho(x, y_i)))^m} 
\lesssim \frac{C_K}{\mu(B_0)^m}. 
\end{align}
Then, taking into account \eqref{BB} and \eqref{KB} we have that for any $x\in B_0$, 
\begin{align*}
|T(\vec{f}\mathbf{1}_{B^*})(x) - T(\vec{f}\mathbf{1}_{B_0^*})(x)| 
\le \int_{(B^*)^m \setminus (B_0^*)^m} |K(x, \vec{y})| \prod_{i=1}^m |f_i(y_i)| \, d\mu(\vec{y}) 
\\
\lesssim C_K \prod_{i=1}^m \frac{1}{\mu(B_0)} \int_{B^{*}}|f_i(y_i)| \, d\mu(y_i) 
\lesssim C_K \prod_{i=1}^m \fint_{B^{*}}|f_i(y_i)| \, d\mu(y_i), 
\end{align*}
where we have used the doubling property of $\mu$. This shows the condition \eqref{list:T-size} holds with $\C_1(T) \lesssim C_K$, 

By the smoothness condition \eqref{eq:smooth-1}, we have for any $B=B(x_0, r_0) \in \B_{\rho}$ with $B\neq \Sigma$ and for any $x \in B$, 
\begin{align*}
\big| &\big(T(\vec{f}) - T(\vec{f}\mathbf{1}_{B^*}) \big)(x) 
- \big(T(\vec{f}) - T(\vec{f}\mathbf{1}_{B^*}) \big)(x_0) \big|
\\ 
& =\bigg|\int_{\Sigma^m \setminus (B^{*})^m} 
(K(x, \vec{y}) - K(x_0, \vec{y}) ) \prod_{i=1}^m f_i(y_i) \, d\mu(\vec{y}) \bigg|
\\ 
& \leq\sum_{k=0}^{\infty} \int_{(2^{k+1}B^*)^m \setminus (2^k B^*)^m} 
|K(x, \vec{y}) - K(x_0, \vec{y})| \prod_{i=1}^m |f_i(y_i)| \, d\mu(\vec{y}) 
\\
& \lesssim \sum_{k=0}^{\infty} \int_{(2^{k+1}B^*)^m \setminus (2^k B^*)^m} 
\frac{\w \big(\frac{\rho(x, x_0)}{\max\limits_{1 \le i \le m} \rho(x_0, y_i)}\big) 
\prod_{i=1}^m |f_i(y_i)| \, d\mu(y_i)}{\big(\sum_{i=1}^m \mu(B(x_0, \rho(x_0, y_i)))\big)^m} 
\\
& \lesssim\sum_{k=0}^{\infty} \omega(2^{-k-1}) 
\prod_{i=1}^m \fint_{2^{k+1}B^*} |f_i| \, d\mu 
\\
&\leq \sum_{k=1}^{\infty} \omega(2^{-k}) \prod_{i=1}^m \lfloor f_i \rfloor_{B^*}
\lesssim \|\omega\|_{{\rm Dini}} \prod_{i=1}^m \lfloor f_i \rfloor_{B^*}, 
\end{align*}
where we have used $\rho(x, x_0) < r_0$, $\max\limits_{1 \le i \le m} \rho(x_0, y_i) \ge 2^k R_0$ for all $\vec{y} \in (2^{k+1}B^*)^m \setminus (2^k B^*)^m$, and 
\[
\sum_{k=1}^{\infty} \w(2^{-k}) 
\simeq \int_0^1 \w(t) \frac{dt}{t}
= \|\omega\|_{{\rm Dini}} <\infty.
\]  
This implies the condition \eqref{list:T-reg} holds with $\C_2(T) \lesssim \|\omega\|_{{\rm Dini}}$. 
\end{proof}

By \cite[Theorem 3.3]{GLMY}, \eqref{T:assumption} implies that 
\begin{align}\label{Tweak11}
\text{$T$ is bounded from $L^1(\Sigma, \mu) \times \cdots \times L^1(\Sigma, \mu)$ to $L^{\frac1m, \infty}(\Sigma, \mu)$}.
\end{align}
Although one can use Theorem \ref{thm:T} to obtain quantitative weighted norm inequalities, that result is not sharp. To get the optimal weighted bounds, we utilize Theorem \ref{thm:sparse} and Lemma \ref{lem:sparse-AS} below.

%%%%%%%%%%%%%%%%%%%%%%%%%% LEMMA LEMMA LEMMA %%%%%%%%%%%%%%%%%%%%%%
\begin{lemma}\label{lem:sparse-AS}
Let $(\Sigma, \rho, \mu)$ be a space of homogeneous type such that $\B_{\rho}$ satisfies the density condition. Then for all $\vec{p}=(p_1, \dots, p_m)$ with $1<p_1, \ldots, p_m<\infty$ and for all $\vec{w}\in A_{\vec{p}, \B_{\rho}}$, 
\begin{equation*}
\sup_{\S \subset \B_{\rho}: \text{sparse}} 
\|\A_{\S}\|_{L^{p_1}(\Sigma, w_1) \times \cdots \times L^{p_m}(\Sigma, w_m) \rightarrow L^p(\Sigma, w)} 
\lesssim [\vec{w}]_{A_{\vec{p}, \B_{\rho}}}^{\max\{p, p'_1, \ldots, p'_m\}}, 
\end{equation*}
where $\frac{1}{p}=\sum_{i=1}^{m}\frac{1}{p_i}$ and $w=\prod_{i=1}^m w_i^{\frac{p}{p_i}}$. 
\end{lemma}
%%%%%%%%%%%%%%%%%%%%%%%%%% LEMMA LEMMA LEMMA %%%%%%%%%%%%%%%%%%%%%%

%%%%%%%%%%%%%%%%%%%%%%%%%% PROOF PROOF PROOF %%%%%%%%%%%%%%%%%%%%%%%
\begin{proof}
 By \cite[Lemma 4.12]{HK}, for each $B \in \B_{\rho}$ there exist some $k \in \{1, \ldots, K_0\}$ and a dyadic cube $Q_B \in \B_k$ such that 
\begin{align}\label{BQB-1}
B \subset Q_B \subset \lambda_0 B \quad\text{ for some uniform constant } \lambda_0 \ge 1, 
\end{align}
which together with \eqref{Du} gives 
\begin{align}\label{BQB-2}
\mu(B) \le \mu(Q_B) \le (2\lambda_0)^{D_{\mu}} \mu(B).
\end{align}
Given a dyadic system $\B_k$, we set $\S_k := \{Q_B \in \B_k: B \in \S\}$, $k=1, \ldots, K_0$. By definition, \eqref{BQB-1}, and \eqref{BQB-2}, we see that for each $B \in \S$ there exists $E_B \subset B \subset Q_B$ such that $\{E_B\}_{B \in \S}$ is a disjoint family and 
\[
\mu(E_B) \ge \mu(B) \ge \frac{\eta}{(2\lambda_0)^{D_{\mu}}} \mu(Q_B), 
\]
which means that 
\begin{equation}\label{Skk}
\text{$\S_k$ is a sparse family, $k=1, \ldots, K_0$. }
\end{equation}

On the other hand, it follows from \eqref{BQB-1} and \eqref{BQB-2} that 
\begin{align}\label{AShomo-1}
\A_{\S}(\vec{f})
\lesssim \sum_{k=1}^{K_0} \sum_{B \in \S: Q_B \in \B_k} 
\prod_{i=1}^m \langle f_i \rangle_{Q_B} \mathbf{1}_{Q_B} 
=: \sum_{k=1}^{K_0} \A_{\S_k}(\vec{f}).
\end{align}
By the construction of dyadic cubes in \cite[Theorem 2.2]{HK}, it is clear that each $\mathfrak{B}_k$ satisfies the Besicovitch condition with the constant $N_0=1$. From this, \eqref{Skk}, and Lemma \ref{lem:Sparse-Besi} below, we conclude that  
\begin{align}\label{AShomo-2}
\sup_{1 \le k \le K_0} \|\A_{\S_k}(\vec{f})\|_{L^p(\Sigma, w)} 
\lesssim [\vec{w}]_{A_{\vec{p}, \B_{\rho}}}^{\max\{p, p'_1, \ldots, p'_m\} }\prod_{i=1}^m \|f_i\|_{L^{p_i}(\Sigma, w_i)}. 
\end{align}
Therefore, the desired estimate is a consequence of \eqref{AShomo-1} and \eqref{AShomo-2}. 
\end{proof}
%%%%%%%%%%%%%%%%%%%%%%%%%% END END END PROOF %%%%%%%%%%%%%%%%%%%%%%%

Let us recall the sharp reverse H\"{o}lder inequality from \cite{HPR}. 
%%%%%%%%%%%%%%%%%%%%%%%% LEMMA LEMMA LEMMA %%%%%%%%%%%%%%%%%%%%%%%%
\begin{lemma}\label{lem:RH} 
Let $(\Sigma, \rho, \mu)$ be a space of homogeneous type. For every $w \in A''_{\infty, \B_{\rho}}$,   
\begin{align*}
\bigg(\fint_B w^{r_w} d\mu \bigg)^{\frac{1}{r_w}} 
\le 2 \cdot 8^{D_{\mu}} \fint_{2B} w \, d\mu, \quad\forall B \in \B_{\rho}, 
\end{align*}
where $r_w := 1 + \frac{1}{c_{\mu} [w]_{A''_{\infty, \B_{\rho}}}}$.
\end{lemma}
%%%%%%%%%%%%%%%%%%%%%%%% LEMMA LEMMA LEMMA %%%%%%%%%%%%%%%%%%%%%%%%

In the current setting, $A_{1, \B_{\rho}} \subset \bigcup_{s>1} RH_{s, \B_{\rho}}$. 
Now considering Theorem \ref{thm:CZOBO}, \eqref{Tweak11}, and Lemma \ref{lem:sparse-AS}, we use Theorems \ref{thm:local}--\ref{thm:weak} and Theorem \ref{thm:sparse} to conclude the following result.

%%%%%%%%%%%%%%%%%%%%%%% THEOREM THEOREM THEOREM %%%%%%%%%%%%%%%%%%%%%
\begin{theorem}\label{thm:CZOlocal}
Let $(\Sigma, \rho, \mu)$ be a space of homogeneous type such that $\B_{\rho}$ satisfies the density condition. Let $T$ be an $m$-linear $\omega$-Calder\'{o}n-Zygmund operator with $\omega \in \mathrm{Dini}$. Then the following hold: 
\begin{list}{\rm (\theenumi)}{\usecounter{enumi}\leftmargin=1.2cm \labelwidth=1cm \itemsep=0.2cm \topsep=.2cm \renewcommand{\theenumi}{\alph{enumi}}}

\item There exists $\gamma>0$ such that for all $B \in \B_{\rho}$ and $f_i \in L^{\infty}_c(\Sigma, \mu)$ with $\supp(f_i) \subset B$, $1 \leq i \leq m$, 
\begin{align*}
\mu\big(\big\{x \in B:  |T(\vec{f})(x)| > t \, \M_{\B_{\rho}}(\vec{f})(x)  \big\}\big) 
& \lesssim e^{- \gamma t}  \mu(B), \quad t>0.
\end{align*}

\item Let $\vec{w}=(w_1, \ldots, w_m)$ and $w=\prod_{i=1}^m w_i^{\frac1m}$. If $\vec{w} \in A_{\vec{1}, \B_{\rho}}$ and $w v^{\frac1m} \in A_{\infty, \B_{\rho}}$, or $w_1,\ldots, w_m \in A_{1, \B_{\rho}}$ and $v \in A_{\infty, \B_{\rho}}$, then 
\begin{align*}
\bigg\|\frac{T(\vec{f})}{v}\bigg\|_{L^{\frac1m,\infty}(\Sigma, \, w v^{\frac1m})}
&\lesssim \bigg\|\frac{\M_{\B_{\rho}}(\vec{f})}{v}\bigg\|_{L^{\frac1m,\infty}(\Sigma, \, w v^{\frac1m})}.
\end{align*}

\item For all $\vec{p}=(p_1, \ldots, p_m)$ with $1<p_1, \ldots, p_m <\infty$ and for all $\vec{w} \in A_{\vec{p}, \B_{\rho}}$, 
\begin{align*}
\|T\|_{L^{p_1}(\Sigma, w_1) \times \cdots \times L^{p_m}(\Sigma, w_m) \to L^p(\Sigma, w)} 
\lesssim [\vec{w}]_{A_{\vec{p}}}^{\max\{p, p'_1, \ldots, p'_m\}}, 
\end{align*}
where $\frac1p=\sum_{i=1}^m \frac{1}{p_i}$ and $w=\prod_{i=1}^m w_i^{\frac{p}{p_i}}$. 
\end{list}
\end{theorem} 
%%%%%%%%%%%%%%%%%%%%%%% THEOREM THEOREM THEOREM %%%%%%%%%%%%%%%%%%%%%

Next, we use Theorem \ref{thm:CZOlocal} to establish weighted compactness for commutators of m-linear $\omega$-Calder\'{o}n-Zygmund operators. 

%%%%%%%%%%%%%%%%%%%%%%% THEOREM THEOREM THEOREM %%%%%%%%%%%%%%%%%%%%%
\begin{theorem}\label{thm:CZOcom}
Let $(\Sigma, \rho, \mu)$ be a space of homogeneous type such that $\B_{\rho}$ satisfies the density condition. Assume that $\mu$ satisfies the reverse doubling condition. Let $T$ be an m-linear $\omega$-Calder\'{o}n-Zygmund operator with $\omega \in \text{Dini}$. Then for any $b \in \CMO_{\B_{\rho}}$ and for each $j=1,\ldots,m$, $[T, b]_{e_j}$ is compact from $L^{p_1}(\Sigma, w_1) \times \cdots \times L^{p_m}(\Sigma, w_m)$ to $L^p(\Sigma, w)$ for all $\vec{p}=(p_1, \ldots, p_m)$ with $1<p_1,\ldots,p_m<\infty$, for all $\vec{w} \in A_{\vec{p}, \B_{\rho}}$, where $\frac1p=\sum_{i=1}^m \frac{1}{p_i}$ and $w=\prod_{i=1}^m w_i^{\frac{p}{p_i}}$.
\end{theorem}
%%%%%%%%%%%%%%%%%%%%%%% THEOREM THEOREM THEOREM %%%%%%%%%%%%%%%%%%%%%

%%%%%%%%%%%%%%%%%%%%%%%%% PROOF PROOF PROOF %%%%%%%%%%%%%%%%%%%%%%%%
\begin{proof}
Let $\omega \in \text{Dini}$ and $T$ be an m-linear $\omega$-Calder\'{o}n-Zygmund operator. By Theorem \ref{thm:CZOlocal}, 
\begin{align}\label{Tbdd-1}
T \text{ is bounded from $L^{p_1}(\Sigma, w_1) \times \cdots \times L^{p_m}(\Sigma, w_m)$ to $L^p(\Sigma, w)$}, 
\end{align}
for all $\vec{p}=(p_1, \ldots, p_m)$ with $1<p_1, \ldots, p_m<\infty$, and for all $\vec{w} \in A_{\vec{p}, \B_{\rho}}$, where $\frac1p=\sum_{i=1}^m \frac{1}{p_i}$ and $w=\prod_{i=1}^m w_i^{\frac{p}{p_i}}$. Then using Theorem \ref{thm:TTb} and \eqref{Tbdd-1}, we have for all $\vec{p}=(p_1, \ldots, p_m)$ with $1<p_1, \ldots, p_m<\infty$, for all $b \in \BMO_{\B_{\rho}}$, and for each $j=1, \ldots, m$, 
\begin{align}\label{eq:wTBMO}
\|[T, b]_{e_j}\|_{L^p(\Sigma, \mu)} 
\lesssim \|b\|_{\BMO_{\B_{\rho}}} \prod_{i=1}^m \|f_i\|_{L^{p_i}(\Sigma, \mu)}.  
\end{align}
In view of Theorem \ref{thm:Tb} and \eqref{Tbdd-1}, it suffices to demonstrate that 
\begin{align}\label{eq:wTbdd}
[T, b]_{e_j} \text{ is compact from $L^{p_1}(\Sigma, \mu) \times \cdots \times L^{p_m}(\Sigma, \mu)$ to $L^p(\Sigma, \mu)$}, 
\end{align}
for all (or for some) $\frac1p=\sum_{i=1}^m \frac{1}{p_i} < 1$ with $1<p_1,\ldots,p_m<\infty$. 

It remains to prove \eqref{eq:wTbdd}. For convenience, we only present the proof in the case $m=2$ and $j=1$. Fix $b \in \CMO_{\B_{\rho}}$ and $\frac1p= \frac{1}{p_1} + \frac{1}{p_2} < 1$ with $1<p_1, p_2<\infty$. Considering Theorem \ref{thm:FKhs-1}, \eqref{eq:wTBMO}, and the fact that $\mathscr{C}_b^{\alpha}(\Sigma, \mu)$ is dense in $\CMO_{\B_{\rho}}$ for any $0< \alpha \le 1$, we may assume that $b \in \mathscr{C}_b^{\alpha}(\Sigma, \mu)$ with $\supp(b) \subset B(x_0, A_0)$ for some $A_0>0$. Define the maximal truncated bilinear $\omega$-Calder\'{o}n-Zygmund operator by 
\begin{align*}
T_*(\vec{f})(x) 
:=\sup_{\eta>0} \bigg|\int_{\sum_{i=1}^2 \rho(x, y_i)>\eta} 
K(x,\vec{y}) \prod_{i=1}^2 f_i(y_i) \, d\mu(\vec{y}) \bigg|. 
\end{align*}
By \eqref{eq:wTBMO}, one has 
\begin{align}\label{CZOFK-1}
\sup_{\|f_i\|_{L^{p_i}(\Sigma, \mu)} \le 1 \atop i=1, 2} 
\|[T, b]_{e_1}(\vec{f})\|_{L^p(\Sigma, \mu)} 
\lesssim 1. 
\end{align}
Let $A>2 A_0$. We may assume that $\diam(\Sigma)=\infty$, otherwise, $\Sigma \setminus B(x_0, A) = \emptyset$ for sufficiently large $A>0$. Fix $x \in \Sigma \setminus B(x_0, A)$. Then for all $y_1 \in \supp b$, we have $\rho(x, x_0) \ge A > 2 A_0 \ge 2 \rho(y_1, x_0)$, and 
\begin{align}\label{rxy}
\rho(x, y_1) 
\ge \rho(x, x_0) - \rho(y_1, x_0) 
\ge \rho(x, x_0)/2
=: r_0. 
\end{align}
Note that the doubling property of $\mu$ implies the for any $r>0$, 
\begin{align}\label{Bzz}
\mu(B(z, r \rho(z, z_0))) 
\simeq \mu(B(z_0, r \rho(z, z_0))), \quad z, z_0 \in \Sigma. 
\end{align}
We claim that 
\begin{align}
\label{fu-1} &\int_{\supp b} \frac{|f_1(y_1)|}{\mu(B(x, \rho(x, y_1)))} d\mu(y_1) 
\lesssim \|f_1\|_{L^{p_1}(\Sigma, \mu)} \bigg(\frac{\mu(B(x_0, A_0))}{\mu(B(x_0, r_0))^{p'_1}} \bigg)^{\frac{1}{p'_1}}, 
\\
\label{fu-2} &\int_{\Sigma} \frac{|f_2(y_2)|}{\sum_{i=1}^2 \mu(B(x, \rho(x, y_i)))} d\mu(y_2) 
\lesssim \|f_2\|_{L^{p_2}(\Sigma, \mu)} \mu(B(x_0, r_0))^{-\frac{1}{p_2}},  
\\
\label{fu-3} &\int_{\Sigma \setminus B(x_0, A)} \frac{d\mu(x)}{\mu(B(x_0, \rho(x, x_0)))^{\gamma+1}} 
\lesssim \mu(B(x_0, A))^{-\gamma}, \quad \gamma>0. 
\end{align}
Indeed, \eqref{fu-1} is a direct consequence of H\"{o}lder's inequality and \eqref{Bzz}. To get \eqref{fu-2}, we utilize \eqref{rxy} and \eqref{Bzz} to arrive at 
\begin{align}\label{bba}
\mu(B(x, \rho(x, y_1))) 
\ge \mu(B(x, r_0)) 
\ge c_0 \mu(B(x_0, r_0)) 
=: a_0. 
\end{align}
By \eqref{Bzz}, 
\begin{align}\label{HH-1}
H_0 &:= \int_{B(x, r_0)} \frac{d\mu(y_2)}{(a_0 + \mu(B(x, \rho(x, y_2))))^{p'_2}} 
\nonumber \\
&\lesssim \frac{\mu(B(x, r_0))}{\mu(B(x_0, r_0))^{p'_2}} 
\lesssim \frac{1}{\mu(B(x_0, r_0))^{p'_2-1}} 
\end{align}
and by the reverse doubling condition \eqref{RD}, 
\begin{align}\label{HH-2}
H_j &:= \int_{B(x, 2^j r_0) \setminus B(x, 2^{j-1} r_0)} 
\frac{d\mu(y_2)}{(a_0 + \mu(B(x, \rho(x, y_2))))^{p'_2}} 
\nonumber \\
&\lesssim \frac{\mu(B(x, 2^j r_0))}{\mu(B(x, 2^{j-1} r_0))^{p'_2}}
\lesssim \frac{1}{\mu(B(x, 2^j r_0))^{p'_2-1}}
\nonumber \\
&\lesssim \frac{2^{-j (p'_2-1) D'_{\mu}}}{\mu(B(x, r_0))^{p'_2-1}}
\simeq \frac{2^{-j (p'_2-1) D'_{\mu}}}{\mu(B(x_0, r_0))^{p'_2-1}}. 
\end{align}
Hence, the estimates \eqref{bba}--\eqref{HH-2} along with H\"{o}lder's inequality give 
\begin{align*}
&\int_{\Sigma} \frac{|f_2(y_2)|}{\sum_{i=1}^2 \mu(B(x, \rho(x, y_i)))} d\mu(y_2) 
\\
&\le \|f_2\|_{L^{p_2}(\Sigma, \mu)}
\bigg(\int_{\Sigma} \frac{d\mu(y_2)}{(a_0 + \mu(B(x, \rho(x, y_2))))^{p'_2}}  \bigg)^{\frac{1}{p'_2}}
\\
&\le \|f_2\|_{L^{p_2}(\Sigma, \mu)} 
\Big(H_0 + \sum_{j=1}^{\infty} H_j  \Big)^{\frac{1}{p'_2}}
\lesssim \|f_2\|_{L^{p_2}(\Sigma, \mu)} \mu(B(x_0, r_0))^{-\frac{1}{p_2}}. 
\end{align*}
This shows \eqref{fu-2}. Analogously, one can get \eqref{fu-3}. 

Now take $\gamma := p(1+\frac{1}{p_2})-1=p/p'_1>0$. We use size condition \eqref{eq:size} and \eqref{fu-1}--\eqref{fu-3} to obtain 
\begin{align*}
&\|[T, b]_{e_1}(\vec{f}) \mathbf{1}_{\Sigma \setminus B(x_0, A)} \|_{L^p(\Sigma, \mu)} 
\\
&\lesssim \bigg(\int_{\Sigma \setminus B(x_0, A)} 
\frac{\mu(B(x_0, A_0))^{\frac{p}{p'_1}} d\mu(x)}{\mu(B(x_0, \rho(x, x_0)))^{\gamma+1}}  \bigg)^{\frac1p} 
\prod_{i=1}^2 \|f_i\|_{L^{p_i}(\Sigma, \mu)} 
\\
&\lesssim \frac{\mu(B(x_0, A_0))^{\frac{1}{p'_1}}}{\mu(B(x_0, A))^{\gamma/p}} 
\prod_{i=1}^2 \|f_i\|_{L^{p_i}(\Sigma, \mu)} 
\\
&\lesssim \frac{\mu(B(x_0, A_0))^{\frac{1}{p'_1}-\frac{\gamma}{p}}}{(A/A_0)^{\gamma D'_{\mu}/p}} 
\prod_{i=1}^2 \|f_i\|_{L^{p_i}(\Sigma, \mu)}, 
\end{align*}
where \eqref{RD} is used in the last step and the implicit constants are independent of $A$. This proves 
\begin{align}\label{CZOFK-2}
\lim_{A \to \infty} \sup_{\|f_i\|_{L^{p_i}(\Sigma, \mu)} \le 1 \atop i=1, 2} 
\|[T, b]_{e_1}(\vec{f}) \mathbf{1}_{\Sigma \setminus B(x_0, A)} \|_{L^p(\Sigma, \mu)} 
=0. 
\end{align}

Let $\varepsilon\in (0,1)$. Given $\eta>0$ chosen later, $r \in (0, \frac{\eta}{2})$, and $x, x' \in \Sigma$ with $\rho(x, x') <r$, we write  
\begin{align}\label{Tebe}
&[T,b]_{e_1}(\vec{f})(x) - [T,b]_{e_1}(\vec{f})(x')
=: \mathscr{I}_1 + \mathscr{I}_2 + \mathscr{I}_3 + \mathscr{I}_4, 
\end{align}
where
\begin{align*}
\mathscr{I}_1 &: =(b(x)-b(x')) \int_{\sum_{i=1}^2 \rho(x, y_i)> \eta} 
K(x,\vec{y}) \prod_{i=1}^2 f_i(y_i) \, d\mu(\vec{y}), 
\\ 
\mathscr{I}_2 & := \int_{\sum_{i=1}^2 \rho(x, y_i)> \eta} 
(K(x, \vec{y}) - K(x', \vec{y})) (b(x') - b(y_1)) \prod_{i=1}^2 f_i(y_i) \, d\mu(\vec{y}),  
\\ 
\mathscr{I}_3 & := \int_{\sum_{i=1}^2 \rho(x, y_i) \le \eta} 
K(x, \vec{y}) (b(x) - b(y_1)) \prod_{i=1}^2 f_i(y_i) \, d\mu(\vec{y}), 
\\ 
\mathscr{I}_4 & := \int_{\sum_{i=1}^2 \rho(x, y_i) \le \eta} 
K(x', \vec{y}) (b(y_1) - b(x')) \prod_{i=1}^2 f_i(y_i) \, d\mu(\vec{y}). 
\end{align*}
By definition, one has 
\begin{align}\label{Tebe-1}
\mathscr{I}_1 
\lesssim \rho(x, x')^{\alpha} \|b\|_{\mathscr{C}_b^{\alpha}(\Sigma, \mu)} T_*(\vec{f})(x) 
\lesssim \eta^{\alpha} T_*(\vec{f})(x).
\end{align}
The condition $\omega \in \mathrm{Dini}$ implies that there exists $t_0=t_0(\varepsilon) \in (0, 1)$ small enough such that 
\begin{align}\label{wtt}
\int_{0}^{t_0} \omega(t)\frac{dt}{t} < \varepsilon.
\end{align}
Since $\mu$ is doubling, the smoothness condition \eqref{eq:smooth-1} gives that 
\begin{align}\label{Tebe-2}
\mathscr{I}_2 
&\lesssim \int_{\max\limits_{i=1, 2} \rho(x, y_i)>\eta/2}
\frac{\prod_{i=1}^2 |f_i(y_i)|}{(\sum_{i=1}^2 \mu(B(x, \rho(x, y_i))))^2} 
\w\bigg(\frac{r}{\max\limits_{i=1, 2} \rho(x, y_i)}\bigg)d\mu(\vec{y}) 
\nonumber \\
&\lesssim \sum_{k=0}^{\infty} \omega\bigg(\frac{2r}{2^k \eta}\bigg) 
\int_{2^{k-1} \eta < \max\limits_{i=1, 2} \rho(x, y_i) \le 2^k \eta} 
\prod_{i=1}^2 \frac{|f_i(y_i)|}{\mu(B(x, 2^{k-1} \eta))} d\mu(\vec{y})
\nonumber \\
&\lesssim \sum^{\infty}_{k=0}\omega \bigg(\frac{2r}{2^k \eta}\bigg) 
\prod_{i=1}^m \fint_{B(x, 2^k \eta)} |f_i| \, d\mu
\lesssim \int_{0}^{2r/\eta} \w(t) \, \frac{dt}{t} \, \M_{\B_{\rho}}(\vec{f})(x). 
\end{align}
To control $\mathscr{I}_3$, we use the size condition \eqref{eq:size} and the doubling property of $\mu$ to arrive at  
\begin{align}\label{Tebe-3}
\mathscr{I}_3 
& \lesssim \|b\|_{\mathscr{C}_b^{\alpha}(\Sigma, \mu)} \int_{\sum_{i=1}^2 \rho(x, y_i) \le \eta} 
\frac{\rho(x, y_1)^{\alpha} \prod_{i=1}^2 |f_i(y_i)|}{(\sum_{i=1}^2 \mu(B(x, \rho(x, y_i))))^2} d\mu(\vec{y}) 
\nonumber \\ 
&\lesssim \sum_{k=0}^{\infty} \int_{2^{-k-1} \eta \leq \sum_{i=1}^2 \rho(x, y_i) < 2^{-k}\eta} 
\frac{\rho(x, y_1)^{\alpha} \prod_{i=1}^2 |f_i(y_i)|}{(\sum_{i=1}^2 \mu(B(x, \rho(x, y_i))))^2} d\mu(\vec{y}) 
\nonumber \\ 
&\lesssim \sum_{k=0}^{\infty} (2^{-k}\eta)^{\alpha} \prod_{i=1}^2 \frac{1}{\mu(B(x, 2^{-k-1} \eta))} 
\int_{B(x,2^{-k}\eta)}|f_i| \, d\mu 
\nonumber \\ 
&\lesssim \eta^{\alpha} \M_{\B_{\rho}}(\vec{f})(x).
\end{align}
Since $\sum_{i=1}^2 \rho(x, y_i) \le \eta$ implies $\sum_{i=1}^2 \rho(x', y_j) \le 2 \eta$, the same argument as for $\mathscr{I}_3$ leads 
\begin{align}\label{Tebe-4}
\mathscr{I}_4 
\lesssim \eta^{\alpha} \M_{\B_{\rho}}(\vec{f})(x'). 
\end{align}

Now choose $0< \delta < \min\{\varepsilon^{\frac{1}{\alpha}}, \frac{\eta t_0}{2} \}$. Then, for any $0<r<\delta$, there hold $0<r<\frac{\eta}{2}$ and $2r/\eta < t_0$. Collecting the estimates \eqref{Tebe}--\eqref{Tebe-4}, we deduce that for any $0<r<\delta$, 
\begin{align}\label{Tbxx}
\big|&[T, b]_{e_1}(\vec{f})(x) - [T, b]_{e_1}(\vec{f})(x') \big|
\nonumber\\
&\lesssim \varepsilon 
\big[T_*(\vec{f})(x) + \M_{\B_{\rho}}(\vec{f})(x) + \M_{\B_{\rho}}(\vec{f})(x') \big]. 
\end{align}
It follows from Theorem \ref{Mapp} and \cite[Theorem 4.16]{GLMY} that 
\begin{align}\label{Tmax}
\text{$\M_{\B_{\rho}}$ and $T_*$ are bounded from $L^{s_1}(\Sigma, \mu) \times L^{s_2}(\Sigma, \mu)$ to $L^s(\Sigma, \mu)$}, 
\end{align} 
for all $\frac1s = \frac{1}{s_1} + \frac{1}{s_2}$ with $1<s_1, s_2<\infty$. Then \eqref{Tbxx}--\eqref{Tmax} yield  
\begin{align*}
& \|[T, b]_{e_j}(\vec{f}) - ([T, b]_{e_j}(\vec{f}))_{B(\cdot, r)}\|_{L^p(\Sigma, \mu)}
\\
&\le \bigg[\int_{\Sigma} \bigg(\fint_{B(x, r)} |[T,b]_{e_1}(\vec{f})(x) 
- [T,b]_{e_1}(\vec{f})(x')| \, d\mu(x') \bigg)^p d\mu(x)  \bigg]^{\frac1p}
\\
&\lesssim \varepsilon
\big[\|T_*(\vec{f})\|_{L^p(\Sigma, \mu)} 
+ \|\M_{\B_{\rho}}(\vec{f})\|_{L^p(\Sigma, \mu)} 
+ \big\|M_{\B_{\rho}} \big(\M_{\B_{\rho}}(\vec{f}) \big) \big\|_{L^p(\Sigma, \mu)} \big] 
\\
&\lesssim \varepsilon \prod_{i=1}^2 \|f_i\|_{L^{p_i}(\Sigma, \mu)},  
\end{align*}
which shows 
\begin{align}\label{CZOFK-3}
\lim_{r \to 0} \sup_{\|f_i\|_{L^{p_i}(\Sigma, \mu)} \le 1 \atop i=1, 2} 
\|[T, b]_{e_j}(\vec{f}) - ([T, b]_{e_j}(\vec{f}))_{B(\cdot, r)}\|_{L^p(\Sigma, \mu)}
=0. 
\end{align}
Thus, \eqref{eq:wTbdd} is a consequence of \eqref{CZOFK-1}, \eqref{CZOFK-2}, \eqref{CZOFK-3}, and Theorem \ref{thm:FKhs-1}. 
\end{proof} 
%%%%%%%%%%%%%%%%%%%%%%% END END END PROOF %%%%%%%%%%%%%%%%%%%%%%%%%%

Let us next present two examples of $\omega$-Calder\'{o}n-Zygmund operators. 

\begin{example}
Given $k \in \N_+$ and $A \in \mathscr{C}^k(\Rn)$, we consider the singular integral operators
\begin{align*}
T_{A, k} f(x) := {\rm p.v.} \int_{\Rn} \frac{1}{|x-y|^n} \frac{R_k(A; x, y)}{|x-y|^k} f(y) \, dy, 
\end{align*}
where 
\[
R_k(A; x, y) := A(x) - \sum_{|\alpha|<k} \frac{A_{\alpha}(y)}{\alpha!} (x-y)^{\alpha}, \quad 
A_{\alpha}(x) := \partial^{\alpha}_x A(x). 
\]
This kind of operators was introduced by Baj$\check{\rm s}$anski and Coifman \cite{BC} to generalize Calder\'{o}n commutators. 
 
Assume that $A_{\alpha} \in L^{\infty}(\Rn)$ for each $|\alpha|=k$. If we denote the kernel of $T_{A, k}$ by   
\begin{align*}
K_{A, k}(x, y)  := \frac{1}{|x-y|^n} \frac{R_k(A; x, y)}{|x-y|^k}, 
\end{align*}
then it follows from \cite[p. 1671]{DL} that 
\begin{align}\label{KAK-1}
\text{$K_{A, k}$ is a standard Calder\'{o}n-Zygmund kernel.} 
\end{align}
Moreover, if $k$ is odd, then by \cite[Theorem 5.5]{DL},  
\begin{align}\label{KAK-2}
\text{$T_{A, k}$ is bounded from $L^1(\Rn)$ to $L^{1, \infty}(\Rn)$.}
\end{align}
Note that \eqref{KAK-2} verifies \eqref{Tweak11}. Thus, in view of \eqref{KAK-1}, Theorems \ref{thm:CZOlocal}--\ref{thm:CZOcom} can be applied to the operator $T_{A, k}$ with $k$ being odd.  
\end{example}

\begin{example}\label{ex:UR}
Let $\Omega \subset \Rn$ be a uniformly rectifiable domain. Let $\sigma=\mathcal{H}^{n-1}|_{\partial \Omega}$ and $\nu$ denote respectively the surface measure and the geometric measure theoretic outward unit normal to $\Omega$. For these geometric concepts, the reader is referred to \cite{CMM, MMMMM}. We consider the principal-value singular integral operator  $T$ (whenever it exists): 
\begin{align}\label{eq:layer-1}
Tf(x) &:= \lim_{\varepsilon \to 0^{+}} \int_{\substack{y \in \partial \Omega \\ |x-y|>\varepsilon}} 
\langle x-y, \nu(y) \rangle K(x-y) f(y) \, d\sigma(y), 
\end{align}
for any $x \in \partial \Omega$, where $K \in \mathscr{C}^N(\Rn \backslash \{0\})$ is a complex-valued function which is even and positive homogeneous of degree $-n$, and $N=N(n) \in \N$ is large enough.

Let $\w_{n-1}$ denote the area of the unit sphere in $\Rn$. Taking $K(x) := \w_{n-1}^{-1}\,|x|^{-n}$, we see that 
$K \in \mathscr{C}^N(\Rn \backslash \{0\})$ is even and homogeneous of degree $-n$ for any $N \in \N$, and that the operator in \eqref{eq:layer-1} coincides with the harmonic double layer potential: 
\begin{align}\label{eq:layer-double} 
\mathcal{K}_{\Delta}f(x) 
:= \lim_{\varepsilon \to 0^{+}} \frac{1}{\w_{n-1}} 
\int_{\substack{y \in \partial \Omega \\ |x-y|>\varepsilon}} 
\frac{\langle x-y, \nu(y) \rangle}{|x-y|^n} f(y) \, d\sigma(y). 
\end{align}
A particular case is the Riesz transform: 
\begin{align}\label{eq:Rj-f}
\mathcal{R}_j f(x) := \lim_{\varepsilon \to 0^{+}} \frac{1}{\w_{n-1}} 
\int_{\substack{y \in \partial \Omega \\ |x-y|>\varepsilon}} 
\frac{x_j-y_j}{|x-y|^n} f(y) \, d\sigma(y), \quad j=1, \ldots, n.  
\end{align}

Observe that $\nu=(\nu_1, \ldots, \nu_n)$ with $|\nu(x)|=1$ at $\H^{n-1}$-a.e. $x \in \partial_{*}\Omega$. Then we rewrite 
\begin{align}\label{eq:layer-3}
Tf(x) =\sum_{i=1}^n T_j f(x), 
\end{align}
where 
\begin{align*}
T_j f(x) := \lim_{\varepsilon \to 0^{+}} 
\int_{\substack{y \in \partial \Omega \\ |x-y|>\varepsilon}} 
(x_j - y_j) K(x-y) (\nu_j f)(y) \, d\sigma(y). 
\end{align*}
It is easy to check that $K_j(x) := x_j K(x) \in \mathscr{C}^N(\Rn \backslash \{0\})$ is a complex-valued function which is odd and positive homogeneous of degree $1-n$, which implies  
\begin{align}\label{eq:layer-4}
\text{$K_j$ is a standard Calder\'{o}n-Zygmund kernel.} 
\end{align}
It was shown in \cite[Proposition 3.19]{HMT} that 
\begin{align}\label{eq:layer-5}
\text{$T_j$ is bounded from $L^1(\partial \Omega, \sigma)$ to $L^{1, \infty}(\partial \Omega, \sigma)$.} 
\end{align}
Additionally, by definition, 
\begin{align}\label{eq:layer-6}
(\partial \Omega, \sigma) \text{ is a space of homogeneous type}. 
\end{align}
As a consequence of \eqref{eq:layer-3}--\eqref{eq:layer-6}, Theorems \ref{thm:CZOlocal}--\ref{thm:CZOcom} hold for the operator $T$ in \eqref{eq:layer-1}.   
\end{example}

%%%%%%%%%%%%%%%%%%%%% SUBSECTION SUBSECTION SUBSECTION %%%%%%%%%%%%%%%%%%
%%%%%%%%%%%%%%%%%%%%% SUBSECTION SUBSECTION SUBSECTION %%%%%%%%%%%%%%%%%%
\subsection{Multilinear Littlewood-Paley square operators}\label{sec:LP} 
Throughout this subsection, let $\lambda>2m$ and $\Gamma(x) := \{(y, t) \in \R^{n+1}_+: |y-x| < \eta t\}$ for some fixed $\eta \in (0, \infty)$. We define three kinds of {\tt multilinear Littlewood-Paley square operators}: 
\begin{align*}
\mathbf{S}^{(1)}(\vec{f})(x) 
& := \bigg(\int_{0}^{\infty} |T_t(\vec{f})(x)|^2 \frac{dt}{t} \bigg)^{\frac12},
\\
\mathbf{S}^{(2)}(\vec{f})(x) 
& := \bigg(\iint_{\Gamma(x)} \big|T_t(\vec{f})(z) \big|^2 \frac{dz dt}{t^{n+1}} \bigg)^{\frac12},
\\ 
\mathbf{S}^{(3)}(\vec{f})(x)
& := \bigg( \iint_{\R^{n+1}_+} \Big( \frac{t}{t+|x-z|} \Big)^{n\lambda}
\big|T_t(\vec{f})(z) \big|^2 \frac{dz dt}{t^{n+1}} \bigg)^{\frac12}, 
\end{align*} 
where $T_t$ is an $m$-linear operator from $\S(\Rn) \times \cdots \times \S(\Rn)$ to the set of measurable functions on $\Rn$, with the kernel $K_t: \R^{nm} \setminus \{x=y_1=\cdots=y_m\} \to \mathbb{C}$ satisfying
\begin{equation*}
T_t(\vec{f})(x) := \int_{\R^{nm}} K_t (x, \vec{y}) \prod_{i=1}^m f_i (y_i) \, d\vec{y}, \qquad  
x \notin \bigcap_{i=1}^m \supp f_i. 
\end{equation*}

For simplicity, we introduce two Banach spaces by
\begin{align*}
\bB_1 
&:=\bigg\{F: \R_+ \to \R: \|F\|_{\bB_1} 
:= \bigg(\int_0^{\infty} |F(t)|^2 \, \frac{dt}{t} \bigg)^{\frac12} < \infty \bigg\}, 
\\
\bB_2 
& :=\bigg\{G:\R^{n+1}_+ \to \R : \|G\|_{\bB_2} 
:= \bigg(\int_{\R^{n+1}_+} |G(z, t)|^2 \, \frac{dzdt}{t^{n+1}} \bigg)^{\frac12} < \infty \bigg\}. 
\end{align*}
Set $\bB_3 := \bB_2$. Given the function $K_t$ and $(z, t) \in \R^{n+1}_+$, we denote 
\begin{align*}
K_{z, t}^{(2)}(x, \vec{y}) 
& := \mathbf{1}_{\Gamma(0)}(z) K_t(x-z, \vec{y}), 
\\
K_{z, t}^{(3)}(x, \vec{y}) 
& := \bigg(\frac{t}{t+|z|}\bigg)^{n\lambda/2} K_t(x-z, \vec{y}), 
\\
T_{z, t}^{(k)}(\vec{f})(x) 
&:= \int_{\R^{nm}} K_{z, t}^{(k)}(x, \vec{y}) \prod_{i=1}^m f_i(y_i) \, d\vec{y}, \quad k=2, 3. 
\end{align*}
For $ k=2, 3$, writing 
\begin{align}
\K^{(1)}(x, \vec{y}) &:= \big\{K_t(x, \vec{y}) \big\}_{t>0}, 
\\  
\label{def:KK} \K^{(k)}(x, \vec{y}) &:= \big\{K_{z, t}^{(k)}(x, \vec{y}) \big\}_{(z, t) \in \R^{n+1}_+}, 
\\ 
\T^{(1)}(\vec{f})(x) &:= \big\{T_t(\vec{f})(x) \big\}_{t>0}, 
\\
\label{def:TT} \T^{(k)}(\vec{f})(x) &:= \big\{T_{z, t}^{(k)}(\vec{f})(x) \big\}_{(z, t) \in \R^{n+1}_+}, 
\end{align}
we see that for each $k=1, 2, 3$, 
\begin{equation*}
\mathbf{S}^{(k)}(\vec{f})(x)  
= \|\T^{(k)} (\vec{f})(x) \|_{\bB_k}
\quad\text{ and }\quad 
\text{$\K^{(k)}$ is the kernel of $\T^{(k)}$}. 
\end{equation*}

In this subsection we work with $\B$ being the collection of all cubes in $\Rn$ with sides parallel to the coordinate axes. As in Example \ref{example}, we see that $\B$ is a ball-basis of $(\Rn, \Ln)$ with $Q^*=5Q$ for any cube $Q \in \B$. We will simply write $\mathcal{M}_r$, $A_{\vec{p}}$, and $\BMO$ instead of $\mathcal{M}_{\B, r}$, $A_{\vec{p}, \B}$, and $\BMO_{\B}$, respectively. 

%%%%%%%%%%%%%%%%%%%% DEFINITION DEFINITION DEFINITION %%%%%%%%%%%%%%%%%%%%%
\begin{definition}
Given $r \in [1, \infty)$ and $k \in \{1, 2, 3\}$, we say $\K^{(k)}$ satisfies the $\bB_k$-valued multilinear $L^r$-H\"{o}rmander condition, written $\K^{(k)} \in \mathbb{H}_r^{(k)}$, if 
\begin{align}
\label{Lr-size}& \sup_{Q \in \B, x \in \frac12 Q} |Q|^m 
\bigg(\fint_{R_1(Q)} \|\K^{(k)}(x,\vec{y})\|_{\bB_k}^{r'} d\vec{y} \bigg)^{\frac{1}{r'}} 
< \infty, 
\\ 
\label{Lr-smooth} \sup_{Q \in \B \atop x, x' \in \frac12 Q} & \sum_{j=1}^{\infty} |2^j Q|^m 
\bigg(\fint_{R_j(Q)} \|\K^{(k)}(x, \vec{y}) - \K^{(k)}(x', \vec{y}) \|_{\bB_k}^{r'} d\vec{y} \bigg)^{\frac{1}{r'}} < \infty, 
\end{align}
where $R_j(Q) := (2^j Q)^m \backslash (2^{j-1} Q)^m$ and $\fint_{R_j(Q)} := \frac{1}{|2^j Q|^m} \int_{R_j(Q)}$, $j=1, \ldots, m$. When $r=1$, the corresponding integral is understood as $\esssup_{\vec{y} \in R_j(Q)}$. 
\end{definition}
%%%%%%%%%%%%%%%%%%%%%%%% DEFINITION DEFINITION DEFINITION %%%%%%%%%%%%%%%%%

Observe that \eqref{Lr-smooth} coincides with the definition in \cite{CY}. Although we require the condition \eqref{Lr-size}, it can be easily checked for plenty of operators in practice, which will be discussed later.

%%%%%%%%%%%%%%%%%%%%%%%%%% THEOREM THEOREM THEOREM %%%%%%%%%%%%%%%%%%
\begin{theorem}\label{thm:Sk}
Let $r \in [1, \infty)$ and $k \in \{1, 2, 3\}$. Let $\mathbf{S}^{(k)}$ be the multilinear Littlewood-Paley square operators with the kernel $\K^{(k)}$ satisfying the $\bB_k$-valued multilinear $L^r$-H\"{o}rmander condition. Then $\mathbf{S}^{(k)}$ is a $\bB_k$-valued multilinear bounded oscillation operator with respect to $\B$ and the exponent $r$.
\end{theorem}
%%%%%%%%%%%%%%%%%%%%%%%%%% THEOREM THEOREM THEOREM %%%%%%%%%%%%%%%%%%

%%%%%%%%%%%%%%%%%%%%%%%%%%%% PROOF PROOF PROOF %%%%%%%%%%%%%%%%%%%%%
\begin{proof}
Let $Q_0 \in \B$ and $x \in Q_0$. Picking $Q=2Q_0$, we see that 
\begin{align}\label{xyy}
\sum_{i=1}^m |x-y_i| \simeq \ell(Q^*) \quad\text{ for all }\quad \vec{y} \in (Q^*)^m \setminus (Q_0^*)^m. 
\end{align}
Then, in light of \eqref{Lr-size} and \eqref{xyy}, we use Minkowski's and H\"{o}lder's inequalities to deduce that 
\begin{align*}
\big\| &\T^{(k)}(\vec{f} \mathbf{1}_{Q^*})(x) - \T^{(k)}(\vec{f} \mathbf{1}_{Q_0^*})(x)\big\|_{\bB_k} 
\\
&=\bigg\|\int_{(Q^*)^m \setminus (Q_0^*)^m} \K^{(k)}(x, \vec{y}) \prod_{i=1}^m f_i(y_i) \, d\vec{y} \bigg\|_{\bB_k}
\\
&\le \int_{(Q^*)^m \setminus (Q_0^*)^m} \|\K^{(k)}(x, \vec{y}) \|_{\bB_k} \prod_{i=1}^m |f_i(y_i)| \, d\vec{y} 
\\
&\le \bigg(\int_{(Q^*)^m \setminus (Q_0^*)^m} \|\K^{(k)}(x, \vec{y})\|_{\bB_k}^{r'}  d\vec{y} \bigg)^{\frac{1}{r'}} 
\bigg(\int_{(Q^*)^m} \prod_{i=1}^m |f_i(y_i)|^r \, d\vec{y}  \bigg)^{\frac1r} 
\\
&\lesssim \prod_{i=1}^m \bigg(\fint_{Q^*} |f_i(y_i)|^r \, dy_i\bigg)^{\frac1r}, 
\end{align*}
which shows the condition \eqref{list:T-size}. 

To proceed, fix $Q \in \B$ and $x, x' \in Q$. It follows from Minkowski's inequality and the condition \eqref{Lr-smooth} that
\begin{align*}
\big\| &\big(\T^{(k)}(\vec{f}) - \T^{(k)}(\vec{f} \mathbf{1}_{Q^*}) \big)(x) 
- \big(\T^{(k)}(\vec{f}) - \T^{(k)}(\vec{f} \mathbf{1}_{Q^*}) \big)(x') \big\|_{\bB_k} 
\\
&= \bigg\| \int_{(\Rn)^m \setminus (5Q)^m}
(\K^{(k)}(x, \vec{y}) - \K^{(k)}(x', \vec{y})) \prod_{i=1}^m f_i(y_i) \, d\vec{y} \bigg\|_{\bB_k} 
\\ 
& \leq \sum_{j=1}^{\infty} \int_{R_j(Q)}
\|\K^{(k)}(x, \vec{y}) - \K^{(k)}(x', \vec{y}) \|_{\bB_k} \prod_{i=1}^m |f_i(y_i)| \, d\vec{y}
\\ 
& \leq \sum_{j=0}^{\infty} |2^j Q|^{\frac{m}{r}}
\bigg(\int_{R_j(Q)} \|\K^{(k)}(x, \vec{y}) - \K^{(k)}(x', \vec{y}) \|_{\bB_k}^{r'} d\vec{y} \bigg)^{\frac{1}{r'}}
\\ 
&\qquad\quad \times
\bigg(\frac{1}{|2^j Q|^{m}} \int_{(2^j Q)^m} \prod_{i=1}^m |f_i(y_i)|^r \, d\vec{y} \bigg)^{\frac{1}{r}}
\\ 
& \lesssim \sup_{j \ge 0} \prod_{i=1}^m \bigg(\fint_{2^j Q} |f_i(y_i)|^r dy_i \bigg)^{\frac1r}
\le \prod_{i=1}^m \lfloor f_i \rfloor_{Q, r}. 
\end{align*}
This shows the condition \eqref{list:T-reg}. 
\end{proof}
%%%%%%%%%%%%%%%%%%%%%%%%%%%% END END END PROOF %%%%%%%%%%%%%%%%%%%%%

Note that in the current scenario, by \eqref{OSC}, $\osc_{\exp L}=\BMO$ with equivalent norms, and \eqref{cube} gives that $A_{\infty}$ satisfies the sharp reverse H\"{o}lder property. Then the following conclusion is a consequence of Theorems \ref{thm:local}--\ref{thm:T-Besi}, \ref{thm:Sk} and \eqref{HLS}. 

%%%%%%%%%%%%%%%%%%%%%%%%%% THEOREM THEOREM THEOREM %%%%%%%%%%%%%%%%%%
\begin{theorem}\label{thm:SSS}
Let $1 \le r <\infty$ and $k \in \{1, 2, 3\}$. Let $\mathbf{S}^{(k)}$ be the multilinear Littlewood-Paley square operators with the kernel $\K^{(k)}$ satisfying the $\bB_k$-valued multilinear $L^r$-H\"{o}rmander condition. Assume that $\mathbf{S}^{(k)}$ is bounded from $L^r(\Rn) \times \cdots \times L^r(\Rn)$ to $L^{\frac{r}{m}, \infty}(\Rn)$. Then the following hold: 
\begin{list}{\rm (\theenumi)}{\usecounter{enumi}\leftmargin=1.2cm \labelwidth=1cm \itemsep=0.2cm \topsep=.2cm \renewcommand{\theenumi}{\alph{enumi}}}

\item There exists $\gamma>0$ such that for all cubes $Q \subset \Rn$ and $f_i \in L^{\infty}_c(\Rn)$ with $\supp(f_i) \subset Q$, $1 \leq i \leq m$, 
\begin{align*}
\big|\big\{x \in Q:  |\mathbf{S}^{(k)}(\vec{f})(x)| > t \, \M_r(\vec{f})(x)  \big\}\big| 
& \lesssim \, e^{- \alpha t}  |Q|, \quad t>0, 
\end{align*}
and for any $\alpha \in \{0, 1\}^m$, 
\begin{align*}
\big|\big\{x \in Q:  |[\mathbf{S}^{(k)}, \b]_{\alpha}(\vec{f})(x)| > t \, \M_r(\vec{f^*})(x)  \big\}\big|
\lesssim e^{-(\frac{\gamma t}{\|\b\|_{\tau}})^{\frac{1}{|\tau|+1}}} |Q|,  
\end{align*}
for all $t>0$, where $\tau=\{i: \alpha_i \neq 0\}$, $\|\b\|_{\tau} = \prod_{i \in \tau} \|b_i\|_{\BMO}$, and $f_i^*=M^{\lfloor r \rfloor}(|f_i|^r)^{\frac1r}$, $i=1, \ldots, m$.

\item Let $\vec{w}=(w_1, \ldots, w_m)$ and $w=\prod_{i=1}^m w_i^{\frac1m}$. If $\vec{w} \in A_{\vec{1}}$ and $w v^{\frac{r}{m}} \in A_{\infty}$, or $\vec{w} \in A_1 \times \cdots \times A_1$ and $v^r \in A_{\infty}$, then 
\begin{align*}
\bigg\|\frac{\mathbf{S}^{(k)}(\vec{f})}{v}\bigg\|_{L^{\frac{r}{m},\infty}(\Rn, \, w v^{\frac{r}{m}})}
&\lesssim \prod_{i=1}^m \|f_i\|_{L^r(\Rn, \, w_i)}.
\end{align*}

\item For all $\vec{p}=(p_1, \ldots, p_m)$ with $r<p_1, \ldots, p_m <\infty$, for all $\vec{w} \in A_{\vec{p}/r}$, and for any $\alpha \in \{0, 1\}^m$, 
\begin{align*}
\|\mathbf{S}^{(k)}\|_{L^{p_1}(\Rn, w_1) \times \cdots \times L^{p_m}(\Rn, w_m) \to L^p(\Rn, w)} 
&\lesssim [\vec{w}]_{A_{\vec{p}}}^{\max\limits_{1 \le i \le m}\{p, (\frac{p_i}{r})'\}}, 
\end{align*}
and 
\begin{align*}
\|[\mathbf{S}^{(k)}, \b]_{\alpha}&\|_{L^{p_1}(\Rn, w_1) \times \cdots \times L^{p_m}(\Rn, w_m) \to L^p(\Rn, w)} 
\lesssim \|\b\|_{\tau} [\vec{w}]_{A_{\vec{p}}}^{(|\tau|+1)\max\limits_{1 \le i \le m}\{p, (\frac{p_i}{r})'\}}, 
\end{align*}
where $\frac1p=\sum_{i=1}^m \frac{1}{p_i}$, $w=\prod_{i=1}^m w_i^{\frac{p}{p_i}}$, $\tau=\{i: \alpha_i \neq 0\}$, and $\|\b\|_{\tau} = \prod_{i \in \tau} \|b_i\|_{\BMO}$. 
\end{list}
\end{theorem}
%%%%%%%%%%%%%%%%%%%%%%%%%% THEOREM THEOREM THEOREM %%%%%%%%%%%%%%%%%%

Next, let us see how Theorem \ref{thm:SSS} recovers many known results. 

%%%%%%%%%%%%%%%%%%%%%%%% DEFINITION DEFINITION DEFINITION %%%%%%%%%%%%%%%%%
\begin{definition}
Let $\{K_t(x,\vec{y})\}_{t>0}$ be a collection of locally integrable functions defined on $(\mathbb{R}^n)^{m+1}$ off the diagonal $x = y_1= \dots = y_m$. Define $\K^{(k)}$ as in \eqref{def:KK}, $k=1, 2, 3$.
\begin{enumerate}
\item We say that $\K^{(1)}$ satisfies the multilinear Littlewood-Paley condition, written $\K^{(1)} \in \mathbb{K}_1$, if for all $t>0$ and for some $\gamma, \delta>0$,
\begin{align*}
|K_t(x,\vec{y})| &
\lesssim \frac{t^{\delta}}{(t+\sum_{j=1}^{m}|x-y_{j}|)^{mn+\delta}},
\\ 
|K_t(x,\vec{y}) - K_t(x',\vec{y})|
& \lesssim \frac{t^{\delta} |x-x'|^{\gamma}}{(t+\sum_{j=1}^{m}|x-y_{j}|)^{mn+\delta+\gamma}},
\end{align*}
whenever
$|x-x'| \leq \frac12 \max\limits_{1\leq j \leq m}|x-y_{j}|$, and for all $1 \leq i \leq m$, 
\begin{equation*}
|K_t(x,\vec{y}) - K_t(x, \vec{y}')|
\lesssim \frac{t^{\delta} |y_i-y_i'|^{\gamma}}{(t+\sum_{j=1}^{m}|x-y_{j}|)^{mn+\delta+\gamma}},
\end{equation*}
where $\vec{y}'=(y_1, \ldots, y'_i, \ldots, y_m)$, whenever $|y_i-y_i'| \leq \frac12 \max\limits_{1 \le j \le m} |x-y_j|$.

\item We say $\K^{(k)}$ satisfies the integral condition of Calder\'{o}n-Zygmund type, written $\K^{(k)} \in \mathbb{K}^{(k)}_2$, if for some  $\gamma>0$, 
\begin{align*}
\|\K^{(k)}(x,\vec{y})\|_{\bB_k} 
& \lesssim \frac{1}{(\sum_{j=1}^{m}|x-y_{j}|)^{mn}},
\\ 
\|\K^{(k)}(x,\vec{y}) - \K^{(k)}(x',\vec{y})\|_{\bB_k}
& \lesssim \frac{|x-x'|^{\gamma}}{(\sum_{j=1}^{m}|x-y_{j}|)^{mn+\gamma}},
\end{align*}
whenever $|x-x'| \leq \frac12 \max\limits_{1 \leq j \leq m} \{|x-y_j|\}$, and for all $1 \leq i \leq m$, 
\begin{align*}
\|\K^{(k)}(x,\vec{y}) - \K^{(k)}(x, \vec{y}')\|_{\bB_k}
\lesssim \frac{|y_i-y_i'|^{\gamma}}{(\sum_{j=1}^{m} |x-y_{j}|)^{mn+\gamma}}, 
\end{align*}
where $\vec{y}'=(y_1, \ldots, y'_i, \ldots, y_m)$, whenever $|y_i-y_i'| \leq \frac12 \max\limits_{1 \le j \le m}|x-y_j|$.
\end{enumerate} 
\end{definition}
%%%%%%%%%%%%%%%%%%%%%%%% DEFINITION DEFINITION DEFINITION %%%%%%%%%%%%%%%%%

The class $\mathbb{K}_1$ contains the classical multilinear Littlewood-Paley kernels, which were studied by many authors (cf. \cite{CHIRSY, ChXY, SXY, XPY}). Afterwards it was improved by Xue and Yan \cite{XY} to conditions $\mathbb{K}_2^{(1)}$ and $\mathbb{K}_2^{(3)}$, while $\mathbb{K}_2^{(2)}$ is new in this paper.

%%%%%%%%%%%%%%%%%%%%%%%% LEMMA LEMMA LEMMA %%%%%%%%%%%%%%%%%%%%%%%%
\begin{lemma}\label{lem:KK}
For any $1 \leq r_1 < r_2 < \infty$, 
\begin{align*}
\K^{(1)} \in \mathbb{K}_1 \, \Longrightarrow \, \K^{(k)} \in \mathbb{K}_2^{(k)},  
\quad\text{ and }\quad 
\mathbb{K}_2^{(k)} \subset \mathbb{H}_{r_1}^{(k)} \subset \mathbb{H}_{r_2}^{(k)}, 
\quad k=1, 2, 3.
\end{align*}
\end{lemma}
%%%%%%%%%%%%%%%%%%%%%%%% LEMMA LEMMA LEMMA %%%%%%%%%%%%%%%%%%%%%%%%

%%%%%%%%%%%%%%%%%%%%%%%%% PROOF PROOF PROOF %%%%%%%%%%%%%%%%%%%%%%%%
\begin{proof} 
The implication $\K^{(1)} \in \mathbb{K}_1 \Longrightarrow \K^{(k)} \in \mathbb{K}_2^{(k)}$, $k=1, 3$, was proved in \cite[Theorems 1.4 and 1.7]{XY}, while $\mathbb{H}_{r_1}^{(k)} \subsetneq \mathbb{H}_{r_2}^{(k)}$ was given in \cite[Proposition 6.4]{CY}, and $\mathbb{K}_2^{(k)} \subset \mathbb{H}_1^{(k)}$ is trivial. Hence, it suffices to prove 
$\K^{(1)} \in \mathbb{K}_1 \Longrightarrow \K^{(2)} \in \mathbb{K}_2^{(2)}$. 

Let $\K^{(1)}=\{K_t\}_{t>0} \in \mathbb{K}_1$. Note that for all $z \in \Gamma(x)$, 
\[
t_0 := \sum_{j=1}^m |x-y_j| 
\le \sum_{j=1}^m (|x-z| + |z-y_j|)
\lesssim t + \sum_{j=1}^m |z-y_j|. 
\]
Then, by the regularity condition for $\K^{(1)}$, 
\begin{align*}
\|&\K^{(2)}(x,\vec{y}) - \K^{(2)}(x',\vec{y})\|_{\bB_k}^2 
\\
&=\int_{\Gamma(0)} |K_t(x'-z, \vec{y}) - K_t(x-z, \vec{y})|^2 \frac{dzdt}{t^{n+1}}
\\
&= \int_{\Gamma(x)} |K_t(x'-x + z, \vec{y}) - K_t(z, \vec{y})|^2 \frac{dzdt}{t^{n+1}}
\\
&\lesssim \int_{\Gamma(x)} \frac{t^{2\delta} |x-x'|^{2\gamma}}{(t+\sum_{j=1}^m |z-y_j|)^{2(mn+\delta+\gamma)}} \frac{dzdt}{t^{n+1}}
\\
&\lesssim \int_0^{t_0} \int_{B(x, \eta t)} \frac{t^{2\delta} |x-x'|^{2\gamma}}{t_0^{2(mn+\delta+\gamma)}} \frac{dz dt}{t^{n+1}} 
+ \int_{t_0}^{\infty} \int_{B(x, \eta t)} \frac{|x-x'|^{2\gamma}}{t^{2(mn + \gamma)}} \frac{dz dt}{t^{n+1}} 
\\
&\lesssim \frac{|x-x'|^{2\gamma}}{t_0^{2(mn+\gamma)}}
=\frac{|x-x'|^{2\gamma}}{(\sum_{j=1}^m |x-y_j|)^{2(mn+\gamma)}}. 
\end{align*}
Much as we did above, it is easy to show the size estimate and the regularity estimate in variables $y_i$ for $\K^{(2)}$. 
\end{proof} 
%%%%%%%%%%%%%%%%%%%%%%%%% END END END PROOF %%%%%%%%%%%%%%%%%%%%%%%%

In light of Lemma \ref{lem:KK}, Theorem \ref{thm:SSS} covers the results in \cite{CXY, ChXY, SXY, XPY}. Let us present  another example. 

\begin{example}
Define the bilinear square Fourier multiplier operator as 
\begin{align*}
\mathbf{S}^{\m}(f_1, f_2) (x) 
:= \bigg(\int_0^{\infty} \bigg|\int_{\R^{2n}} e^{2\pi i x \cdot (\xi_1 + \xi_2)} 
\m(t \vec{\xi}) \widehat{f}_1(\xi_1) \widehat{f}_2(\xi_2) \, d\xi_1 d\xi_2 \bigg|^2 \frac{dt}{t} \bigg)^{\frac12}, 
\end{align*}
where $\m \in L^{\infty}(\R^{2n})$ satisfies 
\begin{align*}
|\partial^{\alpha} \m(\xi_1, \xi_2)| 
\lesssim \frac{(|\xi_1| + |\xi_2|)^{-|\alpha|+\varepsilon_1}}{(1 + |\xi_1| + |\xi_2|)^{\varepsilon_1+\varepsilon_2}}
\end{align*}
for all $|\alpha| \le s$ for some $s \in \N_+$, and for some $\varepsilon_1, \varepsilon_2 >0$. 

Setting
\begin{align*}
&K_t^{\m}(x, y_1, y_2) 
:= \frac{1}{t^{2n}} \check{\m} \bigg(\frac{x-y_1}{t}, \frac{x-y_2}{t} \bigg), 
\\
&T_t^{\m}(f_1, f_2)(x) 
:= \int_{\R^{2n}} K_t^{\m}(x, y_1, y_2) f_1(y_1) f_2(y_2) \, dy_1 dy_2, 
\\
&\K^{\m} :=\{K_t^{\m}\}_{t>0}, \quad\text{ and }\quad 
\T^{\m} := \{T_t^{\m}\}_{t>0}, 
\end{align*}
we rewrite 
\begin{align*}
\mathbf{S}^{\m}(f_1, f_2)(x) = \|\T^{\m}(f_1, f_2)(x)\|_{\bB_1}. 
\end{align*}
Additionally, it follows from \cite[Propositions 3.1--3.2]{SiXY} that 
\begin{align*}
\text{$\K^{\m}$ satisfies the $\bB_1$-valued bilinear $L^r$-H\"{o}rmander condition}, 
\end{align*}
whenever $s \in [n+1, 2n]$ and $2n/s<r \le 2$. Assume in addition that 
\begin{align*}
\text{$\mathbf{S}^{\m}$ is bounded from $L^r(\Rn) \times L^r(\Rn)$ to $L^{\frac{r}{2}, \infty}(\Rn)$}.
\end{align*}
See \cite[Section 2]{SiXY} for the reasonability of this assumption. Consequently, Theorem \ref{thm:SSS} holds for $\mathbf{S}^{\m}$. 
\end{example}

Beyond that, we obtain weighted compactness for commutators of $\mathbf{S}^{(k)}$ as follows. 

%%%%%%%%%%%%%%%%%%%%%%% THEOREM THEOREM THEOREM %%%%%%%%%%%%%%%%%%%%%
\begin{theorem}\label{thm:LP-com}
Let $k \in \{1, 2, 3\}$ and $\mathbf{S}^{(k)}$ be the multilinear Littlewood-Paley square operators with the kernel $\K^{(k)} \in \mathbb{K}_2^{(k)}$. Assume that $\mathbf{S}^{(k)}$ is bounded from $L^{q_1}(\Rn) \times \cdots \times L^{q_m}(\Rn)$ to $L^q(\Rn)$ for some $\frac1q = \sum_{i=1}^m \frac{1}{q_i}$ with $1<q_1, \ldots, q_m <\infty$. If $b \in \CMO(\Rn)$, then for each $j=1,\ldots,m$, $[\mathbf{S}^{(k)}, b]_{e_j}$ is compact from $L^{p_1}(\Rn, w_1) \times \cdots \times L^{p_m}(\Rn, w_m)$ to $L^p(\Rn, w)$ for all $\vec{p}=(p_1, \ldots, p_m)$ with $1<p_1, \ldots, p_m < \infty$, and for all $\vec{w} \in A_{\vec{p}}$, where $\frac1p = \sum_{i=1}^m \frac{1}{p_i}$ and $w=\prod_{i=1}^m w_i^{\frac{p}{p_i}}$.
\end{theorem}
%%%%%%%%%%%%%%%%%%%%%%% THEOREM THEOREM THEOREM %%%%%%%%%%%%%%%%%%%%%

%%%%%%%%%%%%%%%%%%%%%%% PROOF PROOF PROOF %%%%%%%%%%%%%%%%%%%%%%%%%%
\begin{proof}
It was shown in \cite{XY} that the boundedness of $\mathbf{S}^{(k)}$ from $L^{q_1}(\Rn) \times \cdots \times L^{q_m}(\Rn)$ to $L^q(\Rn)$ implies that $\mathbf{S}^{(k)}$ is bounded from $L^1(\Rn) \times \cdots \times L^1(\Rn)$ to $L^{\frac{1}{m}, \infty}(\Rn)$. Let $b \in \CMO(\Rn)$. Then it follows from Theorem \ref{thm:SSS} and Lemma \ref{lem:KK} that 
\begin{align}\label{Sk-1}
\|\mathbf{S}^{(k)}\|_{L^{p_1}(\Rn, w_1) \times \cdots \times L^{p_m}(\Rn, w_m) \to L^p(\Rn, w)} 
&\lesssim [\vec{w}]_{A_{\vec{p}}}^{\max\limits_{1 \le i \le m}\{p, p'_i\}}, 
\end{align}
and 
\begin{align}\label{Sk-2}
\|[\mathbf{S}^{(k)}, &b]_{e_j}\|_{L^{p_1}(\Rn, w_1) \times \cdots \times L^{p_m}(\Rn, w_m) \to L^p(\Rn, w)} 
\lesssim \|b\|_{\BMO} [\vec{w}]_{A_{\vec{p}}}^{2\max\limits_{1 \le i \le m}\{p, p'_i\}}, 
\end{align}
for all $\vec{p}=(p_1, \ldots, p_m)$ with $1<p_1,\ldots,p_m<\infty$, and for all $\vec{w} \in A_{\vec{p}}$, where $\frac1p =\sum_{i=1}^m \frac{1}{p_i}$ and $w=\prod_{i=1}^m w_i^{\frac{p}{p_i}}$. In view of Theorem \ref{thm:Tb} and \eqref{Sk-1}, it suffices to show 
\begin{align}\label{Sk-3}
[\mathbf{S}^{(k)}, b]_{e_j} \text{ is compact from $L^{p_1}(\Rn) \times \cdots \times L^{p_m}(\Rn)$ to $L^p(\Rn)$}, 
\end{align}
for all (or for some) $\frac1p = \sum_{i=1}^m \frac{1}{p_i} < 1$ with $1<p_1, \ldots, p_m<\infty$. 

By definition, it is easy to verify that for any $0< \alpha < \beta <1$, 
\begin{align}\label{CcCa}
\mathscr{C}_c^{\infty}(\Rn) 
\subset \mathscr{C}_b^1(\Rn)
&\subset \mathscr{C}_b^{\beta}(\Rn)
\subset \mathscr{C}_b^{\alpha}(\Rn)
\nonumber \\
&\subset \mathscr{C}_c(\Rn) 
\subset L^{\infty}(\Rn) 
\subset \BMO(\Rn). 
\end{align}
It was shown in \cite[eq. (3.16)]{CY22} that the $\BMO(\Rn)$-closure of $\mathscr{C}_c^{\infty}(\Rn)$ equals the $\BMO(\Rn)$-closure of $\mathscr{C}_c(\Rn)$. Recall that $\CMO(\Rn)$ is the $\BMO(\Rn)$-closure of $\bigcup\limits_{0< \alpha \le 1} \mathscr{C}_b^{\alpha}(\Rn)$. Then, this together with \eqref{CcCa} implies that for any $0 < \alpha \le 1$, 
\begin{align}\label{CCC}
\CMO(\Rn)
=\overline{\mathscr{C}_b^{\alpha}(\Rn)}^{\BMO} 
=\overline{\mathscr{C}_c^{\infty}(\Rn)}^{\BMO} 
=\overline{\mathscr{C}_c(\Rn)}^{\BMO}.
\end{align}

Fix $\frac1p = \sum_{i=1}^m \frac{1}{p_i} < 1$ with $1<p_1, \ldots, p_m<\infty$. Applying Theorem \ref{thm:FKhs-1}, \eqref{Sk-2}, and \eqref{CCC}, we are reduced to showing that for any $b\in \mathscr{C}_c^{\infty }(\Rn)$ and $\varepsilon \in (0, 1)$, the following hold:
\begin{itemize}
\item there exists $A=A(\varepsilon)>0$ independent of $\vec{f}$ such that 
\begin{align}\label{SFK-1}
\big\|[\mathbf{S}^{(k)}, b]_{e_j}(\vec{f})\mathbf{1}_{\{|x|>A\}} \big\|_{L^p(\Rn)} 
\lesssim \varepsilon \prod_{i=1}^m \|f_i\|_{L^{p_i}(\Rn)}. 
\end{align}

\item there exists $\delta=\delta(\varepsilon)$ independent of $\vec{f}$ such that for all $0<r<\delta$, 
\begin{align}\label{SFK-2}
\big\|[\mathbf{S}^{(k)}, b]_{e_j}(\vec{f}) 
- \big([\mathbf{S}^{(k)}, b]_{e_j}(\vec{f}) \big)_{B(\cdot, r)} \big\|_{L^p(\Rn)} 
\lesssim \varepsilon \prod_{i=1}^m \|f_i\|_{L^{p_i}(\Rn)}. 
\end{align}
\end{itemize} 

We only focus on the case $j=1$. Let $\varepsilon \in (0,1)$ and $b \in \mathscr{C}_c^{\infty}(\Rn)$ with $\supp b \subset B(0, A_0)$ for some $A_0 \ge 1$. Then for any $|x|>A \ge 2A_0$, the size condition of $\K^{(k)}$ and \eqref{kf-3} below give 
\begin{align}\label{Skbe}
&[\mathbf{S}^{(k)}, b]_{e_1}(\vec{f})(x) 
=\|[\T^{(k)}, b]_{e_1}(\vec{f})(x)\|_{\bB_k} 
\nonumber \\
&=\bigg\| \int_{\R^{nm}} (b(x) - b(y_1)) \K^{(k)}(x, \vec{y}) \prod_{i=1}^m f_i(y_i) \, d\vec{y}\bigg\|_{\bB_k}
\nonumber \\
&\le \int_{\R^{nm}} |b(x) - b(y_1)| \|\K^{(k)}(x, \vec{y})\|_{\bB_k} \prod_{i=1}^m |f_i(y_i)| \, d\vec{y}
\nonumber \\
&\lesssim \|b\|_{L^{\infty}(\Rn)} \int_{\R^{nm}} \frac{\prod_{i=1}^m |f_i(y_i)|}{(\sum_{i=1}^m |x-y_i|)^{mn}} d\vec{y}
\nonumber \\
&\lesssim |x|^{-n} \prod_{i=1}^m \|f_i\|_{L^{p_i}(\Rn)}, 
\end{align}
which along with $A > \max\{2A_0, \varepsilon^{-p'/n}\}$ implies \eqref{SFK-1} as desired. 

Let $\eta>0$ be chosen later and $0<r<\eta$. Fix $x \in \Rn$ and $x' \in B(x, r)$. We split
\begin{align*}
&[\T^{(k)}, b]_{e_1}(\vec{f})(x) - [\T^{(k)}, b]_{e_1}(\vec{f})(x') 
\\ \nonumber
&=(b(x)-b(x')) \int_{\sum_{i=1}^m |x-y_i|>\eta} \K^{(k)}(x,\vec{y}) \prod_{i=1}^m f_i (y_i) \, d\vec{y}
\\ \nonumber
&\quad+ \int_{\sum_{i=1}^m |x-y_i|>\eta} (b(x') - b(y_1)) 
(\K^{(k)}(x,\vec{y}) - \K^{(k)}(x',\vec{y}))  \prod_{i=1}^m f_i (y_i) \, d\vec{y}
\\ \nonumber
&\quad+ \int_{\sum_{i=1}^m |x-y_i|\le \eta} (b(x)- b(y_1)) \K^{(k)}(x,\vec{y}) \prod_{i=1}^m f_i (y_i) \, d\vec{y} 
\\ \nonumber
&\quad+ \int_{\sum_{i=1}^m |x-y_i| \le \eta} (b(y_1) - b(x')) \K^{(k)}(x', \vec{y}) \prod_{i=1}^m f_i (y_i) \, d\vec{y}. 
\end{align*}
Then by Minkowski's inequality, 
\begin{align}\label{Tkb}
\big\|[\T^{(k)}, b]_{e_1}(\vec{f})(x) - [\T^{(k)}, b]_{e_1}(\vec{f})(x')\big\|_{\bB_k}  
\lesssim \mathscr{J}_1 + \mathscr{J}_2 + \mathscr{J}_3 + \mathscr{J}_4, 
\end{align}
where 
\begin{align*}
\mathscr{J}_1 & = r \, T_*^{(k)}(\vec{f})(x) 
:= r \, \sup_{\eta > 0} \bigg\|\int_{\sum_{i=1}^m |x-y_i|>\eta} 
\K^{(k)}(x,\vec{y}) \prod_{i=1}^m f_i(y_i) \, d\vec{y}\bigg\|_{\bB_k}, 
\\
\mathscr{J}_2 & := \int_{\sum_{i=1}^m |x-y_i|>\eta} 
\|K(x,\vec{y}) - K(x',\vec{y})\|_{\bB_k} \prod_{i=1}^m |f_i(y_i)| \, d\vec{y}, 
\\
\mathscr{J}_3 & := \int_{\sum_{i=1}^m |x-y_i|\le \eta} 
|x - y_1| \|K(x,\vec{y})\|_{\bB_k} \prod_{i=1}^m |f_i(y_i)| \, d\vec{y}, 
\\
\mathscr{J}_4 & := \int_{\sum_{i=1}^m |x-y_i| \le \eta} 
|x' - y_1| \|K(x', \vec{y})\|_{\bB_k} \prod_{i=1}^m |f_i(y_i)| \, d\vec{y}. 
\end{align*}
Note that 
\begin{align}\label{Tkb-1}
\text{$T_*^{(k)}$ is bounded from $L^{r_1}(\Rn) \times \cdots \times L^{r_m}(\Rn)$ to $L^r(\Rn)$}, 
\end{align}
for all $\frac1r=\sum_{i=1}^m \frac{1}{r_i}$ with $1<r_1,\ldots,r_m<\infty$. For $\mathscr{J}_2$, the smoothness condition of $\K^{(k)}$ and \eqref{kf-1} give that 
\begin{align}\label{Tkb-2}
\mathscr{J}_2 
\lesssim r^{\gamma} \int_{\sum_{i=1}^m |x-y_i|>\eta} 
\frac{\prod_{i=1}^m |f_i(y_i)|}{(\sum_{i=1}^m |x-y_i|)^{mn+\gamma}} \, d\vec{y}
\lesssim (r/\eta)^{\gamma} \M(\vec{f})(x). 
\end{align}
To control $\mathscr{J}_3$, we use the size condition and \eqref{kf-2}: 
\begin{align}\label{Tkb-3}
\mathscr{J}_3 & \lesssim \int_{\sum_{i=1}^m |x-y_i|<\eta} 
\frac{\prod_{i=1}^m |f_i(y_i)|}{(\sum_{i=1}^m |x-y_i|)^{mn-1}} d\vec{y}
\lesssim \eta \, \M(\vec{f})(x).
\end{align}
Since $\sum_{i=1}^m |x-y_i| \le \eta$ implies $\sum_{i=1}^m |x'-y_i| \le \eta+mr$, the same argument as $\mathscr{J}_3$ leads 
\begin{align}\label{Tkb-4}
\mathscr{J}_4 \lesssim (\eta+mr) \M(\vec{f})(x') 
\lesssim \eta \, \M(\vec{f})(x'). 
\end{align}

Observing that 
\begin{align*}
&\big|[\mathbf{S}^{(k)}, b]_{e_j}(\vec{f})(x) 
- \big([\mathbf{S}^{(k)}, b]_{e_j}(\vec{f}) \big)_{B(x, r)} \big| 
\\
&=\big| \big\|[\T^{(k)}, b]_{e_1}(\vec{f})\big\|_{\bB_k} 
- \big( \big\|[\T^{(k)}, b]_{e_1}(\vec{f})\big\|_{\bB_k}\big)_{B(\cdot, r)} \big| 
\\
&\le \fint_{B(x, r)} \big\| [\T^{(k)}, b]_{e_1}(\vec{f})(x) 
- [\T^{(k)}, b]_{e_1}(\vec{f})(x') \big\|_{\bB_k} \, dx', 
\end{align*}
and 
\begin{align}\label{tauM}
\bigg\|\fint_{B(\cdot, r)} \M(\vec{f})(x') \, dx' \bigg\|_{L^p(\Rn)} 
\le \|M(\M(\vec{f}))\|_{L^p(\Rn)} 
\lesssim \|\M(\vec{f})\|_{L^p(\Rn)}, 
\end{align}
we invoke Theorem \ref{Mapp} and \eqref{Tkb}--\eqref{tauM} to deduce 
\begin{align*}
&\big\|[\mathbf{S}^{(k)}, b]_{e_j}(\vec{f}) 
- \big([\mathbf{S}^{(k)}, b]_{e_j}(\vec{f}) \big)_{B(\cdot, r)} \big\|_{L^p(\Rn)} 
\\
&\lesssim [\eta + (r/\eta)^{\gamma}] 
\big(\|T_*^{(k)}(\vec{f})\|_{L^p(\Rn)} + \|\M(\vec{f})\|_{L^p(\Rn)} \big) 
\\
&\lesssim \varepsilon \prod_{i=1}^m \|f_i\|_{L^{p_i}(\Rn)}, 
\end{align*}
provided that $\eta :=\varepsilon$ and $0< r< \delta :=\varepsilon^{1+1/\gamma}$. This shows \eqref{SFK-2} and completes the proof.  
\end{proof} 
%%%%%%%%%%%%%%%%%%%%%%% END END END PROOF %%%%%%%%%%%%%%%%%%%%%%%%%%

%%%%%%%%%%%%%%%%%%%%%%% LEMMA LEMMA LEMMA %%%%%%%%%%%%%%%%%%%%%%%%%
\begin{lemma}\label{lem:kf}
Let $\eta, \gamma>0$ and $f_i \in L^{p_i}(\Rn)$, $i=1, \ldots, m$. Then, 
\begin{align}
\label{kf-1}
&\int_{\sum_{i=1}^m |x-y_i| > \eta} \frac{\prod_{i=1}^m |f_i(y_i)|}{(\sum_{i=1}^m |x-y_i|)^{mn+\gamma}} d\vec{y}
\lesssim \eta^{-\gamma} \M(\vec{f})(x), 
\\
\label{kf-2}
&\int_{\sum_{i=1}^m |x-y_i| \le \eta} \frac{\prod_{i=1}^m |f_i(y_i)|}{(\sum_{i=1}^m |x-y_i|)^{mn-\gamma}} d\vec{y}
\lesssim \eta^{\gamma} \, \M(\vec{f})(x).  
\end{align}
If we assume in addition that $\supp(f_1) \subset B(0, A)$ for some $A \ge 1$, then for all $|x|>2A$, 
\begin{align}\label{kf-3}
\int_{\R^{nm}} \frac{\prod_{i=1}^m |f_i(y_i)|}{(\sum_{i=1}^m |x-y_i|)^{mn}} d\vec{y}
\lesssim |x|^{-n} A^{\frac{n}{p'_1}} \prod_{i=1}^m \|f_i\|_{L^{p_i}(\Rn)}. 
\end{align}
\end{lemma}
%%%%%%%%%%%%%%%%%%%%%%% LEMMA LEMMA LEMMA %%%%%%%%%%%%%%%%%%%%%%%%%

%%%%%%%%%%%%%%%%%%%%%%% PROOF PROOF PROOF %%%%%%%%%%%%%%%%%%%%%%%%%%
\begin{proof}
The inequalities \eqref{kf-1} and \eqref{kf-2} can be obtained by splitting the region into annular subregions. To get \eqref{kf-3}, we note that 
\begin{align}\label{aa-1}
(a_1 \cdots a_m)^{\frac1m} 
\le (a_1 + \cdots + a_m)/m 
\quad\text{ for all } a_1, \ldots, a_m \ge 0.
\end{align} 
Observe that for all $|x|>2A$ and $|y_1| \le A$, $|x-y_1| \ge |x| - |y_1| \ge \max\{1, |x|/2\}$ and 
\begin{align}\label{aa-2}
\sum_{i=1}^m |x-y_i| 
\ge \frac12 |x-y_1| + \frac12 \sum_{i=1}^m |x-y_i| 
\ge \frac12 \Big(1+ \sum_{i=1}^m |x-y_i|\Big). 
\end{align}
Hence, we use \eqref{aa-1}--\eqref{aa-2} and H\"{o}lder's inequality to deduce that 
\begin{align*}
&\int_{\R^{nm}} \frac{\prod_{i=1}^m |f_i(y_i)|}{(\sum_{i=1}^m |x-y_i|)^{mn}} d\vec{y}
\lesssim \int_{\R^{nm}} \prod_{i=1}^m \frac{|f_i(y_i)|}{(1 + |x-y_i|)^n} d\vec{y}
\\
&\lesssim \prod_{i=1}^m \|f_i\|_{L^{p_i}(\Rn)} 
\bigg(\int_{\Rn}  \frac{\mathbf{1}_{B(0, A)}(y_1) \, dy_i}{(1 + |x-y_i|)^{np'_i}} \bigg)^{\frac{1}{p'_i}}
\lesssim |x|^{-n} A^{\frac{n}{p'_1}} \prod_{i=1}^m \|f_i\|_{L^{p_i}(\Rn)}. 
\end{align*}
This shows \eqref{kf-3}. 
\end{proof}
%%%%%%%%%%%%%%%%%%%%%%% END END END PROOF %%%%%%%%%%%%%%%%%%%%%%%%%%

%%%%%%%%%%%%%%%%%%%%% SUBSECTION SUBSECTION SUBSECTION %%%%%%%%%%%%%%%%%%
%%%%%%%%%%%%%%%%%%%%% SUBSECTION SUBSECTION SUBSECTION %%%%%%%%%%%%%%%%%%
\subsection{Multilinear Fourier integral operators} 
Given a function $a$ on $\Rn \times \R^{nm}$, the $m$-linear Fourier integral operator $T_a$ is defined by
\begin{align*}
T_a(\vec{f})(x) := \int_{\R^{nm}} a(x, \vec{\xi})
e^{2\pi i x \cdot (\xi_1+\cdots+\xi_m)} \widehat{f}_1(\xi_1) \cdots \widehat{f}_m(\xi_m) \, d\vec{\xi},
\end{align*}
for all $f_i \in \S(\Rn)$, $i=1,\ldots,m$, where $\widehat{f}$ is the Fourier transform of $f$. If $a(x, \vec{\xi}) \equiv \m(\vec{\xi})$ for all $x \in \Rn$, we write $T_a=T_{\m}$.

Given $\tau \in \R$ and $\rho, \delta \in [0, 1]$, we say $a \in \mathcal{S}_{\rho,\delta}^{\tau}(\Rn \times \R^{nm})$ if for each triple of multi-indices $\alpha$ and $\beta_1,\ldots,\beta_m$ there exists a constant $C_{\alpha, \vec{\beta}}$ such that
\begin{align*}
\big| \partial_{x}^{\alpha} \partial_{\xi_1}^{\beta_1} \cdots
\partial_{\xi_m}^{\beta_m}  a(x,\vec{\xi}) \big|
\leq C_{\alpha, \vec{\beta}} (1+ |\vec{\xi}|)^{\tau - \rho \sum_{j=1}^m |\beta_j| + \delta|\alpha|}, 
\qquad x \in \Rn.
\end{align*}
Let $\Phi \in \S(\R^{nm})$ satisfy $\supp(\Phi) \subset \{\vec{\xi} \in \R^{nm}: \frac12 \le |\vec{\xi}| \le 2\}$ and $\sum_{j \in \Z} \Phi(2^{-j}\vec{\xi})=1$ for each $\vec{\xi} \in \R^{nm} \setminus\{0\}$. Denote 
\begin{align}\label{def:mj}
\m_j(\vec{\xi}) :=\Phi(\vec{\xi}) \m(2^{-j} \vec{\xi}), \quad j \in \Z.
\end{align} 
Given $s \in \R$, set 
\begin{align*}
\mathcal{H}^s(\R^{nm}) 
:= \big\{\m \in L^{\infty}(\R^{nm}): 
\|\m\|_{\mathcal{H}^s(\R^{nm})} := \sup_{j \in \Z} \|\m_j\|_{H^s(\R^{nm})}<\infty\big\}, 
\end{align*}
where the Sobolev norm is given by $\|f\|_{H^s(\R^{nm})} := \|(I-\Delta)^{s/2}f\|_{L^2(\R^{nm})}$. 

Throughout this subsection we let $\Sigma=\Rn$, $\mu=\Ln$, and $\B$ be the collection of all cubes in $\Rn$ with sides parallel to the coordinate axes. Then $\B$ is a ball-basis of $(\Rn, \Ln)$ with $Q^*=5Q$ for any cube $Q \in \B$. We will drop the subscript $\B$ in all notation.

%%%%%%%%%%%%%%%%%%%%%%% THEOREM THEOREM THEOREM %%%%%%%%%%%%%%%%%%%%%
\begin{theorem}\label{thm:Ta}
Let $\rho, \delta \in [0, 1]$, $\tau < -mn(1-\rho)$, and $a \in \mathcal{S}^{\tau}_{\rho, \delta}(\Rn \times \R^{nm})$. Then $T_a$ is a multilinear bounded oscillation operator with respect to $\B$ and the exponent $r=1$.
\end{theorem}
%%%%%%%%%%%%%%%%%%%%%%% THEOREM THEOREM THEOREM %%%%%%%%%%%%%%%%%%%%%

%%%%%%%%%%%%%%%%%%%%%%%%% PROOF PROOF PROOF %%%%%%%%%%%%%%%%%%%%%%%%
\begin{proof} 
Set
\begin{equation*}
K_a(x,\vec y) 
:= \int_{\R^{nm}} a(x,\vec\xi)\, e^{2\pi i\vec{y} \cdot \vec{\xi}} \, d\vec\xi.
\end{equation*}
Then, $K_a(x, x-\vec{y})$ is the kernel of $T_a$, and for any multi-index $\alpha$, 
\begin{equation*}
\vec{y}^\alpha K_a(x,\vec y)
=C \int_{\R^{nn}} \partial^\alpha_{\vec\xi} a(x,\vec\xi)\,
e^{2\pi i\vec{y}\cdot \vec{\xi}} \, d\vec{\xi}.
\end{equation*}
From this and the condition $\tau<-mn(1-\rho)$, we obtain 
\begin{equation*}
|\vec y|^{mn}|K_a(x, \vec{y})| 
\lesssim \int_{\R^{mn}}(1+|\vec{\xi}|)^{\tau - \rho mn} \, d\vec{\xi} < \infty,
\end{equation*}
which in turn implies  
\begin{equation}\label{Ksig}
|K_a(x, x-\vec{y})|
\lesssim \frac{1}{(\sum_{i=1}^m |x-y_i|)^{mn}}.
\end{equation}
Hence, for any $Q_0 \in \B$ and $x \in Q_0$, we take $Q=2Q_0$ and use \eqref{Ksig} to deduce 
\begin{align}\label{Tsig}
\big|&T_a(\vec{f} \mathbf{1}_{Q^*})(x) - T_a(\vec{f} \mathbf{1}_{Q_0^*})(x)\big|
\nonumber \\
&=\bigg|\int_{(Q^*)^m \setminus (Q_0^*)^m} K(x, x -\vec{y}) \prod_{i=1}^m f_i(y_i) \, d\vec{y}\bigg|
\nonumber \\
&\lesssim \int_{(Q^*)^m \setminus (Q_0^*)^m} \frac{\prod_{i=1}^m |f_i(y_i)|}{(\sum_{i=1}^m |x-y_i|)^{mn}} \, d\vec{y} 
\lesssim \prod_{i=1}^m \fint_{Q^*} |f_i(y_i)| \, dy_i. 
\end{align}
That is, $T_a$ satisfies the condition \eqref{list:T-size}. Checking the proof of \cite[(4.2)]{CXY}, one can see that the condition \eqref{list:T-reg} is justified for $T_a$. 
\end{proof}
%%%%%%%%%%%%%%%%%%%%%%%%% END END NED PROOF %%%%%%%%%%%%%%%%%%%%%%%%

%%%%%%%%%%%%%%%%%%%%%%% THEOREM THEOREM THEOREM %%%%%%%%%%%%%%%%%%%%%
\begin{theorem}\label{thm:Tm}
Let $s \in (mn/2, mn]$, $\m \in \mathcal{H}^s(\R^{nm})$, and $mn/s < r <\min\{2, mn\}$. Then $T_{\m}$ is a multilinear bounded oscillation operator with respect to $\B$ and the exponent $r$.
\end{theorem}
%%%%%%%%%%%%%%%%%%%%%%% THEOREM THEOREM THEOREM %%%%%%%%%%%%%%%%%%%%%

%%%%%%%%%%%%%%%%%%%%%%%%% PROOF PROOF PROOF %%%%%%%%%%%%%%%%%%%%%%%%
\begin{proof} 
Following the proof of \cite[(2.2)]{LiS}, we see that $T_{\m}$ satisfies the condition \eqref{list:T-reg}. To check the condition \eqref{list:T-size}, we fix $Q_0 \in \B$ and $x \in Q_0$, then take $Q=2Q_0$. Defined $\m_j$ as in \eqref{def:mj} for each $j \in \Z$. Then 
\begin{align}\label{Tmm}
&|T_{\m}(\vec{f} \mathbf{1}_{Q^*})(x) - T_{\m}(\vec{f} \mathbf{1}_{Q_0^*})(x)| 
\nonumber \\
&\le \sum_{j \in \Z} |T_{\m_j}(\vec{f} \mathbf{1}_{Q^*})(x) - T_{\m_j}(\vec{f} \mathbf{1}_{Q_0^*})(x)| 
\nonumber \\
&\le \sum_{j \in \Z} \int_{(Q^*)^m \setminus (Q_0^*)^m} |\check{m}_j(x-\vec{y})| \prod_{i=1}^m |f_i(y_i)| \, d\vec{y}
\nonumber \\
&\le \sum_{j \in \Z} \bigg(\int_{(Q^*)^m \setminus (Q_0^*)^m} |\check{m}_j(x-\vec{y})|^{r'} d\vec{y} \bigg)^{\frac{1}{r'}}  \bigg(\int_{(Q^*)^m} \prod_{i=1}^m |f_i(y_i)|^r \, d\vec{y} \bigg)^{\frac1r}. 
\end{align}
Setting $\widetilde{Q} := x -Q^*$ and $\widetilde{Q}_0 := x -Q_0^*$, we have 
\begin{align}\label{mmj}
I_j 
&:= \bigg(\int_{(Q^*)^m \setminus (Q_0^*)^m} 
|\check{m}_j(x-\vec{y})|^{r'} d\vec{y} \bigg)^{\frac{1}{r'}}
\nonumber \\
&= \bigg(\int_{\widetilde{Q}^m \setminus \widetilde{Q}_0^m} 
|\check{m}_j(\vec{y})|^{r'} d\vec{y} \bigg)^{\frac{1}{r'}}
\nonumber \\
&\le \bigg(\int_{c_1\ell(Q) \le |\vec{y}| \le c_2 \ell(Q)} 
|\check{m}_j(\vec{y})|^{r'} d\vec{y} \bigg)^{\frac{1}{r'}}
\nonumber \\
&\lesssim 2^{j(mn/r - s)} \ell(Q)^{-s} \bigg(\int_{\R^{nm}} (1+|\vec{y}|^2)^{sr'/2} 
|2^{-jmn} \check{m}_j(2^{-j}\vec{y})|^{r'} d\vec{y} \bigg)^{\frac{1}{r'}}
\nonumber \\
&\lesssim 2^{j(mn/r - s)} \ell(Q)^{-s} \|\m_j\|_{H^s(\R^{nm})},  
\end{align}
where we have used \cite[Lemma 3.3]{Tom} for $r'>2$ in the last step. Choose $j_0 \in \Z$ satisfying $2^{-j_0} \le \ell(Q) < 2^{-j_0+1}$. We then use \eqref{mmj} and $mn/r<s$ to obtain 
\begin{align}\label{mmj-1}
\sum_{j \ge j_0} I_j 
&\lesssim \sup_{j \in \Z} \|\m_j\|_{H^s(\R^{nm})} \ell(Q)^{-s} \sum_{j \ge j_0} 2^{j(mn/r-s)} 
\nonumber \\
&\lesssim \|\m\|_{\mathcal{H}^s(\R^{nm})} \ell(Q)^{-mn/r}. 
\end{align}
On the other hand, \eqref{mmj} applied to $s=1$ yields  
\begin{align}\label{mmj-2}
\sum_{j < j_0} I_j 
&\lesssim \sum_{j < j_0} \ell(Q)^{-1} 2^{j(mn/r-1)} \|\m_j\|_{H^1(\R^{nm})}
\nonumber \\
&\lesssim \sum_{j < j_0} \ell(Q)^{-1} 2^{j(mn/r-1)} \|\m_j\|_{H^s(\R^{nm})}
\nonumber \\
&\lesssim  \sup_{j \in \Z} \|\m_j\|_{H^s(\R^{nm})} \ell(Q)^{-1} \sum_{j < j_0} 2^{j(mn/r-1)} 
\nonumber \\
&\lesssim \|\m\|_{\mathcal{H}^s(\R^{nm})} \ell(Q)^{-mn/r}, 
\end{align}
provided $mn>r$. Summing \eqref{Tmm}--\eqref{mmj-2} up, we conclude that 
\begin{align*}
|T_{\m}&(\vec{f} \mathbf{1}_{Q^*})(x) - T_{\m}(\vec{f} \mathbf{1}_{Q_0^*})(x)| 
\lesssim \|\m\|_{\mathcal{H}^s(\R^{nm})} \prod_{i=1}^m \langle f_i \rangle_{Q^*, r}, 
\end{align*}
which shows the condition \eqref{list:T-size} holds.
\end{proof}
%%%%%%%%%%%%%%%%%%%%%%%%% END END NED PROOF %%%%%%%%%%%%%%%%%%%%%%%%

Note that in the current setting, $A_1 \subset A_{\infty} = \bigcup_{s>1} RH_s$. Considering Theorems \ref{thm:Ta}--\ref{thm:Tm} and \eqref{HLS}, we apply Theorems \ref{thm:local}--\ref{thm:T-Besi} to get the following results.  

%%%%%%%%%%%%%%%%%%%%%%% THEOREM THEOREM THEOREM %%%%%%%%%%%%%%%%%%%%%
\begin{theorem}\label{thm:FIOlocal}
Let $\rho, \delta \in [0, 1]$, $\tau < -mn(1-\rho)$, and $a \in \mathcal{S}^{\tau}_{\rho, \delta}(\Rn \times \R^{nm})$. Let $s \in (mn/2, mn]$ and $\m \in \mathcal{H}^s(\R^{nm})$. For each $\mathbb{T}=T_a$ with $r=1$ or $\mathbb{T}=T_{\m}$ with $mn/s < r <\min\{2, mn\}$, the following hold: 
\begin{list}{\rm (\theenumi)}{\usecounter{enumi}\leftmargin=1.2cm \labelwidth=1cm \itemsep=0.2cm \topsep=.2cm \renewcommand{\theenumi}{\alph{enumi}}}

\item There exists $\gamma>0$ such that for all $Q \in \B$ and $f_i \in L^{\infty}_c(\Rn)$ with $\supp(f_i) \subset Q$, $1 \leq i \leq m$, 
\begin{align*}
\big|\big\{x \in Q:  |\mathbb{T}(\vec{f})(x)| > t \, \M_r(\vec{f})(x)  \big\}\big| 
& \lesssim \, e^{- \gamma t}  |Q|, \quad t>0, 
\end{align*}
and for any $\alpha \in \{0, 1\}^m$, 
\begin{align*}
\big|\big\{x \in Q:  |[\mathbb{T}, \b]_{\alpha}(\vec{f})(x)| > t \, \mathcal{M}_r(\vec{f^*})(x)  \big\}\big|
\lesssim e^{-(\frac{\gamma t}{\|\b\|_{\tau}})^{\frac{1}{|\tau|+1}}} |Q|, 
\end{align*}
for all $t>0$, where $\tau=\{i: \alpha_i \neq 0\}$, $\|\b\|_{\tau} = \prod_{i \in \tau} \|b_i\|_{\BMO}$, and $f_i^*=M^{\lfloor r \rfloor}(|f_i|^r)^{\frac1r}$, $i=1, \ldots, m$.

\item Let $\vec{w}=(w_1, \ldots, w_m)$ and $w=\prod_{i=1}^m w_i^{\frac1m}$. If $\vec{w} \in A_{\vec{1}}$ and $w v^{\frac{r}{m}} \in A_{\infty}$, or $\vec{w} \in A_1 \times \cdots \times A_1$ and $v^r \in A_{\infty}$, then 
\begin{align*}
\bigg\|\frac{\mathbb{T}(\vec{f})}{v}\bigg\|_{L^{\frac{r}{m},\infty}(\Rn, \, w v^{\frac{r}{m}})}
&\lesssim \prod_{i=1}^m \|f_i\|_{L^r(\Rn, \, w_i)}.
\end{align*}

\item\label{list:TTb} For all $\vec{p}=(p_1, \ldots, p_m)$ with $r<p_1, \ldots, p_m <\infty$, for all $\vec{w} \in A_{\vec{p}/r}$, and for any $\alpha \in \{0, 1\}^m$, 
\begin{align*}
\|\mathbb{T}\|_{L^{p_1}(\Rn, w_1) \times \cdots \times L^{p_m}(\Rn, w_m) \to L^p(\Rn, w)} 
&\lesssim [\vec{w}]_{A_{\vec{p}}}^{\max\limits_{1 \le i \le m}\{p, (\frac{p_i}{r})'\}}, 
\\ 
\|[\mathbb{T}, \b]_{\alpha}\|_{L^{p_1}(\Rn, w_1) \times \cdots \times L^{p_m}(\Rn, w_m) \to L^p(\Rn, w)} 
&\lesssim \|\b\|_{\tau} [\vec{w}]_{A_{\vec{p}}}^{(|\tau|+1)\max\limits_{1 \le i \le m}\{p, (\frac{p_i}{r})'\}}, 
\end{align*}
where $\frac1p=\sum_{i=1}^m \frac{1}{p_i}$, $w=\prod_{i=1}^m w_i^{\frac{p}{p_i}}$, $\tau=\{i: \alpha_i \neq 0\}$, and $\|\b\|_{\tau} = \prod_{i \in \tau} \|b_i\|_{\BMO}$. 
\end{list}
\end{theorem} 
%%%%%%%%%%%%%%%%%%%%%%% THEOREM THEOREM THEOREM %%%%%%%%%%%%%%%%%%%%%

Based on Theorem \ref{thm:FIOlocal}, we establish weighted compactness for commutators of $T_a$ and $T_{\m}$ as follows. 
%%%%%%%%%%%%%%%%%%%%%%% THEOREM THEOREM THEOREM %%%%%%%%%%%%%%%%%%%%%
\begin{theorem}\label{thm:FIO}
Let $\rho, \delta \in [0, 1]$, $\tau < -mn(1-\rho)$, and $a \in \mathcal{S}^{\tau}_{\rho, \delta}(\Rn \times \R^{nm})$. Let $s \in (mn/2, mn]$, $\m \in \mathcal{H}^s(\R^{nm})$, and $b \in \CMO(\Rn)$. Then for each $\mathbb{T}=T_a$ with $r=1$ or $\mathbb{T}=T_{\m}$ with $mn/s < r <\min\{2, mn\}$, and for each $k=1,\ldots,m$, $[\mathbb{T}, b]_{e_k}$ is compact from $L^{p_1}(\Rn, w_1) \times \cdots \times L^{p_m}(\Rn, w_m)$ to $L^p(\Rn, w)$ for all $\vec{p}=(p_1, \ldots, p_m)$ with $r<p_1, \ldots, p_m < \infty$, and for all $\vec{w} \in A_{\vec{p}/r}$, where $\frac1p = \sum_{i=1}^m \frac{1}{p_i}$ and $w=\prod_{i=1}^m w_i^{\frac{p}{p_i}}$.
\end{theorem}
%%%%%%%%%%%%%%%%%%%%%%% THEOREM THEOREM THEOREM %%%%%%%%%%%%%%%%%%%%%

%%%%%%%%%%%%%%%%%%%%%%%%% PROOF PROOF PROOF %%%%%%%%%%%%%%%%%%%%%%%%
\begin{proof}
Let us first prove the weighted compactness of $T_{\m}$ with $mn/s < r <\min\{2, mn\}$. Indeed, modifying the proof of \cite[Theorem~1.1]{Hu17} to the $m$-linear case, we get that 
\begin{align}\label{Tmm-1}
[T_{\m}, b]_{e_k} \text{ is compact from $L^{p_1}(\Rn) \times \cdots \times L^{p_m}(\Rn)$ to $L^p(\Rn)$}, 
\end{align}
for all $\frac1p = \sum_{i=1}^m \frac{1}{p_i}<1$ with $r_i<p_i<\infty$, $i=1, \ldots, m$, where $1\le r_1,\ldots,r_m<2$ so that $\frac{s}{n}=\sum_{i=1}^m \frac{1}{r_i}$. Recall that $mn/s < r< mn$. Now pick $r < p_i <\infty$, $i=1, \ldots, m$, so that $\frac1p =\sum_{i=1}^m \frac{1}{p_i}<1$. Then $mn/s =: r_i < p_i <\infty$, $1 \le r_i < 2$, and $\sum_{i=1}^m \frac{1}{r_i}= \frac{s}{n}$, which together with \eqref{Tmm-1} gives 
\begin{align}\label{Tmm-2}
[T_{\m}, b]_{e_k} \text{ is compact from $L^{p_1}(\Rn) \times \cdots \times L^{p_m}(\Rn)$ to $L^p(\Rn)$}. 
\end{align}
In light of Theorem \ref{thm:Tb}, the weighted compactness of $T_{\m}$ follows from Theorem \ref{thm:FIOlocal} and \eqref{Tmm-2}. 

Next, we treat $T_a$. By Theorems \ref{thm:Tb} and \ref{thm:FIOlocal}, to obtain the weighted compactness of $T_a$, it suffices to demonstrate that  
\begin{align}\label{Tmm-3}
[T_a, b]_{e_k} \text{ is compact from $L^{p_1}(\Rn) \times \cdots \times L^{p_m}(\Rn)$ to $L^p(\Rn)$},  
\end{align}
for all (or for some) $\frac1p=\sum_{i=1}^m \frac{1}{p_i}$ with $1<p_1, \ldots, p_m<\infty$. Let $\phi_0$ be a nonnegative, radial, $\mathscr{C}_c^{\infty}(\R^{nm})$ function with compact support such that $\phi_0(\vec\xi) = 1$ for $|\vec{\xi}| \leq  1$ and $\phi_0(\vec{\xi}) = 0$ for $|\vec{\xi}|  \geq 2$. Define $\phi(\vec{\xi}) := \phi_0(\vec{\xi}) - \phi_0(2\vec{\xi})$ and $\phi_j(\vec{\xi}) := \phi(2^{-j} \vec{\xi})$, $j \ge 1$. Then, $\sum_{j=0}^{\infty} \phi_j(\vec{\xi}) = 1$ for any $\vec{\xi} \in \R^{nm}$. Setting $a_j(x,\vec{\xi}) := a(x,\vec{\xi}) \phi_j(\vec{\xi})$ for each $j=0, 1, \ldots$, we have 
\begin{align}\label{Taj-1}
T_a = \sum_{j=0}^{\infty} T_{a_j}, 
\end{align}
and by \cite[Proposition 3.1]{CXY}, 
\begin{align}\label{Taj-2}
\|T_{a_j}\|_{L^{p_1}(\Rn) \times \cdots \times L^{p_m}(\Rn) \to L^p(\Rn)} 
\lesssim 2^{j[\tau + mn(1-\rho)]}, \quad j \ge 0.  
\end{align}
In view of \eqref{Taj-1}, \eqref{Taj-2}, and \cite[Lemma 2.11]{COY}, \eqref{Tmm-3} is reduced to showing that for each $j \ge 0$, 
\begin{align}\label{Taj-3}
[T_{a_j}, b]_{e_k} \text{ is compact from $L^{p_1}(\Rn) \times \cdots \times L^{p_m}(\Rn)$ to $L^p(\Rn)$}. 
\end{align}

To proceed, in light of Theorem \ref{thm:FIOlocal} part \eqref{list:TTb} and \eqref{CCC}, we may assume that $b \in \mathscr{C}_c^{\infty}(\Rn)$ with $\supp(b) \subset B(0, A_0)$ for some $A_0 \ge 1$. We will only focus on the case $k=1$. It follows from \cite[Lemma 3.4]{CXY} that for any $s \geq 0$ and $j \geq 0$, 
\begin{equation}\label{eq:phi}
\sup_{x,y_1,\dots,y_m \in \Rn}
(|\vec y|)^{s}
\bigg| \int_{\R^{mn}} a(x,\vec{\xi}) \phi_j(\vec{\xi}) e^{2\pi i \vec{y} \cdot \vec{\xi}} \, d\vec{\xi} \bigg|
\lesssim 2^{j(\tau + mn- \rho s)},
\end{equation}
where the implicit constant is independent of $j$. Writing  
\begin{align*}
K_a(x, \vec{y}) := \mathbf{K}_a(x, x- \vec{y}) 
\quad\text{ and }\quad 
\mathbf{K}_a(x, \vec{y}) := \int_{\R^{nm}} a(x, \vec{\xi}) e^{2\pi i \vec{y} \cdot \vec{\xi}} \, d\vec{\xi}, 
\end{align*}
we see that $K_a$ is the kernel of $T_a$. Then, \eqref{kf-1}, \eqref{kf-2}, and \eqref{eq:phi} lead 
\begin{align*}
|[T_{a_j}, b]_{e_1}(\vec{f})(x)|
&=\bigg|\int_{\R^{nm}} (b(x)- b(y_1)) K_{a_j}(x, \vec{y}) \prod_{i=1}^m f_i(y_i) \, d\vec{y} \bigg|
\\
&\lesssim \int_{\sum_{i=1}^m |x-y_i| < 1} \frac{\prod_{i=1}^m |f_i(y_i)|}{|x-\vec{y}|^{mn-\gamma_0}} \, d\vec{y}
\\
&\qquad\quad+ \int_{\sum_{i=1}^m |x-y_i| \ge 1} \frac{\prod_{i=1}^m |f_i(y_i)|}{|x-\vec{y}|^{mn+\gamma_0}} \, d\vec{y}
\\ 
&\lesssim \M(\vec{f})(x), 
\end{align*}
where $\gamma_0 \in (0, 1)$ is an auxiliary parameter. This in turn implies 
\begin{align}\label{TajFK-1}
&\sup_{\|f_i\|_{L^{p_i}(\Rn)} \le 1 \atop i=1, \ldots, m} 
\|[T_{a_j}, b]_{e_1}(\vec{f})\|_{L^p(\Rn)} 
\nonumber \\
&\qquad\lesssim \|\M\|_{L^{p_1}(\Rn) \times \cdots \times L^{p_m}(\Rn) \to L^p(\Rn)} 
\lesssim 1. 
\end{align}
Let $A > 2A_0$ and $\gamma_1>0$. Then for any $|x|>A$, we utilize \eqref{kf-3} and \eqref{eq:phi} applied to $s=mn+\gamma_1$ to obtain 
\begin{align*}
|[T_{a_j}, b]_{e_1}(\vec{f})(x)|
&\lesssim \int_{B(0, A_0) \times \R^{n(m-1)}} \frac{\prod_{i=1}^m |f_i(y_i)|}{|x-\vec{y}|^{mn+\gamma_1}} \, d\vec{y}
\\
&\lesssim A^{-\gamma_1} \int_{B(0, A_0) \times \R^{n(m-1)}} 
\frac{\prod_{i=1}^m |f_i(y_i)|}{(\sum_{i=1}^m |x-y_i|)^{mn}} \, d\vec{y}
\\
&\lesssim A^{-\gamma_1} |x|^{-n} \prod_{i=1}^m \|f_i\|_{L^{p_i}(\Rn)}, 
\end{align*}
which gives 
\begin{align*}
\big\|[T_{a_j}, b]_{e_1}(\vec{f}) \mathbf{1}_{\{|x|>A\}}\big\|_{L^p(\Rn)} 
\lesssim A^{-[\gamma_1 + n(1-\frac1p)]} \prod_{i=1}^m \|f_i\|_{L^{p_i}(\Rn)}. 
\end{align*}
Hence, choosing $\gamma_1>\max\{0, -n(1-\frac1p)\}$, we arrive at 
\begin{align}\label{TajFK-2}
\lim_{A \to \infty} \sup_{\|f_i\|_{L^{p_i}(\Rn)} \le 1 \atop i=1, \ldots, m} 
\big\|[T_{a_j}, b]_{e_1}(\vec{f}) \mathbf{1}_{\{|x|>A\}}\big\|_{L^p(\Rn)} 
=0. 
\end{align}

Let $\eta>0$ be chosen later and $0<r<\frac{\eta}{2n}$. Fix $x \in \Rn$ and $x' \in B(x, r)$. We split
\begin{align}\label{TaI}
[T_{a_j}, b]_{e_1}(\vec{f})(x) - [T_{a_j}, b]_{e_1}(\vec{f})(x') 
= \mathscr{P}_1 + \mathscr{P}_2 + \mathscr{P}_3 + \mathscr{P}_4, 
\end{align}
where 
\begin{align*}
\mathscr{P}_1 &:= (b(x)-b(x')) \int_{\sum_{i=1}^m |x-y_i|>\eta} 
K_{a_j}(x,\vec{y}) \prod_{i=1}^m f_i (y_i) \, d\vec{y}, 
\\ 
\mathscr{P}_2 &:= \int_{\sum_{i=1}^m |x-y_i|>\eta} (b(x') - b(y_1)) 
(K_{a_j}(x,\vec{y}) - K_{a_j}(x',\vec{y})) \prod_{i=1}^m f_i (y_i) \, d\vec{y}, 
\\ 
\mathscr{P}_3 &:= \int_{\sum_{i=1}^m |x-y_i|\le \eta} 
(b(x)- b(y_1)) K_{a_j}(x,\vec{y}) \prod_{i=1}^m f_i (y_i) \, d\vec{y}, 
\\ 
\mathscr{P}_4 &:= \int_{\sum_{i=1}^m |x-y_i| \le \eta} 
(b(y_1) - b(x')) K_{a_j}(x', \vec{y}) \prod_{i=1}^m f_i (y_i) \, d\vec{y}.
\end{align*}
Since $b \in L^{\infty}(\Rn)$, the estimates \eqref{eq:phi} and \eqref{kf-1} imply 
\begin{align}\label{TaI-1}
|\mathscr{P}_1| 
\lesssim \int_{\sum_{i=1}^m |x-y_i|>\eta} 
\frac{\prod_{i=1}^m |f_i(y_i)|}{|x-\vec{y}|^{mn+1}} \, d\vec{y}
\lesssim \eta^{-1} \M(\vec{f})(x). 
\end{align}
By \eqref{eq:phi} and \eqref{kf-2}, 
\begin{align}\label{TaI-3}
|\mathscr{P}_3| 
\lesssim \int_{\sum_{i=1}^m |x-y_i| \le \eta} 
\frac{\prod_{i=1}^m |f_i(y_i)|}{|x-\vec{y}|^{mn-1}} \, d\vec{y}
\lesssim \eta \, \M(\vec{f})(x). 
\end{align}
Analogously, 
\begin{align}\label{TaI-4}
|\mathscr{P}_4| 
\lesssim \int_{\sum_{i=1}^m |x-y_i| \le \eta+mr} 
\frac{\prod_{i=1}^m |f_i(y_i)|}{|x'-\vec{y}|^{mn-1}} \, d\vec{y}
\lesssim (\eta + mr) \M(\vec{f})(x'). 
\end{align}
Together with the mean value theorem, \eqref{eq:phi} applied to $\partial_k a$ yields 
\begin{align}\label{Kaj-1}
&|\mathbf{K}_{a_j}(x, x- \vec{y}) - \mathbf{K}_{a_j}(x', x- \vec{y})|
\nonumber \\
&=\bigg|\int_{\R^{nm}} (a_j(x, \vec{\xi}) - a_j(x', \vec{\xi})) e^{2\pi i (x-\vec{y}) \cdot \vec{\xi}} \, d\vec{\xi}\bigg| 
\nonumber \\ 
&=\bigg|\int_{\R^{nm}} \int_0^1 (x-x') \cdot \nabla_x a(x(t), \vec{\xi}) \phi_j(\vec{\xi}) e^{2\pi i (x-\vec{y}) \cdot \vec{\xi}} \, dt \, d\vec{\xi}\bigg| 
\nonumber \\ 
&\le \sum_{k=1}^n |x_k-x'_k|  \int_0^1 \bigg|\int_{\R^{nm}} \partial_{x_k} a(x(t), \vec{\xi}) \phi_j(\vec{\xi}) e^{2\pi i (x-\vec{y}) \cdot \vec{\xi}} \, d\vec{\xi}\bigg| \, dt
\nonumber \\ 
&\lesssim \frac{r}{|x-\vec{y}|^{mn+1}} 
\simeq \frac{r}{(\sum_{i=1}^m |x-y_i|)^{mn+1}}, 
\end{align}
where $x(t) := (1-t) x +t x'$. Note that 
\begin{align*}
|x-x(t)| \le t |x-x'|< r \le \frac{\eta}{2\sqrt{n}} < \frac{1}{2\sqrt{n}} \sum_{i=1}^m |x-y_i|, 
\end{align*}
and 
\begin{align*}
|x(t)-\vec{y}| 
&\ge |x-\vec{y}| - |x-x(t)|
\\
&\ge \frac{1}{\sqrt{n}} \sum_{i=1}^m |x-y_i| - |x-x(t)| 
\ge \frac{1}{2\sqrt{n}} \sum_{i=1}^m |x-y_i|. 
\end{align*}
Since $\xi_k a(x, \vec{\xi}) \in \mathcal{S}_{\rho, \delta}^{\tau+1}$, we use the mean value theorem and \eqref{eq:phi} applied to $\xi_k a(x, \vec{\xi})$ to deduce 
\begin{align}\label{Kaj-2}
|&\mathbf{K}_{a_j}(x', x- \vec{y}) - \mathbf{K}_{a_j}(x', x'- \vec{y})| 
\nonumber \\ 
&=\bigg|\int_{\R^{nm}} a_j(x', \vec{\xi}) \big(e^{2\pi i (x-\vec{y}) \cdot \vec{\xi}} - e^{2\pi i (x'-\vec{y}) \cdot \vec{\xi}}\big) \, d\vec{\xi}\bigg| 
\nonumber \\ 
&=\bigg|\int_{\R^{nm}} \int_0^1 2\pi i \sum_{k=1}^n (x - x') \xi_k a(x', \vec{\xi}) \phi_j(\vec{\xi}) e^{2\pi i (x(t)-\vec{y}) \cdot \vec{\xi}} dt \, d\vec{\xi}\bigg| 
\nonumber \\ 
&\le 2\pi \sum_{k=1}^n |x - x'| \int_0^1 \bigg|\int_{\R^{nm}} \xi_k a(x', \vec{\xi}) \phi_j(\vec{\xi}) e^{2\pi i (x(t)-\vec{y}) \cdot \vec{\xi}} \, d\vec{\xi}\bigg| \, dt
\nonumber \\ 
&\lesssim \frac{r}{|x(t)-\vec{y}|^{mn+1}} 
\lesssim \frac{r}{(\sum_{i=1}^m |x-y_i|)^{mn+1}},  
\end{align}
where $x(t) := (1-t) x +t x'$. Since 
\begin{multline*}
K_{a_j}(x,\vec{y}) - K_{a_j}(x',\vec{y})
= \mathbf{K}_{a_j}(x, x- \vec{y}) - \mathbf{K}_{a_j}(x', x'- \vec{y}) 
\\ 
= \mathbf{K}_{a_j}(x, x- \vec{y}) - \mathbf{K}_{a_j}(x', x- \vec{y}) 
+ \mathbf{K}_{a_j}(x', x- \vec{y}) - \mathbf{K}_{a_j}(x', x'- \vec{y}), 
\end{multline*}
we invoke \eqref{Kaj-1}--\eqref{Kaj-2} and \eqref{kf-1} to arrive at 
\begin{align}\label{TaI-2}
|\mathscr{P}_2| 
\lesssim \int_{\sum_{i=1}^m |x-y_i|>\eta} 
\frac{\prod_{i=1}^m |f_i(y_i)|}{(\sum_{i=1}^m |x-y_i|)^{mn+1}} \, d\vec{y}
\lesssim \eta^{-1} \M(\vec{f})(x).  
\end{align}
Now gathering \eqref{TaI-1}--\eqref{TaI-4} and \eqref{TaI-2}, we change variables to obtain
\begin{align*}
&\bigg[\int_{\Rn} \bigg(\fint_{B(x, r)} |T_{a_j}, b]_{e_1}(\vec{f})(x) 
- [T_{a_j}, b]_{e_1}(\vec{f})(x')|^{\frac{p}{p_0}} dx' \bigg)^{p_0} dx \bigg]^{\frac1p}
\\ 
&\lesssim \bigg[\int_{\Rn} \bigg(\fint_{B(x, r)} |\M(\vec{f})(x) + \M(\vec{f})(x')|^{\frac{p}{p_0}} dx' \bigg)^{p_0} dx \bigg]^{\frac1p}
\\ 
&\le \bigg[\int_{\Rn} \fint_{B(x, r)} |\M(\vec{f})(x) + \M(\vec{f})(x')|^p dx' \, dx \bigg]^{\frac1p}
\\
&\lesssim \eta \|\M(\vec{f})\|_{L^p(\Rn)} 
\lesssim \varepsilon \prod_{i=1}^m \|f_i\|_{L^{p_i}(\Rn)}, 
\end{align*}
provided $\eta=\varepsilon$ and $0<r<\delta=\frac{\eta}{2n}=\frac{\varepsilon}{2n}$. This means 
\begin{align}\label{TajFK-3}
\lim_{r \to 0} \sup_{\|f_i\|_{L^{p_i}(\Rn)} \le 1 \atop i=1, \ldots, m} 
\bigg[\int_{\Rn} \bigg(&\fint_{B(x, r)} |[T_{a_j}, b]_{e_1}(\vec{f})(x) 
\nonumber \\
&- [T_{a_j}, b]_{e_1}(\vec{f})(x')|^{\frac{p}{p_0}} dx' \bigg)^{p_0} dx \bigg]^{\frac1p}
=0. 
\end{align}
As a consequence, \eqref{Tmm-3} follows from Theorem \ref{thm:FKhs-2}, \eqref{TajFK-1}, \eqref{TajFK-2}, and \eqref{TajFK-3}. 
\end{proof} 
%%%%%%%%%%%%%%%%%%%%%%% END END END PROOF %%%%%%%%%%%%%%%%%%%%%%%%%%

%%%%%%%%%%%%%%%%%%%%% SUBSECTION SUBSECTION SUBSECTION %%%%%%%%%%%%%%%%%%
%%%%%%%%%%%%%%%%%%%%% SUBSECTION SUBSECTION SUBSECTION %%%%%%%%%%%%%%%%%%
\subsection{Higher order Calder\'{o}n commutators}\label{sec:Calderon}
In this subsection, we will consider higher order Calder\'{o}n commutators. Let $A_1,\ldots, A_m$ be functions defined on $\R$ such that $a_j=A'_j$, $j=1, \ldots, m-1$. We define 
\begin{align*}
\mathcal{C}_m(a_1, \ldots, a_{m-1}, f)(x) 
:= \mathrm{p.v. } \int_{\R} \frac{\prod_{j=1}^{m-1} (A_j(x)-A_j(y))}{(x-y)^{m-1}} f(y) \, dy. 
\end{align*}
Using the strategy in \cite[p. 2106]{DGY}, we rewrite $\mathcal{C}_m$ as the multilinear singular integral operator:  
\begin{align*}
\mathcal{C}_m(a_1, \ldots, a_{m-1}, f)(x) 
= \int_{\R^m} K(x, \vec{y})  \prod_{j=1}^{m-1} a_j(y_j) f(y_m) \, d\vec{y}, 
\end{align*}
where 
\begin{align}\label{eq:KA-2} 
K(x, \vec{y}) 
&:= \frac{(-1)^{(m-1)e(y_m - x)}}{(x-y_m)^m} 
\prod_{j=1}^{m-1} {\bf 1}_{(x \land y_m, x \lor y_m)}(y_j). 
\end{align}
Here, $e(x)={\bf 1}_{(0, \infty)}(x)$, $x \land y=\min\{x, y\}$, and $x \lor y=\max\{x, y\}$. It follows from \cite{HZ} that 
\begin{equation}\label{eq:CCk-1}
|K(x, \vec{y})| \lesssim \frac{1}{(\sum_{j=1}^m |x-y_j|)^m}, 
\end{equation}
and 
\begin{align}\label{eq:CCk-2}
|K(x,\vec{y}) - K(x',\vec{y})| \lesssim \frac{|x-x'|}{(\sum_{j=1}^{m}|x-y_j|)^{m+1}},
\end{align}
whenever $|x-x'| \leq \frac{1}{8} \min\limits_{1\le j \le m} |x-y_j|$. 

To generalize $\mathcal{C}_m$, we define the multilinear singular integral 
\begin{align}\label{def:C}
\mathscr{C}(\vec{f})(x) := \int_{\R^m} K(x, \vec{y})  \prod_{j=1}^m f_j(y_j) \, d\vec{y}, 
\end{align}
where the kernel $K$ is given in \eqref{eq:KA-2}. It was shown in \cite[Corollary 4.2]{DGY} that 
\begin{align}\label{Car-endpoint}
\text{$\mathscr{C}$ is bounded from $L^1(\R) \times \cdots \times L^1(\R)$ to $L^{\frac1m, \infty}(\R)$.}
\end{align}

In this subsection we work in Euclidean space $(\R, \L)$, and let $\B$ be the collection of all intervals in $\R$. As before, we will drop the subscript $\B$ in all notation.

%%%%%%%%%%%%%%%%%%%%%%% THEOREM THEOREM THEOREM %%%%%%%%%%%%%%%%%%%%%
\begin{theorem}\label{thm:C}
The operator $\mathscr{C}$ is a multilinear bounded oscillation operator with respect to the basis $\B$ and the exponent $r=1$.  
\end{theorem}
%%%%%%%%%%%%%%%%%%%%%%% THEOREM THEOREM THEOREM %%%%%%%%%%%%%%%%%%%%%

%%%%%%%%%%%%%%%%%%%%%%%%% PROOF PROOF PROOF %%%%%%%%%%%%%%%%%%%%%%%%
\begin{proof} 
In the current scenario, we set $Q^*=18Q$ for each interval $Q \subset \R$. Considering \eqref{eq:CCk-1}, one can follow much as in \eqref{Tsig} to obtain that $\mathscr{C}$ satisfies the condition \eqref{list:T-size}.

Let $Q \in \B$ and $x, x' \in Q$. By definition, we have 
\begin{align*}
\big| \big(&\mathscr{C}(\vec{f}) - \mathscr{C}(\vec{f} \mathbf{1}_{Q^*}) \big)(x) 
- \big(\mathscr{C}(\vec{f}) - \mathscr{C}(\vec{f} \mathbf{1}_{Q^*}) \big)(x') \big| 
\\
&= \bigg|\int_{\R^m \setminus (Q^*)^m} (K(x, \vec{y}) - K(x', \vec{y})) \prod_{i=1}^m f_i(y_i) \, d\vec{y} \bigg|
\\ 
&\le \sum_{E_1, \ldots, E_m} \int_{E_1 \times \cdots \times E_m} 
|K(x, \vec{y}) - K(x', \vec{y})| \prod_{i=1}^m |f_i(y_i)| \, d\vec{y}
\\
&=: \sum_{E_1, \ldots, E_m} \mathcal{I}_{E_1, \ldots, E_m}, 
\end{align*}
where the summation is taken over all $E_1 \times \cdots \times E_m \in \{\R \setminus Q^*, Q^*\}^m$ with some $E_i=\R  \setminus Q^*$. If $E_1=\cdots=E_m=\R \setminus Q^*$, then $|x-x'| \le \ell(Q) < \frac18 \min\limits_{1 \le j \le m} |x-y_j|$ for all $\vec{y} \in (\R \setminus Q^*)^m$, which together with \eqref{eq:CCk-2} and that $\frac1k \sum_{i=1}^k |a_i| \ge \prod_{i=1}^k |a_i|^{\frac1k}$ gives 
\begin{align*}
&\mathcal{I}_{E_1, \ldots, E_m} 
\lesssim \int_{(\R \setminus Q^*)^m}  
\frac{|x-x'|}{(\sum_{j=1}^{m}|x-y_j|)^{m+1}} \prod_{i=1}^m |f_i(y_i)| \, d\vec{y}
\\
&\lesssim \ell(Q) \sum_{k_1, \ldots, k_m \ge 1} \int_{2^{k_1+1} Q^* \setminus 2^{k_1} Q^*}  
\cdots \int_{2^{k_m+1} Q^* \setminus 2^{k_m} Q^*} 
\prod_{i=1}^m \frac{|f_i(y_i)| \, dy_i}{|x-y_i|^{1+1/m}} 
\\
&\lesssim \prod_{i=1}^m \sum_{k_i \ge 1} 2^{-k_i/m} \fint_{2^{k_i+1} Q^*} |f_i(y_i)| \, dy_i
\lesssim \prod_{i=1}^m \lfloor f_i \rfloor_{Q^*}. 
\end{align*}
To deal with the general case, we may assume that $E_1=\cdots=E_{i_0}=\R \setminus Q^*$ and $E_{i_0+1}=\cdots=E_m=Q^*$ for some $1 \le i_0 \le m-1$. Note that for all $z \in Q$ and $y_i \in 2^{k_i+1} Q^* \setminus 2^{k_i} Q^*$, $i=1, \ldots, i_0$, 
\[
\sum_{j=1}^m |z-y_j| 
\gtrsim \max\{2^{k_1}, \ldots, 2^{k_{i_0}}\} \ell(Q^*) 
=: 2^{k_*} \ell(Q^*). 
\]
Then it follows from \eqref{eq:CCk-1} that  
\begin{align*}
&\mathcal{I}_{E_1, \ldots, E_m} 
\lesssim \sum_{k_1, \ldots, k_{i_0} \ge 1} \int_{2^{k_1+1} Q^* \setminus 2^{k_1} Q^*}  
\cdots \int_{2^{k_{i_0}+1} Q^* \setminus 2^{k_{i_0}} Q^*} 
\\
&\qquad\qquad\qquad\times \int_{(Q^*)^{m-i_0}} \prod_{i=1}^m \frac{|f_i(y_i)|}{2^{k_*} \ell(Q^*)} \, d\vec{y}
\\
&\lesssim \prod_{i=1}^{i_0} \sum_{k_i \ge 1} 2^{k_i - k_*m/i_0} \fint_{2^{k_i+1} Q^*} |f_i| \, dy_i
\times \prod_{i=i_0+1}^m \fint_{Q^*} |f_i| \, dy_i 
\\
&\lesssim \prod_{i=1}^m \lfloor f_i \rfloor_{Q^*},  
\end{align*}
where we have used that $k_i - k_*m/i_0 \le k_i - k_im/(m-1) = -\frac{k_i}{m-1}$. Collecting the estimates above, we show the condition \eqref{list:T-reg}. 
\end{proof}
%%%%%%%%%%%%%%%%%%%%%%%%% END END END PROOF %%%%%%%%%%%%%%%%%%%%%%%%

Note that in the current setting, $A_1 \subset A_{\infty} = \bigcup_{s>1} RH_s$. With Theorem \ref{thm:C}, \eqref{Car-endpoint}, and \eqref{HLS} in hand, Theorems \ref{thm:local}--\ref{thm:T-Besi} imply the following result.  

%%%%%%%%%%%%%%%%%%%%%%% THEOREM THEOREM THEOREM %%%%%%%%%%%%%%%%%%%%%
\begin{theorem}\label{thm:Clocal}
Let $\mathscr{C}$ be the operator in \eqref{def:C} with the kernel $K$ given by \eqref{eq:KA-2}. Then the following hold: 
\begin{list}{\rm (\theenumi)}{\usecounter{enumi}\leftmargin=1.2cm \labelwidth=1cm \itemsep=0.2cm \topsep=.2cm \renewcommand{\theenumi}{\alph{enumi}}}

\item\label{list:C1} There exists $\gamma>0$ such that for all $Q \in \B$ and $f_i \in L^{\infty}_c(\R)$ with $\supp(f_i) \subset Q$, $1 \leq i \leq m$, 
\begin{align*}
\big|\big\{x \in Q:  |\mathscr{C}(\vec{f})(x)| > t \, \M(\vec{f})(x)  \big\}\big| 
& \lesssim \, e^{- \gamma t}  |Q|, \quad t>0, 
\end{align*}
and for any $\alpha \in \{0, 1\}^m$, 
\begin{align*}
\big|\big\{x \in Q:  |[\mathscr{C}, \b]_{\alpha}(\vec{f})(x)| > t \, \mathcal{M}(\vec{f^*})(x)  \big\}\big|
\lesssim e^{-(\frac{\gamma t}{\|\b\|_{\tau}})^{\frac{1}{|\tau|+1}}} |Q|,  
\end{align*}
for all $t>0$, where $\tau=\{i: \alpha_i \neq 0\}$, $\|\b\|_{\tau} = \prod_{i \in \tau} \|b_i\|_{\BMO}$, and $f_i^*=Mf_i$, $i=1, \ldots, m$.

\item\label{list:C2} Let $\vec{w}=(w_1, \ldots, w_m)$ and $w=\prod_{i=1}^m w_i^{\frac1m}$. If $\vec{w} \in A_{\vec{1}}$ and $w v^{\frac1m} \in A_{\infty}$, or $\vec{w} \in A_1 \times \cdots \times A_1$ and $v \in A_{\infty}$, then 
\begin{align*}
\bigg\|\frac{\mathscr{C}(\vec{f})}{v}\bigg\|_{L^{\frac1m,\infty}(\R, \, w v^{\frac1m})}
&\lesssim  \prod_{i=1}^m \|f_i\|_{L^1(\R, \, w_i)}. 
\end{align*}

\item\label{list:C3} For all $\vec{p}=(p_1, \ldots, p_m)$ with $1<p_1, \ldots, p_m <\infty$, for all $\vec{w} \in A_{\vec{p}}$, and for any $\alpha \in \{0, 1\}^m$, 
\begin{align*}
\|\mathscr{C}\|_{L^{p_1}(\R, w_1) \times \cdots \times L^{p_m}(\R, w_m) \to L^p(\R, w)} 
&\lesssim [\vec{w}]_{A_{\vec{p}}}^{\max\limits_{1 \le i \le m}\{p, p'_i\}}, 
\\
\|[\mathscr{C}, \b]_{\alpha}\|_{L^{p_1}(\R, w_1) \times \cdots \times L^{p_m}(\R, w_m) \to L^p(\R, w)} 
&\lesssim \|\b\|_{\tau} [\vec{w}]_{A_{\vec{p}}}^{(|\tau|+1)\max\limits_{1 \le i \le m}\{p, p'_i\}}, 
\end{align*}
where $\frac1p=\sum_{i=1}^m \frac{1}{p_i}$, $w=\prod_{i=1}^m w_i^{\frac{p}{p_i}}$, $\tau=\{i: \alpha_i \neq 0\}$, and $\|\b\|_{\tau} = \prod_{i \in \tau} \|b_i\|_{\BMO}$. 
\end{list}
\end{theorem} 
%%%%%%%%%%%%%%%%%%%%%%% THEOREM THEOREM THEOREM %%%%%%%%%%%%%%%%%%%%%

Note that parts \eqref{list:C2} and \eqref{list:C3} improve \cite[Corollary 4.2]{DGY}, while part \eqref{list:C1} is totally novel.  
Let us next investigate weighted compactness for commutators of $\mathscr{C}$. The following extends the result in \cite[Theorem 1.1]{BuC} to the quasi-Banach range while we present an alternative proof.

%%%%%%%%%%%%%%%%%%%%%%% THEOREM THEOREM THEOREM %%%%%%%%%%%%%%%%%%%%%
\begin{theorem}\label{thm:CC}
Let $\mathscr{C}$ be the operator in \eqref{def:C} with the kernel $K$ given by \eqref{eq:KA-2}. Then for any $b \in \CMO(\R)$ and for each $j=1, \ldots, m$, $[\mathscr{C}, b]_{e_j}$ is compact from $L^{p_1}(\R, w_1) \times \cdots \times L^{p_m}(\R, w_m)$ to $L^p(\R, w)$ for all $\vec{p}=(p_1, \ldots, p_m)$ with $1<p_1, \ldots, p_m < \infty$, and for all $\vec{w} \in A_{\vec{p}}$, where $\frac1p = \sum_{i=1}^m \frac{1}{p_i}$ and $w=\prod_{i=1}^m w_i^{\frac{p}{p_i}}$.
\end{theorem}
%%%%%%%%%%%%%%%%%%%%%%% THEOREM THEOREM THEOREM %%%%%%%%%%%%%%%%%%%%%

%%%%%%%%%%%%%%%%%%%%%%%%% PROOF PROOF PROOF %%%%%%%%%%%%%%%%%%%%%%%%
\begin{proof}
It follows from Theorem \ref{thm:Clocal} that 
\begin{align}
\label{Car-1}
\|\mathscr{C}\|_{L^{p_1}(\R, w_1) \times \cdots \times L^{p_m}(\R, w_m) \to L^p(\R, w)} 
&\lesssim [\vec{w}]_{A_{\vec{p}}}^{\max\limits_{1 \le i \le m}\{p, p'_i\}}, 
\\
\label{Car-2}
\|[\mathscr{C}, b]_{e_j}\|_{L^{p_1}(\R, w_1) \times \cdots \times L^{p_m}(\R, w_m) \to L^p(\R, w)} 
&\lesssim \|b\|_{\BMO} [\vec{w}]_{A_{\vec{p}}}^{2\max\limits_{1 \le i \le m}\{p, p'_i\}}, 
\end{align}
for all $\frac1p=\sum_{i=1}^m \frac{1}{p_i}$ and $\vec{w} \in A_{\vec{p}}$, where $\frac1p=\sum_{i=1}^m \frac{1}{p_i}$, $w=\prod_{i=1}^m w_i^{\frac{p}{p_i}}$. Thus, by Theorem \ref{thm:Tb}, the matter is reduced to showing 
\begin{align*}
[\mathscr{C}, b]_{e_j} \text{ is compact from $L^{p_1}(\R) \times \cdots \times L^{p_m}(\R)$ to $L^p(\R)$}, 
\end{align*}
for all (or for some) $\frac1p = \sum_{i=1}^m \frac{1}{p_i}<1$ with $1<p_1,\ldots,p_m<\infty$. Fix $\frac1p = \sum_{i=1}^m \frac{1}{p_i}<1$ with $1<p_1,\ldots,p_m<\infty$. By Theorem \ref{thm:FKhs-1}, \eqref{Car-2}, and \eqref{CCC}, it is enough to show that for any $\varepsilon\in (0,1)$ and $b \in \mathscr{C}_b^{\alpha}(\R)$ with $\supp(b) \subset B(0, A_0)$ for some $0<\alpha < 1 \le A_0 < \infty$, 
\begin{itemize}
\item there exists $A=A(\varepsilon)>0$ independent of $\vec{f}$ such that 
\begin{align}\label{eq:CC-3}
\big\|[\mathscr{C}, b]_{e_1}](\vec{f})\mathbf{1}_{\{|x|>A\}} \big\|_{L^p(\R)} 
\lesssim \varepsilon \prod_{i=1}^m \|f_i\|_{L^{p_i}(\R)}. 
\end{align}

\item there exists $\delta=\delta(\varepsilon)>0$ independent of $\vec{f}$ such that for all $r \in (0, \delta)$, 
\begin{align}\label{eq:CC-4}
\big\|[\mathscr{C}, b]_{e_1}(\vec{f}) - ([\mathscr{C}, b]_{e_1}](\vec{f}) )_{B(\cdot, r)} \big\|_{L^p(\R)} 
\lesssim \varepsilon \prod_{i=1}^m \|f_i\|_{L^{p_i}(\R)}. 
\end{align}
\end{itemize}

Let $A>2A_0$ and $|x|>A$. Using the size condition \eqref{eq:CCk-1} and \eqref{kf-3}, we deduce that 
\begin{align}\label{eq:CC-5}
|[\mathscr{C}, b]_{e_1}(\vec{f})(x)| 
&\lesssim \int_{B(0, A_0) \times \R^{m-1}} \frac{\prod_{i=1}^m |f_i(y_i)|}{(\sum_{i=1}^m |x-y_i|)^m} d\vec{y} 
\nonumber \\
&\lesssim |x|^{-1} \prod_{i=1}^m \|f_i\|_{L^{p_i}(\Rn)}. 
\end{align}
Taking $A > \max\{2A_0, \varepsilon^{-p'}\}$, we see that \eqref{eq:CC-5} implies \eqref{eq:CC-3}. 

In order to demonstrate \eqref{eq:CC-4}, we set $\eta>0$ chosen later and $0<r<\frac{\eta}{8m}$. Let $x \in \R$ and $x' \in B(x, r)$. Then, we write  
\begin{align}\label{eq:CJJJ}
[\mathscr{C}(\vec{f}), b]_{e_1}(x) - [\mathscr{C}(\vec{f}), b]_{e_1}(x')
= \mathscr{Q}_1 + \mathscr{Q}_2 + \mathscr{Q}_3 + \mathscr{Q}_4, 
\end{align}
where 
\begin{align*}
\mathscr{Q}_1 &:= (b(x)-b(x')) \int_{\sum_{i=1}^m |x-y_i|>\eta} 
K(x,\vec{y}) \prod_{i=1}^m f_i(y_i) \, d\vec{y}, 
\\
\mathscr{Q}_2 &:= \int_{\sum_{i=1}^m |x-y_i|>\eta} (b(x')-b(y_1)) 
(K(x,\vec{y}) - K(x',\vec{y})) \prod_{i=1}^m f_i(y_i) \, d\vec{y}, 
\\
\mathscr{Q}_3 &:= \int_{\sum_{i=1}^m |x-y_i|\le \eta} 
(b(x) - b(y_1)) K(x,\vec{y}) \prod_{i=1}^m f_i(y_i)\, d\vec{y}, 
\\
\mathscr{Q}_4 &:= \int_{\sum_{i=1}^m |x-y_i| \le \eta} 
(b(y_1) - b(x')) K(x', \vec{y}) \prod_{i=1}^m f_i(y_i) \, d\vec{y}.
\end{align*}
Denote 
\[
\mathscr{C}^*(\vec{f})(x) 
:= \sup_{\eta>0} \bigg|\int_{\sum_{i=1}^m |x-y_i|>\eta} 
K(x,\vec{y}) \prod_{i=1}^m f_i(y_i) \, d\vec{y} \bigg|. 
\]
Note that 
\begin{align}\label{CJ-1}
|\mathscr{Q}_1| 
\lesssim |x-x'|^{\alpha} \|b\|_{\mathscr{C}^{\alpha}(\R)} \mathscr{C}^*(\vec{f})(x)
\lesssim \eta^{\alpha} \mathscr{C}^*(\vec{f})(x), 
\end{align}
and by \cite[Theorem 4.3]{DGGLY}, 
\begin{align}\label{Cmax}
\|\mathscr{C}^*\|_{L^{p_1}(\R) \times \cdots \times L^{p_m}(\R) \to L^p(\R)}  
\lesssim 1. 
\end{align} 
Much as \eqref{Tkb-3} and \eqref{Tkb-4}, we have 
\begin{align}\label{CJ-3}
|\mathscr{Q}_3| \lesssim \eta^{\alpha} \, \mathcal{M}(\vec{f})(x) 
\quad\text{ and }\quad 
|\mathscr{Q}_4| \lesssim \eta^{\alpha} \, \mathcal{M}(\vec{f})(x'). 
\end{align}

To analyze $\mathscr{Q}_2$, we write 
\begin{align*}
\mathscr{Q}_{2,1} &:= \int_{\forall i: |x-y_i|>\frac{\eta}{m}} 
|b(x')-b(y_1)| |K(x, \vec{y}) - K(x', \vec{y})| \prod_{i=1}^m |f_i(y_i)| \, d\vec{y}, 
\\ 
\mathscr{Q}_{2,2, j} &:= \int_{|x-y_j| \le \frac{\eta}{m} \atop \sum_{i=1}^m |x-y_i|>\eta} 
|b(x')-b(y_1)| |K(x, \vec{y}) - K(x', \vec{y})| \prod_{i=1}^m |f_i(y_i)| \, d\vec{y}. 
\end{align*}
The smoothness condition \eqref{eq:CCk-2} and \eqref{kf-1} lead 
\begin{align}\label{CJ-21}
\mathscr{Q}_{2,1} 
\lesssim r \int_{\sum_{i=1}^m |x-y_i|>\eta} 
\frac{\prod_{i=1}^m |f_i(y_i)|}{(\sum_{i=1}^m |x-y_i|)^{m+1}} d\vec{y} 
\lesssim r \eta^{-1} \M(\vec{f})(x). 
\end{align} 
To control $\mathscr{Q}_{2,2,j}$, we observe that $\sum_{i=1}^m |x'-y_i| \simeq \sum_{i=1}^m |x-y_i|$ whenever $x' \in B(x, r)$ and $\sum_{i=1}^m |x-y_i| > \eta$. Then, the size condition \eqref{eq:CCk-1} implies 
\begin{align}\label{CJ-22j}
\mathscr{Q}_{2,2,j} 
&\lesssim  \|b\|_{\mathscr{C}^{\alpha}(\R)} \sum_{k=1}^{\infty} 
\int_{|x-y_j|<\eta \atop 2^{k-1} \eta < \sum_{i=1}^m |x-y_i| \le 2^k \eta} 
\frac{\prod_{i=1}^m |f_i(y_i)|}{(\sum_{i=1}^m |x-y_i|)^{m-\alpha}} d\vec{y} 
\nonumber \\
&\lesssim \eta^{\alpha} \sum_{k=1}^{\infty} 2^{-k(1-\alpha)}
\bigg(\fint_{B(x, \eta)} |f_j| \, dy_j \bigg)
\prod_{1 \le i \le m \atop i \neq j} \bigg(\fint_{B(x, 2^k \eta)} |f_i| \, dy_i \bigg)
\nonumber \\
&\lesssim \eta^{\alpha} \prod_{i=1}^m Mf_i(x). 
\end{align} 
It follows from \eqref{CJ-21} and \eqref{CJ-22j} that 
\begin{align}\label{CJ-2}
|\mathscr{Q}_2| 
\le \mathscr{Q}_{2,1} + \sum_{j=1}^m \mathscr{Q}_{2,2,j} 
\lesssim (\eta^{\alpha} + r\eta^{-1}) \prod_{i=1}^m Mf_i(x). 
\end{align}
Summing \eqref{eq:CJJJ}, \eqref{CJ-1}, \eqref{CJ-3}, \eqref{CJ-2} up, we have 
\begin{align*}
\big|[\mathscr{C}(\vec{f}), &b]_{e_1}(x) - [\mathscr{C}(\vec{f}), b]_{e_1}(x')\big|
\\
&\lesssim (\eta^{\alpha} + r\eta^{-1}) \Big[\mathscr{C}^*(\vec{f})(x) + \prod_{i=1}^m Mf_i(x) + \M(\vec{f})(x')\Big]. 
\end{align*}
Now taking $\eta := \varepsilon^{\frac{1}{\alpha}}$ and $\delta := \min\{\frac{\eta}{8m}, \eta^{1+\alpha}\} = \min\{\frac{\varepsilon^{\frac{1}{\alpha}}}{8m}, \varepsilon^{1+\frac{1}{\alpha}}\}$, we see that for any $r \in (0, \delta)$, $r\eta^{-1} \le \eta^{\alpha}= \varepsilon$. Hence, in light of \eqref{tauM} and \eqref{Cmax}, the preceding inequality indicates 
\begin{align*}
&\big\|[\mathscr{C}, b]_{e_1}(\vec{f}) - ([\mathscr{C}, b]_{e_1}](\vec{f}) )_{B(\cdot, r)} \big\|_{L^p(\R)} 
\\
&\le \bigg[\int_{\R} \bigg(\fint_{B(x, r)} \big|[\mathscr{C}, b]_{e_1}(\vec{f})(x) 
- [\mathscr{C}, b]_{e_1}](\vec{f})(x')\big| \, dx' \bigg)^p dx \bigg]^{\frac1p}
\\
&\lesssim \varepsilon \|\mathscr{C}^*(\vec{f})\|_{L^p(\R)} 
+ \varepsilon \bigg\|\prod_{i=1}^m Mf_i \bigg\|_{L^p(\R)} 
+ \varepsilon \bigg\|\fint_{B(\cdot, r)} \M(\vec{f})(x') \, dx' \bigg\|_{L^p(\R)}
\\
&\lesssim \varepsilon \|\mathscr{C}^*(\vec{f})\|_{L^p(\R)} 
+ \varepsilon \prod_{i=1}^m \|Mf_i\|_{L^{p_i}(\R)} 
+ \varepsilon \|\M(\vec{f})\|_{L^p(\R)}
\\
&\lesssim \varepsilon \prod_{i=1}^m \|f_i\|_{L^{p_i}(\R)}. 
\end{align*}
This shows \eqref{eq:CC-4} and completes the proof. 
\end{proof}
%%%%%%%%%%%%%%%%%%%%%%%%% END END END PROOF %%%%%%%%%%%%%%%%%%%%%%%%

%%%%%%%%%%%%%%%%%%%%% SUBSECTION SUBSECTION SUBSECTION %%%%%%%%%%%%%%%%%%
%%%%%%%%%%%%%%%%%%%%% SUBSECTION SUBSECTION SUBSECTION %%%%%%%%%%%%%%%%%%
\subsection{Maximally modulated multilinear singular integrals}\label{sec:mod}
Given a family of multilinear bounded oscillation operators $\{T_{\alpha}\}_{\alpha \in \A}$ on a measure space $(\Sigma, \mu)$, we define the associated maximal operator 
\begin{align}\label{def:TA}
T^{\A}(\vec{f})(x) := \sup_{\alpha \in \A} |T_{\alpha}(\vec{f})(x)|. 
\end{align}
Define a Banach space $\bB$ by 
\begin{align*}
\bB :=\big\{F: \A \to \mathbb{C}: \|F\|_{\bB} := \sup_{\alpha \in \A} |F(\alpha)| < \infty\big\}. 
\end{align*}
Then by Definition \ref{def:MBO}, one can obtain the following result.  

%%%%%%%%%%%%%%%%%%%%%%% THEOREM THEOREM THEOREM %%%%%%%%%%%%%%%%%%%%%
\begin{theorem}\label{thm:TA}
Let $(\Sigma, \mu)$ be a measure space with a ball-basis $\B$. If $\{T_{\alpha}\}_{\alpha \in \A}$ is a family of  multilinear bounded oscillation operators satisfying 
\begin{equation}\label{CTCT}
 \C_1^{\A} := \sup_{\alpha \in \A} \C_1(T_{\alpha}) <\infty 
 \quad\text{and}\quad 
\C_2^{\A} := \sup_{\alpha \in \A} \C_2(T_{\alpha}) <\infty, 
\end{equation}
where $r \in [1, \infty)$ is some fixed exponent, then $T^{\A}$ is a $\bB$-valued multilinear bounded oscillation operator with respect to $\B$ and $r$, with constants $\C_1(T^{\A}) \le \C_1^{\A}$ and $\C_2(T^{\A}) \le \C_2^{\A}$. 
\end{theorem}
%%%%%%%%%%%%%%%%%%%%%%% THEOREM THEOREM THEOREM %%%%%%%%%%%%%%%%%%%%%

Then invoking Theorem \ref{thm:TA} and Theorems \ref{thm:local}--\ref{thm:T}, we derive the following results. 

%%%%%%%%%%%%%%%%%%%%%%% THEOREM THEOREM THEOREM %%%%%%%%%%%%%%%%%%%%%
\begin{theorem}\label{thm:TATA}
Let $(\Sigma, \mu)$ be a measure space with a ball-basis $\B$, and let $\{T_{\alpha}\}_{\alpha \in \A}$ be a family of multilinear bounded oscillation operators satisfying \eqref{CTCT} and that $T^{\A}$ is bounded from $L^r(\Sigma, \mu) \times \cdots \times L^r(\Sigma, \mu)$ to $L^{\frac{r}{m}, \infty}(\Sigma, \mu)$. Then the following hold: 
\begin{list}{\rm (\theenumi)}{\usecounter{enumi}\leftmargin=1.2cm \labelwidth=1cm \itemsep=0.2cm \topsep=.2cm \renewcommand{\theenumi}{\alph{enumi}}}

\item There exists $\gamma>0$ such that for all $B \in \B$ and $f_i \in L^{\infty}_c(\Sigma, \mu)$ with $\supp(f_i) \subset B$, $1 \leq i \leq m$, 
\begin{align*}
\mu\big(\big\{x \in B:  |T^{\A}(\vec{f})(x)| > t \, \M_{\B, r}(\vec{f})(x)  \big\}\big) 
\lesssim e^{- \gamma t}  \mu(B), \quad t>0.
\end{align*}

\item Assume that $A_{1, \B} \subset \bigcup_{s>1} RH_{s, \B}$. If $w \in A_{1, \B}$ and $v^{\frac{r}{m}} \in A_{\infty, \B}$, then 
\begin{align*}
\bigg\|\frac{T^{\A}(\vec{f})}{v}\bigg\|_{L^{\frac{r}{m},\infty}(\Sigma, \, w v^{\frac{r}{m}})}
&\lesssim \bigg\|\frac{\M_{\B, r}(\vec{f})}{v}\bigg\|_{L^{\frac{r}{m},\infty}(\Sigma, \, w v^{\frac{r}{m}})}.
\end{align*}

\item For all $\vec{p}=(p_1, \ldots, p_m)$ with $r<p_1, \ldots, p_m<\infty$ and for all $\vec{w} \in A_{\vec{p}/r, \B}$, 
\begin{align*}
&\|T^{\A}\|_{L^{p_1}(\Sigma, w_1) \times \cdots \times L^{p_m}(\Sigma, w_m) \to L^p(\Sigma, w)} 
\lesssim \mathcal{N}_1(r, \vec{p}, \vec{w}) 
[\vec{w}]_{A_{\vec{p}/r, \B}}^{\max\limits_{1 \le i \le m}\{p, (\frac{p_i}{r})'\}}, 
\end{align*}
where $\frac1p=\sum_{i=1}^m \frac{1}{p_i}$ and $w=\prod_{i=1}^m w_i^{\frac{p}{p_i}}$. If in addition $\B$ satisfies the Besicovitch condition, then 
 \begin{align*}
 \|T^{\A}\|_{L^{p_1}(\Sigma, w_1) \times \cdots \times L^{p_m}(\Sigma, w_m) \to L^p(\Sigma, w)} 
 \lesssim [\vec{w}]_{A_{\vec{p}/r, \B}}^{\max\limits_{1 \le i \le m}\{p, (\frac{p_i}{r})'\}}. 
 \end{align*}
 \end{list} 
\end{theorem}
%%%%%%%%%%%%%%%%%%%%%%% THEOREM THEOREM THEOREM %%%%%%%%%%%%%%%%%%%%%

Let us next present some examples of operators $T^{\A}$ in \eqref{def:TA}. 

\begin{example}\label{exm:mod}
Let $T$ be a multilinear bounded oscillation operator which is bounded from $L^r(\Sigma, \mu) \times \cdots \times L^r(\Sigma, \mu)$ to $L^{\frac{r}{m}, \infty}(\Sigma, \mu) < \infty$, where $1 \le r <\infty$. Let 
\[
\Lambda := \{\vec{\lambda}^{\alpha} = (\lambda_1^{\alpha}, \ldots, \lambda_m^{\alpha})\}_{\alpha \in \A} 
\subset L^{\infty}(\Sigma, \mu)^m  
\]
be a family of measurable functions such that 
\begin{align}\label{KA}
K_0 := \max_{1 \le k \le m} \sup_{\alpha \in \A} \|\lambda^{\alpha}_k\|_{L^{\infty}(\Sigma, \mu)}<\infty.
\end{align} 
Define the maximal modulation of the operator $T$ by 
\begin{equation}\label{def:TG}
T^{\Lambda}(\vec{f})(x) 
:= \sup_{\vec{\lambda}_{\alpha} \in \Lambda} |T(\lambda_1^{\alpha}f_1, \ldots, \lambda_m^{\alpha}f_m)(x)|
=: \sup_{\alpha \in \A} |T_{\alpha}(\vec{f})(x)|. 
\end{equation}
It is not hard to check that 
\begin{align}\label{CAK}
\C_1^{\A} \leq K_0 \C_1(T) 
\quad\text{ and }\quad 
\C_2^{\A} \leq K_0 \C_2(T).
\end{align}

A particular case of \eqref{KA} is 
\[
\lambda_k^{\alpha}(x) 
:= e^{2\pi i \phi_k^{\alpha}(x)}, \quad k=1, \ldots, m, 
\] 
where $\{\phi_k^{\alpha}\}_{\alpha \in \A}$ is a family of measurable real-valued functions indexed by an arbitrary set $\A$. As we have shown in Section \ref{sec:CZO}, the multilinear $\omega$-Calder\'{o}n-Zygmund operators are multilinear bounded oscillation operators. In this case, we extend the maximal modulated singular integral in \cite{GMS} to the multilinear setting. Taking $T$ as the Hilbert transform and $\phi_{\alpha}=\alpha x$ for $\alpha \in \R$, we obtain the {\tt Carleson operator} \cite{HY}: 
\begin{align*}
\mathscr{C} f(x) 
:= \sup_{\alpha \in \R} \bigg| \mathrm{p.v. } \int_{\R} \frac{e^{2\pi i \alpha y}}{x-y} f(y) \, dy \bigg|. 
\end{align*}
It also recovers the {\tt lacunary Carleson operator}
\begin{align*}
\mathscr{C}_{\A} f(x) 
:= \sup_{j \in \N} \bigg| \mathrm{p.v. } \int_{\R} \frac{e^{2\pi i \alpha_j y}}{x-y} f(y) \, dy \bigg|, 
\end{align*}
where $\A=\{\alpha_j\}_{j \in \N}$ is a lacunary sequence of integers, that is, $\alpha_{j+1} \ge \theta \alpha_j$ for all $j \in \N$ and for some $\theta > 1$. 

The last example is the {\tt polynomial Carleson operator}: 
\begin{align*}
\mathscr{C}_d f(x) 
:= \sup_{\mathrm{deg}(P) \le d} \bigg| \mathrm{p.v. } \int_{\mathbb{T}} \frac{e^{iP(y)}}{x-y} f(y) \, dy \bigg|, 
\end{align*}
where the supremum is taken over all real-coefficient polynomials $P$ of degree at most $d \in \N_+$. 

By Theorem \ref{thm:TA} and \eqref{CAK}, Theorem \ref{thm:TATA} can be applied to the operators $T^{\Lambda}$ in \eqref{def:TG}. Additionally, $T$ can be taken as specific multilinear bounded oscillation operators discussed in Sections \ref{sec:CZO}--\ref{sec:Calderon}.
\end{example}

Let us end up with weighted compactness for maximally modulated multilinear Calder\'{o}n-Zygmund operators. 
%%%%%%%%%%%%%%%%%%%%%%% THEOREM THEOREM THEOREM %%%%%%%%%%%%%%%%%%%%%
\begin{theorem}\label{thm:CZOmod}
Let $T$ be an $m$-linear $\w$-Calder\'{o}n-Zygmund operators with $\w \in \mathrm{Dini}$. Define $T^{\Lambda}$ as in  \eqref{def:TG} with $\Lambda := \{\vec{\lambda}^{\alpha} = (\lambda_1^{\alpha}, \ldots, \lambda_m^{\alpha})\}_{\alpha \in \A} \subset L^{\infty}(\Rn)^m$ being a family of measurable functions such that 
\begin{align*}
K_0 := \max_{1 \le k \le m} \sup_{\alpha \in \A} \|\lambda^{\alpha}_k\|_{L^{\infty}(\Rn)}<\infty.
\end{align*}
Assume that $T^{\Lambda}$ is bounded from $L^r(\Sigma, \mu) \times \cdots \times L^r(\Sigma, \mu)$ to $L^{\frac{r}{m}, \infty}(\Sigma, \mu)$ for some $r \in [1, \infty)$. Then for any $b \in \CMO(\Rn)$ and each $j=1, \ldots, m$, $[T^{\Lambda}, b]_{e_j}$ is compact from $L^{p_1}(\Rn, w_1) \times \cdots \times L^{p_m}(\Rn, w_m)$ to $L^p(\Rn, w)$ for all $\vec{p}=(p_1, \ldots, p_m)$ with $r<p_1, \ldots, p_m < \infty$, and for all $\vec{w} \in A_{\vec{p}/r}$, where $\frac1p = \sum_{i=1}^m \frac{1}{p_i}$ and $w=\prod_{i=1}^m w_i^{\frac{p}{p_i}}$.
\end{theorem}
%%%%%%%%%%%%%%%%%%%%%%% THEOREM THEOREM THEOREM %%%%%%%%%%%%%%%%%%%%%

%%%%%%%%%%%%%%%%%%%%%%%%% PROOF PROOF PROOF %%%%%%%%%%%%%%%%%%%%%%%%
\begin{proof}
By a rescaling argument, we may assume that $r=1$. Let $b \in \CMO(\Rn)$, $\frac1p=\sum_{i=1}^m \frac{1}{p_i}$ with $1<p_1, \ldots, p_m < \infty$, and $\vec{w} \in A_{\vec{p}}$. We only focus on the case $j=1$. As argued in Example \ref{exm:mod}, we have  
\begin{align}\label{CZOmod-1}
\|T^{\Lambda}\|_{L^{p_1}(\Rn, w_1) \times \cdots \times L^{p_m}(\Rn, w_m) \to L^p(\Rn, w)} 
&\lesssim [\vec{w}]_{A_{\vec{p}}}^{\max\limits_{1 \le i \le m}\{p, p'_i \}}, 
\\
\label{CZOmod-2}
\|[T^{\Lambda}, b]_{e_1}\|_{L^{p_1}(\Rn, w_1) \times \cdots \times L^{p_m}(\Rn, w_m) \to L^p(\Rn, w)} 
&\lesssim \|b\|_{\BMO} [\vec{w}]_{A_{\vec{p}}}^{2\max\limits_{1 \le i \le m}\{p, p'_i\}}. 
\end{align}
As before, considering \eqref{CZOmod-2} and \eqref{CCC}, we may assume that $b \in \mathscr{C}_c^{\infty}(\Rn)$ with $\supp(b) \subset B(0, A_0)$ for some $A_0 \ge 1$. 

Given $\vec{f}=(f_1, \ldots, f_m)$ and $\vec{\lambda}^{\alpha} \in \Lambda$, we always denote 
\begin{align*}
\vec{\f}^{\alpha} := (\f_1^{\alpha}, \ldots, \f_m^{\alpha}) 
:= (\lambda_1^{\alpha}f_1, \ldots, \lambda_m^{\alpha}f_m). 
\end{align*}
Define 
\begin{align*}
T_*(\vec{f})(x) 
:=\sup_{\eta>0} |T_{\eta}(\vec{f})(x)| 
\quad\text{and}\quad 
T_*^{\Lambda}(\vec{f})(x) 
:=\sup_{\eta>0} |T_{\eta}^{\Lambda}(\vec{f})(x)|, 
\end{align*}
where 
\begin{align*}
T_{\eta}(\vec{f})(x) 
&:= \int_{\sum_{i=1}^m |x-y_i|>2\eta} K(x, \vec{y}) \prod_{i=1}^m f_i(y_i) \, d\vec{y}, 
\\
T_{\eta}^{\Lambda}(\vec{f})(x) 
&:= \sup_{\vec{\lambda}^{\alpha} \in \Lambda} \bigg|\int_{\sum_{i=1}^m |x-y_i|>2\eta} K(x, \vec{y}) 
\prod_{i=1}^m (\lambda_i^{\alpha} f_i)(y_i) \, d\vec{y} \bigg|.  
\end{align*}
By definition, 
\begin{align*}
[T^{\Lambda}, b]_{e_1}(\vec{f})(x) 
&=\sup_{\vec{\lambda}^{\alpha} \in \Lambda} 
\big|[T, b]_{e_1}(\vec{\f}^{\alpha})(x) \big|, 
\\
[T_{\eta}^{\Lambda}, b]_{e_1}(\vec{f})(x) 
&=\sup_{\vec{\lambda}^{\alpha} \in \Lambda} 
\big|[T_{\eta}, b]_{e_1}(\vec{\f}^{\alpha})(x) \big|. 
\end{align*}

It follows from \eqref{CZOmod-2} that 
\begin{align}\label{modFK-1}
\sup_{\|f_i\|_{L^{p_i}(\Rn, w_i)} \le 1 \atop i=1, \ldots, m} 
\big\|[T^{\Lambda}, b]_{e_1}](\vec{f})\big\|_{L^p(\Rn, w)} 
\lesssim 1. 
\end{align}
Let $A>2A_0$ and $|x|>A$. By definition and size condition \eqref{eq:size}, 
\begin{align}\label{Telm}
\big\|&[T^{\Lambda}, b]_{e_1}(\vec{f}) \mathbf{1}_{\{|x|>A\}}\big\|_{L^p(\Rn, w)}
\nonumber \\
&\lesssim \bigg[\sum_{k =0}^{\infty} \int_{2^k A <|x| \le 2^{k+1} A} \bigg(\sum_{j_2, \ldots, j_m=0}^{\infty} 
\mathscr{T}_{j_2, \ldots, j_m}(\vec{f})(x) \bigg)^p w(x) dx \bigg]^{\frac1p}, 
\end{align}
where 
\begin{align*}
\mathscr{T}_{j_2, \ldots, j_m}(\vec{f})(x)
&:= \int_{B(0, A_0) \times R_{j_2}^x \times \cdots \times R_{j_m}^x} 
\frac{\prod_{i=1}^m |f_i(y_i)|}{(\sum_{i=1}^m |x-y_i|)^{mn}} d\vec{y},  
\end{align*}
with $R_{j_i}^x := \{y_i \in \Rn: 2^{j_i-1} |x| \le |y_i| < 2^{j_i} |x|\}$ if $j_i \ge 1$, and $R_{j_i}^x:=\{y_i \in \Rn: |y_i| < |x|\}$ if $j_i=0$. Let us analyze $\mathscr{T}_{j_2, \ldots, j_m}$. Given $j_2, \ldots, j_m \ge 0$, denote $j_* :=\max\{j_2, \ldots, j_m\}$. Then for all $\vec{y} \in B(0, A_0) \times R_{j_2}^x \times \cdots \times R_{j_m}^x$, we have $\sum_{i=1}^m |x-y_i| \simeq 2^{j_*} |x|$ and 
\begin{align}\label{Sjf}
& \mathscr{T}_{j_2, \ldots, j_m}(\vec{f})(x)
\le (2^{j_*} |x|)^{-mn} \prod_{i=1}^m \|f_i\|_{L^{p_i}(\Rn, w_i)}  
\nonumber \\
&\qquad\quad\times 
\bigg(\int_{B(0, A_0)} w_1^{1-p'_1} dx\bigg)^{\frac{1}{p'_1}}
\bigg(\int_{B(0, 2^{j_i} |x|)} w_i^{1-p'_i} dx\bigg)^{\frac{1}{p'_i}}
\nonumber \\
&\lesssim (2^{k+j_*} A)^{-n\theta/p'_1} (2^{j_*} |x|)^{-mn} \prod_{i=1}^m \|f_i\|_{L^{p_i}(\Rn, w_i)} 
\nonumber \\
&\quad\times 
\bigg(\int_{B(0, 2^{k+j_*}A)} w_1^{1-p'_1} dx\bigg)^{\frac{1}{p'_1}} 
\bigg(\int_{B(0, 2^{j_i} |x|)} w_i^{1-p'_i} dx\bigg)^{\frac{1}{p'_i}}, 
\end{align}
where we have used that $w_1^{1-p'_1} \in A_{\infty}$ and the fact that for any $v \in A_{\infty}$, there exists $\theta>0$ so that $v(E)/v(B) \lesssim (|E|/|B|)^{\theta}$ for every ball $B \subset \Rn$ and every measurable set $E \subset B$ (cf. \cite[Theorem 7.3.3]{Gra-1}). Then \eqref{Sjf} implies 
\begin{align}\label{Sjkk}
&\mathscr{S}_{j_2, \ldots, j_m}
:=\bigg(\int_{2^k A< |x| \le 2^{k+1} A} 
\big|\mathscr{T}_{j_2, \ldots, j_m}(\vec{f})(x)\big|^p w(x)dx \bigg)^{\frac1p}
\nonumber\\
&\lesssim (2^{k+j_*} A)^{-n\theta/p'_1} 
\prod_{i=1}^m \|f_i\|_{L^{p_i}(\Rn, w_i)} 
\bigg(\fint_{B(0, 2^{k+j_*+1} A)} w \, dx \bigg)^{\frac1p} 
\nonumber\\
&\qquad\times \prod_{i=1}^m \bigg(\fint_{B(0, 2^{k+j_*+1}A)} w_i^{1-p'_i} dx\bigg)^{\frac{1}{p'_i}}
\nonumber\\
&\lesssim (2^{k+j_*} A)^{-n \theta/p'_1} [\vec{w}]_{A_{\vec{p}}} 
\prod_{i=1}^m \|f_i\|_{L^{p_i}(\Rn, w_i)}.  
\end{align}
Consequently, \eqref{Telm} and \eqref{Sjkk} give that when $p<1$, 
\begin{align*}
\big\|[T^{\Lambda}, b]_{e_1}(\vec{f}) \mathbf{1}_{\{|x|>A\}}\big\|_{L^p(\Rn, w)}
&\lesssim \bigg(\sum_{k, j_2, \ldots, j_m=0}^{\infty} |\mathscr{S}_{j_2, \ldots, j_m}|^p\bigg)^{\frac1p}
\\
&\lesssim A^{-n \theta/p'_1} \prod_{i=1}^m \|f_i\|_{L^{p_i}(\Rn, w_i)},  
\end{align*}
and when $p \ge 1$, 
\begin{align*}
\big\|[T^{\Lambda}, b]_{e_1}(\vec{f}) \mathbf{1}_{\{|x|>A\}}\big\|_{L^p(\Rn, w)}
&\lesssim \sum_{k, j_2, \ldots, j_m=0}^{\infty} \mathscr{S}_{j_2, \ldots, j_m}
\\
&\lesssim A^{-n \theta/p'_1} \prod_{i=1}^m \|f_i\|_{L^{p_i}(\Rn, w_i)}. 
\end{align*}
This immediately implies 
\begin{align}\label{modFK-2}
\lim_{A \to \infty} \sup_{\|f_i\|_{L^{p_i}(\Rn, w_i)} \le 1 \atop i=1, \ldots, m} 
\big\|[T^{\Lambda}, b]_{e_1}](\vec{f})\mathbf{1}_{\{|x|>A\}} \big\|_{L^p(\Rn, w)} 
=0. 
\end{align}

On the other hand, observe that 
\begin{align}\label{Cotlar}
T_*^{\Lambda}(\vec{f})(x) 
=\sup_{\vec{\lambda}^{\alpha} \in \Lambda} T_*(\vec{\f}^{\alpha})(x) 
\lesssim K_0^m \M(\vec{f})(x) + M_s(T^{\Lambda}(\vec{f}))(x),  
\end{align}
for any $s>0$, which is a sequence of the Cotlar inequality in \cite[Lemma 2.1]{Chen} adapted to our case: 
\begin{align*}
T_*(\vec{f})(x) 
\lesssim \M(\vec{f})(x) + M_s(T(\vec{f}))(x), \quad s>0. 
\end{align*}
By \eqref{Tbxx} and \eqref{Cotlar}, 
\begin{align}\label{TTM}
\big|[&T^{\Lambda}, b]_{e_1}(\vec{f})(x) 
- [T^{\Lambda}, b]_{e_1}](\vec{f})(x') \big|
\nonumber \\ 
&\le \sup_{\vec{\lambda}^{\alpha} \in \Lambda} 
\big|[T, b]_{e_1}(\vec{\f}^{\alpha})(x) - [T, b]_{e_1}](\vec{\f}^{\alpha})(x')\big| 
\nonumber \\ 
&\lesssim \varepsilon \sup_{\vec{\lambda}^{\alpha} \in \Lambda} 
\big[T_*(\vec{\f}^{\alpha})(x) + \M(\vec{\f}^{\alpha})(x) + \M(\vec{\f}^{\alpha})(x') \big] 
\nonumber \\ 
&\le \varepsilon K_0^m \M(\vec{f})(x) + \varepsilon K_0^m \M(\vec{f})(x') + \varepsilon M_s(T^{\Lambda}(\vec{f}))(x). 
\end{align}
Taking $0<s<1/m$ and $p/s<p_0<\infty$, and using $w \in A_{mp} \subset A_{p/s} \subset A_{p_0}$ and \eqref{CZOmod-1}, we have 
\begin{align}
\big\|M \big(|\M(\vec{f})|^{\frac{p}{p_0}} \big)\big\|_{L^{p_0}(\Rn, w)}^{\frac{p_0}{p}} 
\lesssim  \|\M(\vec{f})\|_{L^p(\Rn, w)} 
\lesssim \prod_{i=1}^m \|f_i\|_{L^{p_i}(\Rn, w_i)}, 
\end{align}
and 
\begin{align}\label{Meta}
\|M_s(T^{\Lambda}(\vec{f}))\|_{L^p(\Rn, w)}
&=\|M(|T^{\Lambda}(\vec{f}))|^s\|_{L^{\frac{p}{s}}(\Rn, w)}^{\frac{1}{s}}
\nonumber \\
&\lesssim \|T^{\Lambda}(\vec{f})\|_{L^p(\Rn, w)} 
\lesssim \prod_{i=1}^m \|f_i\|_{L^{p_i}(\Rn, w_i)}. 
\end{align}
Then, \eqref{TTM} and \eqref{Meta} yield that 
\begin{align*}
&\bigg[\int_{\Rn} \bigg(\fint_{B(x, r)} |[T^{\Lambda}, b]_{e_1}(\vec{f})(x) 
- [T^{\Lambda}, b]_{e_1}(\vec{f})(x')|^{\frac{p}{p_0}} dx' \bigg)^{p_0} w(x) \, dx \bigg]^{\frac1p}
\\ 
&\lesssim \varepsilon \|\M(\vec{f})\|_{L^p(\Rn, w)} 
+ \varepsilon \big\|M \big(|\M(\vec{f})|^{\frac{p}{p_0}} \big)\big\|_{L^{p_0}(\Rn, w)}^{\frac{p_0}{p}} 
\\
&\qquad\qquad+ \varepsilon \|M_s(T^{\Lambda}(\vec{f}))\|_{L^p(\Rn, w)} 
\\
&\lesssim \varepsilon \prod_{i=1}^m \|f_i\|_{L^{p_i}(\Rn, w_i)}, 
\end{align*}
which gives  
\begin{align}\label{modFK-3}
\lim_{r \to 0} \sup_{\|f_i\|_{L^{p_i}(\Rn, w_i)} \le 1 \atop i=1, \ldots, m} 
\bigg[\int_{\Rn} &\bigg(\fint_{B(x, r)} |[T^{\Lambda}, b]_{e_1}(\vec{f})(x) 
\nonumber \\
&- [T^{\Lambda}, b]_{e_1}(\vec{f})(x')|^{\frac{p}{p_0}} dx' \bigg)^{p_0} w(x) dx \bigg]^{\frac1p}
=0. 
\end{align}
Therefore, it follows from Theorem \ref{thm:FKhs-2}, \eqref{modFK-1}, \eqref{modFK-2}, and \eqref{modFK-3} that $[T^{\Lambda}, b]_{e_1}$ is compact from $L^{p_1}(\Rn, w_1) \times \cdots \times L^{p_m}(\Rn, w_m)$ to $L^p(\Rn, w)$. 
\end{proof}
%%%%%%%%%%%%%%%%%%%%%%%%% END END END PROOF %%%%%%%%%%%%%%%%%%%%%%%%

%%%%%%%%%%%%%%%%%%%%% SUBSECTION SUBSECTION SUBSECTION %%%%%%%%%%%%%%%%%%
%%%%%%%%%%%%%%%%%%%%% SUBSECTION SUBSECTION SUBSECTION %%%%%%%%%%%%%%%%%%
\subsection{$q$-variation of $\w$-Calder\'{o}n-Zygmund operators}\label{sec:var}
Given $q>1$ and a family of real numbers $\mathbf{a} := \{a_{\eta}\}_{\eta>0}$, the $q$-variation norm of $\mathbf{a}$ is defined by
\begin{align*}
\|\mathbf{a}\|_{\bB_q}
=\sup_{\{\eta_j\} \searrow 0} \bigg(\sum_{j=0}^{\infty}
\big|a_{\eta_j} - a_{\eta_{j+1}} \big|^q \bigg)^{\frac1q},
\end{align*}
where the supremum runs over all positive sequences $\{\eta_j\}$ decreasing to zero. Note that $\|\mathbf{a}-c_0\|_{\bB_q} = \|\mathbf{a}\|_{\bB_q}$ for any constant $c_0$, that is, $\|\cdot\|_{\bB_q}$ with $q>1$ is a semi-norm. Denote 
\begin{align*}
\bB_q &:=\{\mathbf{a} :=\{a_{\eta}\}_{\eta>0}: \|\mathbf{a}\|_{\bB_q}<\infty\},  
\\
\widetilde{\bB}_q &:=\{[\mathbf{a}]: \|[\mathbf{a}]\|_{\widetilde{\bB}_q} := \|\mathbf{a}\|_{\bB_q} <\infty\},
\end{align*} 
where the equivalence class is given by $[\mathbf{a}] :=\{\mathbf{a}+c_0: c_0 \in \R\}$. For any $q>1$, $\bB_q$ is not a Banach space, but $\widetilde{\bB}_q$ is. In what follows, for convenience, we always use $\bB_q$ instead of $\widetilde{\bB}_q$.

Let $T$ be an $\w$-Calder\'{o}n-Zygmund operator on $\Rn$ (cf. Definition \ref{def:wCZO}), and let $K$ be the kernel of $T$. 
Given $\eta>0$, we define the hard truncation of $T$ by 
\begin{equation*}
T_{\eta}f(x) :=\int_{|x-y| >\eta} K(x, y) f(y) \, dy.
\end{equation*}
By means of the above notation, we denote
\begin{align*}
&\T f(x) := \big\{T_{\eta}f(x) \big\}_{\eta > 0}, \quad 
\T_b f(x) := \big\{[T_{\eta}, b]f(x) \big\}_{\eta > 0}, 
\\
&\text{and} \quad
\K(x, y) :=\big\{K_{\eta}(x, y) := K(x, y) \mathbf{1}_{\{|x-y|>\eta\}}\big\}_{\eta>0}.
\end{align*}
For $1 \le q < \infty$, the $q$-variation of $T$ and $[T, b]$ are defined as
\begin{align*}
\V_q(\T f)(x) :=\| \T f \|_{\bB_q}
\quad\text{and}\quad 
\V_q(\T_b f)(x) :=\| \T_b f(x) \|_{\bB_q}.
\end{align*}

We are going to work in the space $(\Rn, \Ln)$. Let $\B$ be the collection of all cubes in $\Rn$, and set $Q^*=5Q$ for any $Q \in \B$. As in Example \ref{example}, $\B$ is a ball-basis of $(\Rn, \Ln)$.

%%%%%%%%%%%%%%%%%%%%%%% THEOREM THEOREM THEOREM %%%%%%%%%%%%%%%%%%%%%
\begin{theorem}\label{thm:Var}
Let $q>2$ and $T$ be an $\w$-Calder\'{o}n-Zygmund operator with $\w \in \mathrm{Dini}$. Then $\V_q \circ \T$ is a $\bB_q$-valued bounded oscillation operator with respect to $\B$ and $r \in (1, q]$, with constants $\C_1(\V_q \circ \T) \lesssim C_K$ and $\C_2(\V_q \circ \T) \lesssim 1+\|\w\|_{\mathrm{Dini}}$.
\end{theorem}
%%%%%%%%%%%%%%%%%%%%%%% THEOREM THEOREM THEOREM %%%%%%%%%%%%%%%%%%%%%

%%%%%%%%%%%%%%%%%%%%%%% PROOF PROOF PROOF %%%%%%%%%%%%%%%%%%%%%%%%%%
\begin{proof}
Let $Q_0 \in \B$, $x \in Q_0$, and choose $Q:= 2Q_0$. Observe that 
\begin{align*}
\|\K(x, y)\|_{\bB_q}
&=\sup_{\{\eta_j\} \searrow 0} \bigg(\sum_{j=0}^{\infty}
\big|K_{\eta_j}(x, y) - K_{\eta_{j+1}}(x, y) \big|^q \bigg)^{\frac1q}
\\
&\le |K(x, y)| \sup_{\{\eta_j\} \searrow 0} \sum_{j=0}^{\infty} 
\mathbf{1}_{\{\eta_{j+1} < |x-y| \le \eta_j \}} 
\le |K(x, y)|, 
\end{align*}
which along with the size condition \eqref{eq:size} gives 
\begin{align*}
\|&\T(f \mathbf{1}_{Q^*})(x) - \T(f \mathbf{1}_{Q_0^*})(x)\|_{\bB_q}
=\bigg\|\int_{Q^* \setminus Q_0^*} \K(x, y) f(y) \, dy \bigg\|_{\bB_q}
\\ 
&\le \int_{Q^* \setminus Q_0^*} \|\K(x, y)\|_{\bB_q} |f(y)| \, dy 
\le C_K \int_{Q^* \setminus Q_0^*} \frac{|f(y)|}{|x-y|^n} \, dy 
\lesssim C_K \langle f \rangle_{Q^*}. 
\end{align*}
This shows the condition \eqref{list:T-size}. 

To demonstrate the condition \eqref{list:T-reg}, we let $Q \in \B$ and $x, x' \in Q$. 
\begin{align*}
\|&(\T f - \T(f \mathbf{1}_{Q^*}))(x) - (\T f - \T(f \mathbf{1}_{Q^*}))(x')\|_{\bB_q}
\\
&\le \sum_{k=1}^{\infty} \bigg\|\int_{2^k Q \setminus 2^{k-1} Q} (\K(x, y) - \K(x', y)) f(y) \, dy \bigg\|_{\bB_q}
=: \sum_{k=1}^{\infty} \mathscr{G}_k,  
\end{align*}
Set $f_k := f \mathbf{1}_{2^k Q \setminus 2^{k-1} Q}$. Note that 
\begin{align}\label{GGkk}
\mathscr{G}_k 
\le \mathscr{G}_k^1 + \mathscr{G}_k^2,   
\end{align}
where 
\begin{align*}
\mathscr{G}_k^1 
&:=\sup_{\{\eta_j\} \searrow 0} \bigg(\sum_{j=0}^{\infty} \bigg|\int_{\Rn}
\big(K(x, y) - K(x', y)) \phi_j^1(y) f_k(y) \, dy\bigg|^q \bigg)^{\frac1q}
\\
&\text{with } \phi_j^1(y) = \mathbf{1}_{\{\eta_{j+1}<|x-y| \le \eta_j\}}, 
\end{align*}
and 
\begin{align*}
\mathscr{G}_k^2 
&:=\sup_{\{\eta_j\} \searrow 0} \bigg(\sum_{j=0}^{\infty} \bigg|\int_{\Rn}
K(x', y) \phi_j^2(y) f_k(y) \, dy\bigg|^q \bigg)^{\frac1q} 
\\
&\text{with } \phi_j^2(y) = \mathbf{1}_{\{\eta_{j+1}<|x-y| \le \eta_j\}} - \mathbf{1}_{\{\eta_{j+1}<|x'-y| \le \eta_j\}}. 
\end{align*}
For the first term, it follows from the smoothness condition \eqref{eq:smooth-1} that 
\begin{align}\label{GGkk-1}
\mathscr{G}_k^1 
\le \int_{\Rn} \big|K(x, y) - K(x', y)| |f_k(y)| \, dy
\lesssim \w(2^{-k}) \langle f \rangle_{2^k Q}. 
\end{align}
Use the techniques in \cite{DLY, MTX}, one can get that 
\begin{align}\label{GGkk-2}
\mathscr{G}_k^2 
&\lesssim \ell(Q)^{\frac{1}{r'}} \bigg(\int_{\Rn} \frac{|f_k(y)|^r}{|x-y|^{r+n-1}} dy \bigg)^{\frac1r}
\nonumber \\ 
&\simeq \ell(Q)^{\frac{1}{r'}} \bigg(\frac{(2^k \ell(Q))^n}{(2^{k-1}\ell(Q))^{r+n-1}} \fint_{2^k Q} |f(y)|^r dy \bigg)^{\frac1r}
\nonumber \\
&\lesssim 2^{-k/r'} \langle f \rangle_{2^k Q, r}. 
\end{align}
Now collecting \eqref{GGkk}--\eqref{GGkk-2}, we have 
\begin{align*}
&\|(\T f - \T(f \mathbf{1}_{Q^*}))(x) - (\T f - \T(f \mathbf{1}_{Q^*}))(x')\|_{\bB_q}
\lesssim (1+\|\w\|_{\mathrm{Dini}}) \lfloor f \rfloor_{Q, r}, 
\end{align*}
which shows $\C_2(\V_q \circ \T) \lesssim 1 + \|\w\|_{\mathrm{Dini}}$. 
\end{proof}
%%%%%%%%%%%%%%%%%%%%%%% END END END PROOF %%%%%%%%%%%%%%%%%%%%%%%%%%

If we assume that $\V_q \circ \T$ is bounded on $L^{q_0}(\Rn)$ for some $1 < q_0 < \infty$, then by \cite[Proposition 2.1]{DLY}, 
\begin{align}\label{Var-weak}
\text{$\V_q \circ \T$ is bounded from $L^1(\Rn)$ to $L^{1, \infty}(\Rn)$.} 
\end{align}
Hence, in view of Theorem \ref{thm:Var} and \eqref{Var-weak}, Theorems \ref{thm:local}--\ref{thm:T} imply the following results. 

%%%%%%%%%%%%%%%%%%%%%%% THEOREM THEOREM THEOREM %%%%%%%%%%%%%%%%%%%%%
\begin{theorem}\label{thm:Var-local}
Let $q>2$, $1<r \le q$, and $T$ be an $\w$-Calder\'{o}n-Zygmund operator with $\w \in \mathrm{Dini}$. Assume that $\V_q \circ \T$ is bounded on $L^{q_0}(\Rn)$ for some $q_0 \in (1, \infty)$. Then the following hold: 
\begin{list}{\rm (\theenumi)}{\usecounter{enumi}\leftmargin=1.2cm \labelwidth=1cm \itemsep=0.2cm \topsep=.2cm \renewcommand{\theenumi}{\alph{enumi}}}

\item There exists $\gamma>0$ such that for all $Q \in \B$ and $f_i \in L^{\infty}_c(\Rn)$ with $\supp(f) \subset Q$, 
\begin{align*}
\big|\big\{x \in Q:  |\V_q (\T f)(x)| > t \, M_r f(x)  \big\}\big| 
&\lesssim e^{- \gamma t}  |Q|, 
\\ 
\big|\big\{x \in Q:  |\V_q(\T_b f)(x)| > t \, M_r f^*(x)  \big\}\big|
&\lesssim e^{-(\frac{\gamma t}{\|b\|_{\BMO}})^{\frac12}} |Q|, 
\end{align*}
for all $t>0$, where $f^*=M^{\lfloor r \rfloor}(|f|^r)^{\frac1r}$.
\item If $w \in A_1$ and $v^r \in A_{\infty}$, then 
\begin{align*}
\bigg\|\frac{\V_q(\T f)}{v}\bigg\|_{L^{r, \infty}(\Rn, \, w v^r)}
\lesssim \|f\|_{L^r(\Rn, \, w)}. 
\end{align*}

\item For all $r<p<\infty$ and for all $w \in A_{p/r}$, 
\begin{align*}
\|\V_q \circ \T\|_{L^p(\Rn, w) \to L^p(\Rn, w)} 
&\lesssim [w]_{A_{p/r}}^{\max\{p, (\frac{p}{r})'\}}, 
\\
\|\V_q \circ \T_b\|_{L^p(\Rn, w) \to L^p(\Rn, w)} 
&\lesssim \|b\|_{\BMO} [w]_{A_{p/r}}^{2\max\{p, (\frac{p}{r})'\}}. 
\end{align*}
 \end{list} 
\end{theorem}
%%%%%%%%%%%%%%%%%%%%%%% THEOREM THEOREM THEOREM %%%%%%%%%%%%%%%%%%%%%

%%%%%%%%%%%%%%%%%%%%%%% THEOREM THEOREM THEOREM %%%%%%%%%%%%%%%%%%%%%
\begin{theorem}\label{thm:Var-comp}
Let $q>2$, $1<r \le q$, and $T$ be an $\w$-Calder\'{o}n-Zygmund operator with $\w \in \mathrm{Dini}$. Assume that $\V_q \circ \T$ is bounded on $L^{q_0}(\Rn)$ for some $q_0 \in (1, \infty)$. If $b \in \CMO(\Rn)$, then $\V_q \circ \T_b$ is compact on $L^p(\Rn, w)$ for all $r<p < \infty$ and for all $w \in A_{p/r}$.
\end{theorem}
%%%%%%%%%%%%%%%%%%%%%%% THEOREM THEOREM THEOREM %%%%%%%%%%%%%%%%%%%%%

%%%%%%%%%%%%%%%%%%%%%%% PROOF PROOF PROOF %%%%%%%%%%%%%%%%%%%%%%%%%%
\begin{proof}
Let $r<p<\infty$, $w \in A_{p/r}$, and $b \in \CMO(\Rn)$. Then by Theorem \ref{thm:Var-local}, 
\begin{align}
\label{VV-1}
\|\V_q \circ \T\|_{L^p(\Rn, w) \to L^p(\Rn, w)} 
& \lesssim [w]_{A_{p/r}}^{\max\{p, (\frac{p}{r})'\}}, 
\\
\label{VV-2}
\|\V_q \circ \T_b\|_{L^p(\Rn, w) \to L^p(\Rn, w)} 
& \lesssim \|b\|_{\BMO} [w]_{A_{p/r}}^{2\max\{p, (\frac{p}{r})'\}}. 
\end{align} 
In view of \eqref{VV-2}, we may assume that $b\in \mathscr{C}_c^{\infty }(\Rn)$ with $\supp b \subset B(0, A_0)$ for some $A_0 \ge 1$. Let $\varphi \in \mathscr{C}_c^{\infty}(\Rn)$ be a nonnegative and radial function satisfying $\mathbf{1}_{B(0, 2)} \le \varphi \le \mathbf{1}_{B(0, 4)}$. Given $\tau>0$, we denote 
\begin{align*}
K^{\tau}(x, y) &:= K(x, y) [1-\varphi(\tau^{-1} |x-y|)], 
\\ 
T_{\eta}^{\tau}f(x) &:= \int_{|x-y|>\eta} K^{\tau}(x, y) f(y) \, dy, 
\\ 
\T^{\tau} f(x) &:= \big\{T_{\eta}^{\tau}f(x) \big\}_{\eta > 0}, 
\\
\T_b^{\tau} f(x) &:= \big\{[T_{\eta}^{\tau}, b]f(x) \big\}_{\eta > 0}. 
\end{align*}
Note that $\supp(K^{\tau}(x, y)) \subset \{|x-y|>2\tau\}$ and 
\begin{align}\label{suppK}
K^{\tau}(x, y) = K(x, y) \text{ for all } |x-y|>4\tau. 
\end{align}
By definition, we may write  
\begin{equation}\label{KKK}
\begin{aligned}
&K^{\tau}= K^{\tau, 1} + K^{\tau, 2}, 
\\
&\text{where $K^{\tau, 1}$ an $\w$-Calder\'{o}n-Zygmund kernel} 
\\
&\text{and $K^{\tau, 2}$ is a standard Calder\'{o}n-Zygmund kernel}. 
\end{aligned}
\end{equation}
Hence, \eqref{VV-1} and \eqref{KKK} imply 
\begin{align}\label{VV-3}
\|\V_q (\T^{\tau} f)\|_{L^p(\Rn, w)}
\le \sum_{i=1}^2 \|\V_q (\T^{\tau, i} f)\|_{L^p(\Rn, w)} 
\lesssim \|f\|_{L^p(\Rn, w)}. 
\end{align}
On the other hand, for any $\tau>0$ and $x \in \Rn$, if we write $\Delta_K(y) := K(x, y) - K^{\tau}(x, y)$, then it follows from size condition \eqref{eq:size} and \eqref{kf-2} that 
\begin{align*}
\big|&\V_q(\T_b f)(x) - \V_q(\T_b^{\tau} f)(x)\big|
\\
&\le \sup_{\{\eta_j\} \searrow 0} \bigg(\sum_{j=0}^{\infty} \bigg|
\int_{\{\eta_{j+1}< |x-y| \le \eta_j\}} (b(x)-b(y)) \Delta_K(y) f(y) \, dy \bigg|^q \bigg)^{\frac1q}
\\
&\lesssim \|\nabla b\|_{L^{\infty}(\Rn)} \sup_{\{\eta_j\} \searrow 0} \sum_{j=0}^{\infty} 
\int_{|x-y|< 4 \tau \atop \eta_{j+1}< |x-y| \le \eta_j} \frac{|f(y)|}{|x-y|^{n-1}} \, dy 
\\
&\lesssim \tau \, Mf(x)
\le \tau \, M_r f(x), 
\end{align*}
which in turn gives 
\begin{align}\label{VV-4}
\lim_{\tau \to 0} \sup_{\|f\|_{L^p(\Rn, w)} \le 1} \big\|\V_q(\T_b f) - \V_q(\T_b^{\tau} f)\big\|_{L^p(\Rn, w)}
=0. 
\end{align}
This means that it suffices to show that given $\tau>0$, 
\begin{align}\label{VTU-1}
\text{$\V_q \circ \T_b^{\tau}$ is compact on $L^p(\Rn, w)$}.
\end{align}
Since $\V_q \circ \T_b^{\tau}$ is linearizable, in light of Theorem \ref{thm:Tb} and \eqref{VV-3}, \eqref{VTU-1} is reduced to showing that 
\begin{align}\label{VTU-2}
\text{$\V_q \circ \T_b^{\tau}$ is compact on $L^p(\Rn)$}. 
\end{align}

It remains to justify \eqref{VTU-2}. By \eqref{VV-2} and \eqref{VV-4}, 
\begin{align}\label{VFK-1}
\sup_{\|f\|_{L^p(\Rn)} \le 1} \|\V_q(\T_b^{\tau} f)\|_{L^p(\Rn)} 
\lesssim 1. 
\end{align}
Let $\varepsilon \in (0,1)$. Then for any $|x|>A \ge 2A_0$, the size condition \eqref{eq:size} and \eqref{kf-3} give 
\begin{align*}
\V_q(\T_b^{\tau} f)(x)
&\le \sup_{\{\eta_j\} \searrow 0} \sum_{j=0}^{\infty} 
\big|[T_{\eta_j}^{\tau}, b]f(x) - [T_{\eta_{j+1}}^{\tau}, b]f(x)\big| 
\nonumber \\
&\le \sup_{\{\eta_j\} \searrow 0} \sum_{j=0}^{\infty} 
\int_{\eta_{j+1}<|x-y| \le \eta_j \atop y \in B(0, A_0)} |b(x) - b(y)| |K^{\tau}(x, y)| |f(y)| \, dy 
\nonumber \\
&\lesssim \|b\|_{L^{\infty}(\Rn)} \int_{B(0, A_0)} \frac{|f(y)|}{|x-y|^n} dy
\lesssim |x|^{-n} A_0^{\frac{n}{p'}} \|f\|_{L^p(\Rn)}, 
\end{align*}
which along with $A > \max\{2A_0, \varepsilon^{-p'/n}\}$ implies 
\begin{align}\label{VFK-2}
\lim_{A \to \infty} \sup_{\|f\|_{L^p(\Rn)} \le 1}
\big\|\V_q(\T_b^{\tau} f) \mathbf{1}_{\{|x|>A\}} \big\|_{L^p(\Rn)} 
=0. 
\end{align}

To proceed, let $0<s<\tau$, and fix $x \in \Rn$ and $x' \in B(x, s)$. We split
\begin{align}\label{VHH}
|\V_q(\T_b^{\tau} f)(x) - \V_q(\T_b^{\tau} f)(x')| 
\le \mathscr{H}_1 + \mathscr{H}_2, 
\end{align}
where 
\begin{align*}
\mathscr{H}_1 
&:=\sup_{\{\eta_j\} \searrow 0} \bigg(\sum_{j=0}^{\infty} \bigg|\int_{\Rn}
(b(x')-b(y)) K^{\tau}(x', y) 
\\
&\qquad\qquad\qquad\quad\times \big(\mathbf{1}_{\{\eta_{j+1}<|x-y| \le \eta_j\}} - \mathbf{1}_{\{\eta_{j+1}<|x'-y| \le \eta_j\}} \big) 
f(y) \, dy\bigg|^q \bigg)^{\frac1q}, 
\\
\mathscr{H}_2 
&:=\sup_{\{\eta_j\} \searrow 0} \bigg(\sum_{j=0}^{\infty} \bigg|\int_{\Rn}
\big[(b(x)-b(y))K^{\tau}(x, y) - (b(x')-b(y)) K^{\tau}(x', y) \big] 
\\
&\qquad\qquad\qquad\qquad\times \mathbf{1}_{\{\eta_{j+1}<|x-y| \le \eta_j\}} f(y) \, dy\bigg|^q \bigg)^{\frac1q}.
\end{align*}
Much as in the first estimate in \eqref{GGkk-2}, it follows from \eqref{suppK} that  
\begin{align}\label{VHH-1}
\mathscr{H}_1 
&\lesssim s^{\frac{1}{r'}} \bigg(\int_{|y-x'|>2\tau} \frac{|(b(x')-b(y))f(y)|^r}{|x-y|^{r+n-1}} dy \bigg)^{\frac1r}
\nonumber \\ 
&\lesssim s^{\frac{1}{r'}} \bigg(\int_{|y-x|>\tau} \frac{|f(y)|^r}{|x-y|^{r+n-1}} dy \bigg)^{\frac1r}
\lesssim s^{\frac{1}{r'}} M_r f(x). 
\end{align}
For $\mathscr{H}_2$, by the fact that $\supp(K^{\tau}(x, y)) \cup \supp(K^{\tau}(x', y)) \subset \{|x-y|>\tau\}$, we have 
\begin{align}\label{VHH-2}
\mathscr{H}_2 
\le \mathscr{H}_{2, 1} + \mathscr{H}_{2, 2},   
\end{align}
where 
\begin{align*}
\mathscr{H}_{2,1} 
&:= |b(x) -b(x')| \sup_{\{\eta_j\} \searrow 0} \bigg(\sum_{j=0}^{\infty} \bigg|\int_{|x-y| >\tau \atop \eta_{j+1}<|x-y| \le \eta_j} 
K^{\tau}(x, y) f(y) \, dy\bigg|^q \bigg)^{\frac1q}, 
\\ 
\mathscr{H}_{2,2}
&:= \sup_{\{\eta_j\} \searrow 0} \sum_{j=0}^{\infty} \bigg|
\int_{|x-y|>\tau \atop \eta_{j+1}<|x-y| \le \eta_j} (b(x') - b(y)) \Delta_{K^\tau}(y) f(y) \, dy\bigg|. 
\end{align*}
Here, $\Delta_{K^\tau}(y) := K^{\tau}(x, y) - K^{\tau}(x', y)$. Observe that 
\begin{align}\label{VHH-21}
\mathscr{H}_{2,1} 
\lesssim s \, \V_q (\T^{\tau} f)(x), 
\end{align}
and 
\begin{align}\label{VHH-22}
\mathscr{H}_{2,2}
\lesssim \int_{|x-y|>\tau} \frac{|x-x'|^{\delta}}{|x-y|^{n+\delta}}  |f(y)| \, dy
\lesssim s^{\delta} Mf(x) 
\le s^{\delta} M_r f(x). 
\end{align}
Gathering \eqref{VHH}--\eqref{VHH-22}, we obtain 
\begin{align*}
|\V_q(\T_b^{\tau} f)(x) - \V_q(\T_b^{\tau} f)(x')| 
\lesssim s \, \V_q (\T^{\tau} f)(x) + (s^{\frac{1}{r'}} + s^{\delta}) M_rf(x), 
\end{align*}
which together with \eqref{VV-3} gives that 
\begin{align*}
&\big\|\V_q(\T_b^{\tau} f) - \big(\V_q(\T_b^{\tau} f) \big)_{B(\cdot, s)}\big\|_{L^p(\Rn)} 
\\
&\le \bigg\|\fint_{B(x, s)} |\V_q(\T_b^{\tau} f)(x) - \V_q(\T_b^{\tau} f)(x')| dx'\bigg\|_{L^p(\Rn)} 
\\ 
&\lesssim s \|\V_q (\T^{\tau} f)\|_{L^p(\Rn)}
+ (s^{\frac{1}{r'}} + s^{\delta}) \|M_r f\|_{L^p(\Rn)}
\lesssim \varepsilon \|f\|_{L^p(\Rn)},  
\end{align*}
provided $\delta_0 := \min\{\tau, \varepsilon^{r'}, \varepsilon^{\frac{1}{\delta}}\}$ and $0< s< \delta_0$. This shows 
\begin{align}\label{VFK-3}
\lim_{s \to 0} \sup_{\|f\|_{L^p(\Rn)} \le 1}
\big\|\V_q(\T_b f) - \big(\V_q(\T_b f)\big)_{B(\cdot, s)} \big\|_{L^p(\Rn)} 
=0. 
\end{align}
Therefore, \eqref{VTU-2} is a consequence of Theorem \ref{thm:FKhs-1}, \eqref{VFK-1}, \eqref{VFK-2}, and \eqref{VFK-3}. 
\end{proof} 
%%%%%%%%%%%%%%%%%%%%%%% END END END PROOF %%%%%%%%%%%%%%%%%%%%%%%%%%

%%%%%%%%%%%%%%%%%%%%%%%%% SECTION SECTION SECTION %%%%%%%%%%%%%%%%%%%%%
%%%%%%%%%%%%%%%%%%%%%%%%% SECTION SECTION SECTION %%%%%%%%%%%%%%%%%%%%%
\section{Preliminaries}\label{sec:pre}
This section contains some preliminaries including the geometry of measure spaces and the properties of Muckenhoupt weights on measure spaces.

%%%%%%%%%%%%%%%%%%%%% SUBSECTION SUBSECTION SUBSECTION %%%%%%%%%%%%%%%%%%
%%%%%%%%%%%%%%%%%%%%% SUBSECTION SUBSECTION SUBSECTION %%%%%%%%%%%%%%%%%%
\subsection{Geometry of measure spaces}\label{sec:geomtry}
Let $(\Sigma, \mu)$ be a measure space. Assume that a basis $\B$ satisfies the property \eqref{list:B4}. It follows from the property \eqref{list:B4} that 
\begin{align}\label{ball} 
\text{if $A, B \in \B$ satisfy $A \cap B \neq \emptyset$ and $\mu(A) \le 2\mu(B)$, then $A \subset B^*$.} 
\end{align}
Given $B \in \B$, we denote 
\[
B^{(0)} := B, \quad 
B^{(1)} := B^*, \quad 
B^{(k+1)} := (B^{(k)})^*, \quad k \ge 1. 
\]
Then the property \eqref{list:B4} implies  
\begin{align}\label{BkBk}
\mu(B^{(k+1)}) \le \C_0 \, \mu(B^{(k)}) \quad\text{ and }\quad  
\mu(B^{(k)}) \le \C_0^k \mu(B), \quad k \ge 0. 
\end{align}

A set $E \subset \Sigma$ is called {\tt bounded} if there exists some $B \in \B$ so that $E \subset B$. We say that a measurable set $E$ is {\tt almost surely} a subset of a measurable set $F$ if $\mu(E \setminus F)=0$, in which case we write $E \subset F$ a.s. 
Let $\#A$ denote the cardinality of a finite set $A$. 

We collect the geometric properties of measure spaces given in \cite{Kar} as follows.
%%%%%%%%%%%%%%%%%%%%%%%%% LEMMA LEMMA LEMMA %%%%%%%%%%%%%%%%%%%%%%%
\begin{lemma}\label{lem:BG}
Let $(\Sigma, \mu)$ be a measure space with a ball-basis $\B$. If $B, G_k \in \B$, $k=1, 2, \ldots$, satisfy $G_k \cap B \neq \emptyset$ and $\lim_{k \to \infty} \mu(G_k) =r := \sup_{A \in \B} \mu(A)$, then $\Sigma \subset \bigcup_k G_k^*$. Moreover, for any ball $A \in \B$, we have $A \subset G_k$ for any $k \ge k_0$, where $k_0$ is some integer. 
\end{lemma}
%%%%%%%%%%%%%%%%%%%%%%%%% LEMMA LEMMA LEMMA %%%%%%%%%%%%%%%%%%%%%%%

%%%%%%%%%%%%%%%%%%%%%% DEFINITION DEFINITION DEFINITION %%%%%%%%%%%%%%%%%%%
\begin{definition}\label{def:density}
Given a measurable set $E \subset \Sigma$, a point $x \in E$ is said to be {\tt a density point} if for any $\varepsilon>0$ there exists a ball $B$ containing $x$ such that $\mu(B \cap E) > (1-\varepsilon) \mu(B)$. We say that a measure space $(\Sigma, \mu)$ satisfies {\tt the density property} if for any measurable set $E$, almost all points $x \in E$ are density points. 
\end{definition}
%%%%%%%%%%%%%%%%%%%%%% DEFINITION DEFINITION DEFINITION %%%%%%%%%%%%%%%%%%%

%%%%%%%%%%%%%%%%%%%%%%%%% LEMMA LEMMA LEMMA %%%%%%%%%%%%%%%%%%%%%%%
\begin{lemma}\label{lem:BE}
Let $(\Sigma, \mu)$ be a measurable space with a basis $\B$ satisfying the property \eqref{list:B4}. Then the following hold: 
\begin{list}{\rm (\theenumi)}{\usecounter{enumi}\leftmargin=1.2cm \labelwidth=1cm \itemsep=0.2cm \topsep=.2cm \renewcommand{\theenumi}{\alph{enumi}}}

\item\label{list-1} If $E \subset \Sigma$ is bounded such that $E \subset \bigcup_{G \in \mathcal{G}} G$ for some $\mathcal{G} \subset \B$, then there exists a (finite or infinite) pairwise disjoint subcollection $\mathcal{G}' \subset \mathcal{G}$ such that $E \subset \bigcup_{G \in \mathcal{G}'} G^*$.

\item\label{list-2} If $A \in \B$ and $\mathcal{G} \subset \B$ is a pairwise disjoint family satisfying 
\[
A \cap G^* \neq \emptyset 
\quad\text{ and }\quad 
0<c_1 \le \mu(G) \le c_2 < \infty, 
\quad\forall G \in \mathcal{G}, 
\]
with some positive constants $c_1, c_2$, then $\#\mathcal{G} \le c_1^{-1} \min\{c_2 \C_0^3, \C_0 \mu(A)\}$. 

\item\label{list-3} The property \eqref{list:B3} is equivalent to the density condition.  

\item\label{list-4} If in addition $\B$ satisfies the property \eqref{list:B3}, then for any bounded measurable set $E$ with $\mu(E)>0$, and $\varepsilon>0$ there is a sequence $\{B_k\} \subset \B$ such that 
\begin{align*}
\mu \bigg(\bigcup_k B_k \setminus E \bigg) < \varepsilon
\quad\text{ and }\quad 
\mu \bigg(E \setminus \bigcup_k B_k\bigg) < \alpha \, \mu(E), 
\end{align*}
where $\alpha \in (0, 1)$ is an allowable constant. 

\item\label{list-5} If in addition $\B$ satisfies the property \eqref{list:B3}, then for any bounded measurable set $E$ there exists a sequence $\{B_k\} \subset \B$ such that 
\begin{align*}
E \subset \bigcup_k B_k \text{ a.s. } 
\quad\text{ and }\quad 
\sum_k \mu(B_k) \le 2\C_0 \, \mu(E). 
\end{align*}
\end{list} 
\end{lemma}
%%%%%%%%%%%%%%%%%%%%%%%%% LEMMA LEMMA LEMMA %%%%%%%%%%%%%%%%%%%%%%%

%%%%%%%%%%%%%%%%%%%%% SUBSECTION SUBSECTION SUBSECTION %%%%%%%%%%%%%%%%%%
%%%%%%%%%%%%%%%%%%%%% SUBSECTION SUBSECTION SUBSECTION %%%%%%%%%%%%%%%%%%
\subsection{Muckenhoupt weights}\label{sec:Mucken}
Throughout this section, we always assume that $(\Sigma, \mu)$ is a measure space with $\mu(\Sigma)>0$. Recall that a  collection $\B$ of measurable sets is called a basis of $(\Sigma, \mu)$, if $0<\mu(B)<\infty$ for every $B \in \B$. Given a basis $\B$, the Hardy-Littlewood maximal operator $M_{\B}$  on $(\Sigma, \mu)$ associated with $\B$ is defined for each measurable function $f$ on $\Sigma$  by 
\begin{equation*}
M_{\B}f(x) := \sup_{x\in B \in \B} \fint_B |f| \, d\mu, \quad\text{if } x \in \Sigma_\B:=\bigcup_{B \in \B} B, 
\end{equation*} 
and $M_{\B}f(x)=0$ otherwise.

A measurable function $w$ on $\Sigma$ is called a {\tt weight} on $(\Sigma, \mu)$ if $0<w(x)<\infty$ for $\mu$-a.e.~$x \in \Sigma_\B$. 
Given $p \in (1, \infty)$ and a basis $\B$ on $(\Sigma, \mu)$, we define the {\tt Muckenhoupt class} $A_{p, \B}$ as the collection of all weights $w$ on $(\Sigma, \mu)$ satisfying 
\begin{equation*}
[w]_{A_{p,\B}}:=\sup_{B \in \B} \bigg(\fint_{B} w\, d\mu \bigg) 
\bigg(\fint_{B}w^{1-p'}\, d\mu \bigg)^{p-1}<\infty,
\end{equation*} 
where $p'$ is the H\"older conjugate exponent of $p$, i.e., $\frac1p+\frac{1}{p'}=1$. As for the case $p=1$, we say that $w\in A_{1,\B}$ if  
\begin{equation*}
[w]_{A_{1,\B}} := \|(M_{\B}w)\,w^{-1}\,\mathbf{1}_{\Sigma_\B}\|_{L^{\infty}(\Sigma, \mu)} <\infty.
\end{equation*}
Finally, we define 
\begin{equation*}
A_{\infty,\B} := \bigcup_{p\geq 1}A_{p,\B} 
\quad\text{ with }\quad 
[w]_{A_{\infty,\B}} := \inf\big\{[w]_{A_{p,\B}}: w \in A_{p, \B}\big\}. 
\end{equation*} 

For every $p\in(1,\infty)$ and every weight $w$ on $(\Sigma, \mu)$, we define the associated weighted Lebesgue space 
$L^p(\Sigma, w) := L^p(\Sigma,\,w d\mu)$ as the collection of measurable functions $f$ with 
\[
\|f\|_{L^p(\Sigma, w)} :=\Big(\int_\Sigma |f|^p\,w\,d\mu\Big)^{\frac1p}
<\infty.
\]
The {\tt reverse H\"{o}lder classes} are defined in the following way: we say that $w\in RH_{s, \B}$ for $s\in(1,\infty)$ and a basis $\B$ on $(\Sigma, \mu)$, if 
\begin{equation*}
[w]_{RH_{s, \B}} 
:=\sup_{B \in \B} \bigg(\fint_B w^s\, d\mu \bigg)^{\frac1s} \bigg(\fint_B w\,d\mu \bigg)^{-1} < \infty. 
\end{equation*} 
Regarding the endpoint $s=\infty$, $w \in RH_{\infty, \B}$ means that
\begin{equation*}
[w]_{RH_{\infty, \B}} 
:= \sup_{B \in \B} \|w\mathbf{1}_{B}\|_{L^{\infty}(\Sigma, \mu)} \bigg(\fint_{B}w\,d\mu\bigg)^{-1}<\infty.  
\end{equation*}

\begin{remark}
Note that, we do not define $A_{p,\B}$ and $RH_{s,\B}$ in the set $\Sigma\setminus\Sigma_\B$. This may create some technical issues in the arguments below and to avoid them, we will assume from now on that $\mu(\Sigma \setminus \Sigma_\B)=0$.  With this assumption in place, $w$ is a weight on $(\Sigma, \mu)$ if $0<w(x)<\infty$ for $\mu$-a.e.~$x \in \Sigma$. In the general situation where $\mu(\Sigma \setminus \Sigma_\B)>0$ one can alternatively work with $\Sigma_\B$ in place of $\Sigma$ (or, what is the same, restrict all functions and weights to the set $\Sigma_\B$).
\end{remark}

The following properties follow much as in the Euclidean case (see, \cite{GR, ST}).

%%%%%%%%%%%%%%%%%%%%%%%%% LEMMA LEMMA LEMMA %%%%%%%%%%%%%%%%%%%%%%%
\begin{lemma}\label{lem:Appp}
Let  $(\Sigma, \mu)$ be a measure space with a basis $\B$. Then the following hold: 
\begin{list}{\textup{(\theenumi)}}{\usecounter{enumi}\leftmargin=1cm \labelwidth=1cm \itemsep=0.2cm 
			\topsep=.2cm \renewcommand{\theenumi}{\alph{enumi}}}
			
\item\label{Appp-1} $A_{1, \B} \subset A_{p, \B} \subset A_{q, \B}\subset A_{\infty, \B}$ for every $1 <p \leq q <\infty$.  

\item\label{Appp-2} $RH_{\infty, \B} \subset RH_{q, \B} \subset RH_{p, \B}$ for every $1 <p \leq q <\infty$. 

\item\label{Appp-3} For any $p\in(1,\infty)$, $w\in A_{p,\B}$ if and only if $w^{1-p'}\in A_{p', \B}$, and $[w^{1-p\,'}]_{A_{p', \B}}=[w]^{\frac{1}{p-1}}_{A_{p,\B}}$. 

\item\label{Appp-4} If $w_1, w_2\in A_{p,\B}$ and $0\le \theta\le 1$, then $w_1^{\theta}\,w_2^{1-\theta}\in A_{p,\B}$ with
\[
[w_1^\theta\,w_2^{1-\theta}\,]_{A_{p,\B}} \leq [w_1]_{A_{p, \B}}^\theta\,[w_2]^{1-\theta}_{A_{p, \B}}. 
\]

\item\label{Appp-5} For all $p\in[1,\infty)$ and $s\in(1,\infty)$,
\begin{align*}
w\in A_{p,\B} \cap RH_{s, \B} \quad 
\Longleftrightarrow\quad w^s\in A_{\tau, \B}, \quad \tau=s(p-1)+1.
\end{align*}

\end{list}
\end{lemma}
%%%%%%%%%%%%%%%%%%%%%%%%% LEMMA LEMMA LEMMA %%%%%%%%%%%%%%%%%%%%%%%

%%%%%%%%%%%%%%%%%%%%%%%%% LEMMA LEMMA LEMMA %%%%%%%%%%%%%%%%%%%%%%%
\begin{lemma}\label{lem:A1}
Let $(\Sigma, \mu)$ be a measure space with a ball-basis $\B$. Then the following hold: 
\begin{list}{\textup{(\theenumi)}}{\usecounter{enumi}\leftmargin=1cm \labelwidth=1cm \itemsep=0.2cm 
			\topsep=.2cm \renewcommand{\theenumi}{\roman{enumi}}}
			
\item\label{list:A11} Given $p \in (1,\infty)$, $w \in A_{p, \B}$ if and only if there exist $w_1, w_2 \in A_{1, \B}$ such that $w=w_1 w_2^{1-p}$.

\item\label{list:ARH} If $w \in A_{1, \B}$, then $w^{-1} \in RH_{\infty, \B}$. 

\item\label{list:A13} If $u \in A_{\infty, \B}$ and $v \in A_{\infty, \B} \cap RH_{\infty, \B}$, then $uv \in A_{\infty, \B}$. 
 
\item\label{list:A12} For any $u \in A_{1, \B} \cap RH_{s, \B}$ for some $s \in (1, \infty)$, there exists $\varepsilon_0 \in (0, 1)$ such that $u v^{\varepsilon} \in A_{1, \B}$ for all $v \in A_{1, \B}$ and $\varepsilon \in(0, \varepsilon_0)$.

\item\label{list:A14} If $w \in RH_{\infty, \B}$, then for any $B \in \B$ and measurable set $E \subset B$, 
\begin{align*}
\frac{w(E)}{w(B)} 
\le [w]_{RH_{\infty, \B}} \frac{\mu(E)}{\mu(B)}. 
\end{align*}
\end{list}
\end{lemma}
%%%%%%%%%%%%%%%%%%%%%%%%% LEMMA LEMMA LEMMA %%%%%%%%%%%%%%%%%%%%%%%

%%%%%%%%%%%%%%%%%%%%%%%%%% PROOF PROOF PROOF %%%%%%%%%%%%%%%%%%%%%%%
\begin{proof}
For part \eqref{list:A11}, in view of Theorem \ref{Mapp}, the forward implication can be shown as in \cite[Theorem 7.5.1]{Gra-1}, while the reverse direction is trivial.  

Let $w \in A_{1, \B}$. For any $B \in \B$, 
\begin{align*}
\esssup_B w^{-1} 
=(\essinf_B w)^{-1} 
&\le [w]_{A_{1, \B}} \bigg(\fint_B w \, d\mu \bigg)^{-1} 
\le [w]_{A_{1, \B}} \bigg(\fint_B w^{-1} \, d\mu \bigg), 
\end{align*}
provided that 
\begin{align*}
1= \bigg(\fint_B w^{\frac12} w^{-\frac12} \, d\mu \bigg)^2 
\le \bigg(\fint_B w \, d\mu \bigg)  \bigg(\fint_B w^{-1} \, d\mu \bigg). 
\end{align*}
This proves $[w^{-1}]_{RH_{\infty, \B}} \le [w]_{A_{1, \B}}$. 

To show part \eqref{list:A13}, let $u \in A_{\infty, \B}$ and $v \in A_{\infty, \B} \cap RH_{\infty, \B}$. Then $u \in A_{p, \B}$ and $v \in A_{s, \B}$ for some $1<p, s<\infty$. Pick $q=p+s-1$ and $r=1+\frac{s-1}{p-1}$. Then, for any $B \in \B$, 
\begin{align}\label{Buv-1}
\fint_B uv \, d\mu 
\le \big(\esssup_B v\big) \fint_B u \, d\mu 
\le [v]_{RH_{\infty, \B}} \bigg(\fint_B u \, d\mu \bigg) \bigg(\fint_B v \, d\mu \bigg), 
\end{align}
and by H\"{o}lder's inequality, 
\begin{align}\label{Buv-2}
\bigg(\fint_B (uv)^{1-q'} \, d\mu \bigg)^{q-1} 
&\le \bigg(\fint_B u^{(1-q')r} \, d\mu \bigg)^{\frac{q-1}{r}} 
\bigg(\fint_B v^{(1-q')r'} \, d\mu \bigg)^{\frac{q-1}{r'}} 
\nonumber \\ 
&=\bigg(\fint_B u^{1-p'} \, d\mu \bigg)^{p-1} 
\bigg(\fint_B v^{1-s'} \, d\mu \bigg)^{s-1}. 
\end{align}
Collecting \eqref{Buv-1} and \eqref{Buv-2}, we get $[uv]_{A_{q, \B}} \le [u]_{A_{p, \B}} [v]_{A_{s, \B}} [v]_{RH_{\infty, \B}}$.

To continue, we let $u \in A_{1, \B} \cap RH_{s, \B}$ for some $s \in (1, \infty)$ and $v \in A_{1, \B}$. Set $\varepsilon_0 = \frac{1}{s'} \in (0, 1)$ and fix $\varepsilon \in (0, \varepsilon_0)$. Write $r:=(\frac{1}{\varepsilon})' < s$. Then by Lemma \ref{lem:Appp} part \eqref{Appp-2}, $u \in A_{1, \B} \cap RH_{r, \B}$. Thus, for any $B \in \B$ and a.e. $x \in \B$, 
\begin{align*}
\fint_B u v^{\varepsilon} \, d\mu 
&\le \bigg(\fint_B u^r \, d\mu \bigg)^{\frac1r} \bigg(\fint_B v^{\varepsilon r'} \, d\mu \bigg)^{\frac{1}{r'}}  
\\
&\le [u]_{RH_{r, \B}} \bigg(\fint_B u \, d\mu \bigg) \bigg(\fint_B v \, d\mu \bigg)^{\varepsilon} 
\\
&\le [u]_{RH_{r, \B}} [u]_{A_{1, \B}} [v]_{A_{1, \B}}^{\varepsilon} u(x) v(x)^{\varepsilon}, 
\end{align*}
which implies $uv^{\varepsilon} \in A_{1, \B}$.

Now let $w \in RH_{\infty, \B}$. Then for any $B \in \B$ and measurable set $E \subset B$, 
\begin{align*}
\frac{w(E)}{w(B)} 
=\frac{1}{w(B)} \int_B \mathbf{1}_E \, w \, d\mu 
\le \big(\esssup_B w \big) \frac{\mu(E)}{w(B)} 
\le [w]_{RH_{\infty, \B}} \frac{\mu(E)}{\mu(B)}. 
\end{align*} 
This proves part \eqref{list:A14}.
\end{proof}
%%%%%%%%%%%%%%%%%%%%%%%%%% END END END PROOF %%%%%%%%%%%%%%%%%%%%%%%

%%%%%%%%%%%%%%%%%%%%%%%%% LEMMA LEMMA LEMMA %%%%%%%%%%%%%%%%%%%%%%%
\begin{lemma}\label{lem:Ainfty}
Let $(\Sigma, \mu)$ be a measure space with a basis $\B$. Let $w$ be a weight on $(\Sigma, \mu)$. Consider the following conditions: 
\begin{list}{\rm (\theenumi)}{\usecounter{enumi}\leftmargin=1.2cm \labelwidth=1cm \itemsep=0.2cm \topsep=.2cm \renewcommand{\theenumi}{\alph{enumi}}}

\item\label{Ai-1} $w \in A_{\infty, \B}$. 

\item\label{Ai-2} There exist $0<\theta \le 1 \le C_0 <\infty$ such that for any $B \in \B$ and measurable set $E \subset B$, 
\[
\frac{\mu(E)}{\mu(B)} 
\le C_0 \bigg(\frac{w(E)}{w(B)}\bigg)^{\theta}. 
\]

\item\label{Ai-3} For any $\alpha \in (0, 1)$, there exists $\beta \in (0, 1)$ such that for each $B \in \B$ and measurable set $E \subset B$, 
\[
\mu(E) \ge \alpha \, \mu(B) \quad\Longrightarrow\quad 
w(E) \ge \beta \, w(B). 
\]

\item\label{Ai-4} For any $\alpha \in (0, 1)$, there exists $\beta \in (0, 1)$ such that for each $B \in \B$ and measurable set $E \subset B$, 
\[
\mu(E) \le \alpha \, \mu(B) \quad\Longrightarrow\quad 
w(E) \le \beta \, w(B). 
\]
\end{list}
Then, $\eqref{Ai-1} \Longrightarrow \eqref{Ai-2} \Longrightarrow \eqref{Ai-3} \iff \eqref{Ai-4}$. 
\end{lemma}
%%%%%%%%%%%%%%%%%%%%%%%%% LEMMA LEMMA LEMMA %%%%%%%%%%%%%%%%%%%%%%%

%%%%%%%%%%%%%%%%%%%%%%%%%% PROOF PROOF PROOF %%%%%%%%%%%%%%%%%%%%%%%
\begin{proof}
To show $\eqref{Ai-1} \Longrightarrow \eqref{Ai-2}$, we let $w \in A_p$ for some $1<p<\infty$. Let $B \in \B$ and $E \subset B$ be a measurable set. Then H\"{o}lder's inequality implies 
\begin{multline*}
\frac{\mu(E)}{\mu(B)}
=\fint_B \mathbf{1}_E w^{\frac1p} \, w^{-\frac1p} \, d\mu
\le \bigg(\fint_B \mathbf{1}_E w \, d\mu \bigg)^{\frac1p} 
\bigg(\fint_B w^{-\frac{p'}{p}} d\mu\bigg)^{\frac{1}{p'}}
\\
=\bigg(\frac{w(E)}{w(B)}\bigg)^{\frac1p} 
\bigg[\bigg(\fint_B w \, d\mu\bigg) \bigg(\fint_B w^{1-p'} d\mu\bigg)^{p-1}\bigg]^{\frac1p}
\le [w]_{A_p}^{\frac1p} \bigg(\frac{w(E)}{w(B)}\bigg)^{\frac1p}. 
\end{multline*}
This shows \eqref{Ai-2}. The implication $\eqref{Ai-2} \Longrightarrow \eqref{Ai-3}$ follows from that whenever $\mu(E)>\alpha \mu(B)$, 
\begin{align*}
\frac{w(E)}{w(B)} 
\ge C_0^{-1} \bigg(\frac{\mu(E)}{\mu(B)}\bigg)^{\frac{1}{\theta}}
\ge C_0^{-1} \alpha^{\frac{1}{\theta}} 
=: \beta. 
\end{align*}
Note that $\alpha \in (0, 1)$ implies $\beta \in (0, 1)$. It is trivial that $\eqref{Ai-3} \iff \eqref{Ai-4}$. 
\end{proof}
%%%%%%%%%%%%%%%%%%%%%%%%%% END END END PROOF %%%%%%%%%%%%%%%%%%%%%%%

Define 
\begin{align*}
[w]_{A'_{\infty, \B}} 
&:= \sup_{B \in \B} \bigg(\fint_B w \, d\mu \bigg) \exp\bigg(\fint_B \log w^{-1} \, d\mu \bigg), 
\end{align*}
and 
\begin{align*}
[w]_{A''_{\infty, \B}} 
&:= \sup_{B \in \B} \frac{1}{w(B^*)} \int_B M_{\B}(w \mathbf{1}_{B^*}) \, d\mu.   
\end{align*}
We present the following relationship between $A_{\infty, \B}$, $A'_{\infty, \B}$, and $A''_{\infty, \B}$. See \cite{DMO} for more properties about $A_{\infty, \B}$ weights. 

%%%%%%%%%%%%%%%%%%%%%%%%% LEMMA LEMMA LEMMA %%%%%%%%%%%%%%%%%%%%%%%
\begin{lemma}\label{lem:AiAi} 
Let $(\Sigma, \mu)$ be a measure space with a ball-basis $\B$. Then there holds $A_{\infty, \B} \subset A'_{\infty, \B} \subset A''_{\infty, \B}$. 
\end{lemma}
%%%%%%%%%%%%%%%%%%%%%%%%% LEMMA LEMMA LEMMA %%%%%%%%%%%%%%%%%%%%%%%

%%%%%%%%%%%%%%%%%%%%%%%%%% PROOF PROOF PROOF %%%%%%%%%%%%%%%%%%%%%%%
\begin{proof}
By Jensen's inequality, we have 
\begin{align*}
\exp \bigg(\fint_B \log w^{-1} \bigg) 
\le \bigg(\fint_B w^{-\frac{1}{p-1}} d\mu \bigg)^{p-1}, \quad 1<p<\infty, 
\end{align*}
which immediately implies $[w]_{A'_{\infty, \B}}  \le [w]_{A_{p, \B}}$ for any $1<p<\infty$. Thus, $A_{\infty, \B} \subset A'_{\infty, \B}$. 

To proceed, we observe that for any $B \in \B$, 
\begin{align}\label{MBequiv}
M_{\B}(w \mathbf{1}_{B^*})(x) 
= \sup_{x \in B' \in \B \atop B' \subset B^*} \fint_{B'} w \, d\mu, \quad x \in B. 
\end{align}
Indeed, fix $x \in B$ and let $x \in B' \in \B$. If $\mu(B) \le \mu(B')$, then 
\begin{align*}
\fint_{B'} w \mathbf{1}_{B^*} \, d\mu 
\lesssim \fint_{B^*} w \, d\mu 
\le \sup_{x \in B' \in \B \atop B' \subset B^*} \fint_{B'} w \, d\mu. 
\end{align*}
If $\mu(B') < \mu(B)$, the property \eqref{list:B4} gives $B' \subset B^*$, which in turn implies  
\begin{align*}
\fint_{B'} w \mathbf{1}_{B^*} \, d\mu 
\le \fint_{B'} w \, d\mu 
\le \sup_{x \in B' \in \B \atop B' \subset B^*} \fint_{B'} w \, d\mu. 
\end{align*}

If we define 
\begin{align*}
\mathbb{M}_{\B} f(x) 
:= \sup_{B \in \B: B \ni x} \exp\bigg(\fint_B \log |f| \, d\mu \bigg), 
\end{align*}
then by Jensen's inequality, we obtain for any $1<p<\infty$, 
\begin{multline}\label{MMB-1}
\|\mathbb{M}_{\B}f\|_{L^1(\Sigma, \mu)}
=\|(\mathbb{M}_{\B}(|f|^{1/p}))^p\|_{L^1(\Sigma, \mu)}
=\|\mathbb{M}_{\B}(|f|^{1/p})\|_{L^p(\Sigma, \mu)}^p
\\ 
\le \|M_{\B}(|f|^{1/p})\|_{L^p(\Sigma, \mu)}^p
\lesssim \||f|^{1/p}\|_{L^p(\Sigma, \mu)}^p
=\|f\|_{L^1(\Sigma, \mu)}, 
\end{multline}
where we have used Lemma \ref{lem:M}. Fix $B \in \B$. Then it follows from \eqref{MBequiv} that for all $x \in B$, 
\begin{align}\label{MMB-2}
M_{\B}(w \mathbf{1}_{B^*})(x)
&\le [w]_{A'_{\infty, \B}} \sup_{x \in B' \in \B \atop B' \subset B^*} \exp\bigg(\fint_{B'} \log w \, d\mu \bigg)
\nonumber \\
&\le [w]_{A'_{\infty, \B}} \mathbb{M}_{\B} (w \mathbf{1}_{B^*})(x).
\end{align}
Thus, \eqref{MMB-1} and \eqref{MMB-2} imply 
\begin{align*}
\int_B M_{\B}(w \mathbf{1}_{B^*}) \, d\mu  
\le [w]_{A'_{\infty, \B}} \|\mathbb{M}_{\B}(w \mathbf{1}_{B^*})\|_{L^1(\Sigma, \mu)}
\lesssim [w]_{A'_{\infty, \B}} w(B^*), 
\end{align*}
which gives $[w]_{A''_{\infty, \B}} \le [w]_{A'_{\infty, \B}}$.
\end{proof}
%%%%%%%%%%%%%%%%%%%%%%%%%% END END END PROOF %%%%%%%%%%%%%%%%%%%%%%%

In view of Lemma \ref{lem:M}, let us present a multilinear extension of Coifmann-Rochberg theorem. 

%%%%%%%%%%%%%%%%%%%%%%%%% LEMMA LEMMA LEMMA %%%%%%%%%%%%%%%%%%%%%%%
\begin{lemma}\label{lem:CR}
Let $(\Sigma, \mu)$ be a measure space with a ball-basis $\B$. Then 
\begin{align}\label{eq:CR}
[(\mathcal{M}_{\B}(\vec{f}))^{\delta}]_{A_{1, \B}} 
\leq \frac{c_m}{1-m\delta}, \quad \text{for any } 0<\delta<\frac{1}{m}.
\end{align}
\end{lemma}
%%%%%%%%%%%%%%%%%%%%%%%%% LEMMA LEMMA LEMMA %%%%%%%%%%%%%%%%%%%%%%%

%%%%%%%%%%%%%%%%%%%%%%%%%% PROOF PROOF PROOF %%%%%%%%%%%%%%%%%%%%%%%
\begin{proof}
Mimicking the proof of \cite[Theorem 3.4, p. 158]{GR} and \cite[Lemma~1]{OPR}, we present the details for the convenience of the reader. Fix $B \in \B$ and $x \in B$. Given $\vec{f}=(f_1, \ldots, f_m)$, let $f_i^0 := f_i \mathbf{1}_{B^*}$ and $f_i^{\infty} := f_i \mathbf{1}_{\Sigma \setminus B^*}$, $i=1, \ldots, m$. Then, writing $\vec{f}^0=(f_1^0, \ldots, f_m^0)$, we have 
\begin{align}\label{eq:CR-1}
\fint_B \mathcal{M}_{\B}(\vec{f})^{\delta} d\mu 
&\le \fint_B \mathcal{M}_{\B}(\vec{f}^0)^{\delta} d\mu 
\nonumber \\
&\quad+ \sum_{\alpha_1, \ldots, \alpha_m \in \{0, \infty\} \atop \exists \alpha_i = \infty} 
\fint_B \mathcal{M}_{\B}(f_1^{\alpha_1}, \ldots, f_m^{\alpha_m})^{\delta} d\mu. 
\end{align}
To control the first term, we set 
\begin{equation*}
\lambda_0 := \prod_{i=1}^m \fint_{B^*} |f_i| \, d\mu
\le \mathcal{M}_{\B}(\vec{f})(x), 
\end{equation*}
and invoke Lemma \ref{lem:M} to arrive at 
\begin{align}\label{eq:CR-2}
\fint_B \mathcal{M}_{\B}(\vec{f}^0)^{\delta} d\mu 
&=\frac{\delta}{\mu(B)} \int_0^{\infty} \lambda^{\delta} 
\mu(\{x \in B: \mathcal{M}_{\B}(\vec{f}^0)(x)^{\delta}>\lambda\}) \frac{d\lambda}{\lambda}
\nonumber\\ 
&\le \lambda_0^{\delta} + \frac{\delta}{\mu(B)} \int_{\lambda_0}^{\infty} \lambda^{\delta} 
\mu(\{x \in \Sigma: \mathcal{M}_{\B}(\vec{f}^0)(x)^{\delta}>\lambda\}) \frac{d\lambda}{\lambda}
\nonumber\\ 
&\le \lambda_0^{\delta} + \frac{c \, \delta}{\mu(B^*)} \int_{\lambda_0}^{\infty} \lambda^{\delta-\frac1m} 
\prod_{i=1}^m \|f_i^0\|_{L^1(\Sigma, \mu)}^{\frac1m} \frac{d\lambda}{\lambda}
\nonumber\\ 
&= \lambda_0^{\delta} \bigg[1 + \frac{c \, \delta^2}{1-m\delta} \Big(\lambda_0^{-1} 
\prod_{i=1}^m \fint_{B^*} |f_i| \, d\mu \Big)^{\frac1m} \bigg]
\nonumber\\ 
&\le \frac{c}{1-m\delta} \mathcal{M}_{\B}(\vec{f})(x)^{\delta}. 
\end{align}
To proceed, let $\alpha_1, \ldots, \alpha_m \in \{0, \infty\}$ with $\alpha_i = \infty$ for some $i$.  Let $A \in \B$ and $y \in A \cap B$. Observe that 
\begin{align}\label{AFB}
\int_A |f_i^{\infty}| \, d\mu \neq 0 
\quad\Longrightarrow\quad 
\mu(A) > 2\mu(B) 
\quad\Longrightarrow\quad 
x \in B \subset A^*.
\end{align}
Indeed, if $\mu(A) \le 2\mu(B)$, then the property \eqref{list:B4} gives $A \subset B^*$, hence, $\int_A |f_i^{\infty}| \, d\mu = \int_{A\setminus B^*} |f_i| \, d\mu =0$. The second implication in \eqref{AFB} is just a consequence of the property \eqref{list:B4} and that $A \cap B \neq \emptyset$. Hence, by \eqref{list:B4}, 
\begin{align*}
\prod_{i=1}^m \fint_A |f_i^{\alpha_i}| \, d\mu 
\le \prod_{i=1}^m \bigg(\frac{\mu(A^*)}{\mu(A)}\bigg) \fint_{A^*} |f_i^{\alpha_i}| \, d\mu 
\le \C_0^m \mathcal{M}_{\B}(f_1^{\alpha_1}, \ldots, f_m^{\alpha_m})(x), 
\end{align*}
which in turn yields   
\begin{align}\label{eq:CR-3}
\mathcal{M}_{\B}(f_1^{\alpha_1}, \ldots, f_m^{\alpha_m})(y)
\le \C_0^m \mathcal{M}_{\B}(\vec{f})(x), \quad x, y \in B. 
\end{align}
Consequently, collecting \eqref{eq:CR-1}, \eqref{eq:CR-2}, and \eqref{eq:CR-3}, we conlude 
\begin{align*}
\fint_B \mathcal{M}_{\B}(\vec{f})^{\delta} d\mu 
\le \frac{c}{1-m\delta} \mathcal{M}_{\B}(\vec{f})(x)^{\delta}, \quad x \in B, 
\end{align*}
which implies \eqref{eq:CR} as desired. 
\end{proof}
%%%%%%%%%%%%%%%%%%%%%%%%%% END END END PROOF %%%%%%%%%%%%%%%%%%%%%%%

%%%%%%%%%%%%%%%%%%%%%% DEFINITION DEFINITION DEFINITION %%%%%%%%%%%%%%%%%%%
\begin{definition}
Let $(\Sigma, \mu)$ be a measure space with a ball-basis $\B$. Let $\vec{p}=(p_1, \ldots, p_m)$ with $1 \leq p_1, \ldots, p_m<\infty$, and let $\vec{w} = (w_1, \ldots, w_m)$ be a vector of weights. We say that $\vec{w} \in A_{\vec{p}, \B}$ if 
\begin{align*}
[\vec{w}]_{A_{\vec{p}, \B}} 
:= \sup_{B \in \B} \bigg(\fint_{B} w \, d\mu \bigg)^{\frac1p} 
\prod_{i=1}^m \bigg(\fint_{B} w_i^{1-p'_i} d\mu\bigg)^{\frac{1}{p'_i}} < \infty, 
\end{align*}
where $w=\prod_{i=1}^m w^{\frac{p}{p_i}}$ and $\frac1p=\sum_{i=1}^m \frac{1}{p_i}$. When $p_i=1$, $\big(\fint_B w_i^{1-p'_i}d\mu\big)^{\frac{1}{p'_i}}$ is understood as $(\essinf_B w_i)^{-1}$. 
\end{definition}
%%%%%%%%%%%%%%%%%%%%%% DEFINITION DEFINITION DEFINITION %%%%%%%%%%%%%%%%%%%

As shown in \cite[Theorem 3.6]{LOPTT}, one has the following characterization of $A_{\vec{p}, \B}$. 

%%%%%%%%%%%%%%%%%%%%%%%%% LEMMA LEMMA LEMMA %%%%%%%%%%%%%%%%%%%%%%%
\begin{lemma}\label{lem:ApAp}
Let $(\Sigma, \mu)$ be a measure space with a ball-basis $\B$. Let $\vec{w}=(w_1, \ldots, w_m)$ and $\vec{p}=(p_1, \ldots, p_m)$ with $1 \leq p_1, \ldots, p_m<\infty$. Then $\vec{w}\in A_{\vec{p}, \B}$ if and only if $w \in A_{mp, \B}$ and $\sigma_i := w_i^{1-p'_i} \in A_{mp'_i, \B}$, $i=1, \ldots, m$. When $p_i=1$, $w_i^{1-p'_i}\in A_{mp'_i, \B}$ is understood as $w_i^{\frac1m} \in A_{1, \B}$. Moreover, 
\begin{align}
\label{ww-1} &[\vec{w}]_{A_{\vec{p}, \B}} 
\le [w]_{A_{mp, \B}}^{\frac1p} \prod_{i=1}^m [\sigma_i]_{A_{mp'_i, \B}}^{\frac{1}{p'_i}}, 
\quad \text{ and } \quad
\\ 
\label{ww-2} &[w]_{A_{mp, \B}}^{\frac1p} \leq [\vec{w}]_{A_{\vec{p}, \B}}, \quad 
[\sigma_i]_{A_{mp'_i, \B}}^{\frac{1}{p'_i}}  \leq [\vec{w}]_{A_{\vec{p}, \B}},  \quad i=1,\ldots,m,
\end{align}
\end{lemma}
%%%%%%%%%%%%%%%%%%%%%%%%% LEMMA LEMMA LEMMA %%%%%%%%%%%%%%%%%%%%%%%

%%%%%%%%%%%%%%%%%%%%% SUBSECTION SUBSECTION SUBSECTION %%%%%%%%%%%%%%%%%%
%%%%%%%%%%%%%%%%%%%%% SUBSECTION SUBSECTION SUBSECTION %%%%%%%%%%%%%%%%%%
\subsection{Young functions and Orlicz spaces}
Let $(\Sigma, \mu)$ be a measure space with a basis $\B$. A function $\Phi : [0, \infty) \rightarrow [0, \infty)$ is said to be a Young function, if $\Phi$ is continuous, convex, increasing function such that $\Phi(0)=0$ and $\Phi(t)/t \rightarrow \infty$ as $t \rightarrow \infty$.

Given a Young function $\Phi$ and $B \in \B$, we define the normalized Luxemburg norm of $f$ on $B$ by 
\begin{equation*}
\|f\|_{\Phi, B}
= \inf \bigg\{\lambda>0; \fint_B \Phi \Big(\frac{|f(x)|}{\lambda}\Big) d\mu \leq 1 \bigg\}.
\end{equation*}
Three useful examples are the following: 
\begin{itemize}

\item $\Phi(t)=t^p$, $p>1$, then $\|f\|_{L^p, B} := \|f\|_{\Phi, B}$.

\item $\Phi(t)=t \log^r(e+t)$, $r>0$, then $\|f\|_{L(\log L)^r, B} := \|f\|_{\Phi, B}$. 

\item $\Phi(t)=e^{t^r}-1$, $r>0$, then $\|f\|_{\exp L^r, B} := \|f\|_{\Phi, B}$. 
\end{itemize}

For a given Young function $\phi$, one can define a complementary function
\[
\bar{\Phi}(s)=\sup_{t>0} \{ st - \Phi(t)\}, \quad s \geq 0.
\]
Then the following H\"{o}lder's inequality holds
\begin{equation*}
\fint_B |fg| \, d\mu 
\leq 2 \|f\|_{\Phi, B} \|g\|_{\bar{\Phi}, B}.
\end{equation*}
Let us present more generalized H\"{o}lder's inequalities on Orlicz spaces. 

%%%%%%%%%%%%%%%%%%%%%%%% LEMMA LEMMA LEMMA %%%%%%%%%%%%%%%%%%%%%%%%
\begin{lemma}[\cite{PRiv}]\label{lem:PhiPhi}
If $\Phi_0, \Phi_1, \ldots, \Phi_k$ are Young functions satisfying 
\[
\Phi_1^{-1}(t) \cdots \Phi_k^{-1}(t) \lesssim \Phi_0^{-1}(t) \quad\text{ for all } t>0, 
\] 
then for any $B \in \B$, 
\begin{equation*}
\|f_1 \cdots f_k\|_{\Phi_0, B} 
\lesssim \|f_1\|_{\Phi_1, B} \cdots \|f_k\|_{\Phi_k, B} .
\end{equation*}
In particular, for any $B \in \B$ and $\frac1s = \sum_{i=1}^k \frac{1}{s_i}$ with $s_1, \ldots, s_k \geq 1$, 
\begin{align*}
\fint_B |f_1 \ldots f_k g| \, d\mu
& \lesssim \| f_1 \|_{\exp L^{s_1}, B} \cdots \| f_k \|_{\exp L^{s_k}, B} \|g\|_{L(\log L)^{\frac{1}{s}}, B}. 
\end{align*}
\end{lemma}
%%%%%%%%%%%%%%%%%%%%%%%% LEMMA LEMMA LEMMA %%%%%%%%%%%%%%%%%%%%%%%%

Given $r>0$, define the space $\osc_{\exp L^r, \B}$ as the collection of all functions $f \in L^1_{\loc}(\Sigma, \mu)$ satisfying 
\begin{align*}
\|f\|_{\osc_{\exp L^r, \B}}:= \sup_{B \in \B} \|f-f_B\|_{\exp L^r, B} < \infty.  
\end{align*}
We say that a measurable function $f \in \BMO_{\B}$ if $\int_B |f| \, d\mu< \infty$ for every $B\in \B$ and 
\begin{align}\label{def:BMO}
\|f\|_{\BMO_{\B}} 
:= \sup_{B \in \B} \fint_B |f-f_B| \, d\mu < \infty. 
\end{align}
When $\Sigma = \Rn$, $d\mu=dx$, and $\B$ denotes the collection of all balls or cubes in $\Rn$, the John-Nirenberg's theorem (cf. \cite[Corollary 3.1.7]{Gra-2}) implies that 
\begin{align}\label{OSC}
\text{$\osc_{\exp L, \B}=\BMO_{\B}$ \, with comparable norms.}
\end{align}

A basis $\B$ is called {\tt uniform} if there exists a constant $c_0 \in (0, 1)$ such that for all $B \in \B$, there exist a disjoint collection $\{B_j\} \subset \B$ such that for all $j$, $B_j \subset B$, $\mu(B_j) \ge c_0 \mu(B)$, and $\mu(B \setminus \bigcup_j B_j) = 0$. Moreover, a basis $\B$ is said to be {\tt differentiable} if for each point $x \in \Sigma$, there exists a sequence $\{B_j\} \subset \B$ such that $x \in \bigcap_j B_j$ and $\lim_{j \to \infty} \mu(B_j) = 0$. 

Note that in general $\osc_{\exp L^r, \B} \subset \BMO_{\B}$ for any $r \ge 1$.  If a basis $\B$ is uniform and differentiable, then \cite[Theorem 3.1]{DLOPW} gives 
\begin{align}\label{OSCBMO} 
\text{$\osc_{\exp L, \B} = \BMO_{\B}$ \, with comparable norms.}
\end{align}

%%%%%%%%%%%%%%%%%%%%%% DEFINITION DEFINITION DEFINITION %%%%%%%%%%%%%%%%%%%
\begin{definition}
We say that $A_{\infty, \B}$ satisfies the {\tt sharp reverse H\"{o}lder property} if there exists a constant $c_0>0$ such that for every $w \in A_{\infty, \B}$ one has $w \in RH_{r_w, \B}$ with $r_w = 1 + c_0 [w]_{A_{\infty, \B}}^{-1}$.  
\end{definition}
%%%%%%%%%%%%%%%%%%%%%% DEFINITION DEFINITION DEFINITION %%%%%%%%%%%%%%%%%%%

If we take $\Sigma = \Rn$, $d\mu=dx$, and $\B$ as the collection of all cubes in $\Rn$, then \cite[Theorem 2.3]{HP} gives that 
\begin{align}\label{cube}
\text{$A_{\infty, \B}$ satisfies the sharp reverse H\"{o}lder property}.
\end{align}

%%%%%%%%%%%%%%%%%%%%%%%%%% LEMMA LEMMA LEMMA %%%%%%%%%%%%%%%%%%%%%%
\begin{lemma}\label{lem:LlogL}
Assume that $A_{\infty, \B}$ satisfies the sharp reverse H\"{o}lder property. Let $s>1$, $t>0$, and $w \in A_{\infty, \B}$. Then for any $B \in \B$, 
\begin{align}
\label{e:w} &\| w^{\frac1s} \|_{L^s(\log L)^{st}, B}
\lesssim  [w]_{A_{\infty, \B}}^t \langle w \rangle_B^{\frac1s},
\\
\label{e:fw} &\| f w \|_{L(\log L)^t, B}
\lesssim  [w]_{A_{\infty, \B}}^t \langle w \rangle_B \inf_{x \in B} M_{w}(|f|^s)(x)^{\frac1s}. 
\end{align}
\end{lemma} 
%%%%%%%%%%%%%%%%%%%%%%%%%% LEMMA LEMMA LEMMA %%%%%%%%%%%%%%%%%%%%%%

%%%%%%%%%%%%%%%%%%%%%%%%%% PROOF PROOF PROOF %%%%%%%%%%%%%%%%%%%%%%%
\begin{proof}
Let $w \in A_{\infty, \B}$. By assumption, there exists $r_w \simeq 1 + [w]_{A_{\infty, \B}}^{-1}$ such that $w \in RH_{r_w, \B}$. This implies 
\begin{align*}
&\|w^{\frac1s} \|_{L^s(\log L)^{st}, B}^s
\lesssim \|w\|_{L(\log L)^{st}, B}
\\ 
&\quad \lesssim \bigg(1 + \frac{1}{(r_w - 1)^{st}} \bigg)
\bigg( \fint_B w^{r_w} dx  \bigg)^{\frac{1}{r_w}}
\lesssim  [w]_{A_{\infty}}^{st} \langle w \rangle_B.
\end{align*}
Now by Lemma \ref{lem:PhiPhi} and \eqref{e:w}, we get 
\begin{align*}
\|fw\|_{L(\log L)^{t}, B}
&\lesssim \bigg( \fint_B |f|^s w \, d\mu \bigg)^{\frac1s}
\|w^{\frac{1}{s'}} \|_{L^{s'}(\log L)^{s't}, B}
\\ 
&\lesssim  [w]_{A_{\infty}}^t
\bigg( \frac{1}{w(B)} \int_B |f|^s w \, d\mu \bigg)^{\frac1s} \langle w \rangle_B
\\
&\lesssim  [w]_{A_{\infty}}^t \inf_{x \in B}
M_{\B, w}(|f|^s)^{\frac1s} \langle w \rangle_B.
\end{align*}
This completes the proof. 
\end{proof}
%%%%%%%%%%%%%%%%%%%%%%%%%% END END END PROOF %%%%%%%%%%%%%%%%%%%%%%%

%%%%%%%%%%%%%%%%%%%%%%%% SECTION SECTION SECTION %%%%%%%%%%%%%%%%%%%%%%
%%%%%%%%%%%%%%%%%%%%%%%% SECTION SECTION SECTION %%%%%%%%%%%%%%%%%%%%%%
\section{Properties of bounded oscillation operators}\label{sec:BO}

Given a $\bB$-valued operator $T$ satisfying $T(\vec{f})(x) = \|\T(\vec{f})(x)\|_{\bB}$, we define its maximal operator by  
\begin{align*}
T_*(\vec{f})(x)  
:= \sup_{x \in B \in \B} \big\|\T(\vec{f})(x) - \T(\vec{f} \mathbf{1}_{B^*})(x)\big\|_{\bB}. 
\end{align*}

We say that a ball-basis $\B$ in a measure space satisfies {\tt the doubling condition} if there is a constant $\gamma>1$ such that for any ball $B \in \B$ with $B^* \subsetneq \Sigma$, one can find $B' \in \B$ satisfying $B \subset B'$ and $\mu(B') \le \gamma \, \mu(B)$. 

%%%%%%%%%%%%%%%%%%%%%%% THEOREM THEOREM THEOREM %%%%%%%%%%%%%%%%%%%%%
\begin{theorem}\label{thm:doubling}
Let $(\Sigma, \mu)$ be a measure space with a ball-basis $\B$. Assume that $T$ is a $\bB$-valued multilinear operator satisfying the condition \eqref{list:T-reg} such that $T$ is bounded from $L^r(\Sigma, \mu) \times \cdots \times L^r(\Sigma, \mu)$ to $L^{\frac{r}{m}, \infty}(\Sigma, \mu)$ for some $r \in [1, \infty)$. Then the following hold: 
\begin{list}{\rm (\theenumi)}{\usecounter{enumi}\leftmargin=1.2cm \labelwidth=1cm \itemsep=0.2cm \topsep=.2cm \renewcommand{\theenumi}{\roman{enumi}}}

\item\label{dou-1} $T_*$ is bounded from $L^r(\Sigma, \mu) \times \cdots \times L^r(\Sigma, \mu)$ to $L^{\frac{r}{m}, \infty}(\Sigma, \mu)$ with bound $(\C_2(T) + \|T\|)$. 

\item\label{dou-2} If in addition $\B$ satisfies the doubling condition, then $T$ satisfies the condition \eqref{list:T-size}. In particular, $T$ is a $\bB$-valued multilinear bounded oscillation operator with respect to $\B$ and $r$.  
\end{list}
\end{theorem} 
%%%%%%%%%%%%%%%%%%%%%%% THEOREM THEOREM THEOREM %%%%%%%%%%%%%%%%%%%%%

Considering the condition \eqref{list:T-size}, given $A, B \in \B$ with $A \subset B$, we denote 
\begin{align*}
\Delta(A, B)
:= \sup_{f_i \in L^r(\mu) \atop i=1, \ldots, m}\sup_{x \in A} 
\big\|\T(\vec{f} {\bf 1}_{B^*})(x) - \T(\vec{f} {\bf 1}_{A^*})(x) \big\|_{\bB} 
\bigg/ \prod_{i=1}^m \langle f_i \rangle_{B^*, r}. 
\end{align*}

%%%%%%%%%%%%%%%%%%%%%%%%% LEMMA LEMMA LEMMA %%%%%%%%%%%%%%%%%%%%%%%
\begin{lemma}\label{lem:ABF}
Let $(\Sigma, \mu)$ be a measure space with a ball-basis $\B$. Then we have 
\begin{list}{\rm (\theenumi)}{\usecounter{enumi}\leftmargin=1.2cm \labelwidth=1cm \itemsep=0.2cm \topsep=.2cm \renewcommand{\theenumi}{\alph{enumi}}}

\item\label{ABF-1} For any $A, B, C \in \B$ with $A \subset B \subset C$, one has $\Delta(A, B) \le \Delta(A, C)$. 

\item\label{ABF-2} For any $A, B \in \B$ with $A \subset B$, 
\begin{align*}
\lfloor f \mathbf{1}_{B^*} \rfloor_{A, r}
\lesssim [\mu(B)/\mu(A)]^{\frac1r} \langle f \rangle_{B^*, r}. 
\end{align*}
\end{list}
\end{lemma} 
%%%%%%%%%%%%%%%%%%%%%%%%% LEMMA LEMMA LEMMA %%%%%%%%%%%%%%%%%%%%%%%

%%%%%%%%%%%%%%%%%%%%%%%%%% PROOF PROOF PROOF %%%%%%%%%%%%%%%%%%%%%%%
\begin{proof}
To show part \eqref{ABF-1}, we let $A, B, C \in \B$ with $A \subset B \subset C$, $f_i \in L^r(\mu)$, and set $g_i := f_i \mathbf{1}_{B^*}$, $i=1, \ldots, m$. Then for any $x \in A$, 
\begin{align*}
\|\T(\vec{f} {\bf 1}_{B^*})(x) - \T(\vec{f} {\bf 1}_{A^*})(x)\|_{\bB} 
=\|\T(\vec{g} {\bf 1}_{C^*})(x) - \T(\vec{g} {\bf 1}_{A^*})(x)\|_{\bB} 
\\ 
\le \Delta(A, C) \prod_{i=1}^m \langle g_i \rangle_{C^*, r}
= \Delta(A, C) \prod_{i=1}^m [\mu(B^*)/\mu(C^*)]^{\frac1r} \langle f_i \rangle_{B^*, r}
\\ 
\le \Delta(A, C) \prod_{i=1}^m \langle f_i \rangle_{B^*, r}. 
\end{align*}
This gives at once part \eqref{ABF-1}. 

Let $A_1, A, B \in \B$ with $A \subset A_1 \cap B$. If $\mu(A_1) \le \mu(B^*)$, then the property \eqref{list:B4} and $A_1 \cap B^* \neq \emptyset$ give that $A_1 \subset B^{(2)}$. As a consequence, 
\begin{align}\label{fb-1}
\langle f \mathbf{1}_{B^*} \rangle_{A_1, r}
\le [\mu(B^{(2)})/\mu(A_1)]^{\frac1r} \langle f \mathbf{1}_{B^*} \rangle_{B^{(2)}, r}
\lesssim [\mu(B)/\mu(A)]^{\frac1r} \langle f \rangle_{B^*, r}. 
\end{align}
If $\mu(A_1) > \mu(B^*)$, then it follows from the property \eqref{list:B4} and $A_1 \cap B^* \neq \emptyset$ that $B^* \subset A^*_1$, which yields  
\begin{align}\label{fb-2}
\langle f \mathbf{1}_{B^*} \rangle_{A_1, r}
\le [\mu(A^*_1)/\mu(A_1)]^{\frac1r} \langle f \mathbf{1}_{B^*} \rangle_{A^*_1, r}
\lesssim \langle f \rangle_{B^*, r}. 
\end{align}
Thus, \eqref{fb-1} and \eqref{fb-2} imply part \eqref{ABF-2} as desired. 
\end{proof}
%%%%%%%%%%%%%%%%%%%%%%%%%% END END END PROOF %%%%%%%%%%%%%%%%%%%%%%%

%%%%%%%%%%%%%%%%%%%%%%%%% LEMMA LEMMA LEMMA %%%%%%%%%%%%%%%%%%%%%%%
\begin{lemma}\label{lem:ABC}
Let $(\Sigma, \mu)$ be a measure space with a ball-basis $\B$. Assume that $T$ is a $\bB$-valued multilinear operator satisfying the condition \eqref{list:T-reg} such that $T$ is bounded from $L^r(\Sigma, \mu) \times \cdots \times L^r(\Sigma, \mu)$ to $L^{\frac{r}{m}, \infty}(\Sigma, \mu)$ for some $r \in [1, \infty)$. Write $\|T\| := \|T\|_{L^r(\Sigma, \mu) \times \cdots \times L^r(\Sigma, \mu) \to L^{\frac{r}{m}, \infty}(\Sigma, \mu)}$. Then the following statements hold:  
\begin{list}{\rm (\theenumi)}{\usecounter{enumi}\leftmargin=1.2cm \labelwidth=1cm \itemsep=0.2cm \topsep=.2cm \renewcommand{\theenumi}{\arabic{enumi}}}

\item\label{ABC-1} For any $A, B \in \B$ with $A \subset B$, 
\begin{align*}
\Delta(A, B) 
\lesssim (\C_2(T) + \|T\|) [\mu(B)/\mu(A)]^{\frac{m}{r}}.
\end{align*}

\item\label{ABC-2} For any $A, B, C \in \B$ with $A \subset B \subset C$, 
\begin{align*}
\Delta(A, C) 
\lesssim \big(\C_2(T) + \|T\| + \Delta(A, B) \big) [\mu(C)/\mu(B)]^{\frac{m}{r}}.
\end{align*}

\item\label{ABC-3} For any $A, B \in \B$ with $A \subset B$, 
\begin{align*}
\Delta(A, B^{(k)}) 
\lesssim \C_0^{\frac{mk}{r}} \big(\C_2(T) + \|T\| + \Delta(A, B) \big), \quad k \ge 1.
\end{align*}
\end{list}
\end{lemma}
%%%%%%%%%%%%%%%%%%%%%%%%% LEMMA LEMMA LEMMA %%%%%%%%%%%%%%%%%%%%%%%

%%%%%%%%%%%%%%%%%%%%%%%%%% PROOF PROOF PROOF %%%%%%%%%%%%%%%%%%%%%%%
\begin{proof}
Let us first show part \eqref{ABC-1}. Fix $A, B \in \B$ with $A \subset B$. Set 
\begin{align*}
K_0 := \|T\|_{L^r(\Sigma, \mu) \times \cdots \times L^r(\Sigma, \mu) \to L^{\frac{r}{m}, \infty}(\Sigma, \mu)} 
\prod_{i=1}^m \bigg(\frac{4^{r+1}}{\mu(A)} \int_{B^*} |f_i|^r \, d\mu \bigg)^{\frac1r}.
\end{align*}
Since $T$ is bounded from $L^r(\Sigma, \mu) \times \cdots \times L^r(\Sigma, \mu)$ to $L^{\frac{r}{m}, \infty}(\Sigma, \mu)$, we have 
\begin{align*}
&\mu(\{x \in A: \|\T(\vec{f} \mathbf{1}_{B^*})(x)\|_{\bB} > K_0/2\}) 
\\
&= \mu(\{x \in A: |T(\vec{f} \mathbf{1}_{B^*})(x)| > K_0/2\}) 
\\
&\le \bigg(\frac{\|T\|}{K_0/2} \prod_{i=1}^m \|f_i \mathbf{1}_{B^*}\|_{L^r(\Sigma, \mu)} \bigg)^{\frac{r}{m}}
\\
&=\bigg[2 \bigg(\frac{\mu(A)}{4^{r+1}}\bigg)^{\frac{m}{r}} \bigg]^{\frac{r}{m}}
\le \frac14 \mu(A). 
\end{align*}
Similarly,  
\begin{align*}
\mu(\{x \in A: \|\T(\vec{f} \mathbf{1}_{A^*})(x)\|_{\bB} > K_0/2\}) 
\le \mu(A)/4.
\end{align*}
Therefore, 
\begin{align*}
\mu(\{x \in A: \|\T(\vec{f} \mathbf{1}_{B^*})(x) - \T(\vec{f} \mathbf{1}_{A^*})(x)\|_{\bB} > K_0\})
\le \mu(A)/2. 
\end{align*}
This means that there exists $x' \in A$ such that 
\begin{align}\label{TK-1}
&\|\T(\vec{f} \mathbf{1}_{B^*})(x') - \T(\vec{f} \mathbf{1}_{A^*})(x')\|_{\bB} 
\nonumber \\
&\quad\le K_0 \lesssim \|T\| [\mu(B)/\mu(A)]^{\frac{m}{r}} \prod_{i=1}^m \langle f_i \rangle_{B^*, r}. 
\end{align}
On the other hand, for any $x \in A$,  applying the condition \eqref{list:T-reg} to $\vec{f} \mathbf{1}_{B^*}$ and $A^*$ in the place of $\vec{f}$ and $B^*$, we have  
\begin{align}\label{TK-2}
&\|(\T(\vec{f} \mathbf{1}_{B^*}) - \T(\vec{f} \mathbf{1}_{A^*}))(x)  
- (\T(\vec{f} \mathbf{1}_{B^*}) - \T(\vec{f} \mathbf{1}_{A^*}))(x')\|_{\bB} 
\nonumber \\ 
&\quad \le \C_2(T) \prod_{i=1}^m \lfloor f_i  \mathbf{1}_{B^*}\rfloor_{A, r}
\lesssim \C_2(T) [\mu(B)/\mu(A)]^{\frac{m}{r}} \prod_{i=1}^m \langle f_i \rangle_{B^*, r}, 
\end{align}
where we have used Lemma \ref{lem:ABF} part \eqref{ABF-2} in the last inequality. Then combining \eqref{TK-1} with \eqref{TK-2}, we obtain that for all $x \in A$, 
\begin{align*}
&\|\T(\vec{f} \mathbf{1}_{B^*})(x) - \T(\vec{f} \mathbf{1}_{A^*})(x)\|_{\bB} 
\\
&\qquad\lesssim (\C_2(T) +\|T\|) [\mu(B)/\mu(A)]^{\frac{m}{r}} \prod_{i=1}^m \langle f_i \rangle_{B^*, r}. 
\end{align*}
This shows part \eqref{ABC-1}. 

Now let $A, B, C \in \B$ with $A \subset B \subset C$, and let $x \in A$. Then by part \eqref{ABC-1}, 
\begin{align}\label{TCT-1}
&\|\T(\vec{f} \mathbf{1}_{C^*})(x) - \T(\vec{f} \mathbf{1}_{B^*})(x)\|_{\bB} 
\nonumber \\
&\qquad\lesssim (\C_2(T) +\|T\|) [\mu(C)/\mu(B)]^{\frac{m}{r}} \prod_{i=1}^m \langle f_i \rangle_{C^*, r}, 
\end{align}
and by definition,  
\begin{align}\label{TCT-2}
\|\T(\vec{f} \mathbf{1}_{B^*})(x) & - \T(\vec{f} \mathbf{1}_{A^*})(x)\|_{\bB} 
\le \Delta(A, B) \prod_{i=1}^m \langle f_i \rangle_{B^*, r} 
\nonumber \\ 
&\lesssim \Delta(A, B) [\mu(C^*)/\mu(B^*)]^{\frac{m}{r}} \prod_{i=1}^m \langle f_i \rangle_{C^*, r}
\nonumber \\ 
&\lesssim \Delta(A, B) [\mu(C)/\mu(B)]^{\frac{m}{r}} \prod_{i=1}^m \langle f_i \rangle_{C^*, r}. 
\end{align}
Then part \eqref{ABC-2} follows from \eqref{TCT-1} and \eqref{TCT-2}, while part \eqref{ABC-3} is a direct consequence of part \eqref{ABC-2}.
\end{proof}
%%%%%%%%%%%%%%%%%%%%%%%%%% END END END PROOF %%%%%%%%%%%%%%%%%%%%%%%

%%%%%%%%%%%%%%%%%%%%%%%%% LEMMA LEMMA LEMMA %%%%%%%%%%%%%%%%%%%%%%%
\begin{lemma}\label{lem:DBB}
Let $(\Sigma, \mu)$ be a measure space with a ball-basis $\B$. Assume that $T$ is a $\bB$-valued multilinear bounded oscillation operator with respect to $\B$ and $r \in [1, \infty)$ such that $T$ is bounded from $L^r(\Sigma, \mu) \times \cdots \times L^r(\Sigma, \mu)$ to $L^{\frac{r}{m}, \infty}(\Sigma, \mu)$. Denote $\C(T) := \C_1(T) + \C_2(T) 
+ \|T\|_{L^r(\Sigma, \mu) \times \cdots \times L^r(\Sigma, \mu) \to L^{\frac{r}{m}, \infty}(\Sigma, \mu)}$. Then the following properties hold: 
\begin{list}{\rm (\theenumi)}{\usecounter{enumi}\leftmargin=1.2cm \labelwidth=1cm \itemsep=0.2cm \topsep=.2cm \renewcommand{\theenumi}{\alph{enumi}}}

\item\label{DBB-1} For any $B \in \B$ with $B^{(1)}=B$ there exists $\widetilde{B} \in \B$ such that 
\begin{align*}
B^{(2)} \subset \widetilde{B}, \quad 
\Delta(B^{(2)}, \widetilde{B}) \lesssim \C(T), \quad\text{and}\quad  
\mu(\widetilde{B}) \ge 2 \mu(B). 
\end{align*}

\item\label{DBB-2} For any $B \in \B$ there exists $\widetilde{B} \in \B$ such that $B^{(2)} \subset \widetilde{B}$, 
\begin{align*}
\Delta(B^{(2)}, \widetilde{B}) \lesssim \C(T), \quad 
\text{and either $\widetilde{B}^{(1)}=\widetilde{B}$ or $\mu(\widetilde{B}) \ge 2 \mu(B)$}. 
\end{align*}

\item\label{DBB-3} For any $B \in \B$ there exists a sequence $\{B_k\}_{k \ge 0} \subset \B$ such that $\Sigma=\bigcup_{k \ge 0} B_k$ with $B_0=B$, 
\begin{align*}
B^{(2)}_{k-1} \subset B_k, \quad \text{and}\quad
\Delta(B_{k-1}^{(2)}, B_k) \lesssim \C(T), \quad k \ge 1. 
\end{align*}
\end{list}
\end{lemma}
%%%%%%%%%%%%%%%%%%%%%%%%% LEMMA LEMMA LEMMA %%%%%%%%%%%%%%%%%%%%%%%

%%%%%%%%%%%%%%%%%%%%%%%%%% PROOF PROOF PROOF %%%%%%%%%%%%%%%%%%%%%%%
\begin{proof}
We begin with the proof of part \eqref{DBB-1}. Fix $B \in \B$ with $B^{(1)}=B$. By the condition \eqref{list:T-size}, there exists $\widetilde{B} \in \B$ so that $B \subsetneq \widetilde{B}$ and $\Delta(B, \widetilde{B}) \lesssim \C_1(T)$. Then it follows from $B^{(1)}=B$ that 
\begin{equation*}
B^{(2)}=B \subsetneq \widetilde{B}
\quad\text{ and }\quad 
\Delta(B^{(2)}, \widetilde{B}) = \Delta(B, \widetilde{B}) \lesssim \C_1(T). 
\end{equation*}
Also, the property \eqref{list:B4} implies $\mu(\widetilde{B}) \ge 2\mu(B)$, otherwise, $\widetilde{B} \subset B^{(1)}=B$.

To show part \eqref{DBB-2}, fixing $B \in \B$, we set $\B_B := \{A \in \B: B \subset A^{(1)}\}$, 
\begin{align}\label{def:ab}
a := \sup_{A \in \B_B: \mu(A) \le 2 \mu(B)} \mu(A), 
\quad\text{and}\quad 
b := \inf_{A \in \B_B: \mu(A) > 2 \mu(B)} \mu(A).
\end{align} 
Note that $a \le 2 \mu(B) \le b$. Then, there exist $B_1, B_2 \in \B_B$ so that 
\begin{equation}\label{aabb}
a/2 < \mu(B_1) \le a \le 2\mu(B) \le b \le \mu(B_2) < 2b. 
\end{equation}
We first treat the case $b > \C_0^2 a$. In this scenario, choosing $\widetilde{B} := B_1^{(1)} \supset B$, we invoke \eqref{aabb} to get  
\begin{align}\label{BBb}
\mu(\widetilde{B}^{(1)})
=\mu(B_1^{(2)}) 
\le \C_0^2 \mu(B_1) 
\le \C_0^2 a 
<b. 
\end{align}
Observe that by definition in \eqref{def:ab}, there is no $B' \in \B_B$ such that $a<\mu(B')<b$. This and \eqref{BBb} yield $\mu(\widetilde{B}^{(1)}) \le a \le 2\mu(B_1)$, which along with $\widetilde{B}^{(1)} \cap B_1 \neq \emptyset$ and property \eqref{list:B4} implies $\widetilde{B}^{(1)} \subset B_1^{(1)}=\widetilde{B} \subset \widetilde{B}^{(1)}$. That is, $\widetilde{B}^{(1)}=\widetilde{B}$. Hence, $B^{(2)} \subset \widetilde{B}^{(2)}=\widetilde{B}$, and by \eqref{aabb}, 
\begin{align}\label{aabb-1}
\mu(B^{(2)}) 
\le \mu(\widetilde{B})
=\mu(B_1^{(1)})
\le \C_0 \mu(B_1)
\le 2\C_0 \mu(B) 
\le 2\C_0 \mu(B^{(2)}). 
\end{align}
If $b \le \C_0^2 a$, picking $\widetilde{B} := B_2^{(3)} = (B_2^{(1)})^{(2)} \supset B^{(2)}$, we use \eqref{aabb} to see that $\mu(\widetilde{B}) \ge \mu(B_2) \ge b \ge 2\mu(B)$ and 
\begin{align}\label{aabb-2}
\mu(B^{(2)}) 
\le \mu(\widetilde{B})
&=\mu(B_2^{(3)})
\le \C_0^3 \mu(B_2)
\le 2\C_0^3 b 
\nonumber \\
&\le 2\C_0^5 a 
\le 4\C_0^5 \mu(B)
\le 4\C_0^5 \mu(B^{(2)}). 
\end{align}
Thus, in light of \eqref{aabb-1} and \eqref{aabb-2}, we obtain $\mu(\widetilde{B}) \simeq \mu(B^{(2)})$, which together with  Lemma \ref{lem:ABC} part \eqref{ABC-1} concludes that 
\begin{align*}
\Delta(B^{(2)}, \widetilde{B}) 
\lesssim \C(T). 
\end{align*}

Finally, having shown parts \eqref{DBB-1} and \eqref{DBB-2}, to achieve part \eqref{DBB-3}, we just need to set $B_0:=B$ and use induction to construct the desired sequence $\{B_k\}_{k \ge 0}$ in the following way: 
\begin{align*}
B^{(2)}_{k-1} \subset B_k, \quad 
\Delta(B_{k-1}^{(2)}, B_k) \lesssim \C(T), \quad\text{and}\quad 
\mu(B_k) \ge 2\mu(B_{k-1}), \quad k \ge 1. 
\end{align*} 
Additionally, it follows from Lemma \ref{lem:BG} that $\Sigma=\bigcup_{k \ge 0} B_k$. This completes the proof. 
\end{proof}
%%%%%%%%%%%%%%%%%%%%%%%%%% END END END PROOF %%%%%%%%%%%%%%%%%%%%%%%

%%%%%%%%%%%%%%%%%%%%%%%%%% PROOF PROOF PROOF %%%%%%%%%%%%%%%%%%%%%%%
\begin{proof}[\bf Proof of Theorem \ref{thm:doubling}]
Fix $\lambda>0$ and set $E:=\{x \in \Sigma: T_*(\vec{f})(x) > \lambda\}$. We may assume that $E$ is bounded. For any $x \in E$, there exists $B_x \in \B$ containing $x$ so that $\|\T(\vec{f})(x) - \T(\vec{f} \mathbf{1}_{B_x^*})(x)\|_{\bB} \ge \lambda$. Then $E \subset \bigcup_{x \in E} B_x$. Since $E$ is bounded, by Lemma \ref{lem:BE} part \eqref{list-1}, one can find a disjoint subfamily $\{B_k\}$ such that $E \subset \bigcup_k B_k^*$. Denote 
\begin{align*}
B'_k := \Big\{x \in B_k: |T(\vec{f} \mathbf{1}_{B_k^*})(x)| 
< (2\C_0)^{\frac{m}{r}} \|T\| \prod_{i=1}^m \lfloor f_i \rfloor_{B_k, r} \Big\}.
\end{align*}
Since $T$ is bounded from $L^r(\Sigma, \mu) \times \cdots \times L^r(\Sigma, \mu)$ to $L^{\frac{r}{m}, \infty}(\Sigma, \mu)$, 
\begin{align*}
\mu(\Sigma \setminus B'_k)
\le \bigg(\frac{\prod_{i=1}^m \|f_i \mathbf{1}_{B_k^*}\|_{L^r(\Sigma, \mu)}}
{(2\C_0)^{m/r} \prod_{i=1}^m \lfloor f_i \rfloor_{B_k, r}}\bigg)^{\frac{r}{m}}
\le (2\C_0)^{-1} \mu(B_k^*) 
\le \frac12 \mu(B_k). 
\end{align*}
Hence, one has 
\begin{align}\label{BBk-1}
\mu(B'_k) \ge \mu(B_k) - \mu(\Sigma \setminus B'_k) \ge \mu(B_k)/2.
\end{align}
Fix $k$, let $x_k$ be the defining point of $B_k$ and $x \in B'_k \setminus\{\mathcal{M}_{\B, r}^{\otimes}(\vec{f}) > \delta \lambda\}$, where 
\begin{align}\label{dede}
\delta := \frac{1}{4(2\C_0)^{m/r}(\C_2(T) + \|T\|)}.
\end{align}
Then by definition and \eqref{dede}, 
\begin{multline}\label{TS-1}
\|\T(\vec{f} \mathbf{1}_{B_k^*})(x)\|_{\bB} 
= |T(\vec{f} \mathbf{1}_{B_k^*})(x)| 
< (2\C_0)^{\frac{m}{r}} \|T\| \prod_{i=1}^m \lfloor f_i \rfloor_{B_k, r}
\\ 
\le (2\C_0)^{\frac{m}{r}} \|T\| \mathcal{M}_{\B, r}^{\otimes}(\vec{f})(x) 
\le (2\C_0)^{\frac{m}{r}} \|T\| \delta \lambda
\le \lambda/4.  
\end{multline}
Moreover, it follows from the condition \eqref{list:T-reg} and \eqref{dede} that 
\begin{multline}\label{TS-2}
S(x) 
:= \|(\T(\vec{f}) - \T(\vec{f} \mathbf{1}_{B_k^*}))(x) - (\T(\vec{f}) - \T(\vec{f} \mathbf{1}_{B_k^*}))(x_k)\|_{\bB}
\\ 
\le \C_2(T) \prod_{i=1}^m \lfloor f_i \rfloor_{B_k, r}  
\le \C_2(T) \mathcal{M}_{\B, r}^{\otimes}(\vec{f})(x) 
\le \C_2(T) \delta \lambda
\le \lambda/4. 
\end{multline}
Then gathering \eqref{TS-1} and \eqref{TS-2}, we arrive at 
\begin{align*}
\lambda
&\le \|\T(\vec{f})(x_k) - \T(\vec{f} \mathbf{1}_{B_k^*})(x_k)\|_{\bB} 
\\
&\le \|\T(\vec{f})(x)\|_{\bB} + \|\T(\vec{f} \mathbf{1}_{B_k^*})(x)\|_{\bB} + S(x)
\\
&\le \|\T(\vec{f})(x)\|_{\bB} + \lambda/2. 
\end{align*}
That is, $|T(\vec{f})(x)| = \|\T(\vec{f})(x)\|_{\bB} \ge \frac12 \lambda$. Consequently, we have shown that 
\begin{align}\label{BBk-2}
\bigcup_{k} B'_k 
\subset \{x: \mathcal{M}_{\B, r}^{\otimes}(\vec{f})(x) > \delta \lambda\} 
\cup \{x: |T(\vec{f})(x)| \ge \lambda/2\}. 
\end{align}
Recall that both $T$ and $\mathcal{M}_{\B, r}^{\otimes}$ are bounded from $L^r(\Sigma, \mu) \times \cdots \times L^r(\Sigma, \mu)$ to $L^{\frac{r}{m}, \infty}(\Sigma, \mu)$. 
We then use \eqref{BBk-1} and \eqref{BBk-2} to deduce that 
\begin{align*}
\mu(E)
&\le \sum_k \mu(B^*_k) 
\lesssim \sum_k \mu(B_k) 
\le \sum_k \mu(B'_k) 
\\
&\lesssim (\delta^{-1} \|\mathcal{M}_{\B, r}^{\otimes}\| + \|T\|)^{\frac{r}{m}} \lambda^{-\frac{r}{m}}  
\prod_{i=1}^m \|f_i\|_{L^r(\Sigma, \mu)}^{\frac{r}{m}}
\\
&\lesssim (\C_2(T) + \|T\|)^{\frac{r}{m}} \lambda^{-\frac{r}{m}}  \prod_{i=1}^m \|f_i\|_{L^r(\Sigma, \mu)}^{\frac{r}{m}}, 
\end{align*}
where the implicit constants are independent of $\lambda$. This shows part \eqref{dou-1}. 

Next, let us turn to the proof of part \eqref{dou-2}. Fix $A \in \B$ with $A^* \subsetneq \Sigma$. Since $\B$ satisfies the doubling condition, there exist $\gamma>1$ and $B \in \B$ such that $A \subset B$ and $\mu(B) \le \gamma \, \mu(A)$. By Lemma \ref{lem:ABC} part \eqref{ABC-1}, this gives 
\begin{align*}
\Delta(A, B) 
\lesssim (\C_2(T) + \|T\|) [\mu(B)/\mu(A)]^{\frac{m}{r}} 
\le \gamma^{\frac{m}{r}} (\C_2(T) + \|T\|), 
\end{align*}
which means that $T$ satisfies the condition \eqref{list:T-size}. 
\end{proof}
%%%%%%%%%%%%%%%%%%%%%%%%%% END END END PROOF %%%%%%%%%%%%%%%%%%%%%%%

%%%%%%%%%%%%%%%%%%%%%%%% SECTION SECTION SECTION %%%%%%%%%%%%%%%%%%%%%%
%%%%%%%%%%%%%%%%%%%%%%%% SECTION SECTION SECTION %%%%%%%%%%%%%%%%%%%%%%
\section{Sparse bounds for bounded oscillation operators}\label{sec:sparse} 
In what follows, we always let $(\Sigma, \mu)$ be a measure space with a ball-basis $\B$. Assume that $T$ is a multilinear bounded oscillation operator with respect to  $\B$ and $r \in [1, \infty)$ such that $T$ is bounded from $L^r(\Sigma, \mu) \times \cdots \times L^r(\Sigma, \mu)$ to $L^{\frac{r}{m}, \infty}(\Sigma, \mu)$. Denote 
\[
\C(T) := \C_1(T) + \C_2(T) 
+ \|T\|_{L^r(\Sigma, \mu) \times \cdots \times L^r(\Sigma, \mu) \to L^{\frac{r}{m}, \infty}(\Sigma, \mu)}.
\]

%%%%%%%%%%%%%%%%%%%%%%%%% LEMMA LEMMA LEMMA %%%%%%%%%%%%%%%%%%%%%%%
\begin{lemma}\label{lem:GA}
Let $\lambda \ge 3\C_0^4$. Let $F \subset \Sigma$ be a measurable set and $A \in \B$ with $F \cap A \neq \emptyset$ such that $\mu(F) \le \lambda^{-1} \mu(A)$. Then there exists a collection $\G \subset \B$ satisfying 
\begin{align}
&F \cap A^* \cap G \neq \emptyset, \, \forall G \in \G, \quad 
\\
&F \cap A^* \subset \bigcup_{G \in \G} G \text{ a.s.}, \quad
\\ 
\label{GA-1} &\mu \Big(\bigcup_{G \in \G} G^* \Big) \le 3\C_0^2 \lambda^{-1} \mu(A), 
\end{align}
and for every $G \in \G$ there exists $\widetilde{G} \in \B$ such that 
\begin{align}\label{GA-2}
\widetilde{G} \not\subset F, \quad 
G^{(2)} \subset \widetilde{G} \subset A^*, \quad\text{and}\quad 
\Delta(G^{(2)}, \widetilde{G}) \lesssim \C(T), 
\end{align}
where the implicit constants are independent of $F$, $A$, $\lambda$, and $\G$.
\end{lemma}
%%%%%%%%%%%%%%%%%%%%%%%%% LEMMA LEMMA LEMMA %%%%%%%%%%%%%%%%%%%%%%%

%%%%%%%%%%%%%%%%%%%%%%%%%% PROOF PROOF PROOF %%%%%%%%%%%%%%%%%%%%%%%
\begin{proof}
By Lemma \ref{lem:BE} part \eqref{list-5} applied to $E:=F \cap A^*$, we get a subfamily $\B' \subset \B$ such that $E \cap B \neq \emptyset$ for all $B \in \B'$, 
\begin{align}\label{FAG-1}
E \subset \bigcup_{B \in \B'} B \text{ a.s.},  
\quad\text{ and }\quad 
\sum_{B \in \B'} \mu(B) \le 2\C_0 \, \mu(E). 
\end{align}
Given $B \in \B'$, we use Lemma \ref{lem:DBB} part \eqref{DBB-3} to find a sequence $\{B_k\}_{k \ge 0} \subset \B$ such that $\Sigma=\bigcup_{k \ge 0} B_k$ with $B_0=B$, 
\begin{align}\label{FAG-2}
B^{(2)}_{k-1} \subset B_k, \quad \text{and}\quad
\Delta(B_{k-1}^{(2)}, B_k) \lesssim \C(T), \quad k \ge 1. 
\end{align}
For each $B \in \B'$, we attach $G=B_k$, where $k \ge 0$ is the least index satisfying $B_{k+1}^{(1)} \not\subset F$. We then use $\G$ to denote the collection of all such $G$. 

For each $G \in \G$, by construction, there exists $B \in \B'$ such that $G=B$ or $G=B_k \supset B^{(2)}$ for some $k \ge 1$. This together with the first two estimates in \eqref{FAG-1} gives $F \cap A^* \cap G=E \cap G \neq \emptyset$ and $F \cap A^* =E \subset \bigcup_{G \in \G} G \text{ a.s.}$, respectively. Additionally, by definition of $\G$, since $G=B_k \in \G$ for some $k \ge 0$, $G^* \not\subset F$ implies $G=B_0$, otherwise $G=B_k$ with $k \ge 1$ and $G^*=B_k^* \subset F$. This yields  
\begin{align}\label{GGG}
\mu \Big(\bigcup_{G \in \G} G^* \Big) 
&=\mu \Big(\bigcup_{G \in \G: G^* \subset F} G^* \Big) + \mu \Big(\bigcup_{G \in \G: G^* \not\subset F} G^* \Big) 
\nonumber \\
&\le \mu(F) + \mu \Big(\bigcup_{B \in \B'} B^* \Big) 
\le \mu(F) + \sum_{B \in \B'} \mu(B^*) 
\nonumber \\
&\le \mu(F) + \C_0 \sum_{B \in \B'} \mu(B) 
\le (1+2\C_0^2) \mu(F) 
\nonumber \\ 
&\le (1+2\C_0^2) \lambda^{-1} \mu(A) 
\le 3\C_0^2 \lambda^{-1} \mu(A), 
\end{align}
where we have used \eqref{FAG-1} and that $\mu(F) \le \lambda^{-1} \mu(A)$. This shows \eqref{GA-1}.

To proceed, given $G=B_k \in \G$ for some $k \ge 0$, we define 
\begin{equation*}
\widetilde{G} := 
\begin{cases}
A^*, & \text{if } \mu(B_{k+1}^*) > \mu(A), 
\\
B_{k+1}^*, &\text{if } \mu(B_{k+1}^*) \le \mu(A). 
\end{cases}
\end{equation*}
Note that $\mu(A^*) \ge \mu(A) \ge \lambda \mu(F) > \mu(F)$, hence, $A^* \not\subset F$. This and the construction of $\G$ yield that $\widetilde{G} \not\subset F$. Moreover, it follows from \eqref{GGG} that $\mu(G^{(2)}) \le \C_0^2 \mu(G)
\le 3\C_0^4 \lambda^{-1} \mu(A) \le \mu(A)$, which along with $G \cap A^* \neq \emptyset$ and property \eqref{list:B4} gives $G^{(2)} \subset A^*$. If $\mu(B_{k+1}^*) > \mu(A)$, then $G^{(2)} \subset A^*=\widetilde{G} \subset B_{k+1}^{(2)}$, which along with Lemma \ref{lem:ABF} part \eqref{ABF-1} and Lemma \ref{lem:ABC} part \eqref{ABC-3} implies 
\begin{align*}
\Delta(G^{(2)}, \widetilde{G})
\le \Delta(B_k^{(2)}, B_{k+1}^{(2)})
\lesssim \C(T) + \Delta(B_k^{(2)}, B_{k+1})
\lesssim \C(T), 
\end{align*}
where the last estimate in \eqref{FAG-2} was used in the last step.  If $\mu(B_{k+1}^*) \le \mu(A)$, then by $G \cap A^* \neq \emptyset$ and property \eqref{list:B4}, we have $G^{(2)} = B_k^{(2)} \subset B_{k+1} \subset B_{k+1}^* = \widetilde{G} 
\subset A^*$ and as above, 
\begin{align*}
\Delta(G^{(2)}, \widetilde{G})
= \Delta(B_k^{(2)}, B_{k+1}^{(1)})
\lesssim \C(T) + \Delta(B_k^{(2)}, B_{k+1})
\lesssim \C(T).  
\end{align*}
This shows \eqref{GA-2}.
\end{proof}
%%%%%%%%%%%%%%%%%%%%%%%%%% END END END PROOF %%%%%%%%%%%%%%%%%%%%%%%

Define 
\[
\Gamma(\vec{f})(x) := \max\big\{|T(\vec{f})(x)|, \, T_*(\vec{f})(x), \, 
\C(T) \M_{\B, r}^{\otimes}(\vec{f})(x) \big\}. 
\]
Then by Theorem \ref{thm:doubling} and Lemma \ref{lem:M}, we see that 
\begin{align}\label{eq:gam}
\|\Gamma\| := \|\Gamma\|_{L^r(\Sigma, \mu) \times \cdots \times L^r(\Sigma, \mu) \to L^{\frac{r}{m}, \infty}(\Sigma, \mu)} 
\lesssim \C(T). 
\end{align}
In this sequel, we fix $B_0 \in \B$ and $f_i \in L^r(\Sigma, \mu)$ with $\|f_i\|_{L^r(B_0, \mu)}^r \ge \frac12 \|f_i\|_{L^r(\Sigma, \mu)}^r$, $i=1, \ldots, m$. Without loss of generality, we assume that 
\begin{align}\label{support} 
\supp(f_i) \subset B_0^{(3)}, \quad i=1, \ldots, m, 
\end{align}
since 
\begin{equation*}
\|f_i \mathbf{1}_{B_0^{(3)}}\|_{L^r(B_0, \mu)}^r 
= \|f_i\|_{L^r(B_0, \mu)}^r 
\ge \frac12 \|f_i\|_{L^r(\Sigma, \mu)}^r
\ge \frac12 \|f_i \mathbf{1}_{B_0^{(3)}}\|_{L^r(\Sigma, \mu)}^r. 
\end{equation*}

%%%%%%%%%%%%%%%%%%%%%%%%% LEMMA LEMMA LEMMA %%%%%%%%%%%%%%%%%%%%%%%
\begin{lemma}\label{lem:tree}
Let $\lambda \ge 3\C_0^4$. Then there exists a family $\G=\G(B_0) \subset \B$ such that for each $A \in \G$ one can find a subfamily $\F(A) \subset \G$ satisfying the following properties: 
\begin{align}
\label{tree-1}
&A^* \cap B \neq \emptyset, \quad \forall B \in \F(A), 
\\
\label{tree-2}
&\mu \Big(\bigcup_{B \in \F(A)} B^* \Big) \le 3\C_0^2 \lambda^{-1} \mu(A), 
\\
\label{tree-3}
&\Gamma(\vec{f} \mathbf{1}_{A^{(3)}})(x) 
\lesssim \C(T) \,  \lambda^{\frac{m}{r}} \prod_{i=1}^m \langle f_i \rangle_{A^{(3)}, r}, 
\end{align}
for a.e. $x \in A^* \setminus \bigcup_{B \in \F(A)} B$, and for every $B \in \F(A)$ there are $\widetilde{B} \in \B$ and $\xi \in \widetilde{B}$ such that 
\begin{align}
\label{tree-4} & B^{(2)} \subset \widetilde{B} \subset A^*, 
\\
\label{tree-5} &\Gamma(\vec{f} \mathbf{1}_{A^{(3)}})(\xi) 
\lesssim \C(T) \,  \lambda^{\frac{m}{r}} \prod_{i=1}^m \langle f_i \rangle_{A^{(3)}, r}, 
\\
\label{tree-6} & \|\T(\vec{f} \mathbf{1}_{\widetilde{B}^*})(x) - \T(\vec{f} \mathbf{1}_{B^{(3)}})(x)\|_{\bB}
\lesssim \C(T) \,  \lambda^{\frac{m}{r}} \prod_{i=1}^m \langle f_i \rangle_{A^{(3)}, r},  
\end{align}
for all $x \in B^{(2)}$.
\end{lemma}
%%%%%%%%%%%%%%%%%%%%%%%%% LEMMA LEMMA LEMMA %%%%%%%%%%%%%%%%%%%%%%%

%%%%%%%%%%%%%%%%%%%%%%%%%% PROOF PROOF PROOF %%%%%%%%%%%%%%%%%%%%%%%
\begin{proof}
Set 
\begin{align}\label{FAA}
F(B_0) := \bigg\{x \in \Sigma: \Gamma(\vec{f} \mathbf{1}_{B_0^{(3)}})(x) 
> 2^{\frac{m}{r}} \C_0^{\frac{3m}{r}} \lambda^{\frac{m}{r}} 
\|\Gamma\| \prod_{i=1}^m \langle f_i \rangle_{B_0^{(3)}, r} \bigg\}. 
\end{align}
It follows from \eqref{eq:gam} that 
\begin{align*}
\mu(F(B_0)) 
\le \Bigg(\frac{\|\Gamma\| \prod_{i=1}^m \|f_i\|_{L^r(B_0^{(3)}, \mu)}}
{2^{\frac{m}{r}} \C_0^{\frac{3m}{r}} \lambda^{\frac{m}{r}} 
\|\Gamma\| \prod_{i=1}^m \langle f_i \rangle_{B_0^{(3)}, r}}\Bigg)^{\frac{r}{m}} 
\le \frac{\mu(B_0^{(3)})}{\C_0^3 \lambda} 
\le \frac{\mu(B_0)}{\lambda}. 
\end{align*}
Then by Lemma \ref{lem:GA} applied to $A=B_0$ and $F=F(B_0)$, one can find a collection $\F(B_0) \subset \B$ such that 
\begin{align}
\label{GAA-1} & F(B_0) \cap B_0^* \cap B \neq \emptyset, \quad \forall B \in \F(B_0), 
\\
\label{GAA-2} & F(B_0) \cap B_0^* \subset \bigcup_{B \in \F(B_0)} B \text{ a.s.}, 
\\
\label{GAA-3} & \mu \Big(\bigcup_{B \in \F(B_0)} B^* \Big) \le 3\C_0^2 \lambda^{-1} \mu(B_0), 
\end{align}
and for every $B \in \F(B_0)$ there exists $\widetilde{B} \in \B$ such that 
\begin{align}\label{GAA-4}
\widetilde{B} \not\subset F(B_0), \quad 
B^{(2)} \subset \widetilde{B} \subset B_0^*, \quad\text{and}\quad 
\Delta(B^{(2)}, \widetilde{B}) \lesssim \C(T). 
\end{align}
Then, \eqref{GAA-2} means that $\mu\big(F(B_0) \cap B_0^* \setminus \bigcup_{B \in \F(B_0)} B \big)=0$, which along with \eqref{FAA} and \eqref{eq:gam} gives 
\begin{align}\label{}
&\Gamma(\vec{f} \mathbf{1}_{B_0^{(3)}})(x) 
\lesssim \C(T) \lambda^{\frac{m}{r}} \, \prod_{i=1}^m \langle f_i \rangle_{B_0^{(3)}, r}, \quad 
\text{a.e. } x \in B_0^* \setminus \bigcup_{B \in \F(B_0)} B. 
\end{align}
Additionally, by the first estimate in \eqref{GAA-4}, definition of $F(B_0)$, and \eqref{eq:gam}, there exists $\xi \in \widetilde{B} \setminus F(B_0)$ so that 
\begin{align}\label{Gaf}
\Gamma(\vec{f} \mathbf{1}_{B_0^{(3)}})(\xi) 
\lesssim \C_T \,  \lambda^{\frac{m}{r}} \prod_{i=1}^m \langle f_i \rangle_{B_0^{(3)}, r}. 
\end{align}
Moreover, by the last two estimates in \eqref{GAA-4} and \eqref{Gaf}, we have for any $x \in B^{(2)}$, 
\begin{align*}
&\|\T(\vec{f} \mathbf{1}_{\widetilde{B}^*})(x) - \T(\vec{f} \mathbf{1}_{B^{(3)}})(x)\|_{\bB}
\\
&=\|\T(\vec{f} \mathbf{1}_{B_0^{(3)} \cap \widetilde{B}^*})(x) - \T(\vec{f} \mathbf{1}_{B_0^{(3)} \cap B^{(3)}})(x)\|_{\bB}
\\
&\le \Delta(B^{(2)}, \widetilde{B}) \prod_{i=1}^m \langle f_i \mathbf{1}_{B_0^{(3)}} \rangle_{\widetilde{B}^*, r}
\lesssim \C(T) \inf_{\widetilde{B}} \mathcal{M}_{\B, r}^{\otimes}(\vec{f} \mathbf{1}_{B_0^{(3)}})
\\
&\le \Gamma(\vec{f} \mathbf{1}_{B_0^{(3)}})(\xi) 
\lesssim \C(T) \,  \lambda^{\frac{m}{r}} \prod_{i=1}^m \langle f_i \rangle_{B_0^{(3)}, r}. 
\end{align*}
Therefore, we have verified \eqref{tree-1}--\eqref{tree-6} for $A=B_0$. 

Then for each $A \in \F(B_0)$, we repeat the proceeding procedure for $A$ instead of $B_0$ to obtain a family $\F(A)$ satisfying the properties \eqref{tree-1}--\eqref{tree-6}. Define recursively $\F_0(B_0) := \{B_0\}$, 
\begin{align}\label{FFF}
\F_1(B_0) := \F(B_0), \quad\text{and}\quad  
\F_{k+1}(B_0) := \bigcup_{B \in \F_k(B_0)} \F(B), \quad k \ge 1. 
\end{align}
Then we set 
\begin{equation}\label{generation}
\G = \G(B_0) := \bigcup_{k \ge 0} \F_k(B_0). 
\end{equation}
As argued above, the family $\G$ satisfies \eqref{tree-1}--\eqref{tree-6}. This completes the proof. 
\end{proof}
%%%%%%%%%%%%%%%%%%%%%%%%%% END END END PROOF %%%%%%%%%%%%%%%%%%%%%%%

To proceed, we give some notation. The set $\F(B)$ above is the collection of stopping time balls related to $B$. Then $\F_k(B)$  is the family of the $k$-th generation stopping time balls related to $B$, while $\G$ is the collection of all generation balls. Given $B, B' \in \G$ and $k \ge 0$, we denote $\mathscr{F}^k(B')=B$ if $B' \in \F_k(B)$, for which we say that $B$ is the $k$-th generation ancestor of $B'$ in $\G$, or $B'$ is one of the $k$-th generation descendants of $B$ in $\G$. We also denote $\mathscr{F} :=\mathscr{F}^1$. By \eqref{tree-2}, one has 
\begin{align}\label{FkFk}
\mu \bigg(\bigcup_{B' \in \F_k(B)} B' \bigg) 
\le (3\C_0^2 \lambda^{-1})^k \mu(B), \quad \forall k \ge 1. 
\end{align}

%%%%%%%%%%%%%%%%%%%%%%%%%% PROOF PROOF PROOF %%%%%%%%%%%%%%%%%%%%%%%
\begin{proof}[\bf Proof of Theorem \ref{thm:sparse}]
We only focus on the proof of \eqref{Tbsparse} since the proof of \eqref{Tsparse} is similar.  Let $\lambda>3\C_0^6$ be a constant chosen later. Fix $B_0 \in \B$. By Lemma \ref{lem:tree}, one can find a family $\G \subset \B$ satisfying the properties \eqref{tree-1}--\eqref{tree-6}. 

Given $B \in \B$, we write $r(B) := [\frac12 \log_{\C_0} \mu(B)]$,  
\begin{equation*}
\G^{k_0}  := \{B_0\}, \quad\text{ and }\quad 
\G^k := \{B \in \G: r(B)=k\}, \quad k < k_0, 
\end{equation*}
where $k_0 := r(B_0)$. It follows from \eqref{tree-2} that $r(\mathscr{F}(B)) \ge r(B)+1$. We also define $\mathcal{B}$ as the collection of $B \in \G$ satisfying that for some $B' \in \G$, 
\begin{align}\label{FBFB}
B^{(2)}\cap B' \neq \emptyset 
\quad\text{ and }\quad 
r(\mathscr{F}(B)) + 2 \le r(B') \le r(\mathscr{F}(B)) -2. 
\end{align}
Furthermore, we set 
\begin{align*}
\H^{k_0} := \G^{k_0}=\{B_0\} 
\quad\text{ and }\quad
\H^k := \G^k \setminus \bigcup_{B \in \mathcal{B}}  \G(B), 
\quad k < k_0.
\end{align*}
Note that \eqref{tree-4} gives that $\bigcup_{B \in \H^k} B \subset B_0^*$ for each $k \le k_0$. This along with Lemma \ref{lem:BE} part \eqref{list-1} implies that  there exists a disjoint subfamily $\D^k \subset \H^k$ such that 
\begin{align}\label{BGBD}
\bigcup_{B \in \H^k} B \subset \bigcup_{B \in \D^k} B^{(1)}.
\end{align} 
Then we set 
\begin{align}\label{DHG}
\D := \bigcup_{k \le k_0} \D^k
\subset \bigcup_{k \le k_0} \H^k
\subset \bigcup_{k \le k_0} \G^k
\subset \G. 
\end{align}
Writing  
\[
\D_1 := \{B \in \D: r(B) \text{ is odd}\} 
\quad\text{ and }\quad
\D_2 := \{B \in \D: r(B) \text{ is even}\}, 
\]
we see that $\D=\D_1 \cup \D_2$. In view of Lemma \ref{lem:1/2-sparse} below,  both $\D_1$ and $\D_2$ are $\frac12$-sparse, which in turn implies that 
\begin{align}\label{SS}
\text{$\S_i :=\{B^{(3)}: B \in \D_i\}$ is $\frac{1}{2\C_0^3}$-sparse}, \quad i=1, 2. 
\end{align}

By Lemma \ref{lem:GAFA} below, there exists a set $F_0$ of zero measure such that for any $B \in \D$,  
\begin{align}\label{Gafb}
\Gamma(\vec{f} \mathbf{1}_{B^{(3)}}) (x) 
\lesssim \C(T) \prod_{i=1}^m \langle f_i \rangle_{B^{(3)}, r}, \quad 
x \in \bigg(B^* \backslash \bigcup_{G \in \D \atop r(G)<r(B)} G^* \bigg) \backslash F_0. 
\end{align}
Note that 
\[
\mu(F_1) := \mu \bigg(\bigcap_{k \le k_0} \bigcup_{G \in \D: r(G) \le k} G^* \bigg) =0.
\]
Next, fix $x \in B_0 \setminus (F_0 \cup F_1)$. Then, there exists $B \in \D$ such that $x \in B^* \setminus \bigcup_{G \in \D:r(G)<r(B)} G^*$. By the construction of $\G$ and that $B \in \D \subset \G$, one can find a sequence $\{B_j\}_{j=0}^k \subset \D$ such that $B_{j+1} \in \F(B_j)$, $j=0, 1, \ldots, k-1$, and $B_k=B$. By \eqref{tree-4}--\eqref{tree-6}, there exist $\widetilde{B}_j $ and $\xi_j \in \widetilde{B}_{j+1}$, $j=0,1,\ldots,k-1$, such that 
\begin{align}
\label{BBT-1} & B_{j+1}^{(2)} \subset \widetilde{B}_{j+1} \subset B_j^{(1)}, 
\\
\label{BBT-2} & \Gamma(\vec{f} \mathbf{1}_{B_j^{(3)}})(\xi_j) 
\lesssim \C(T) \prod_{i=1}^m \langle f_i \rangle_{B_j^{(3)}, r}, 
\\ 
\label{BBT-3} & \big\|\T(\vec{f} \mathbf{1}_{\widetilde{B}_{j+1}^{(1)}})(x) 
- \T(\vec{f} \mathbf{1}_{B_{j+1}^{(3)}})(x) \big\|_{\bB} 
\lesssim \C(T) \prod_{i=1}^m \langle f_i \rangle_{B_j^{(3)}, r}, 
\end{align}
since $x \in B^*=B_k^{(1)} \subset B_{j+1}^{(2)}$. 

Observe that 
\begin{align*}
\prod_{i \in \tau} (a_i + b_i)
=\sum_{\tau_1 \uplus \tau_2 = \tau(\alpha)} 
\prod_{i_1 \in \tau_1} a_{i_1} \times \prod_{i_2 \in \tau_2} b_{i_2}, 
\end{align*}
which gives 
\begin{align}\label{Tbg}
[\T, \b]_{\alpha}(\vec{f})(x)
=\sum_{\tau_1 \uplus \tau_2 = \tau(\alpha)} \prod_{i_1 \in \tau_1} (b_{i_1}(x) - c_{i_1}) \times \T(\vec{g})(x), 
\end{align}
where 
\begin{equation*}
g_i = 
\begin{cases}
f_i, & i \not\in \tau_2,
\\ 
(b_i-c_i) f_i, & i \in \tau_2. 
\end{cases}
\end{equation*}
Then it follows from \eqref{Gafb} and \eqref{Tbg} that 
\begin{align}\label{TBb}
\|[&\T, \b]_{\alpha}(\vec{f} \mathbf{1}_{B^{(3)}})(x)\|_{\bB} 
\nonumber \\
&\le \sum_{\tau_1 \uplus \tau_2 = \tau(\alpha)} \prod_{i_1 \in \tau_1} |b_{i_1}(x) - c_{i_1}| \times \|\T(\vec{g})(x)\|_{\bB}
\nonumber \\
&\lesssim \sum_{\tau_1 \uplus \tau_2 = \tau(\alpha)} \prod_{i_1 \in \tau_1} |b_{i_1}(x) - c_{i_1}| \times 
\Gamma(\vec{g} \mathbf{1}_{B^{(3)}})(x) 
\nonumber \\
&\lesssim \sum_{\tau_1 \uplus \tau_2 = \tau(\alpha)} \prod_{i_1 \in \tau_1} |b_{i_1}(x) - b_{i_1, B^{(3)}}| \langle f_{i_1} \rangle_{B^{(3)}, r} 
\nonumber \\
&\qquad\times \prod_{i_2 \in \tau_2} \langle (b_{i_2} - b_{i_2, B^{(3)}}) f_{i_2} \rangle_{B^{(3)}, r} 
\times \prod_{i_3 \not\in \tau_1 \cup \tau_2} \langle f_{i_3} \rangle_{B^{(3)}, r}, 
\end{align}
where we have chosen $c_i=b_{i, B^{(3)}}$ in \eqref{Tbg}, $i=1, \ldots, m$. Now gathering \eqref{TBb} and \eqref{TTBj} below, we conclude that 
\begin{align*}
|[&T, \b]_{\alpha}(\vec{f})(x)| 
= \|[\T, \b]_{\alpha}(\vec{f} \mathbf{1}_{B_0^{(3)}})(x)\|_{\bB} 
\\
&\le \sum_{j=0}^{k-1} \big\|[\T, \b]_{\alpha}(\vec{f} \mathbf{1}_{B_j^{(3)}})(x) 
- [\T, \b]_{\alpha}(\vec{f} \mathbf{1}_{B_{j+1}^{(3)}})(x) \big\|_{\bB} 
\nonumber \\
&\qquad\qquad + \|[\T, \b]_{\alpha}(\vec{f} \mathbf{1}_{B^{(3)}})(x)\|_{\bB} 
\\
&\lesssim \C(T) \sum_{j=0}^k \sum_{\tau_1 \uplus \tau_2 = \tau(\alpha)} 
\prod_{i_1 \in \tau_1} |b_{i_1}(x) - b_{i_1, B_j^{(3)}}| \langle f_{i_1} \rangle_{B_j^{(3)}, r} 
\nonumber \\
&\qquad \times \prod_{i_2 \in \tau_2} \langle (b_{i_2} - b_{i_2, B_j^{(3)}}) f_{i_2} \rangle_{B_j^{(3)}, r} 
\times \prod_{i_3 \not\in \tau_1 \cup \tau_2} \langle f_{i_3} \rangle_{B_j^{(3)}, r}
\\
&\le \C(T) \sum_{\tau_1 \uplus \tau_2 = \tau(\alpha)} \sum_{B \in \S_1 \cup \S_2} 
\prod_{i_1 \in \tau_1} |b_{i_1}(x) - b_{i_1, B^{(3)}}| \langle f_{i_1} \rangle_{B^{(3)}, r} 
\nonumber \\
&\qquad \times \prod_{i_2 \in \tau_2} \langle (b_{i_2} - b_{i_2, B^{(3)}}) f_{i_2} \rangle_{B^{(3)}, r} 
\prod_{i_3 \not\in \tau_1 \cup \tau_2} \langle f_{i_3} \rangle_{B^{(3)}, r} \, \mathbf{1}_B(x). 
\end{align*}
This completes the proof of Theorem \ref{thm:sparse}. 
\end{proof}
%%%%%%%%%%%%%%%%%%%%%%%%%% END END END PROOF %%%%%%%%%%%%%%%%%%%%%%%

Finally, let us demonstrate the following three lemmas, which have been used above. 

%%%%%%%%%%%%%%%%%%%%%%%%% LEMMA LEMMA LEMMA %%%%%%%%%%%%%%%%%%%%%%%
\begin{lemma}\label{lem:1/2-sparse}
Both $\D_1$ and $\D_2$ are $\frac12$-sparse whenever $\lambda$ is large enough. 
\end{lemma}
%%%%%%%%%%%%%%%%%%%%%%%%% LEMMA LEMMA LEMMA %%%%%%%%%%%%%%%%%%%%%%%

%%%%%%%%%%%%%%%%%%%%%%%%%% PROOF PROOF PROOF %%%%%%%%%%%%%%%%%%%%%%%
\begin{proof}
We only focus on the proof for $\D_1$. Given $B \in \D_1$, we set 
\begin{align}\label{def:EB}
E_B := B \setminus \bigcup_{B' \in \D_1: B' \cap B \neq \tinyemptyset \atop r(B') \le r(B)-2} B'. 
\end{align}
Let $B_1, B_2 \in \D_1$. We may assume that $B_1 \cap B_2 \neq \emptyset$, since $B_1 \cap B_2 = \emptyset$ implies $E_{B_1} \cap E_{B_2} = \emptyset$. Note that $r(B_1) \neq r(B_2)$ because each $\D^k$ is a pairwise disjoint family. Then we may assume that $r(B_1) \le r(B_2)-2$, which along with \eqref{def:EB} gives that $E_{B_2} \cap B_1 = \emptyset$. Consequently, $E_{B_1} \cap E_{B_2} = \emptyset$. 

To continue, we fix $B, \widetilde{B} \in \D_1 \subset \D$ with $B \cap \widetilde{B} \neq \emptyset$ and $r(\widetilde{B}) \le r(B)-2$. Then $\widetilde{B}^{(2)} \cap B \neq \emptyset$, and it follows from \eqref{FBFB} and \eqref{DHG} that 
\begin{align}\label{FKB}
r(\mathscr{F}(\widetilde{B})) + 2 
\le r(B) \le r(\mathscr{F}(\widetilde{B})) -2. 
\end{align}
If we write 
\[
\mathcal{P}_B 
:= \{P \in \D: P^{(2)} \cap B \neq \emptyset, |r(P)-r(B)| \le 2\},
\] 
then \eqref{FKB} implies $\mathscr{F}(\widetilde{B}) \in \mathcal{P}_B$. Note that $\# \mathcal{P}_B \lesssim 1$ because of Lemma \ref{lem:BE} part \eqref{list-2} and  $-3 \le \log \frac{\mu(P)}{\mu(B)} \le 3$ for any $P \in \mathcal{P}_B$.
As a consequence, these and \eqref{FkFk} imply 
\begin{multline*}
\mu(B \setminus E(B)) 
=\mu \bigg(\bigcup_{B' \in \D_1: B' \cap B \neq \tinyemptyset \atop r(B') \le r(B)-2} B' \bigg)  
\le \mu \bigg(\bigcup_{P \in \mathcal{P}_B} \bigcup_{B'' \in \G(P)} B'' \bigg)  
\\
\le \sum_{P \in \mathcal{P}_B} \sum_{k=1}^{\infty} \mu \bigg(\bigcup_{B'' \in \F_k(P)} B'' \bigg)  
\lesssim \sum_{P \in \mathcal{P}_B} \sum_{k=1}^{\infty} \lambda^{-k} \mu(P)
\lesssim \frac{\mu(P)}{\lambda-1}
\le \frac12 \mu(P), 
\end{multline*}
provided $\lambda>1$ is large enough. 
\end{proof}
%%%%%%%%%%%%%%%%%%%%%%%%%% END END END PROOF %%%%%%%%%%%%%%%%%%%%%%%

%%%%%%%%%%%%%%%%%%%%%%%%% LEMMA LEMMA LEMMA %%%%%%%%%%%%%%%%%%%%%%%
\begin{lemma}\label{lem:GAFA}
For any $B \in \D$, there holds 
\begin{align*}
\Gamma(\vec{f} \mathbf{1}_{B^{(3)}}) (x) 
\lesssim \C(T) \prod_{i=1}^m \langle f_i \rangle_{B^{(3)}, r}, \quad 
\text{a.e. } x \in B^* \setminus \bigcup_{G \in \D:r(G)<r(B)} G^*. 
\end{align*}
\end{lemma}
%%%%%%%%%%%%%%%%%%%%%%%%% LEMMA LEMMA LEMMA %%%%%%%%%%%%%%%%%%%%%%%

%%%%%%%%%%%%%%%%%%%%%%%%%% PROOF PROOF PROOF %%%%%%%%%%%%%%%%%%%%%%%
\begin{proof}
Fix $B \in \D$. We first claim that 
\begin{align}\label{GGsub}
\bigcup_{G \in \F(B)} G \subset \bigcup_{G \in \D:r(G)<r(B)} G^* := H. 
\end{align}
It suffices to show that $G \subset H$ for all $G \in \F(B) \setminus \D$.  Fix $G \in \F(B) \setminus \D$, then necessarily $r(G)<r(B)$. In view of \eqref{DHG}, there are only two cases: 
\[
{\rm (i)} \quad G \in \bigcup_{k \le k_0} (\H^k \setminus \D^k), \qquad 
{\rm (ii)} \quad G \in \bigcup_{k \le k_0} (\G^k \setminus \H^k). 
\]
If $G \in \bigcup_{k \le k_0} (\H^k \setminus \D^k)$, then there exist a unique $k \le k_0$ and some $B'_k \in \H^k$ such that 
\begin{align}\label{GBBB}
G=B'_k \subset \bigcup_{B'_k \in \H^k} B'_k 
\subset \bigcup_{B_k \in \D^k} B_k^{(1)} 
\subset H, 
\end{align}
provided \eqref{BGBD} and that $r(B_k)=k=r(G)<r(B)$ for each $B_k \in \D^k$. 

If $G \in \bigcup_{k \le k_0} (\G^k \setminus \H^k) = \big(\bigcup_{k \le k_0} \G^k \big) \cap \big(\bigcup_{B \in \mathcal{B}} \G(B) \big)$, then it follows from \eqref{FBFB} that there exists $B' \in \G \subset \D$ so that 
\begin{align*}
G^{(2)}\cap B' \neq \emptyset 
\quad\text{ and }\quad 
r(G) + 2 \le r(B') \le r(\mathscr{F}(G)) -2 =  r(B)-2, 
\end{align*}

By definition, 
\begin{align*}
\frac12 \log_{\C_0} \mu(G) - 1+2 
\le r(G) + 2 
\le r(B') 
\le \frac12 \log_{\C_0} \mu(B'), 
\end{align*}
and hence, $\mu(G^{(2)}) 
\le \C_0^2 \mu(G) 
\le \mu(B')$, 
which together with the property \eqref{list:B4} gives that $G \subset G^{(2)} \subset (B')^* \subset H$. 

Consequently, the desired estimate follows at once from \eqref{tree-3} and \eqref{GGsub}.
\end{proof}
%%%%%%%%%%%%%%%%%%%%%%%%%% END END END PROOF %%%%%%%%%%%%%%%%%%%%%%%

%%%%%%%%%%%%%%%%%%%%%%%%% LEMMA LEMMA LEMMA %%%%%%%%%%%%%%%%%%%%%%%
\begin{lemma}\label{lem:TFA}
For all $x \in B_k$ and $j=0, 1, \ldots, k-1$, 
\begin{align}\label{TTBj}
&\big\|[\T, \b]_{\alpha}(\vec{f} \mathbf{1}_{B_j^{(3)}})(x) 
- [\T, \b]_{\alpha}(\vec{f} \mathbf{1}_{B_{j+1}^{(3)}})(x) \big\|_{\bB} 
\nonumber \\
&\quad\lesssim \C(T) \sum_{\tau_1 \uplus \tau_2 = \tau(\alpha)} 
\prod_{i_1 \in \tau_1} |b_{i_1}(x) - b_{i_1, B_j^{(3)}}| \langle f_{i_1} \rangle_{B_j^{(3)}, r} 
\nonumber \\
&\qquad\quad \times \prod_{i_2 \in \tau_2} \langle (b_{i_2} - b_{i_2, B_j^{(3)}}) f_{i_2} \rangle_{B_j^{(3)}, r} 
\times \prod_{i_3 \not\in \tau_1 \cup \tau_2} \langle f_{i_3} \rangle_{B_j^{(3)}, r}. 
\end{align}
\end{lemma}
%%%%%%%%%%%%%%%%%%%%%%%%% LEMMA LEMMA LEMMA %%%%%%%%%%%%%%%%%%%%%%%

%%%%%%%%%%%%%%%%%%%%%%%%%% PROOF PROOF PROOF %%%%%%%%%%%%%%%%%%%%%%%
\begin{proof}
It follows from \eqref{Tbg} that 
\begin{multline}\label{JJ}
[\T, \b]_{\alpha}(\vec{f} \mathbf{1}_{B_j^{(3)}})(x) 
- [\T, \b]_{\alpha}(\vec{f} \mathbf{1}_{B_{j+1}^{(3)}})(x) 
\\
=\sum_{\tau_1 \uplus \tau_2 = \tau(\alpha)} \prod_{i_1 \in \tau_1} (b_{i_1}(x) - c_{i_1})  
\big[\T(\vec{g} \mathbf{1}_{B_j^{(3)}})(x) - \T(\vec{g} \mathbf{1}_{B_{j+1}^{(3)}})(x) \big], 
\end{multline}
where $c_i := b_{i, B_j^{(3)}}$, $i=1, \ldots, m$. Thus, \eqref{TTBj} is a consequence of \eqref{JJ} and the following 
\begin{align}\label{Tgg}
\big\|\T(\vec{g} \mathbf{1}_{B_j^{(3)}})(x) - \T(\vec{g} \mathbf{1}_{B_{j+1}^{(3)}})(x)\big\|_{\bB}
\lesssim \C(T) \prod_{i=1}^m \langle g_i \rangle_{B_j^{(3)}, r}, \quad x \in B_k. 
\end{align}

It remains to show \eqref{Tgg}. We split 
\begin{align}\label{JJJ}
\big\|\T(\vec{g} \mathbf{1}_{B_j^{(3)}})(x) - \T(\vec{g} \mathbf{1}_{B_{j+1}^{(3)}})(x)\big\|_{\bB}
\le \mathscr{J}_1 + \mathscr{J}_2 + \mathscr{J}_3, 
\end{align}
where 
\begin{align*}
\mathscr{J}_1 
&:= \big\| \big(\T(\vec{g} \mathbf{1}_{B_j^{(3)}}) 
- \T(\vec{g} \mathbf{1}_{\widetilde{B}_{j+1}^{(1)}}) \big)(x) 
- \big(\T (\vec{g} \mathbf{1}_{B_j^{(3)}}) 
- \T(\vec{g} \mathbf{1}_{\widetilde{B}_{j+1}^{(1)}})\big)(\xi_j)\big\|_{\bB}, 
\\ \nonumber
\mathscr{J}_2 
&:= \big\|\T (\vec{g} \mathbf{1}_{B_j^{(3)}})(\xi_j) 
- \T(\vec{g} \mathbf{1}_{\widetilde{B}_{j+1}^{(1)}})(\xi_j) \big\|_{\bB}, 
\\
\mathscr{J}_3 
&:= \big\|\T(\vec{g} \mathbf{1}_{\widetilde{B}_{j+1}^{(1)}})(x) 
- \T (\vec{g} \mathbf{1}_{B_{j+1}^{(3)}})(x) \big\|_{\bB}. 
\end{align*}
Using the condition \eqref{list:T-reg} and \eqref{BBT-1}, we have $\xi_j \in \widetilde{B}_{j+1}^{(1)} \subset B_j^{(3)}$ and 
\begin{align}\label{JJ1}
\mathscr{J}_1 
\le \C_1(T) \prod_{i=1}^m \lfloor g_i \mathbf{1}_{B_j^{(3)}} \rfloor_{\widetilde{B}_{j+1}^{(1)}, r} 
\le \C(T) \mathcal{M}_{\B, r}^{\otimes}(\vec{g} \mathbf{1}_{B_j^{(3)}})(\xi_j)
\nonumber \\ 
\le \Gamma(\vec{g} \mathbf{1}_{B_j^{(3)}})(\xi_j)
\lesssim \C(T) \prod_{i=1}^m \langle g_i \rangle_{B_j^{(3)}, r}, 
\end{align}
where we have used \eqref{BBT-2} in the last inequality. To control $\mathcal{J}_2$, by definition of $T_*$ and \eqref{BBT-1}, we see that $\xi_j \in \widetilde{B}_{j+1} \subset B_j^{(1)}$ and 
\begin{align}\label{JJ2}
\mathscr{J}_2 
&= \big\|\T (\vec{g} \mathbf{1}_{B_j^{(3)}})(\xi_j) 
- \T(\vec{g} \mathbf{1}_{B_j^{(3)} \cap \widetilde{B}_{j+1}^{(1)}})(\xi_j) \big\|_{\bB}
\nonumber \\ 
&\le T_*(\vec{g} \mathbf{1}_{B_j^{(3)}})(\xi_j) 
\le \Gamma(\vec{g} \mathbf{1}_{B_j^{(3)}})(\xi_j) 
\lesssim \C(T) \prod_{i=1}^m \langle g_i \rangle_{B_j^{(3)}, r}. 
\end{align}
Besides, \eqref{BBT-3} gives 
\begin{align}\label{JJ3}
\mathscr{J}_3 
\lesssim \C(T) \prod_{i=1}^m \langle g_i \rangle_{B_j^{(3)}, r}. 
\end{align}
Therefore, \eqref{Tgg} immediately follows from \eqref{JJJ}--\eqref{JJ3}.
\end{proof}
%%%%%%%%%%%%%%%%%%%%%%%%%% END END END PROOF %%%%%%%%%%%%%%%%%%%%%%%

%%%%%%%%%%%%%%%%%%%%%%%% SECTION SECTION SECTION %%%%%%%%%%%%%%%%%%%%%%
%%%%%%%%%%%%%%%%%%%%%%%% SECTION SECTION SECTION %%%%%%%%%%%%%%%%%%%%%%
\section{Proof of Theorems \ref{thm:local}--\ref{thm:T-Besi}}
This section is devoted to showing Theorems \ref{thm:local}--\ref{thm:T-Besi}.

%%%%%%%%%%%%%%%%%%%%% SUBSECTION SUBSECTION SUBSECTION %%%%%%%%%%%%%%%%%%
%%%%%%%%%%%%%%%%%%%%% SUBSECTION SUBSECTION SUBSECTION %%%%%%%%%%%%%%%%%%
\subsection{Local exponential decay estimates}\label{sec:local}
In order to prove Theorem \ref{thm:local}, we begin with a local Coifman-Fefferman inequality. Define 
\begin{align*}
\widetilde{\mathcal{M}}_{\B, L(\log L)^r} (\vec{f}) 
:= \sup_{B \in \B: B \ni x} \prod_{i=1}^m \||f_i|^r\|_{L(\log L)^r, B}^{\frac1r}. 
\end{align*}

%%%%%%%%%%%%%%%%%%%%%%%% LEMMA LEMMA LEMMA %%%%%%%%%%%%%%%%%%%%%%%%
\begin{lemma}\label{lem:Ub}
Let $(\Sigma, \mu)$ be a measure space with a ball-basis $\B$. Given $B_0 \in \B$ and functions $\vec{f}$ with $\supp(f_i) \subset B_0$, $i=1, \ldots, m$, we have for all $p \in (1, \infty)$ and $w \in A_{p, \B}$, 
\begin{align}\label{TbM-1} 
\| T(\vec{f}) \|_{L^1(B_0, w)}
\lesssim [w]_{A_{p, \B}} \| \mathcal{M}_{\B, r}(\vec{f}) \|_{L^1(B_0^{(3)}, w)}. 
\end{align}
Moreover, if $A_{\infty, \B}$ satisfies the sharp reverse H\"{o}lder inequality, then for the same $\vec{f}$, $p$, and $w$ as above, and for each $\alpha \in \{0, 1\}^m$, 
\begin{align}\label{TbM-2} 
\|[T, \b]_{\alpha}(\vec{f}) \|_{L^1(B_0, w)}
\lesssim \|\b\|_{\tau} [w]_{A_{p, \B}}^{|\tau|+1} 
\|\widetilde{\mathcal{M}}_{\B, L(\log L)^r}(\vec{f}) \|_{L^1(B_0^{(3)}, w)}, 
\end{align}
where the implicit constant is independent of $B_0$, $[w]_{A_{p, \B}}$, and $\vec{f}$.
\end{lemma}
%%%%%%%%%%%%%%%%%%%%%%%% LEMMA LEMMA LEMMA %%%%%%%%%%%%%%%%%%%%%%%%

%%%%%%%%%%%%%%%%%%%%%%%%% PROOF PROOF PROOF %%%%%%%%%%%%%%%%%%%%%%%%
\begin{proof}
It suffices to show \eqref{TbM-2} since the proof of \eqref{TbM-1} is much easier. Fix $B_0 \in \B$ and functions $\vec{f}$ with $\supp(f_i) \subset B_0$, $i=1, \ldots, m$. We claim that there are two sparse families $\S_1=\S_1(B_0)$ and $\S_2=\S_2(B_0)$ such that for a.e. $x \in B_0$, 
\begin{align}
&\label{Tf-sparse}  |[T, \b]_{\alpha}(\vec{f})(x)| 
\lesssim \C(T) \sum_{\tau_1 \uplus \tau_2=\tau} \big[\A_{\S_1, r}^{\b, \tau_1, \tau_2}(\vec{f})(x) 
+ \A_{\S_2, r}^{\b, \tau_1, \tau_2}(\vec{f})(x) \big], 
\\ 
\label{BBSS} & \text{and} \quad B \subset B_0^{(3)} \, \text{ for all } \, B \in \S_1 \cup \S_2. 
\end{align}
Indeed, considering \eqref{support} and modifying the proof of Theorem \ref{thm:sparse}, we obtain \eqref{Tf-sparse}. To get \eqref{BBSS}, we observe that 
\begin{align}\label{B2B}
B^{(2)} \subset B^*_0 \quad\text{ for all } B \in \G(B_0) \setminus \{B_0\}.
\end{align}
In fact, by \eqref{GAA-4}, 
\begin{align}\label{B2BB}
B^{(2)} \subset B^*_0 \quad \text{ for any } B \in \F(B_0).
\end{align}
In view of \eqref{FFF}, given $B_2 \in \F_2(B_0)$, there exists some $B_1 \in \F_1(B_0)=\F(B_0)$ such that $B_2 \in \F(B_1)$, which together with \eqref{B2BB} yields $B_2^{(2)} \subset B_1^* \subset B_1^{(2)} \subset B_0^*$. Recursively, one can show $B_k^{(2)} \subset B_0^*$ for all $B_k \in \F_k(B_0)$ and $k \ge 1$. This and \eqref{generation} imply \eqref{B2B}. 
Recall the definition of $\S_1$ and $\S_2$ in \eqref{SS} and that $\S_1 \cup \S_2 \subset \D \subset \G(B_0)$ (see \eqref{DHG}). Hence, this and \eqref{B2B} conclude \eqref{BBSS}.

Let $p \in (1, \infty)$ and $w \in A_{p, \B}$. Checking the proof of Lemma \ref{lem:Ainfty}, we conclude that for any $\alpha \in (0, 1)$, for any $B \in \B$, and for any measurable subset $E \subset B$, 
\begin{align}\label{EBAp}
\mu(E) \ge \alpha \mu(B) \quad\Longrightarrow \quad 
w(E) \ge \alpha^p \, [w]_{A_{p, \B}}^{-1} w(B).
\end{align} 
Let $\S$ be an $\eta$-sparse family with $\eta \in (0, 1)$ such that $B \subset B_0^{(3)}$ for all $B \in \S$. Then, there exists a pairwise disjoint family $\{E_B\}_{B \in \S}$ such that $E_B \subset B$ and $\mu(E_B) \ge \eta \mu(B)$ for all $B \in \S$. This and \eqref{EBAp} imply that 
\begin{equation*}
w(B) \le \eta^{-p} [w]_{A_{p, \B}} w(E_B) \quad\text{ for all } B \in \S, 
\end{equation*}
which together with Lemma \ref{lem:PhiPhi} and \eqref{e:fw} immediately gives 
\begin{align}\label{eq:ASB}
\|&\A_{\S, r}^{\b, \tau_1, \tau_2}(\vec{f})\|_{L^1(B_0, w)}
\nonumber \\
&\le \sum_{B \in \S} \mu(B) \bigg(\fint_B \prod_{i \in \tau_1} 
\langle f_i \rangle_{B, r} |b_i - b_{i,B}| w \, d\mu \bigg) 
\nonumber \\
&\qquad\times \prod_{j \in \tau_2} \langle (b_j - b_{j, B}) f_j \rangle_{B, r} 
\prod_{k \not\in \tau_1 \cup \tau_2} \langle f_k \rangle_{B, r} 
\nonumber \\
&\lesssim \sum_{B \in \S} \mu(B) \|w\|_{L(\log L)^{|\tau_1|}, B} 
\prod_{i \in \tau_1} \langle f_i \rangle_{B, r} \|b_i - b_{i, B}\|_{\exp L, B} 
\nonumber \\
&\qquad\times \prod_{j \in \tau_2} \| |b_j - b_{j, B}|^r\|_{\exp L^{\frac1r}}^{\frac1r} 
\||f_j|^r\|_{L(\log L)^r, B}^{\frac1r}
\prod_{k \not\in \tau_1 \cup \tau_2} \langle f_k \rangle_{B, r} 
\nonumber \\
&\lesssim \|\b\|_{\tau_1 \cup \tau_2} [w]_{A_{\infty}}^{|\tau_1|} 
\sum_{B \in \S} \prod_{i \not\in \tau_2} \langle f_i \rangle_{B, r} 
\prod_{j \in \tau_2} \||f_j|^r\|_{L(\log L)^r, B}^{\frac1r} \, w(B)
\nonumber \\
&\lesssim \|\b\|_{\tau_1 \cup \tau_2} [w]_{A_{p, \B}}^{|\tau_1|+1} 
\sum_{B \in \S} \inf_{x \in B} \widetilde{\mathcal{M}}_{\B, L(\log L)^r}(\vec{f})(x)  \, w(E_B)
\nonumber \\
&\le \|\b\|_{\tau_1 \cup \tau_2} [w]_{A_{p, \B}}^{|\tau_1|+1} 
\sum_{B \in \S} \int_{E_B} \widetilde{\mathcal{M}}_{\B, L(\log L)^r}(\vec{f})  \, w \, d\mu
\nonumber \\
&\le \|\b\|_{\tau_1 \cup \tau_2} [w]_{A_{p, \B}}^{|\tau_1|+1} 
\|\widetilde{\mathcal{M}}_{\B, L(\log L)^r}(\vec{f})\|_{L^1(B_0^{(3)}, w)}, 
\end{align} 
where we have used the disjointness of $\{E_B\}_{B \in \S}$ and that $B \subset B_0^{(3)}$ for all $B \in \S$. We here mention that the implicit constants in \eqref{eq:ASB} are independent of $\S$. Therefore, \eqref{TbM-2} is a consequence of \eqref{Tf-sparse}, \eqref{BBSS}, and \eqref{eq:ASB}. 
\end{proof}
%%%%%%%%%%%%%%%%%%%%%%%%% END END END PROOF %%%%%%%%%%%%%%%%%%%%%%%%

\begin{proof}[\bf Proof of Theorem \ref{thm:local}] 
Given $s>1$, for any nonnegative function $h \in L^s(\Sigma, \mu)$, we define the Rubio de Francia algorithm as 
\begin{align*}
\mathcal{R}h 
:= \sum_{k=0}^{\infty} \frac{1}{2^k} \frac{M_{\B}^{k}h}{\|M_{\B}\|^k_{L^s(\mu)\to L^s(\mu)}}. 
\end{align*}
Then one can check that 
\begin{equation}\label{eq:RdF}
\begin{aligned}
&h \leq \mathcal{R}h, \quad 
\|\mathcal{R}h\|_{L^s(\mu)} \leq 2 \|h\|_{L^s(\mu)}, 
\\ 
&\text{and}\quad  
[\mathcal{R}h]_{A_{1, \B}} \leq 2 \|M_{\B}\|_{L^s(\mu) \to L^s(\mu)} \leq 2\C_0^{\frac1s}, 
\end{aligned}
\end{equation}
where the last estimates follows from \eqref{Mr-2}. Let $p>1$ and $q>1$ be chosen later. By Riesz theorem, there exists a nonnegative function $h \in L^{q'}(B, \mu)$ with $\|h\|_{L^{q'}(B, \mu)}=1$ such that 
\begin{align}\label{Iqq}
\mathcal{I}(t)^{\frac1q} 
&:= \mu\big(\big\{x \in B:  [T, \b]_{\alpha}(\vec{f})(x) > t \, 
\mathcal{M}_{\B, r}(\vec{f^*})(x) \big\}\big)^{\frac1q}
\nonumber\\
& \leq \frac{1}{t} \bigg\| \frac{[T, \b]_{\alpha}(\vec{f})}{\mathcal{M}_{\B, r}(\vec{f^*})} \bigg\|_{L^q(B, \mu)}
\leq \frac{1}{t} \int_B |[T, \b]_{\alpha}(\vec{f})|
\frac{h}{\mathcal{M}_{\B, r}(\vec{f^*})} \, d\mu
\nonumber\\
& \leq \frac{1}{t} \int_B |[T, \b]_{\alpha}(\vec{f})|
\frac{\mathcal{R}h}{\mathcal{M}_{\B, r}(\vec{f^*})} \, d\mu
= t^{-1} \|[T, \b]_{\alpha}(\vec{f})\|_{L^1(B, w)},
\end{align}
where we have used the first estimate in \eqref{eq:RdF}, 
\[
w := w_1 w_2^{1-p}, \quad 
w_1 := \mathcal{R}h, 
\quad\text{and}\quad 
w_2 := \mathcal{M}_{\B, r}(\vec{f^*})^{p'-1}.
\] 
In view of Lemma \ref{lem:CR}, we pick $p>1+m/r$ (equivalently, $\frac{p'-1}{r}<\frac1m$) so that 
\begin{align}\label{eq:MMf}
[w_2]_{A_{1, \B}}
=\big[\big(\mathcal{M}_{\B}(|f_1^*|^r, \ldots, |f_m^*|^r) \big)^{\frac{p'-1}{r}}\big]_{A_{1, \B}} 
\leq C_m.
\end{align}
Then \eqref{eq:RdF} and \eqref{eq:MMf} imply that $w=w_1 w_2^{1-p} \in A_{p, \B}$ and 
\begin{align}\label{wapcc}
[w]_{A_{p, \B}} 
\leq [w_1]_{A_{1, \B}} [w_2]_{A_{1, \B}}^{p-1}
\leq 2 \C_0^{\frac{1}{q'}} C_m^{p-1}
\leq 2 \C_0 C_m^{p-1}. 
\end{align}

Note that for any $B' \in \B$ and locally integrable function $f$, there holds that
\begin{align}\label{fLL}
\|f\|_{L(\log L)^r, B'} 
\le \|f\|_{L(\log L)^{\lfloor r \rfloor}, B'} 
\lesssim \fint_{B'} M^{\lfloor r \rfloor}(f \mathbf{1}_{B'}) \, d\mu.
\end{align}
Thus, \eqref{Iqq}, \eqref{wapcc}, \eqref{fLL}, and Lemma \ref{lem:Ub} give 
\begin{align*}
\mathcal{I}(t)^{\frac1q} 
& \le c_0 \, t^{-1} \|\b\|_{\tau} [w]_{A_{p, \B}}^{|\tau|+1}  
\|\widetilde{\mathcal{M}}_{\B, L(\log L)^r}(\vec{f})\|_{L^1(B^{(3)}, w)} 
\\
& \le c_0 \, t^{-1} \|\b\|_{\tau} [w]_{A_{p, \B}}^{|\tau|+1} 
\|\mathcal{M}_{\B, r}(\vec{f^*})\|_{L^1(B^{(3)}, w)} 
\\ 
& = c_0 \, t^{-1} \|\b\|_{\tau} [w]_{A_{p, \B}}^{|\tau|+1} \|\mathcal{R}h\|_{L^1(B^{(3)}, \mu)}
\\ 
&\leq c_0 \, t^{-1} \|\b\|_{\tau} [w]_{A_{p, \B}}^{|\tau|+1}  
\|\mathcal{R}h\|_{L^{q'}(B^{(3)}, \mu)}  \mu(B^{(3)})^{\frac1q}
\\ 
& \le c_0 \, t^{-1} \|\b\|_{\tau} \mu(B)^{\frac1q}
\le c_0 \, t^{-1} \|\b\|_{\tau} q^{|\tau|+1}  \mu(B)^{\frac1q}, 
\end{align*}
where $c_0>1$ varies from line to line and is independent of $q$. As a consequence, for any $t> t_0 := c_0 e \|\b\|_{\tau}$, we choose 
\[
q = \bigg(\frac{t}{c_0 e \|\b\|_{\tau}} \bigg)^{\frac{1}{|\tau| +1}} > 1
\] 
to deduce that 
\begin{align}\label{It-1}
\mathcal{I}(t) 
&\le \big(c_0  t^{-1} \|\b\|_{\tau} q^{|\tau|+1} \big)^q \mu(B)
 = e^{-q} \mu(B)
 \nonumber \\
&= e^{- (\frac{t}{c_0 e \|\b\|_{\tau}} )^{\frac{1}{|\tau| +1}}} \mu(B)
=: e^{- (\frac{\gamma t}{\|\b\|_{\tau}})^{\frac{1}{|\tau|+1}}} \mu(B),
\end{align}
where $\gamma := \frac{1}{c_0 e}$ depends only on $m$, $p$, and $\C_0$. Besides, for all $0<t \le t_0$, 
\begin{align}\label{It-2}
\mathcal{I}(t) 
\le \mu(B)
=e \cdot e^{-(\frac{\gamma t_0}{\|\b\|_{\tau}})^{\frac{1}{|\tau|+1}}} \mu(B)
\le e \cdot e^{-(\frac{\gamma t}{\|\b\|_{\tau}})^{\frac{1}{|\tau|+1}}} \mu(B). 
\end{align}
Now gathering \eqref{It-1} and \eqref{It-2}, we eventually obtain that for all $t>0$, 
\begin{equation*}
\mu \big(\big\{x \in B:  [T, \b]_{\alpha}(\vec{f})(x) > t \, \mathcal{M}_{\B, r}(\vec{f^*})(x)  \big\}\big)
\le e \cdot e^{-(\frac{\gamma t}{\|\b\|_{\tau}})^{\frac{1}{|\tau|+1}}} \mu(B). 
\end{equation*}
This completes the proof.
\end{proof}

%%%%%%%%%%%%%%%%%%%%% SUBSECTION SUBSECTION SUBSECTION %%%%%%%%%%%%%%%%%%
%%%%%%%%%%%%%%%%%%%%% SUBSECTION SUBSECTION SUBSECTION %%%%%%%%%%%%%%%%%%
\subsection{Mixed weak type estimates}\label{sec:weak} 
The goal of this subsection is to demonstrate Theorem \ref{thm:weak}. For this purpose, we first present a Coifman-Fefferman  inequality.

%%%%%%%%%%%%%%%%%%%%%%% LEMMA LEMMA LEMMA %%%%%%%%%%%%%%%%%%%%%%%%%
\begin{lemma}\label{lem:TM} 
Let $(\Sigma, \mu)$ be a measure space with a ball-basis $\B$. For any $p \in (0, \infty)$ and $w \in A_{\infty, \B}$, 
\begin{align}\label{eq:C-F}
\|T(\vec{f})\|_{L^p(\Sigma, w)} 
\lesssim \|\mathcal{M}_{\B, r}(\vec{f})\|_{L^p(\Sigma, w)}.
\end{align}
\end{lemma}
%%%%%%%%%%%%%%%%%%%%%%% LEMMA LEMMA LEMMA %%%%%%%%%%%%%%%%%%%%%%%%%

%%%%%%%%%%%%%%%%%%%%%%%%% PROOF PROOF PROOF %%%%%%%%%%%%%%%%%%%%%%%%
\begin{proof}
Let $w \in A_{\infty, \B}$. By Lemma \ref{lem:Ainfty}, for any $\alpha \in (0, 1)$ there exists $\beta \in (0, 1)$ such that for any $B \in \B$ and any measurable subset $E \subset B$, 
\begin{align}\label{EBB}
\mu(E) \ge \alpha \mu(B) \quad\Longrightarrow \quad 
w(E) \ge \beta w(B).
\end{align} 
Let $\S$ be an $\eta$-sparse family, $\eta \in (0, 1)$. By definition, there exists a pairwise disjoint family $\{E_B\}_{B \in \S}$ such that $E_B \subset B$ and $\mu(E_B) \ge \eta \mu(B)$ for all $B \in \S$. In light of \eqref{EBB}, there is $\beta=\beta(\eta) \in (0, 1)$ such that for $B \in \S$, we have $w(E_B) \ge \beta w(B)$, which in turn yields that
\begin{align}\label{eq:ASp=1}
\|\mathcal{A}_{\S, r}(\vec{f})\|_{L^1(\Sigma, w)}
& \leq \sum_{B \in \S} \prod_{i=1}^m \langle f_i \rangle_{B, r} w(B)
\lesssim \sum_{B \in \S} \Big(\inf_{B} \mathcal{M}_{\B, r}(\vec{f}) \Big) w(E_B)
\nonumber \\ 
& \leq \sum_{B \in \S} \int_{E_B} \mathcal{M}_{\B, r}(\vec{f}) \, w \, d\mu
\leq \|\mathcal{M}_{\B, r}(\vec{f})\|_{L^1(\Sigma, w)}, 
\end{align} 
where the implicit constants are independent of $\S$. As a consequence, Theorem \ref{thm:sparse} and \eqref{eq:ASp=1} imply 
\begin{align}\label{eq:TMp=1}
\|T(\vec{f})\|_{L^1(\Sigma, w)}
\lesssim \|\mathcal{M}_{\B, r}(\vec{f})\|_{L^1(\Sigma, w)}. 
\end{align}
This corresponds to \eqref{eq:C-F} holds for the case $p=1$. 

To show the general case $p \in (0, \infty)$, we apply the $A_{\infty}$ extrapolation theorem: given a family of pairs of functions $\mathcal{F}$, if for some $p_0 \in (0, \infty)$ and for every weight $w_0 \in A_{\infty, \B}$, 
\begin{align}\label{eq:fg-some}
\|f\|_{L^{p_0}(\Sigma, w_0)} 
\leq C_1 \|g\|_{L^{p_0}(\Sigma, w_0)}, \quad \forall (f, g) \in \mathcal{F}, 
\end{align}
then for all $p \in (0, \infty)$ and all $w \in A_{\infty, \B}$, 
\begin{align}\label{eq:fg-every}
\|f\|_{L^p(\Sigma, w)} 
\leq C_2 \|g\|_{L^p(\Sigma, w)}, \quad \forall (f, g) \in \mathcal{F}.  
\end{align}
This is contained in \cite[Theorem 3.34]{CMM}, which established an $A_{\infty, \B}$ extrapolation on general Banach function spaces. Additionally, the estimate \eqref{Mr-2} is needed to verify the hypothesis there. 

Observe that \eqref{eq:TMp=1} verifies \eqref{eq:fg-some} for $p_0=1$ and the pair $(f, g)=(T(\vec{f}), \mathcal{M}_{\B, r}(\vec{f}))$.  Therefore, \eqref{eq:C-F} follows from \eqref{eq:fg-every}. 
\end{proof}
%%%%%%%%%%%%%%%%%%%%%%% END END END PROOF %%%%%%%%%%%%%%%%%%%%%%%%%%

%%%%%%%%%%%%%%%%%%%%%%%% PROOF PROOF PROOF%%%%%%%%%%%%%%%%%%%%%%%%%%
\noindent\textbf{Proof of Theorem \ref{thm:weak}.}
We use a hybrid of the arguments in \cite{CMP} and \cite{LOP}.
Define
\[
\mathcal{R}h := \sum_{j=0}^{\infty} \frac{S^jh}{(2K)^j},
\]
where $K>0$ will be chosen later and $Sf := M_{\B}(f w)/w$. It immediately yields that
\begin{align}\label{e:R-3}
h(x) \leq \mathcal{R}h(x) \quad\text{ and }\quad 
S(\mathcal{R}h)(x) \leq 2K \, \mathcal{R}h(x).
\end{align}
We claim that there exists some $s>1$ such that 
\begin{align}\label{RhAi}
\mathcal{R}h \cdot w v^{\frac{r}{ms'}} \in A_{\infty, \B}
\end{align}  
and 
\begin{align}\label{e:R-2}
\|\mathcal{R}h\|_{L^{s',1}(\Sigma, w v^{\frac{r}{m}})}
\leq 2 \|h\|_{L^{s',1}(\Sigma, w v^{\frac{r}{m}})}.
\end{align}
The proof of \eqref{RhAi} and \eqref{e:R-2} will be given at the end of this section.

Observing that for any weight $\sigma$ on $(\Sigma, \mu)$, 
\begin{equation}\label{eq:Lpq}
\|f^q\|_{L^{p,\infty}(\Sigma, \sigma)}
= \|f\|^q_{L^{pq,\infty}(\Sigma, \sigma)}, \quad 0<p,q<\infty, 
\end{equation}
we have 
\begin{align}\label{Tff-1}
\bigg\| &\frac{T(\vec{f})}{v}
\bigg\|_{L^{\frac{r}{m},\infty}(\Sigma, w v^{\frac{r}{m}})}^{\frac{r}{ms}}
= \bigg\| \bigg|\frac{T(\vec{f})}{v}\bigg|^{\frac{r}{ms}} \bigg\|_{L^{s,\infty}(\Sigma, w v^{\frac{r}{m}})}
\nonumber \\ 
&=\sup_{0 \le h \in L^{s',1}(\Sigma, w v^{\frac{r}{m}}) \atop \|h\|_{L^{s',1}(\Sigma, w v^{\frac{r}{m}})}=1}
\bigg|\int_{\Sigma} |T(\vec{f})|^{\frac{r}{ms}} h \, w \, v^{\frac{r}{ms'}} d\mu  \bigg|
\nonumber \\ 
&\leq \sup_{0 \le h \in h \in L^{s',1}(w v^{\frac{r}{m}}) \atop \|h\|_{L^{s',1}(\Sigma, w v^{\frac{r}{m}})}=1} 
\int_{\Sigma} |T(\vec{f})|^{\frac{r}{ms}} \mathcal{R}h \, w \, v^{\frac{r}{ms'}} d\mu.
\end{align}
Fix a nonnegative function $h \in L^{s',1}(\Sigma, w v^{\frac{r}{m}})$ with $\|h\|_{L^{s',1}(\Sigma, w v^{\frac{r}{m}})}=1$. Then by \eqref{RhAi}, Lemma \ref{lem:TM}, and H\"{o}lder's inequality, we obtain
\begin{align}\label{Tff-2}
\int_{\Sigma} & |T(\vec{f})|^{\frac{r}{ms}} \mathcal{R}h \, w v^{\frac{r}{ms'}} d\mu
\nonumber \\ 
& \lesssim \int_{\Sigma} \mathcal{M}_{\B, r}(\vec{f})^{\frac{r}{ms}} \mathcal{R}h \, w v^{\frac{r}{ms'}} d\mu
\nonumber \\ 
& = \int_{\Sigma} \bigg(\frac{\mathcal{M}_{\B, r}(\vec{f})}{v}\bigg)^{\frac{r}{ms}}
\mathcal{R}h \, w v^{\frac{r}{m}} d\mu
\nonumber \\ 
& \leq \bigg\|\bigg(\frac{\mathcal{M}_{\B, r}(\vec{f})}{v}\bigg)^{\frac{r}{ms}} \bigg\|_{L^{s,\infty}(\Sigma, w v^{\frac{r}{m}})}
\|\mathcal{R}h\|_{L^{s',1}(\Sigma, w v^{\frac{r}{m}})}
\nonumber \\ 
& \lesssim \bigg\|\frac{\mathcal{M}_{\B, r}(\vec{f})}{v}\bigg\|_{L^{\frac{r}{m},\infty} (\Sigma, w v^{\frac{r}{m}})}^{\frac{r}{ms}}
\|h\|_{L^{s',1}(\Sigma, w v^{\frac{r}{m}})},
\end{align}
where \eqref{e:R-2} was used in the last inequality. Consequently, it follows from \eqref{Tff-1} and \eqref{Tff-2} that 
\begin{equation*}
\bigg\|\frac{T(\vec{f})}{v}\bigg\|_{L^{\frac{r}{m},\infty}(\Sigma, \, wv^{\frac{r}{m}})}
\lesssim \bigg\| \frac{\mathcal{M}_{\B, r}(\vec{f})}{v}\bigg\|_{L^{\frac{r}{m},\infty}(\Sigma, \, wv^{\frac{r}{m}})}.
\end{equation*}

It remains to show \eqref{RhAi} and \eqref{e:R-2}. The proof follows the strategy in \cite{CMP}. For the sake of completeness we present the details. %Together with Lemma \ref{lem:A1} parts \eqref{list:ARH}--\eqref{list:A13} and Lemma \ref{lem:ApAp}, the assumption $\vec{w} \in A_{\vec{1}, \B}$ and $w v^{\frac{r}{m}} \in A_{\infty, \B}$ or $\vec{w} \in A_{1, \B} \times \cdots \times A_{1, \B}$ and $v^r \in A_{\infty, \B}$ indicates that 
Since $w \in A_{1, \B}$ and $v^{\frac{r}{m}} \in A_{\infty, \B}$, we have 
\begin{equation}\label{e:S-1}
\begin{aligned}
&\|Sf\|_{L^{\infty}(\Sigma, wv^{\frac{r}{m}})}
\leq [w]_{A_{1, \B}} \|f\|_{L^{\infty}(\Sigma, w v^{\frac{r}{m}})} 
\\
&\text{and}\quad v^{\frac{r}{m}} \in A_{q_0, \B} \text{ for some } q_0>1. 
\end{aligned}
\end{equation}
It follows from Lemma \ref{lem:A1} part \eqref{list:A11} that there exist $v_1, v_2 \in A_{1, \B}$ such that 
\begin{align}\label{vvv}
v^{\frac{r}{m}}=v_1 v_2^{1-q_0}.
\end{align} 

Observe that the second inequality in \eqref{e:R-3} indicates that $\mathcal{R}h \cdot w \in A_{1, \B}$. Recall that $w \in A_{1, \B}$. Then by Lemma \ref{lem:A1} part \eqref{list:A12}, there exists $\varepsilon_0 \in (0, 1)$ such that 
\begin{align}\label{Rww}
\text{$(\mathcal{R}h \cdot w) v_1^{\varepsilon} \in A_{1, \B}$ \, and \, 
$w v_2^{\varepsilon} \in A_{1, \B}$ \, for any $\varepsilon \in (0, \varepsilon_0)$.}
\end{align} 
Choosing $p_0>1+(q_0-1)/{\varepsilon_0}$, $0< \varepsilon < \min\{\varepsilon_0, \frac{1}{2p_0}\}$, and $s :=(\frac{1}{\varepsilon})'>1$, we use \eqref{Rww} and Lemma \ref{lem:A1} part \eqref{list:A11} to see that $w v_2^{\frac{q_0-1}{p_0-1}} \in A_{1, \B}$ and 
\begin{align}\label{wv22}
\mathcal{R}h \cdot w v^{\frac{r}{ms'}}
=[(\mathcal{R}h \cdot w) v_1^{\varepsilon}] \cdot v_2^{1-[(q_0-1)\varepsilon+1]}
\in A_{(q_0-1)\varepsilon+1, \B}.
\end{align}
This shows \eqref{RhAi}. Additionally, in light of \eqref{vvv} and  \eqref{wv22}, Lemma \ref{lem:A1} part \eqref{list:A11} gives 
\begin{align}\label{wpv}
w^{1-p_0} v^{\frac{r}{m}}
=v_1 \Big(w v_2^{\frac{q_0-1}{p_0-1}}\Big)^{1-p_0} \in A_{p_0, \B}.
\end{align}
Thus, invoking \eqref{wpv} and Theorem \ref{Mapp}, we immediately obtain 
\begin{align}\label{e:S-2}
\|Sf\|_{L^{p_0}(\Sigma, \, w v^{\frac{r}{m}})}
= \|M_{\B}(f w)\|_{L^{p_0}(\Sigma, \, w^{1-p_0} v^{\frac{r}{m}})}
\leq c_1 \|f\|_{L^{p_0}(\Sigma, \, w v^{\frac{r}{m}})}.
\end{align}
To proceed, let us recall the Marcinkiewicz interpolation in \cite[~Proposition A.1]{CMP}, which holds in general measure spaces. Then, this, \eqref{e:S-1}, and \eqref{e:S-2} imply that $S$ is bounded on $L^{p,1}(\Sigma, w v^{\frac{r}{m}})$ for all $p \ge p_0$ with the bound 
\begin{align*}
K(p)=2^{\frac1p} \bigg[c_1 \bigg(\frac{1}{p_0} - \frac{1}{p}\bigg)^{-1} + c_2\bigg],
\quad\text{ where } c_2 := [w]_{A_{1, \B}}.
\end{align*}
Since $K(p)$ is decreasing with respect to $p$, there holds 
\begin{equation}\label{eq:Lp1}
\|Sf\|_{L^{p,1}(\Sigma, \, w v^{\frac{r}{m}})}
\leq K \|f\|_{L^{p,1}(\Sigma, \, w v^{\frac{r}{m}})}, \quad p \geq 2p_0
\end{equation}
where $K := 4p_0(c_1+c_2) > K(2p_0) \geq K(p)$. Note that $s'>2p_0$ and apply \eqref{eq:Lp1} to conclude \eqref{e:R-2}. This completes the proof.
\qed

%%%%%%%%%%%%%%%%%%%%% SUBSECTION SUBSECTION SUBSECTION %%%%%%%%%%%%%%%%%%
%%%%%%%%%%%%%%%%%%%%% SUBSECTION SUBSECTION SUBSECTION %%%%%%%%%%%%%%%%%%
\subsection{Quantitative weighted estimates}\label{sec:sharp}
To show Theorems \ref{thm:T} and \ref{thm:T-Besi}, we shall borrow the approach from \cite{LN} to deal with the cases $p>1$ and $p \le 1$ uniformly. We begin with a general result as follows. 

%%%%%%%%%%%%%%%%%%%%%%%% LEMMA LEMMA LEMMA %%%%%%%%%%%%%%%%%%%%%%%%
\begin{lemma}\label{lem:vvq}
Let $(\Sigma, \mu)$ be a measure space with a ball-basis $\B$. Let $1<q_i \leq \infty$ and $0<s_i<\infty$, $i=1, \ldots, m$, such that $\sum_{i=1}^{m}\frac{s_i}{q_i}=1$. Write $\theta_i := s_i (1-\frac{1}{q_i})$, $i=1, \ldots, m$. If $\vec{\sigma} := (\sigma_1, \ldots, \sigma_m)$ are weights such that $\prod_{i=1}^m \sigma_i^{\theta_i}=1$ and
\[
[[\vec{\sigma}]]_{\vec{\theta}} 
:= \sup_{B \in \B} \prod_{i=1}^m \langle \sigma_i \rangle_B^{\theta_i} < \infty,
\]
then we have 
\begin{align*}
&\sup_{\S \subset \B: \, \mathrm{sparse}} \sum_{B \in \S} 
\bigg(\prod_{i=1}^m \langle f_i \sigma_i \rangle_B^{s_i} \bigg) \mu(B) 
\\ 
&\qquad\lesssim [[\vec{\sigma}]]_{\vec{\theta}}^{\max\limits_{1 \le i \le m} \{\frac{q_i}{q_i-1}\}} 
\prod_{i=1}^m \|M_{\B, \sigma_i}\|_{L^{q_i}(\Sigma, \sigma_i)}^{s_i} \|f\|_{L^{q_i}(\Sigma, \sigma_i)}^{s_i}.
\end{align*}
\end{lemma}
%%%%%%%%%%%%%%%%%%%%%%%% LEMMA LEMMA LEMMA %%%%%%%%%%%%%%%%%%%%%%%%

%%%%%%%%%%%%%%%%%%%%%%%%% PROOF PROOF PROOF %%%%%%%%%%%%%%%%%%%%%%%%
\begin{proof}
Let $\S \subset \B$ be a sparse family. Fix $B \in \S$ and write $\theta := \sum_{i=1}^m \theta_i$. By H\"{o}lder's inequality and that $\prod_{i=1}^m \sigma_i^{\theta_i}=1$, 
\begin{align}\label{EBQ}
\mu(B)^q 
\lesssim \mu(E_B)^{\theta} 
= \bigg(\int_{E_B} \prod_{i=1}^m \sigma_i^{\frac{\theta_i}{\theta}} d\mu \bigg)^{\theta}
\leq \prod_{i=1}^m \sigma_i(E_B)^{\theta_i}. 
\end{align}
Since $\theta_i = s_i (1-\frac{1}{q_i})$ and $\sum_{i=1}^m \frac{s_i}{q_i}=1$, we invoke \eqref{EBQ} to get 
\begin{align*}
&\mu(B)\prod_{i=1}^m \frac{\langle \sigma_i \rangle_B^{s_i}}{\sigma_i(E_B)^{s_i/q_i}} 
= \prod_{i=1}^m \bigg(\frac{\sigma_i(B)}{\sigma_i(E_B)} \bigg)^{\frac{s_i}{q_i}} \langle \sigma_i \rangle_B^{\theta_i} 
\\
& \leq \bigg[\prod_{i=1}^m \bigg(\frac{\sigma_i(B)}{\sigma_i(E_B)} \bigg)^{\theta_i} 
\bigg]^{\max\limits_{1 \le i \le m} \{\frac{s_i}{\theta_i q_i}\}} \prod_{i=1}^m \langle \sigma_i \rangle_B^{\theta_i} 
\\
& = \bigg[\bigg(\prod_{i=1}^m \langle \sigma_i \rangle_B^{\theta_i} \bigg) 
\bigg(\frac{\mu(B)^{\theta}}{\prod_{i=1}^m \sigma_i(E_B)^{\theta_i}}\bigg) \bigg]^{\max\limits_{1 \le i \le m} \{\frac{s_i}{\theta_i q_i}\} } \prod_{i=1}^m \langle \sigma_i \rangle_B^{\theta_i} 
\\
& \lesssim [[\vec{\sigma}]]_{\vec{\theta}}^{1 + \max\limits_{1 \le i \le m} \{\frac{s_i}{\theta_i q_i}\}} 
=[[\vec{\sigma}]]_{\vec{\theta}}^{\max\limits_{1 \le i \le m} \{\frac{q_i}{q_i-1}\}}.
\end{align*}
Hence, this in turn gives 
\begin{align*}
&\sum_{B \in \S} \bigg(\prod_{i=1}^m \langle f_i \sigma_i \rangle_B^{s_i} \bigg) \mu(B) 
\\ 
& = \sum_{B \in \S} \prod_{i=1}^m \bigg(\fint_B f_i \, d\sigma_i \bigg)^{s_i} \sigma_i(E_B)^{\frac{s_i}{q_i}} 
\bigg(\mu(B) \prod_{i=1}^m \frac{\langle \sigma_i \rangle_B^{s_i}}{\sigma_i(E_B)^{s_i/q_i}}  \bigg)
\\
& \lesssim [[\vec{\sigma}]]_{\vec{\theta}}^{\max\limits_{1 \le i \le m}\{\frac{q_i}{q_i-1}\}} 
\sum_{B \in \S} \prod_{i=1}^m \bigg(\fint_B f_i \, d\sigma_i \bigg)^{s_i} \sigma_i(E_B)^{\frac{s_i}{q_i}} 
\\
& \le [[\vec{\sigma}]]_{\vec{\theta}}^{\max\limits_{1 \le i \le m}\{\frac{q_i}{q_i-1}\}} 
\prod_{i=1}^m \bigg[\sum_{B \in \S} \bigg(\fint_B f_i \, dv_i \bigg)^{q_i} \sigma_i(E_B) \bigg]^{\frac{s_i}{q_i}} 
\\
& \le [[\vec{\sigma}]]_{\vec{\theta}}^{\max\limits_{1 \le i \le m}\{\frac{q_i}{q_i-1}\}} 
\prod_{i=1}^m \bigg[\sum_{B \in \S} \big(\inf_B M_{\B, \sigma_i} f_i \big)^{q_i} \sigma_i(E_B) \bigg]^{\frac{s_i}{q_i}} 
\\
& \le [[\vec{\sigma}]]_{\vec{\theta}}^{\max\limits_{1 \le i \le m}\{\frac{q_i}{q_i-1}\}} 
\prod_{i=1}^m \bigg[\sum_{B \in \S} \int_{E_B} (M_{\B, \sigma_i} f_i)^{q_i} \, d\sigma_i \bigg]^{\frac{s_i}{q_i}} 
\\
& \le [[\vec{\sigma}]]_{\vec{\theta}}^{\max\limits_{1 \le i \le m}\{\frac{q_i}{q_i-1}\}} 
\prod_{i=1}^m \|M_{\B, \sigma_i}f\|_{L^{q_i}(\Sigma, \sigma_i)}^{s_i}
\\
& \leq [[\vec{\sigma}]]_{\vec{\theta}}^{\max\limits_{1 \le i \le m}\{\frac{q_i}{q_i-1}\}} 
\prod_{i=1}^m \|M_{\B, \sigma_i}\|_{L^{q_i}(\Sigma, \sigma_i)}^{s_i} \|f\|_{L^{q_i}(\Sigma, \sigma_i)}^{s_i}. 
\end{align*}
The proof is complete. 
\end{proof}
%%%%%%%%%%%%%%%%%%%%%%%%% END END END PROOF %%%%%%%%%%%%%%%%%%%%%%%%

Next, we present weighted estimates for multilinear sparse operators $\A_{\S, \vec{r}}$ and $\A_{\S, \vec{r}}^{\b, \tau_1, \tau_2}$. 

%%%%%%%%%%%%%%%%%%%%%%%% LEMMA LEMMA LEMMA %%%%%%%%%%%%%%%%%%%%%%%%
\begin{lemma}\label{lem:Sparse}
Let $(\Sigma, \mu)$ be a measure space with a ball-basis $\B$. For all $\vec{r}=(r_1, \ldots, r_m)$ with $1 \le r_1, \ldots, r_m<\infty$, for all $\vec{p}=(p_1, \ldots, p_m)$ with $r_i<p_i<\infty$, $i=1, \ldots, m$, and for all $\vec{w}=(w_1, \ldots, w_m) \in A_{\vec{p}/\vec{r}}$, we have 
\begin{align*}
\sup_{\S \subset \B: \text{sparse}}
\|\A_{\S, \vec{r}}\|_{L^{p_1}(\Sigma, w_1) \times \ldots \times L^{p_m}(\Sigma, w_m) \to L^p(\Sigma, w)} 
\lesssim \mathcal{N}_1(\vec{r}, \vec{p}, \vec{w}) 
[\vec{w}]_{A_{\vec{p}/\vec{r}, \B}}^{\max\limits_{1\le i \le m}\{p, (\frac{p_i}{r_i})'\}}. 
\end{align*}
Moreover, if $A_{\infty, \B}$ satisfies the sharp reverse H\"{o}lder property, then for the same exponents $\vec{p}$ and weights  $\vec{w}$, 
\begin{align*}
& \sup_{\S \subset \B: \text{sparse}} 
\|\A_{\S, \vec{r}}^{\b, \tau_1, \tau_2} \|_{L^{p_1}(\Sigma, w_1) \times \cdots \times L^{p_m}(\Sigma, w_m) \rightarrow L^p(\Sigma, w)} 
\\ 
&\lesssim \mathcal{N}_1(\vec{r}, \vec{p}, \vec{w}) 
\mathcal{N}_2^{\tau_2}(\vec{r}, \vec{p}, \vec{w}) 
[w]_{A_{\infty, \B}}^{ |\tau_1|}
\\
&\qquad\qquad\times \prod_{j \in \tau_2} [\sigma_j]_{A_{\infty, \B}}
[\vec{w}]_{A_{\vec{p}/\vec{r}}}^{\max\limits_{1 \leq i \leq m} \{p, (\frac{p_i}{r_i})' \}}
\prod_{i \in \tau_1 \uplus \tau_2} \|b_i\|_{\osc_{\exp L}}
\\ 
&\lesssim \mathcal{N}_1(\vec{r}, \vec{p}, \vec{w}) 
\mathcal{N}_2^{\tau_2}(\vec{r}, \vec{p}, \vec{w}) 
[\vec{w}]_{A_{\vec{p}/\vec{r}}}^{\beta  
\max\limits_{1 \leq i \leq m} \{p, (\frac{p_i}{r_i})' \}}
\prod_{i \in \tau_1 \uplus \tau_2} \|b_i\|_{\osc_{\exp L}},   
\end{align*}
where $\beta := |\tau_1| + |\tau_2| +1$.
\end{lemma}
%%%%%%%%%%%%%%%%%%%%%%%% LEMMA LEMMA LEMMA %%%%%%%%%%%%%%%%%%%%%%%%

%%%%%%%%%%%%%%%%%%%%%%%%% PROOF PROOF PROOF %%%%%%%%%%%%%%%%%%%%%%%%
\begin{proof}
Let $\vec{r}=(r_1, \ldots, r_m)$ with $1 \le r_1, \ldots, r_m<\infty$, $\vec{p}=(p_1, \ldots, p_m)$ with $r_i<p_i<\infty$, $i=1, \ldots, m$, and $\vec{w}=(w_1, \ldots, w_m) \in A_{\vec{p}/\vec{r}, \B}$. We begin with the estimate for $\A_{\S, \vec{r}}$. It suffices to show 
\begin{align}\label{fsigma}
\mathcal{J} 
&:= \|\A_{\S, \vec{r}}(f_1 \sigma_1^{\frac{1}{r_1}}, \ldots, f_m \sigma_m^{\frac{1}{r_m}})\|_{L^p(\Sigma, w)} 
\lesssim \mathcal{N}(\vec{r}, \vec{p}, \vec{w}) \prod_{i=1}^m \|f_i\|_{L^{p_i}(\Sigma, \sigma_i)}, 
\end{align}
where $\sigma_i := w_i^{\frac{r_i}{r_i-p_i}}$, $i=1, \ldots, m$. Assume first that $p>1$. By duality, \eqref{fsigma} is reduced to the following inequality 
\begin{align}\label{Npw}
\sum_{B \in \S} \Big(\prod_{i=1}^m \langle f_i \sigma_i^{\frac{1}{r_i}} \rangle_{B, r_i} \Big) \langle g w \rangle_B \, \mu(B) 
\lesssim \mathcal{N}(\vec{r}, \vec{p}, \vec{w}) \prod_{i=1}^m \|f_i\|_{L^{p_i}(\Sigma, \sigma_i)}, 
\end{align}
for all $g \in L^{p'}(\Sigma, w)$ with $\|g\|_{L^{p'}(\Sigma, w)} = 1$. Choosing 
\begin{align*}
& \sigma_{m+1} = w, \quad 
p_{m+1}=p', \quad 
r_{m+1}=1, \quad 
f_{m+1}=g, \quad 
\\ 
&s_i = \frac{1}{r_i}, \quad 
q_i = \frac{p_i}{r_i}, \quad 
\text{and}\quad  
\theta_i := \frac{1}{r_i}-\frac{1}{p_i},\quad  i=1, \ldots, m+1, 
\end{align*} 
we see that
\begin{align*}
&\sum_{i=1}^{m+1} \frac{s_i}{q_i} 
=\sum_{i=1}^{m+1} \frac{1}{p_i} = 1, \quad 
\theta_i = s_i \Big(1-\frac{1}{q_i}\Big), \quad i=1, \ldots, m+1, 
\\ 
&\text{and} \quad \prod_{i=1}^{m+1} \sigma_i^{\theta_i} 
=\prod_{i=1}^m w_i^{\frac{r_i \theta_i}{r_i - p_i}} \times w^{\frac1p} 
= \prod_{i=1}^m w_i^{-\frac{1}{p_i}} \times w^{\frac{1}{p}}=1. 
\end{align*}
By the choices above, we have 
\begin{align}\label{sigsig}
[[\vec{v}]]_{\vec{\theta}} 
= [\vec{w}]_{A_{\vec{p}/\vec{r}, \B}}  
\quad\text{and}\quad 
\max_{1 \le i \le m+1} \Big\{\frac{q_i}{q_i-1} \Big\}  
= \max\limits_{1 \le i \le m} \Big\{\Big(\frac{p_i}{r_i} \Big)', p \Big\}. 
\end{align}
Then in view of  \eqref{sigsig}, Lemma \ref{lem:vvq} applied to $|f_i|^{r_i}$ in place of $f_i$ yields 
\begin{align}\label{JP-1}
&\sum_{B \in \S} \Big(\prod_{i=1}^m \langle f_i \sigma_i^{\frac{1}{r_i}} \rangle_{B, r_i} \Big) 
\langle g w \rangle_B \, \mu(B) 
\nonumber \\ 
&=\sum_{B \in \S} \Big(\prod_{i=1}^m \langle |f_i|^{r_i} \sigma_i \rangle_B^{\frac{1}{r_i}} \Big) 
\langle g w \rangle_B \, \mu(B) 
\nonumber \\ 
&= \sum_{B \in \S} \Big(\prod_{i=1}^{m+1} \langle |f_i|^{r_i} \sigma_i \rangle_B^{s_i} \Big) \, \mu(B) 
\nonumber \\ 
&\lesssim [[\vec{\sigma}]]_{\vec{\theta}}^{\max\limits_{1 \le i \le m+1} \{\frac{q_i}{q_i-1} \}}  
\prod_{i=1}^{m+1} \|M_{\B, \sigma_i}\|_{L^{q_i}(\Sigma, \sigma_i)}^{s_i} \| |f_i|^{r_i}\|_{L^{q_i}(\Sigma, \sigma_i)}^{s_i} 
\nonumber \\ 
&= [\vec{w}]_{A_{\vec{p}/\vec{r}, \B}}^{\max\limits_{1 \le i \le m}\{(\frac{p_i}{r_i})', p\}} 
\|M_{\B, w}\|_{L^{p'}(\Sigma, w)}
\nonumber \\ 
&\qquad\qquad\times \prod_{i=1}^m \|M_{\B, \sigma_i}\|_{L^{p_i/r_i}(\Sigma, \sigma_i)}^{\frac{1}{r_i}} 
\|f_i\|_{L^{p_i}(\Sigma, \sigma_i)}, 
\end{align}
which shows \eqref{Npw}. 

To deal with the case $0 < p \le 1$, we observe that
\begin{align}\label{jpjp}
\mathcal{J}^p 
\le \widetilde{\mathcal{J}}^p 
&:= \sum_{B \in \S} \bigg(\prod_{i=1}^m \langle f_i \sigma_i^{\frac{1}{r_i}} \rangle_{B, r_i}^p \bigg) w(B) 
= \sum_{B \in \S} \bigg(\prod_{i=1}^{m+1} \langle |f_i|^{r_i} \sigma_i \rangle_B^{s_i} \bigg) \mu(B), 
\end{align}
where $\sigma_{m+1} = w$, $f_{m+1}=1$, $s_1=\cdots=s_m=p/r_i$, and $s_{m+1}=1$. Choosing 
\[
q_i = \frac{p_i}{r_i} \quad\text{ and }\quad 
\theta_i=s_i \bigg(1-\frac{1}{q_i}\bigg), \quad i=1, \ldots, m+1, 
\]
where $r_{m+1}=1$ and $p_{m+1}=\infty$, we see that 
\begin{align*}
\sum_{i=1}^{m+1}\frac{s_i}{q_i}
=\sum_{i=1}^m \frac{p}{p_i}= 1
\quad\text{ and}\quad 
\prod_{i=1}^{m+1} \sigma_i^{\theta_i} 
=w\prod_{i=1}^m w_i^{\frac{r_i \theta_i}{r_i-p_i}} 
=w\prod_{i=1}^m w_i^{-\frac{p}{p_i}} =1. 
\end{align*}
Note that
\begin{align*}
[[\vec{\sigma}]]_{\vec{\theta}} = [\vec{w}]_{A_{\vec{p}/\vec{r}, \B}}^p
\quad\text{ and }\quad 
\max\limits_{1 \le i \le m+1} \Big\{\frac{q_i}{q_i-1}\Big\}
=\max\limits_{1 \le i \le m}\Big\{\Big(\frac{p_i}{r_i}\Big)' \Big\}, 
\end{align*}
which along with Lemma \ref{lem:vvq} applied to $|f_i|^{r_i}$ instead of $f_i$ yields 
\begin{align}\label{JP-2}
\widetilde{\mathcal{J}}^p 
&\lesssim [[\vec{\sigma}]]_{\vec{\theta}}^{\max\limits_{1 \le i \le m+1} \{\frac{q_i}{q_i-1} \}}  
\prod_{i=1}^{m+1} \|M_{\B, \sigma_i}\|_{L^{q_i}(\Sigma, \sigma_i)}^{s_i} \||f_i|^{r_i}\|_{L^{q_i}(\Sigma, \sigma_i)}^{s_i} 
\nonumber \\ 
&\le [\vec{w}]_{A_{\vec{p}/\vec{r}, \B}}^{p \max\limits_{1 \le i \le m}\{(\frac{p_i}{r_i})'\}} 
\prod_{i=1}^m \|M_{\B, \sigma_i}\|_{L^{p_i/r_i}(\Sigma, \sigma_i)}^{\frac{p}{r_i}} \|f_i\|_{L^{p_i}(\Sigma, \sigma_i)}^p, 
\end{align}
where we have used that $\|M_{\B, w}\|_{L^{\infty}(\Sigma, w)} \le 1$ and $\|f_{m+1}\|_{L^{\infty}(\Sigma, w)}=1$. Hence, \eqref{fsigma} follows from \eqref{Npw}, \eqref{jpjp}, and \eqref{JP-2}.

Next, let us estimate $\A_{\S, \vec{r}}^{\b, \tau_1, \tau_2}$. Assuming that $\|b_i\|_{\osc_{\exp L}}=1$, $i \in \tau_1 \cup \tau_2$, it is enough to show 
\begin{align}\label{ASNP}
&\|\A_{\S,\vec{r}}^{\b, \tau_1, \tau_2}(f_1 \sigma_1^{\frac{1}{r_1}}, \ldots, f_m \sigma_m^{\frac{1}{r_m}})\|_{L^p(\Sigma, w)} 
\nonumber \\
&\qquad\lesssim \mathcal{N}_1(\vec{r}, \vec{p}, \vec{w}) 
\mathcal{N}_2^{\tau_2}(\vec{r}, \vec{p}, \vec{w}) 
[w]_{A_{\infty, \B}}^{ |\tau_1|}
\nonumber \\
&\qquad\qquad\times \prod_{j \in \tau_2} [\sigma_j]_{A_{\infty, \B}}
[\vec{w}]_{A_{\vec{p}/\vec{r}}}^{\max\limits_{1 \leq i \leq m} \{p, (\frac{p_i}{r_i})' \}}
\prod_{i=1}^m \|f_i\|_{L^{p_i}(\Sigma, \sigma_i)}, 
\end{align}
since \eqref{ww-2} applied to $\vec{p}/\vec{r}$ in place of $\vec{p}$ gives 
\begin{align}\label{www}
[w]_{A_{\infty, \B}}^{ |\tau_1|}
\prod_{j \in \tau_2} [\sigma_j]_{A_{\infty, \B}} 
&\le [\vec{w}]_{A_{\vec{p}, \B}}^{|\tau_1| p}
\prod_{j \in \tau_2} [\vec{w}]_{A_{\vec{p}, \B}}^{(p_j/r_j)'} 
\lesssim [\vec{w}]_{A_{\vec{p}, \B}}^{(|\tau_1| + |\tau_2|) \max\limits_{1 \le i \le m}\{p, (\frac{p_i}{r_i})'\}}.
\end{align}
To treat the case $p > 1$, let $g \in L^{p'}(\Sigma, w)$ be a nonnegative function satisfying $\|g\|_{L^{p'}(\Sigma, w)} = 1$. It follows from Lemma \ref{lem:PhiPhi} that 
\begin{align*}
&\int_{\Sigma} \A_{\S,\vec{r}}^{\b, \tau_1, \tau_2}
(f_1 \sigma_1^{\frac{1}{r_1}}, \ldots, f_m \sigma_m^{\frac{1}{r_m}}) \, gw\, d\mu 
\\ 
& \lesssim \sum_{B \in \S} \mu(B) \bigg(\fint_B \prod_{i \in \tau_1}
\langle f_i \sigma_i^{\frac{1}{r_i}} \rangle_{B, r_i} |b_i - b_{i, B}|  g w\, d\mu \bigg)
\\ 
&\quad\quad \times \prod_{j \in \tau_2} \langle (b_j - b_{j, B}) f_j \sigma_j^{\frac{1}{r_j}}\rangle_{B, r_j} 
\times \prod_{k \not\in \tau_1 \cup \tau_2} \langle f_k \sigma_k^{\frac{1}{r_k}}\rangle_{B, r_k} 
\\ 
& \lesssim \sum_{B \in \S} \mu(B) \|gw\|_{L(\log L)^{|\tau_1|}, B}
\prod_{i \in \tau_1} \langle f_i \sigma^{\frac{1}{r_i}} \rangle_{B, r_i}
\|b_i - b_{i, B} \|_{\exp L, B}
\\ 
&\quad \times \prod_{j \in \tau_2}
\|(b_j - b_{j, B})^{r_j} \|_{\exp L^{\frac{1}{r_j}}, B}^{\frac{1}{r_j}} 
\| |f_j|^{r_j} \sigma_j  \|_{L(\log L)^{r_j}, B}^{\frac{1}{r_j}}
\prod_{k \not\in \tau_1 \cup \tau_2} \langle f_k \sigma_k^{\frac{1}{r_k}}\rangle_{B, r_k} 
\\ 
& \lesssim 
\sum_{B \in \S} \prod_{i \not\in \tau_2} \langle f_i \sigma_i^{\frac{1}{r_i}} \rangle_{B, r_i}
\prod_{j \in \tau_2}  \| |f_j|^{r_j} \sigma_j  \|_{L(\log L)^{r_j}, B}^{\frac{1}{r_j}} 
\|gw\|_{L(\log L)^{|\tau_1|}, B} \mu(B). 
\end{align*}
To control the inner terms, we use \eqref{e:fw} to obtain 
\begin{align}
\label{tau-1} \| |f_j|^{r_j} \sigma_j  \|_{L(\log L)^{r_j}, B}^{\frac{1}{r_j}}
&\lesssim [\sigma_j]_{A_{\infty, \B}} \big\langle M_{\B, \sigma_j}(|f_j|^{r_j s_j})^{\frac{1}{s_j}} \sigma_j \big\rangle_B^{\frac{1}{r_j}} ,
\\ 
\label{tau-2} \|gw\|_{L(\log L)^{|\tau_1|}, B}
&\lesssim [w]_{A_{\infty, \B}}^{|\tau_1|} \big\langle M_{\B, w}(|g|^s)(x)^{\frac1s} w \big\rangle_B, 
\end{align}
where $1 < s_j < p_j/r_j$, $j \in \tau_2$, and $1 < s < p'$. Collecting \eqref{tau-1} and \eqref{tau-2}, we obtain
\begin{align}\label{ASNP-1}
&\int_{\Sigma} \A_{\S,\vec{r}}^{\b, \tau_1, \tau_2}
(f_1 \sigma_1^{\frac{1}{r_1}}, \ldots, f_m \sigma_m^{\frac{1}{r_m}}) \, gw\, d\mu 
\nonumber \\ 
& \lesssim [w]_{A_{\infty, \B}}^{|\tau_1|} \prod_{j \in \tau_2} [\sigma_j]_{A_{\infty, \B}}
\sum_{B \in \S} \prod_{i \not\in \tau_2} \langle f_i \sigma_i^{\frac{1}{r_i}} \rangle_{B, r_i} 
\nonumber \\ 
&\quad \times \prod_{j \in \tau_2} 
\big\langle M_{\B, \sigma_j}(|f_j|^{r_j s_j})^{\frac{1}{r_j s_j}} \cdot \sigma_j^{\frac{1}{r_j}} \big\rangle_{B, r_j}
\big\langle M_{\B, w}(|g|^s)^{\frac1s} w \big\rangle_B \, \mu(B)
\nonumber \\ 
& \lesssim [w]_{A_{\infty, \B}}^{|\tau_1|} \prod_{j \in \tau_2} [\sigma_j]_{A_{\infty, \B}} 
[\vec{w}]_{A_{\vec{p}/\vec{r}, \B}}^{\max\limits_{1 \le i \le m}\{(\frac{p_i}{r_i})', p\}} 
\|M_{\B, w}\|_{L^{p'}(\Sigma, w)}
\nonumber \\ 
&\quad\times\prod_{i=1}^m \|M_{\B, \sigma_i}\|_{L^{p_i/r_i}(\Sigma, \sigma_i)}^{\frac{1}{r_i}}
\prod_{i \not\in \tau_2} \|f_i\|_{L^{p_i}(\Sigma, \sigma_i)}
\nonumber \\ 
&\quad\times
\prod_{j \in \tau_2} \|M_{\B, \sigma_j}(|f_j|^{r_j s_j})^{\frac{1}{r_j s_j}}\|_{L^{p_j}(\Sigma, \sigma_j)}
\|M_{\B, w}(|g|^s)^{\frac1s}\|_{L^{p'}(\Sigma, w)}
\nonumber \\ 
& \lesssim [w]_{A_{\infty, \B}}^{|\tau_1|} \prod_{j \in \tau_2} [\sigma_j]_{A_{\infty, \B}} 
[\vec{w}]_{A_{\vec{p}/\vec{r}, \B}}^{\max\limits_{1 \le i \le m}\{(\frac{p_i}{r_i})', p\}} 
\|M_{\B, w}\|_{L^{p'}(\Sigma, w)}
\nonumber \\ 
&\quad\times \prod_{i=1}^m \|M_{\B, \sigma_i}\|_{L^{p_i/r_i}(\Sigma, \sigma_i)}^{\frac{1}{r_i}}
\prod_{j \in \tau_2} \|M_{\B, \sigma_j}\|_{L^{\frac{p_j}{r_j s_j}}(\Sigma, \sigma_j)}^{\frac{1}{r_j s_j}} 
\nonumber \\ 
&\quad\times \|M_{\B, w}\|_{L^{p'/s}(\Sigma, w)}^{\frac1s} \prod_{i=1}^m \|f_i\|_{L^{p_i}(\Sigma, \sigma_i)}.  
\end{align}
Next, let us turn our attention to the case $0 < p \leq 1$. Using Lemma \ref{lem:PhiPhi} and \eqref{tau-1}, we have
\begin{align*}
\|&\A_{\S,\vec{r}}^{\b, \tau_1, \tau_2}(f_1 \sigma_1^{\frac{1}{r_1}}, \ldots, f_m \sigma_m^{\frac{1}{r_m}})\|_{L^p(w)}^p 
\\
& \leq \sum_{B \in \S} \bigg(\fint_B \prod_{i \in \tau_1} 
\langle f_i \sigma_i^{\frac{1}{r_i}} \rangle_{B, r_i}^p |b_i - b_{i,B}|^p w \, d\mu \bigg)
\\
&\qquad \times \prod_{j \in \tau_2} \langle (b_j - b_{j,B}) f_j \sigma_j^{\frac{1}{r_j}} \rangle_{B, r_j}^p 
\prod_{k \not\in \tau_1 \cup \tau_2} \langle f_k \sigma_k^{\frac{1}{r_k}} \rangle_{B, r_k} \, \mu(B) 
\\ 
& \lesssim \sum_{B \in \S} \|w\|_{L(\log L)^{p |\tau_1|}, B} 
\prod_{i \in \tau_1} \langle f_i \sigma_i^{\frac{1}{r_i}} \rangle_{B, r_i}^p
\|(b_i - b_{i, B})^p\|_{\exp L^{\frac1p}, B}
\\ 
&\quad\quad \times \prod_{j \in \tau_2} \| (b_j - b_{j, B})^{r_j} \|_{\exp L^{\frac{1}{r_j}}, B}^{\frac{p}{r_j}} 
\| |f_j|^{r_j} \sigma_j \|_{L(\log L)^{r_j}, B}^{\frac{p}{r_j}} 
\\
&\qquad\times \prod_{k \not\in \tau_1 \cup \tau_2} \langle f_k \sigma_k^{\frac{1}{r_k}} \rangle_{B, r_k}\, \mu(B)
\\ 
& \lesssim [w]_{A_{\infty}}^{p |\tau_1|} \prod_{j \in \tau_2} [\sigma_j]_{A_{\infty}}^p 
\sum_{B \in \S} \prod_{i \in \tau_1 \cup \tau_3} \langle f_i \sigma_i^{\frac{1}{r_i}} \rangle_{B, r_i}^p 
\\
&\qquad\times \prod_{j \in \tau_2}  \big\langle M_{\B, \sigma_j}(|f_j|^{r_j s_j})^{\frac{1}{r_j s_j}} \cdot \sigma_j^{\frac{1}{r_j}} \big\rangle_{B, r_j}^p w(B). 
\end{align*}
Furthermore, it can be controlled by 
\begin{align}\label{ASNP-2}
\|&\A_{\S,\vec{r}}^{\b, \tau_1, \tau_2}(f_1 \sigma_1^{\frac{1}{r_1}}, \ldots, f_m \sigma_m^{\frac{1}{r_m}})\|_{L^p(w)}^p 
\nonumber \\ 
&\quad\lesssim [w]_{A_{\infty}}^{p |\tau_1|} \prod_{j \in \tau_2} [\sigma_j]_{A_{\infty}}^p
[\vec{w}]_{A_{\vec{p}/\vec{r}, \B}}^{p \max\limits_{1 \le i \le m}\{(\frac{p_i}{r_i})'\}} 
\prod_{i=1}^m \|M_{\B, \sigma_i}\|_{L^{p_i/r_i}(\Sigma, \sigma_i)}^{\frac{p}{r_i}} 
\nonumber \\ 
&\qquad\quad\times \prod_{i \in \tau_1 \cup \tau_3} \|f_i\|_{L^{p_i}(\Sigma, \sigma_i)}^p 
\prod_{j \in \tau_2} \|M_{\B, \sigma_j}(|f_j|^{r_j s_j})^{\frac{1}{r_j s_j}}\|_{L^{p_j}(\Sigma, \sigma_j)}^p
\nonumber \\  
&\quad\lesssim [w]_{A_{\infty}}^{p |\tau_1|} \prod_{j \in \tau_2} [\sigma_j]_{A_{\infty}}^p
[\vec{w}]_{A_{\vec{p}/\vec{r}, \B}}^{p \max\limits_{1 \le i \le m}\{(\frac{p_i}{r_i})'\}} 
\prod_{i=1}^m \|M_{\B, \sigma_i}\|_{L^{p_i/r_i}(\Sigma, \sigma_i)}^{\frac{p}{r_i}} 
\nonumber \\ 
&\qquad\quad\times \prod_{j \in \tau_2} \|M_{\B, \sigma_j}\|_{L^{\frac{p_j}{r_j s_j}}(\Sigma, \sigma_j)}^{\frac{p}{r_j s_j}} 
\prod_{i=1}^m \|f_i\|_{L^{p_i}(\sigma_i)}^p. 
\end{align}
Thus, \eqref{ASNP} is a consequence of \eqref{ASNP-1} and \eqref{ASNP-2}. This completes the proof. 
\end{proof}
%%%%%%%%%%%%%%%%%%%%%%%%% END END END PROOF %%%%%%%%%%%%%%%%%%%%%%%%%

Now, it is clear that Theorem \ref{thm:T} is a direct consequence of Theorem \ref{thm:sparse} and Lemma \ref{lem:Sparse}. 
Additionally, by Lemma \ref{lem:Sparse}, to prove Theorem \ref{thm:T-Besi}, one has to control $\|M_{\B, w}\|_{L^s(\Sigma, w)}$ whenever $\B$ satisfies the Besicovitch condition. This was done in \cite[Theorem 6.3]{Kar}.

%%%%%%%%%%%%%%%%%%%%%%%%% LEMMA LEMMA LEMMA %%%%%%%%%%%%%%%%%%%%%%%
\begin{lemma}\label{lem:M-Besi}
Let $(\Sigma, \mu)$ be a measure space with a ball-basis $\B$ satisfying the Besicovitch condition with the constant $N_0$. Then for any weight $w$,  
\begin{align}
&\|M_{\B, w}\|_{L^1(\Sigma, w) \to L^{1, \infty}(\Sigma, w)} \le N_0, 
\\
&\|M_{\B, w}\|_{L^s(\Sigma, w) \to L^s(\Sigma, w)} \le C_p N_0^{\frac1s}, \quad 1<s<\infty. 
\end{align}
\end{lemma} 
%%%%%%%%%%%%%%%%%%%%%%%%% LEMMA LEMMA LEMMA %%%%%%%%%%%%%%%%%%%%%%%

Then Lemmas \ref{lem:Sparse} and \ref{lem:M-Besi} imply the following result.  

%%%%%%%%%%%%%%%%%%%%%%%% LEMMA LEMMA LEMMA %%%%%%%%%%%%%%%%%%%%%%%%
\begin{lemma}\label{lem:Sparse-Besi}
Let $(\Sigma, \mu)$ be a measure space with a ball-basis $\B$ satisfying the Besicovitch condition. Then for all $\vec{r}=(r_1, \ldots, r_m)$ with $1 \le r_1, \ldots, r_m<\infty$, for all $\vec{p}=(p_1, \ldots, p_m)$ with $r_i<p_i<\infty$, $i=1, \ldots, m$, and for all $\vec{w}=(w_1, \ldots, w_m) \in A_{\vec{p}/\vec{r}}$,  we have 
\begin{align*}
\sup_{\S \subset \B: \text{sparse}} 
\|\mathcal{A}_{\S, \vec{r}}\|_{L^{p_1}(w_1) \times \ldots \times L^{p_m}(w_m) \to L^p(w)} 
\lesssim [\vec{w}]_{\vec{p}/\vec{r}, \B}^{\max\limits_{1 \le i \le m}\{p, (\frac{p_i}{r_i})'\}}. 
\end{align*}
Moreover, if $A_{\infty, \B}$ satisfies the sharp reverse H\"{o}lder property, then for the same exponents $\vec{p}$ and weights  $\vec{w}$, 
\begin{align*}
&\sup_{\S \subset \B: \text{sparse}} 
\|\A_{\S, \vec{r}}^{\b, \tau_1, \tau_2} \|_{L^{p_1}(\Sigma, w_1) \times \cdots \times L^{p_m}(\Sigma, w_m) \rightarrow L^p(\Sigma, w)} 
\\ 
&\quad \lesssim [w]_{A_{\infty, \B}}^{ |\tau_1|}
\prod_{j \in \tau_2} [\sigma_j]_{A_{\infty, \B}}
[\vec{w}]_{A_{\vec{p}/\vec{r}}}^{\max\limits_{1 \leq i \leq m} \{p, (\frac{p_i}{r_i})' \}}
\prod_{i \in \tau_1 \uplus \tau_2} \|b_i\|_{\osc_{\exp L}}
\\ 
&\quad \lesssim 
[\vec{w}]_{A_{\vec{p}/\vec{r}}}^{(|\tau_1| + |\tau_2|+1) 
\max\limits_{1 \leq i \leq m} \{p, (\frac{p_i}{r_i})' \}}
\prod_{i \in \tau_1 \uplus \tau_2} \|b_i\|_{\osc_{\exp L}}. 
\end{align*}
\end{lemma}
%%%%%%%%%%%%%%%%%%%%%%%% LEMMA LEMMA LEMMA %%%%%%%%%%%%%%%%%%%%%%%%

Therefore, Theorem \ref{thm:T-Besi} follows from Theorem \ref{thm:sparse} and Lemma \ref{lem:Sparse-Besi}.

%%%%%%%%%%%%%%%%%%%%%%%% SECTION SECTION SECTION %%%%%%%%%%%%%%%%%%%%%%
%%%%%%%%%%%%%%%%%%%%%%%% SECTION SECTION SECTION %%%%%%%%%%%%%%%%%%%%%%
\section{Compactness of commutators}\label{sec:compact}
Our goal of this section is to show Theorems \ref{thm:Tb}--\ref{thm:FKhs-2}.

%%%%%%%%%%%%%%%%%%%%% SUBSECTION SUBSECTION SUBSECTION %%%%%%%%%%%%%%%%%%
%%%%%%%%%%%%%%%%%%%%% SUBSECTION SUBSECTION SUBSECTION %%%%%%%%%%%%%%%%%%
\subsection{Characterizations of compactness}  
We begin with showing weighted Fr\'{e}chet-Kolmogorov theorems, which characterize the relative compactness of a set in $L^p(\Sigma, w)$ on spaces of homogeneous type $(\Sigma, \rho, \mu)$. In Lebesgue spaces, it was proved by Yosida \cite[p.~275]{Yosida} in the case $1\le p<\infty$, which was extended to the quasi-Banach case $0<p<1$ in \cite{Tsuji} and the weighted Lebesgue spaces in \cite{COY, XYY}.

%%%%%%%%%%%%%%%%%%%%%%%%%% PROOF PROOF PROOF  %%%%%%%%%%%%%%%%%%%%%%%
\begin{proof}[\bf Proof of Theorem \ref{thm:FKhs-1}]
We begin with proving the necessity. Since $\K$ is relatively compact, it is totally bounded. Then, given $\varepsilon>0$, one can find a finite number of functions $\{f_j\}_{j=1}^N \subset \K$ such that $\K \subseteq \bigcup_{k=1}^N B(f_k,\varepsilon)$. This means that given an arbitrary function $f \in \K$, there exists some $k \in \{1, \ldots, N\}$ such that 
\begin{align}\label{eq:fk-f}
\|f_k-f\|_{L^{p}(\Sigma, w)} < \varepsilon, 
\end{align}
which in turn gives 
\begin{align*}
\|f\|_{L^{p}(\Sigma, w)} 
\leq \|f-f_k\|_{L^{p}(\Sigma, w)} + \|f_k\|_{L^{p}(\Sigma, w)}
< \varepsilon + \max_{1 \leq k \leq N} \|f_k\|_{L^{p}(\Sigma, w)}. 
\end{align*}
This justifies the condition \eqref{list:FK1} holds. Since $f_k \in L^p(\Sigma, w)$, there exists $A_k>0$ such that 
\begin{equation}\label{eq:fkAk}
\|f_k \mathbf{1}_{\Sigma \setminus B(x_0, A_k)}\|_{L^p(\Sigma, w)} < \varepsilon, 
\quad k=1,\ldots,N. 
\end{equation}
Set $A:=\max\{A_k: k=1,\ldots,N\}$. Then by \eqref{eq:fk-f} and \eqref{eq:fkAk},  
\begin{align*}
\|f \mathbf{1}_{\Sigma \setminus B(x_0, A)}\|_{L^p(\Sigma, w)} 
\leq \|f-f_k\|_{L^p(\Sigma, w)} + \|f_k \mathbf{1}_{\Sigma \setminus B(x_0, A_k)}\|_{L^p(\Sigma, w)}
< 2\varepsilon. 
\end{align*}
This shows the condition \eqref{list:FK2} holds. To continue, we split 
\begin{align}\label{ffk-1}
\|f-f_{B(\cdot,r)}\|_{L^p(\Sigma, w)} 
&\leq \|f-f_k\|_{L^p(\Sigma, w)} 
+ \|f_k-(f_k)_{B(\cdot,r)}\|_{L^p(\Sigma, w)}
\nonumber \\ 
&\qquad+ \|(f_k)_{B(\cdot,r)}-f_{B(\cdot,r)}\|_{L^p(\Sigma, w)}. 
\end{align}
Note that 
\begin{align*}
|f_k(x)-(f_k)_{B(x,r)}| 
\lesssim |f_k(x)| + M_{\B_{\rho}}f_k(x) \in L^p(\Sigma, w) 
\end{align*}
and $(f_k)_{B(x,r)} \to f_k(x)$ a.e. $x \in \Sigma$ by the Lebesgue differentiation theorem (cf. \cite[p. 12]{GLMY}). Thus, the Lebesgue domination convergence theorem gives that for some $\delta>0$, 
\begin{align}\label{ffk-2}
\|f_k-(f_k)_{B(\cdot,r)}\|_{L^p(\Sigma, w)} 
< \varepsilon,\quad \forall r \in (0, \delta). 
\end{align}
Since 
\begin{align*}
|(f_k)_{B(x,r)}-f_{B(x,r)}| 
\leq \fint_{B(x,r)} |f_k - f| \, d\mu 
\leq M_{\B_{\rho}}(f_k-f)(x), 
\end{align*}
we use \cite[Proposition 7.13]{HK} to arrive at  
\begin{align}\label{ffk-3}
\|(f_k)_{B(\cdot,r)}-f_{B(\cdot,r)}\|_{L^p(\Sigma, w)} 
&\leq \|M_{\B_{\rho}}(f_k-f)\|_{L^p(\Sigma, w)}  
\nonumber \\
&\lesssim \|f_k-f\|_{L^p(\Sigma, w)}  
\lesssim \varepsilon. 
\end{align}
Gathering \eqref{eq:fk-f} and \eqref{ffk-1}--\eqref{ffk-3}, for any $0<r<\delta$, we obtain that 
\begin{align*}
\|f-f_{B(\cdot,r)}\|_{L^p(\Sigma, w)} 
\lesssim \varepsilon,\quad\text{uniformly in } f \in \K.  
\end{align*}
This proves the condition \eqref{list:FK3} holds.

Let us next show the sufficiency. Assume that \eqref{list:FK1}, \eqref{list:FK2}, and \eqref{list:FK3} hold. Then by \eqref{list:FK2} and \eqref{list:FK3}, for any 
fixed $\varepsilon>0$, there exist $A>0$ and $\delta>0$ such that for 
$0<r<\delta$, 
\begin{align}\label{eq:compactset-1} 
\|f \mathbf{1}_{\Sigma \setminus B(x_0, A)}\|_{L^p(\Sigma, w)} < \varepsilon  \, \text{ for all } f \in \K, 
\\ 
\label{ffb} \|f-f_{B(\cdot, r)}\|_{L^p(\Sigma, w)}<\varepsilon \, \text{ for all } f \in \K.
\end{align}
Fix such an $r>0$. By H\"{o}lder's inequality, we have for all $x, y \in \Sigma$, 
\begin{equation}\label{eq:bdd}
|f_{B(x, r)}| 
\le \frac{1}{\mu(B(x, r))} \bigg(\int_{B(x, r)} |f|^p w\, d\mu\bigg)^{\frac1p}
\bigg(\int_{B(x, r)} w^{1-p'} d\mu \bigg)^{\frac{1}{p'}}, 
\end{equation}
and 
\begin{align}\label{eq:continue}
|f_{B(x, r)} - f_{B(y, r)}| 
&\le \biggl|\frac{1}{\mu(B(x,r))} - \frac{1}{\mu(B(y, r))}\biggr| \int_{B(x, r)} |f| \, d\mu
\nonumber \\ 
&\qquad + \frac{1}{\mu(B(y, r)} \int_{\Sigma} |{\bf{1}}_{B(x,r)}-{\bf{1}}_{B(y,r)}| |f| \, d\mu
\nonumber \\ 
&=: \mathscr{I}_1 + \mathscr{I}_2, 
\end{align}
where 
\begin{align*}
\mathscr{I}_1 & := \biggl|\frac{1}{\mu(B(x,r))} - \frac{1}{\mu(B(y,r))}\biggr|
\|f\|_{L^p(\Sigma, w)}
\bigg(\int_{B(x, r)} w^{1-p'} d\mu \bigg)^{\frac1{p'}}, 
\\ 
\mathscr{I}_2 & := \frac{1}{\mu(B(y,r))} \|f\|_{L^p(\Sigma, w)} 
\bigg(\int_{\Sigma} |{\bf{1}}_{B(x,r)} - {\bf{1}}_{B(y,r)}|^{p'} w^{1-p'} d\mu \bigg)^{\frac1{p'}}. 
\end{align*}
Since $w,\,w^{1-p'} \in L^1_{\loc}(\Sigma, \mu)$, there exists $C_0>0$ such that 
 \begin{align}\label{eq:BA}
\int_{B(x, r)} w^{1-p'} d\mu \le C_0 \quad\text{for all } x\in B(x_0, A). 
\end{align}
By \cite[(3.1) and (3.4)]{CW}, every ball $B(x', r')$ in $(\Sigma, \rho, \mu)$ is totally bounded. Since $\overline{B(x_0,\,A)} \subset B(x_0, A+r)$, $\overline{B(x_0, A)}$ is also totally bounded, which along with the completeness of $(\Sigma, \rho, \mu)$  implies that $\overline{B(x_0, A)}$ is compact. Since $\mu$ is metrically continuous, it follows that $\mu(B(\cdot, r))$ is uniformly continuous on $\overline{B(x_0, A)}$. Then in light of \eqref{eq:bdd}--\eqref{eq:BA}, $\{f_{B(x, r)}\}_{f \in \K}$ is equi-bounded and equi-continuous on the closure of $B(x_0, A)$. Then, by Ascoli-Arzel\'a theorem, it is relatively compact and so totally bounded in $\mathscr{C}(B(x_0, A))$. Consequently, there exists a finite number of functions $\{f_j\}_{j=1}^N \subset \K$ such that 
\begin{equation*}
\inf_{j}\sup_{d(x_0,x)\le A}|f_{B(x, r)} - (f_j)_{B(x, r)}| 
< \varepsilon \, w(B(x_0, A))^{-\frac1p} \ \text{ for all }f\in \K, 
\end{equation*}
which implies that for each $f\in \K$ there exists $j \in \{1,\ldots, N\}$ such that 
\begin{equation}\label{eq:compactset-4}
\sup_{d(x_0,x)\le A}|f_{B(x, r)} - (f_j)_{B(x, r)}| 
< \varepsilon \, w(B(x_0,A))^{-\frac1p}.
\end{equation}
To proceed, we split 
\begin{align}\label{ffj-1}
\|f-f_j\|_{L^p(\Sigma, w)}
&\le \|(f-f_j)\mathbf{1}_{B(x_0, A)}\|_{L^p(\Sigma, w)} 
\nonumber \\
&\qquad+ \|(f-f_j) \mathbf{1}_{\Sigma \setminus B(x_0, A)} \|_{L^p(\Sigma, w)}
\nonumber \\
&=: \mathrm{I} + \mathrm{II}. 
\end{align}
By \eqref{ffb} and \eqref{eq:compactset-4}, one has 
\begin{align}\label{ffj-2}
\mathrm{I} 
&\le \|f-f_{B(x, r)}\|_{L^p(\Sigma, w)}
+ \|(f_{B(x, r)} - (f_j)_{B(x, r)}) \mathbf{1}_{B(x_0, A)}\|_{L^p(\Sigma, w)}
\nonumber \\ 
&\qquad + \|(f_j)_{B(x, r)} - f_j\|_{L^p(\Sigma, w)}
\nonumber \\ 
&\le \varepsilon + \varepsilon\, w(B(x_0, A))^{-\frac1p} \|\mathbf{1}_{B(x_0, A)}\|_{L^p(w)} + \varepsilon
\le 3 \varepsilon. 
\end{align}
In view of \eqref{eq:compactset-1}, there holds 
\begin{align}\label{ffj-3}
\mathrm{II}
\le \|f \mathbf{1}_{\Sigma \setminus B(x_0, A)}\|_{L^p(w)} 
+ \|f_j \mathbf{1}_{\Sigma \setminus B(x_0, A)}\|_{L^p(w)} 
\le 2 \varepsilon. 
\end{align}
Collecting \eqref{ffj-1}--\eqref{ffj-3}, we conclude that $\K$ is totally bounded, hence, relatively compact in $L^p(\Sigma, w)$. 
\end{proof}
%%%%%%%%%%%%%%%%%%%%%%%%%% END END END PROOF %%%%%%%%%%%%%%%%%%%%%%%

%%%%%%%%%%%%%%%%%%%%%%%%%% PROOF PROOF PROOF %%%%%%%%%%%%%%%%%%%%%%%
\begin{proof}[\bf Proof of Theorem \ref{thm:FKhs-2}]
Assume that $\K$ is relatively compact in $L^p(\Sigma, w)$. Then following the proof of Theorem \ref{thm:FKhs-1}, one can check that both \eqref{list:FKhs-1} and \eqref{list:FKhs-2} hold. To justify \eqref{list:FKhs-3}, we need a bit more work. Let $\varepsilon>0$. Since $\K$ is relatively compact, there exists a finite number of functions $\{f_j\}_{j=1}^N \subset \K$ such that for any $g \in \K$, one can find $j \in \{1, \ldots, N\}$ satisfying $\|g-f_j\|_{L^p(\Sigma, w)}<\varepsilon$. Fix $f \in \K$. Then there is some $f_j \in \K$ such that 
\begin{align}\label{eq:ffe}
\|f-f_j\|_{L^p(\Sigma, w)} <\varepsilon. 
\end{align}
Note that 
\begin{align}\label{eq:Ifr}
\mathcal{I}
&:= \int_{\Sigma} \bigg(\fint_{B(x, r)} |f(x)-f(y)|^{\frac{p}{p_0}} d\mu(y)\bigg)^{p_0}w(x) \, d\mu(x)
\nonumber \\ 
&\lesssim \int_{\Sigma} \bigg(\fint_{B(x, r)} |f(x) - f_j(x)|^{\frac{p}{p_0}} d\mu(y)\bigg)^{p_0} w(x) \, d\mu(x)
\nonumber \\ 
&\quad+ \int_{\Sigma} \bigg(\fint_{B(x, r)} |f_j(x) - f_j(y)|^{\frac{p}{p_0}} d\mu(y)\bigg)^{p_0} w(x) \, d\mu(x)
\nonumber \\ 
&\quad+ \int_{\Sigma} \bigg(\fint_{B(x, r)} |f_j(y) - f(y)|^{\frac{p}{p_0}} d\mu(y)\bigg)^{p_0} w(x) \, d\mu(x)
\nonumber \\ 
&=: \mathcal{I}_1 + \mathcal{I}_2 + \mathcal{I}_3. 
\end{align}
By \eqref{eq:ffe}, it is easy to control the first term: 
\begin{align}\label{eq:Ifr-1} 
\mathcal{I}_1 
= \int_{\Sigma} |f(x)-f_j(x)|^p w(x) \, d\mu(x) 
< \varepsilon^p. 
\end{align}
For $\mathcal{I}_3$, we use $w \in A_{p_0, \B_{\rho}}$ and the $L^{p_0}(\Sigma, \mu)$ boundedness of $M_{\B_{\rho}}$ (cf. \cite[Proposition 7.13]{HK}) to obtain  
\begin{align}\label{eq:Ifr-3} 
\mathcal{I}_3 
\le \int_{\Sigma} M_{\B_{\rho}}(|f-f_j|^{\frac{p}{p_0}})^{p_0} w \, d\mu
\lesssim \int_{\Sigma} |f-f_j|^p w \, d\mu
< \varepsilon^p, 
\end{align}
where \eqref{eq:ffe} was used in the last step. To deal with $\mathcal{I}_2$, we see that $w \in L^1_{\loc}(\Sigma, \mu)$, and hence, $\mathscr{C}_b^{\alpha}(\Sigma, \mu)$ is dense in $L^p(\Sigma, w)$ for any $p \in (0, \infty)$. So, we can find $g_j \in \mathscr{C}_b^{\alpha}(\Sigma, \mu)$ such that 
\begin{equation}\label{eq:fge}
\|f_j-g_j\|_{L^p(\Sigma, w)} < \varepsilon. 
\end{equation}
Assume that there exist $r_0, A_0>0$ such that $\supp(g_j) \subset B(x_0, A_0)$ and 
\begin{align}\label{eq:gjgj}
|g_j(x) - g_j(y)|< \varepsilon, \quad\text{ whenever}\quad \rho(x, y) < r_0. 
\end{align}
We then use \eqref{eq:fge} and \eqref{eq:gjgj} to deduce that for any $0<r<r_0$, 
\begin{align}\label{eq:Ifr-2}
\mathcal{I}_2  &\le \int_{\Sigma} \bigg(\fint_{B(x, r)} |f_j(x)-g_j(x)|^{\frac{p}{p_0}} d\mu(y)\bigg)^{p_0}w(x) \, d\mu(x)
\nonumber \\ 
&\qquad+ \int_{\Sigma} \bigg(\fint_{B(x, r)} |g_j(x)-g_j(y)|^{\frac{p}{p_0}} d\mu(y)\bigg)^{p_0} w(x) \, d\mu(x)
\nonumber \\ 
&\qquad+ \int_{\Sigma} \bigg(\fint_{B(x, r)} |g_j(y)-f_j(y)|^{\frac{p}{p_0}} d\mu(y)\bigg)^{p_0} w(x) \, d\mu(x)
\nonumber \\ 
&\le  \int_{\Sigma} |f_j-g_j|^p w \, d\mu
+ \varepsilon^p w(B(x_0, A_0+r_0))  
\nonumber \\ 
&\qquad+ \int_{\Sigma} M_{\B_{\rho}}(|g_j-f_j|^{\frac{p}{p_0}})^{p_0} w \, d\mu
\nonumber \\ 
&\lesssim \varepsilon^p + \varepsilon^p w(B(x_0, A_0+r_0) + \|f_j-g_j\|_{L^p(\Sigma, w)}^p 
\lesssim \varepsilon^p. 
\end{align}
Now gathering \eqref{eq:Ifr}--\eqref{eq:Ifr-3} and \eqref{eq:Ifr-2}, we conclude that $\mathcal{I} \lesssim \varepsilon^p$ for all $0<r<r_0$. This proves \eqref{list:FKhs-3}.

To show the sufficiency, we assume that \eqref{list:FKhs-1}--\eqref{list:FKhs-3} hold. Consider the case $p \ge p_0$. Observing that 
\begin{align*}
|f(x)-f_{B(x, r)}| 
&\le \fint_{B(x, r)} |f(x)-f(y)| d\mu(y) 
\\ 
&\le \bigg(\fint_{B(x, r)} |f(x)-f(y)|^{\frac{p}{p_0}}d\mu(y)\bigg)^{\frac{p_0}{p}}, 
\end{align*}
we use the condition \eqref{list:FKhs-3} to get 
\begin{equation}\label{eq:f-fB}
\lim\limits_{r \to 0} \sup\limits_{f \in \K}\|f-f_{B(\cdot,r)}\|_{L^p(\Sigma, w)}=0.
\end{equation} 
Since $p\ge p_0$ and $w,\, w^{1-p'_0}\in L^1_{\loc}(\Sigma, \mu)$, there holds $w,\, w^{1-p'}\in L^1_{\loc}(\Sigma, \mu)$. Then invoking \eqref{list:FKhs-1}, \eqref{list:FKhs-2}, and \eqref{eq:f-fB}, and Theorem \ref{thm:FKhs-1}, we conclude that $\K$ is relatively compact in $L^p(\Sigma, w)$. 

To deal with the case $p<p_0$, we suppose that $\K$ is a family of non-negative functions. Writing $a:=p/p_0<1$ , we have 
\begin{align}
\label{eq:fafa} &\big|f^a(x) - (f^a)_{B(x, r)} \big| 
\le \fint_{B(x, r)} |f(x)-f(y)|^{\frac{p}{p_0}} d\mu(y), 
\end{align}
which together with \eqref{list:FKhs-3} yields 
\begin{equation}\label{eq:fa-fB}
\lim\limits_{r \to 0} 
\sup\limits_{f \in \K}\| f^a - (f^a)_{B(\cdot, r)}\|_{L^{p_0}(\Sigma, w)}=0.
\end{equation} 
Besides, \eqref{list:FKhs-1} and \eqref{list:FKhs-2} can be rewritten as 
\begin{align}
\label{eq:fGfG}
&\sup\limits_{f \in \K} \| f^a\|_{L^{p_0}(\Sigma, w)} < \infty 
\quad\text{and}\quad 
\lim\limits_{A \to \infty} 
\sup\limits_{f \in \K}\| f^a \mathbf{1}_{\Sigma \setminus B(x_0, A)}\|_{L^{p_0}(\Sigma, w)}=0. 
\end{align} 
Consequently, it follows from \eqref{eq:fa-fB}, \eqref{eq:fGfG}, and Theorem \ref{thm:FKhs-1} that 
\begin{align}\label{Kacom}
\text{$\K^a:=\{f^a: f \in \K\}$ is relatively compact in $L^{p_0}(\Sigma, w)$.} 
\end{align}
Now let $\{f_j\}$ be a sequence of functions in $\K$. By \eqref{Kacom}, there exists a Cauchy subsequence of $\{f_j^a\}$, which we denote again by $\{f_j^a\}$ for simplicity. Then for any $\varepsilon>0$, there exists an integer $N$ such that for all $i, j\ge N$, 
\begin{equation}\label{eq:faij}
\int_{\Sigma} \big|f_i^a - f_j^a \big|^{p_0}w \, d\mu 
< \varepsilon^{p_0}.
\end{equation}
For fixed $i, j \in \N$, we set 
\begin{equation*}
E_\varepsilon :=\bigg\{x \in \Sigma: \frac{f_i(x)+f_j(x)}{|f_i(x)-f_j(x)|}
\le \frac{1}{\varepsilon} \bigg\}, \quad\forall\varepsilon>0.
\end{equation*}
By elementary calculation (see \cite[p. 33]{Tsuji}), for any $a \in (0,1)$, there holds 
\begin{equation}\label{eq:sata}
|s^a-t^a|\le |s-t|^a \le \frac{1}{a}\bigg(\frac{s+t}{|s-t|}\bigg)^{1-a}|s^a-t^a|,\quad \forall s,t>0.
\end{equation}
Then, by $p_0 a=p$, \eqref{eq:faij}, and \eqref{eq:sata}, we have
\begin{align*}
\int_{E_\varepsilon}|f_i - f_j|^p w \, d\mu
&\le a^{-p_0}\varepsilon^{(a-1)p_0}
\int_{E_\varepsilon}|f_i^a - f_j^a|^{p_0}w \, d\mu
\le a^{-p_0}\varepsilon^{p}.
\end{align*}
On the other hand, \eqref{eq:sata} and \eqref{list:FKhs-1} give 
\begin{align*}
\int_{E_\varepsilon^c} |f_i - f_j|^p w \, d\mu
&\le \int_{E_\varepsilon^c}|\varepsilon(f_i + f_j)|^p w \, d\mu
\\
&\le \varepsilon^p \bigg(\int_{E_\varepsilon^c}|f_i|^p w \, d\mu
+\int_{E_\varepsilon^c}|f_j|^{p}w \, d\mu \bigg)
\le 2K_0^p\varepsilon^p, 
\end{align*}
where $K_0:=\sup\limits_{f \in \K} \|f\|_{L^p(\Sigma, w)}<\infty$. The two estimates above show that $\{f_j\}$ is a Cauchy sequence in $\K \subset L^p(\Sigma, w)$. Thus $\K$ is relatively compact in $L^p(\Sigma, w)$. 

To handle the general case, we define  
\[
\K^+ := \{f^+: f \in \K\} \quad\text{and}\quad \K^- := \{f^-: f \in \K\}, 
\]
where 
\[
f^+(x) :=\max\{f(x), 0\} \quad\text{and}\quad 
f^-(x) := \max\{-f(x), 0 \}, \quad\forall f \in \K. 
\]
Then, for all $f \in \K$ and $x, y \in \Sigma$, 
\begin{align*}
0 \le f^+(x) \le |f(x)|, \qquad |f^+(x) - f^+(y)| \le |f(x) - f(y)|, 
\\
0 \le f^-(x) \le |f(x)|, \qquad |f^-(x) - f^-(y)| \le |f(x) - f(y)|. 
\end{align*}
This means that 
\begin{align}\label{eq:GG}
\text{both $\K^+$ and $\K^-$ satisfy the conditions \eqref{list:FKhs-1}--\eqref{list:FKhs-3}}.
\end{align} 
Let $\{f_j\}$ be an arbitrary sequence of functions in $\K$. By \eqref{eq:GG} 
and the conclusion in the preceding case, we conclude that for any 
$\varepsilon>0$ there exists $N_0 \in \N$ such that for all $i, j \ge N_0$, 
\begin{equation*}
\|f_i^+ - f_j^+\|_{L^p(\Sigma, w)} < \varepsilon \quad\text{and}\quad 
\|f_i^- - f_j^-\|_{L^p(\Sigma, w)} < \varepsilon, 
\end{equation*}
which implies 
\begin{align*}
\|f_i - f_j\|_{L^p(\Sigma, w)} 
\le \|f_i^+ - f_j^+\|_{L^p(\Sigma, w)} + \|f_i^- - f_j^-\|_{L^p(\Sigma, w)} 
< 2 \varepsilon. 
\end{align*}
Accordingly, $\{f_j\}$ is a Cauchy sequence in $\K$, so $\K$ is relatively compact in $L^p(\Sigma, w)$. 
\end{proof}
%%%%%%%%%%%%%%%%%%%%%%%% END END END PROOF %%%%%%%%%%%%%%%%%%%%%%%%%

%%%%%%%%%%%%%%%%%%%%% SUBSECTION SUBSECTION SUBSECTION %%%%%%%%%%%%%%%%%%
%%%%%%%%%%%%%%%%%%%%% SUBSECTION SUBSECTION SUBSECTION %%%%%%%%%%%%%%%%%%
\subsection{Extrapolation for compact operators} 
We will see that Theorem \ref{thm:Tb} is a consequence of Theorems \ref{thm:Ap} and \ref{thm:TTb} below.

%%%%%%%%%%%%%%%%%%%%%% THEOREM THEOREM THEOREM %%%%%%%%%%%%%%%%%%%%%%
\begin{theorem}\label{thm:TTb}
Let $(\Sigma, \rho, \mu)$ be a space of homogeneous type. Let $T$ be an $m$-linear or $m$-linearizable operator. Assume that there exists $\vec{q}=(q_1, \ldots, q_m)$ with $1<q_1, \ldots, q_m<\infty$ such that for all $\vec{v}=(v_1, \ldots, v_m) \in A_{\vec{q}, \B_{\rho}}$, 
\begin{align}\label{ttb-1}
\|T(\vec{f})\|_{L^q(\Sigma, v)}
\lesssim \prod_{i=1}^m  \|f_i\|_{L^{q_i}(\Sigma, v_i)}, 
\end{align}
where $\frac1q = \sum_{i=1}^m \frac{1}{q_i}$ and $v=\prod_{i=1}^m v_i^{\frac{q}{q_i}}$. Then, for all $\vec{p}=(p_1, \ldots, p_m)$ with $1<p_1, \ldots, p_m<\infty$, for all $\vec{w}=(w_1, \ldots, w_m) \in A_{\vec{p}, \B_{\rho}}$, for all $\b = (b_1, \ldots, b_m) \in \BMO_{\B_{\rho}}^m$, and for each multi-index $\alpha \in \N^m$, 
\begin{align}\label{ttb-2}
\|[T, \b]_{\alpha}(\vec{f})\|_{L^p(\Sigma, w)}
\lesssim \prod_{i=1}^m \|b_i\|_{\BMO_{\B_{\rho}}}^{\alpha_i} \|f_i\|_{L^{p_i}(\Sigma, w_i)}, 
\end{align}
where $\frac1p = \sum_{i=1}^m \frac{1}{p_i}$ and $w=\prod_{i=1}^m w_i^{\frac{p}{p_i}}$. 
\end{theorem}
%%%%%%%%%%%%%%%%%%%%%% THEOREM THEOREM THEOREM %%%%%%%%%%%%%%%%%%%%%%

Now let us see how to deduce Theorem \ref{thm:Tb} from Theorems \ref{thm:Ap}, \ref{thm:WMIP-4}, and \ref{thm:TTb}. 

%%%%%%%%%%%%%%%%%%%%%%%% PROOF PROOF PROOF %%%%%%%%%%%%%%%%%%%%%%%%%
\begin{proof}[\textbf{Proof of Theorem $\ref{thm:Tb}$.}] 
Let $T$ be an $m$-linear or $m$-linearizable operator. Fix $\alpha \in \N^m$ and $\b=(b_1,\ldots, b_m) \in \BMO_{\B_{\rho}}^m$. Let $\vec{r}=(r_1, \ldots, r_m)$ with $1<r_1, \ldots, r_m<\infty$ be the same as in  \eqref{eq:bT-1}. By Theorem \ref{thm:TTb}, the assumption \eqref{eq:bT-1} implies that for all $\vec{r}=(r_1, \ldots, r_m)$ with $1<r_1, \ldots, r_m<\infty$ and for all $\vec{u}=(u_1, \ldots, u_m) \in A_{\vec{r}, \B_{\rho}}$, 
\begin{align}\label{eq:bT-3}
[T, \b]_{\alpha} \text{ is bounded from $L^{r_1}(\Sigma, u_1) \times \cdots \times L^{r_m}(\Sigma, u_m)$ to $L^r(\Sigma, u)$}, 
\end{align}
where $\frac1r=\sum_{i=1}^m \frac{1}{r_i}$ and $u=\prod_{i=1}^m u_i^{\frac{r}{r_i}}$. By Theorem \ref{thm:WMIP-4}, interpolating between \eqref{eq:bT-3} with $\vec{u}=(1, \ldots, 1)$ and \eqref{eq:bT-2} gives that for all $\vec{s}=(s_1, \ldots, s_m)$ with $1<s_1, \ldots, s_m<\infty$, 
\begin{align}\label{eq:bT-4}
[T, \b]_{\alpha} \text{ is compact from $L^{s_1}(\Sigma, \mu) \times \cdots \times L^{s_m}(\Sigma, \mu)$ to $L^s(\Sigma, \mu)$}, 
\end{align} 
where $\frac1s=\sum_{i=1}^m \frac{1}{s_i}$. Thus, \eqref{eq:bT-3} and \eqref{eq:bT-4} respectively verifies \eqref{eq:Ap-1} and \eqref{eq:Ap-2} with $\vec{v}=(1,\ldots,1)$ for $[T, \b]_{\alpha}$ in place of $T$. Therefore, Theorem \ref{thm:Ap} implies Theorem  \ref{thm:Tb}.  
\end{proof} 
%%%%%%%%%%%%%%%%%%%%%%%% END END END PROOF %%%%%%%%%%%%%%%%%%%%%%%%%

\subsection{Proof of Theorem \ref{thm:Ap}} 
Since compactness is stronger than boundedness, to show Theorem \ref{thm:Ap}, we first establish an extrapolation for multilinear Muckenhoupt classes $A_{\vec{p}}$ on spaces of homogeneous type.  

%%%%%%%%%%%%%%%%%%%%%% THEOREM THEOREM THEOREM %%%%%%%%%%%%%%%%%%%%%
\begin{theorem}\label{thm:extraAp}
Let $(\Sigma, \rho, \mu)$ be a space of homogeneous type. Let $\F$ be a collection of $(m+1)$-tuples of nonnegative measurable functions. Assume that there exists $\vec{q}=(q_1, \ldots, q_m)$ with $1 \le q_1, \ldots, q_m <\infty$ such that for all $\vec{v}=(v_1, \ldots, v_m) \in A_{\vec{q}, \B_{\rho}}$, 
\begin{align}\label{ttb-3}
\|f\|_{L^q(\Sigma, v)} 
\lesssim \prod_{i=1}^m \|f_i\|_{L^{q_i}(\Sigma, v_i)}, \quad (f, f_1, \dots, f_m) \in \F, 
\end{align}
where $\frac1q=\sum_{i=1}^m \frac{1}{q_i}$ and $v=\prod_{i=1}^m v_i^{\frac{q}{q_i}}$. Then, for all $\vec{p}=(p_1, \dots, p_m)$ with $1 < p_1, \ldots, p_m <\infty$ and for all $\vec{w}=(w_1, \ldots, w_m) \in A_{\vec{p}, \B_{\rho}}$,  
\begin{align}\label{ttb-4}
\|f\|_{L^p(\Sigma, w)} 
\lesssim \prod_{i=1}^m \|f_i\|_{L^{p_i}(\Sigma, w_i)}, \quad (f, f_1, \dots, f_m) \in \F, 
\end{align}
where $\frac1p=\sum_{i=1}^m \frac{1}{p_i}$ and $w=\prod_{i=1}^m w_i^{\frac{p}{p_i}}$.  
\end{theorem}
%%%%%%%%%%%%%%%%%%%%%% THEOREM THEOREM THEOREM %%%%%%%%%%%%%%%%%%%%%%

%%%%%%%%%%%%%%%%%%%%%%%% PROOF PROOF PROOF %%%%%%%%%%%%%%%%%%%%%%%%%
\begin{proof}
Since $\mu$ is a doubling measure, one can follow the proof of \cite[Corollary 1.5]{LMO} to show Theorem \ref{thm:extraAp}. Note that without the doubling property of $\mu$, this result will be problematic, which only happens in the case $q_i=1$. In the linear case ($m=1$), different definition of $A_1$ class would cause issues, see \cite[p. 2016]{OP}. 
\end{proof}
%%%%%%%%%%%%%%%%%%%%%%%% END END END PROOF %%%%%%%%%%%%%%%%%%%%%%%%%

%%%%%%%%%%%%%%%%%%%%%%%% CLAIM CLAIM CLAIM %%%%%%%%%%%%%%%%%%%%%%%%%
\begin{claim}\label{clm:Banach}
The following statements hold: 
\begin{list}{\rm (\theenumi)}{\usecounter{enumi}\leftmargin=1cm \labelwidth=1cm \itemsep=0.1cm \topsep=.2cm \renewcommand{\theenumi}{\alph{enumi}}}

\item\label{clm-1} Let $\bB$ be a Banach space and $D$ be a domain in the plane. Then for every $\bB$-valued holomorphic function $f(z)$ in $D$, $\|f(z)\|_{\bB}^p$ is sub-harmonic for any $0<p<\infty$.

\item\label{clm-2} Theorems $3.1$ and $3.5$ in \cite{COY} hold for $m$-linear and $m$-linearizable operators. 

\end{list}
\end{claim}
%%%%%%%%%%%%%%%%%%%%%%%% CLAIM CLAIM CLAIM %%%%%%%%%%%%%%%%%%%%%%%%%

%%%%%%%%%%%%%%%%%%%%%%%% PROOF PROOF PROOF %%%%%%%%%%%%%%%%%%%%%%%%%
\begin{proof}
Let $z \in D$ and $B(z, r) \subset D$. Since $f(z) \in \bB$, there holds 
\begin{align}\label{fzb-1}
\|f(z)\|_{\bB} 
=\sup_{\mathfrak{b} \in \bB^*: \|\mathfrak{b}\|_{\bB^*}=1} 
|\langle f(z), \mathfrak{b} \rangle|. 
\end{align}
Let $\mathfrak{b} \in \bB^*$ with $\|\mathfrak{b}\|_{\bB^*}=1$. Since $f(z)$ is a $\bB$-valued holomorphic function in $D$, $\langle f(z), \mathfrak{b} \rangle$ is a complex valued holomorphic function in $D$, and for any $0<p<\infty$, $|\langle f(z), \mathfrak{b} \rangle|^p$ is sub-harmonic in $D$. Hence, we have 
\begin{align}\label{fzb-2}
|\langle f(z), \mathfrak{b} \rangle|^p
&\le \frac{1}{2\pi} \int_0^{2\pi} |\langle f(r e^{it}+z), \mathfrak{b} \rangle|^p \, dt
\nonumber \\
&\le \frac{1}{2\pi} \int_0^{2\pi} \|f(r e^{it}+z)\|_{\bB}^p \|\mathfrak{b}\|_{\bB^*}^p \, dt. 
\end{align}
Then it follows from \eqref{fzb-1} and \eqref{fzb-2} that 
\begin{align*}
\|f(z)\|_{\bB}^p  
\le \frac{1}{2\pi} \int_0^{2\pi} \|f(r e^{it}+z)\|_{\bB}^p \, dt, 
\end{align*}
which implies that $\|f(z)\|_{\bB}^p$ is sub-harmonic in $D$. This shows part \eqref{clm-1}. 

With part \eqref{clm-1} in hand, we can conclude part \eqref{clm-2} following the proof of \cite[Theorems 3.1, 3.5]{COY} with a slight modification. We give some tips. Let $T$ be an $m$-linearizable operator with $T(\vec{f})(x)=\|\mathcal{T}(\vec{f})(x)\|_{\bB}$. In the proof, $U_{\ell}(z)$ is replaced by 
\[
U_{\ell}(z) := e^{k(z^2-1)/\ell}
(A_1M_1)^{z-1}(A_2M_2)^{-z} \T(\vec{F_z})w^{1-z}_0v_0^{z} G_z^k. 
\]
By the $m$-linearity of $\T$, it is not hard to check that $U_{\ell}(z)$ is $\bB$-valued subharmonic in the strip $S:= \{z \in \mathbb{C}:0 < \Re (z) < 1\}$. Then by part \eqref{clm-1}, 
\begin{equation}\label{Uz}
\text{$\|U_{\ell}(z)\|_{\bB}^{\frac1k}$ is subharmonic in the strip $S$.} 
\end{equation}
As did in the proof of \cite[Theorem 3.1]{COY}, \eqref{Uz} implies that $\Phi_{\ell}(z)$ is subharmonic in $S$. More details are left to the interested reader. 
\end{proof}
%%%%%%%%%%%%%%%%%%%%%%%% END END END PROOF %%%%%%%%%%%%%%%%%%%%%%%%%

%%%%%%%%%%%%%%%%%%%%%%%%%% PROOF PROOF PROOF %%%%%%%%%%%%%%%%%%%%%%%
\begin{proof}[\bf Proof of Theorem \ref{thm:WMIP-4}]
The condition $w_0, v_0 \in A_{\infty, \B_{\rho}}$ implies that $w_0, v_0 \in A_{\kappa, \B_{\rho}}$ for some $\kappa \in (1, \infty)$. For any $r>0$, we use the $L^{\kappa}(\Sigma, w_0)$ boundedness of $M$ and \eqref{eq:WMIP-41} to get 
\begin{align}\label{eq:Iff}
\bigg[&\int_{\Sigma} \bigg( \fint_{B(x, r)} |T(\vec{f})(x) - T(\vec{f})(y)|^{\frac{p_0}{\kappa}} \, d\mu(y)\bigg)^{\kappa} 
w_0(x)\, d\mu(x) \bigg]^{\frac{1}{p_0}}
\nonumber \\ 
&\lesssim \bigg(\int_{\Sigma} |T(\vec{f})|^{p_0} w_0 \, d\mu \bigg)^{\frac{1}{p_0}} 
+ \bigg(\int_{\Sigma} M(|T\vec{f}|^{\frac{p_0}{\kappa}})^{\kappa} w_0\, dx \bigg)^{\frac{1}{p_0}} 
\nonumber \\ 
&\lesssim \|T(\vec{f})\|_{L^{p_0}(\Sigma, w_0)}
\le M_1 \prod_{i=1}^m \|f_i\|_{L^{p_i}(\Sigma, w_i)}. 
\end{align}
On the other hand, it follows from Theorem \ref{thm:FKhs-2} and \eqref{eq:WMIP-42} that  
\begin{align}
\label{eq:WMIP-44} &\|T(\vec{f})\|_{L^{q_0}(v_0)} \leq  M_2 \prod_{i=1}^m \|f_i\|_{L^{q_i}(\Sigma, v_i)}, 
\\
\label{eq:WMIP-45} \lim_{A \to \infty} & \|T(\vec{f}) \mathbf{1}_{\Sigma \setminus B(x_0, A)}\|_{L^{q_0}(\Sigma, v_0)} 
\prod_{i=1}^m \|f_i\|_{L^{q_i}(\Sigma, v_i)}^{-1} =0, 
\end{align} 
and 
\begin{align}\label{eq:WMIP-46} 
\lim_{r \to \infty}  \bigg[\int_{\Sigma} \bigg(&\fint_{B(x, r)}  
|T(\vec{f})(x) - T(\vec{f})(y)|^{\frac{q_0}{\kappa}} d\mu(y)\bigg)^{\kappa} v_0(x) d\mu(x)\bigg]^{\frac{1}{q_0}} 
\nonumber \\
&\times \prod_{i=1}^m \|f_i\|_{L^{q_i}(\Sigma, v_i)}^{-1} =0.
\end{align} 
Then in view of \eqref{eq:exp}, Claim \ref{clm:Banach} part \eqref{clm-2}, and \cite[Theorem 3.1]{COY}, we interpolate between \eqref{eq:WMIP-41} with the bound $\widetilde{M}_1$ and \eqref{eq:WMIP-44} to arrive at 
\begin{equation}\label{eq:WMIP-47}
\sup_{\|f_i\|_{L^{s_i}(\Sigma, u_i)} \le 1 \atop i=1, \ldots, m}
\|T(\vec{f})\|_{L^{s_0}(\Sigma, u_0)} 
\leq \widetilde{M}_1^{1-\theta} M_2^{\theta}. 
\end{equation}
Note that \eqref{eq:WMIP-45} implies that for any $\varepsilon>0$ there exists $A_\varepsilon$ such that for all $A >A_\varepsilon$, 
\begin{equation}\label{eq:WMIP-48}
\|T(\vec{f})\mathbf{1}_{\Sigma \setminus B(x_0, A)}\|_{L^{q_0}(\Sigma, v_0)}  
< \varepsilon \prod_{i=1}^m \|f_i\|_{L^{q_i}(\Sigma, v_i)}. 
\end{equation}
Thus, by \eqref{eq:WMIP-41} with bound $\widetilde{M}_1$ and \eqref{eq:WMIP-48}, Claim \ref{clm:Banach} part \eqref{clm-2} and \cite[Theorem 3.1]{COY} applied to $T(\vec{f}) \mathbf{1}_{\Sigma \setminus B(x_0, A)}$ yield 
\[
\|T(\vec{f})\mathbf{1}_{\Sigma \setminus B(x_0, A)}\|_{L^{s_0}(\Sigma, u_0)} 
\leq \widetilde{M}_1^{1-\theta} \varepsilon^{\theta} \prod_{i=1}^m \|f_i\|_{L^{s_i}(\Sigma, u_i)}, 
\]
which is equivalent to   
\begin{equation}\label{eq:WMIP-49}
\lim_{A \to \infty} \sup_{\|f_i\|_{L^{s_i}(\Sigma, u_i)} \le 1 \atop i=1, \ldots, m}
\|T(\vec{f}) \mathbf{1}_{\Sigma \setminus B(x_0, A)}\|_{L^{s_0}(\Sigma, u_0)} 
=0.
\end{equation}
Additionally, \eqref{eq:WMIP-46} implies that for any $\varepsilon>0$ there exists $r_0=r_0(\varepsilon)>0$ such that for all $0<r<r_0$,
\begin{align*}
\bigg[\int_{\Sigma} \bigg(\fint_{B(x, r)} |T(\vec{f})(x) - T(\vec{f})(y)|^{\frac{q_0}{\kappa}} \, 
&d\mu(y)\bigg)^{\kappa} v_0(x)  \, d\mu(x) \bigg]^{\frac{1}{q_0}} 
\\
&\leq  \varepsilon \prod_{i=1}^m \|f_i\|_{L^{q_i}(\Sigma, v_i)}, 
\end{align*}
which along with \eqref{eq:Iff}, Claim \ref{clm:Banach} part \eqref{clm-2}, and \cite[Theorem 3.5]{COY} leads 
\begin{align*}
\bigg[\int_{\Sigma} \bigg(\fint_{B(x, r)} 
|T(\vec{f})(x) -T(\vec{f})(y)|^{\frac{s_0}{\kappa}} \, &d\mu(y)\bigg)^{\kappa} u(x)  d\mu(x)\bigg]^{\frac{1}{s_0}} 
\\ 
&\leq M_1^{1-\theta} \varepsilon^{\theta} \prod_{i=1}^m \|f_i\|_{L^{s_i}(\Sigma, u_i)}. 
\end{align*}
That is, 
\begin{equation}\label{eq:T_rho}
\begin{aligned}
\lim_{r \to 0} \sup_{\|f_i\|_{L^{s_i}(\Sigma, u_i)} \le 1 \atop i=1, \ldots, m}
\bigg[\int_{\Sigma} \bigg(\fint_{B(x, r)} 
&|T(\vec{f})(x) -T(\vec{f})(y)|^{\frac{s_0}{\kappa}} \, 
\nonumber \\ 
& \times d\mu(y)\bigg)^{\kappa} u_0(x)  \, d\mu(x)\bigg]^{\frac{1}{s_0}} 
= 0. 
\end{aligned}
 \end{equation}
Having shown \eqref{eq:WMIP-47}, \eqref{eq:WMIP-49}, and \eqref{eq:T_rho}, we use Theorem  \ref{thm:FKhs-2} to conclude the proof.
\end{proof}
%%%%%%%%%%%%%%%%%%%%%%%%%% END END END PROOF %%%%%%%%%%%%%%%%%%%%%%%

%%%%%%%%%%%%%%%%%%%%%%% LEMMA LEMMA LEMMA %%%%%%%%%%%%%%%%%%%%%%%%%
\begin{lemma}\label{lem:int-Lp} 
Let $(\Sigma, \rho, \mu)$ be a space of homogeneous type. Let $\vec{p}=(p_1, \ldots, p_m)$ with $1<p_1, \ldots, p_m<\infty$ and $\vec{s}=(s_1, \ldots, s_m)$ with with $1<s_1, \ldots, s_m<\infty$. If $\vec{w} \in A_{\vec{p}, \B_{\rho}}$ and $\vec{v} \in A_{\vec{s}, \B_{\rho}}$, then there exist $\theta \in (0, 1)$,  $\vec{r}=(r_1, \ldots, r_m)$ with $1<r_1, \ldots, r_m<\infty$, and $\vec{u} \in A_{\vec{r}, \B_{\rho}}$ such that 
\begin{align*}
L^p(\Sigma, w) &= [L^r(\Sigma, u), L^s(\Sigma, v)]_{\theta}, 
\\
L^{p_i}(\Sigma, w_i) &= [L^{r_i}(\Sigma, u_i), L^{s_i}(\Sigma, v_i)]_{\theta}, 
\end{align*}
for each $i=1,\ldots,m$, where $\frac1p=\sum_{i=1}^m \frac{1}{p_i}$, $\frac1s=\sum_{i=1}^m \frac{1}{s_i}$, $\frac1r = \sum_{i=1}^m \frac{1}{r_i}$, $w=\prod_{i=1}^m w_i^{\frac{p}{p_i}}$, $v=\prod_{i=1}^m v_i^{\frac{s}{s_i}}$, and $u=\prod_{i=1}^m u_i^{\frac{r}{r_i}}$. 
\end{lemma}
%%%%%%%%%%%%%%%%%%%%%%% LEMMA LEMMA LEMMA %%%%%%%%%%%%%%%%%%%%%%%%%

%%%%%%%%%%%%%%%%%%%%%%%% PROOF PROOF PROOF %%%%%%%%%%%%%%%%%%%%%%%%%
\begin{proof}
The proof is similar to that of \cite[Lemma 4.1]{COY}, which heavily depends on the reverse H\"{o}lder inequality. In the current setting, we use the weak reverse H\"{o}lder inequality in Lemma \ref{lem:RH}.  
\end{proof}
%%%%%%%%%%%%%%%%%%%%%%%% END END END PROOF %%%%%%%%%%%%%%%%%%%%%%%%%

Now let us turn to the proof of Theorem \ref{thm:Ap}. Let $\vec{p}=(p_1, \dots, p_m)$ with $1 < p_1, \ldots, p_m<\infty$, $\vec{w}=(w_1, \ldots, w_m) \in A_{\vec{p}, \B_{\rho}}$, and $w=\prod_{i=1}^m w_i^{\frac{p}{p_i}}$. Since $\vec{v}=(v_1, \ldots, v_m) \in A_{\vec{s}, \B_{\rho}}$, Lemma \ref{lem:int-Lp} gives that for every $i=1, \ldots, m$,  
\begin{equation}\label{eq:Lpsq}
\begin{aligned}
L^p(\Sigma, w) &= [L^r(\Sigma, u), L^s(\Sigma, v)]_{\theta}, 
\\  
L^{p_i}(\Sigma, w_i) &= [L^{s_i}(\Sigma, u_i), L^{q_i}(\Sigma, v_i)]_{\theta},
\end{aligned}
\end{equation}
for some $\theta \in (0, 1)$, 
\begin{align}\label{rru}
\text{$\vec{r}=(r_1, \ldots, r_m)$ with $1<r_1, \ldots, r_m < \infty$, and $\vec{u} \in A_{\vec{r}, \B_{\rho}}$,} 
\end{align} 
where $u=\prod_{i=1}^m u_i^{\frac{r}{r_i}}$.

On the other hand, by Theorem \ref{thm:extraAp}, the assumption \eqref{eq:Ap-1} implies that 
\begin{align}\label{eq:Apmu}
T \text{ is bounded from $L^{t_1}(\Sigma, \sigma_1) \times \cdots \times L^{t_m}(\Sigma, \sigma_m)$ to $L^t(\Sigma, \sigma)$}, 
\end{align}
for all $\vec{t}=(t_1, \dots, t_m)$ with $1<t_1, \ldots, t_m<\infty$ and for all $\vec{\sigma} \in A_{\vec{t}, \B_{\rho}}$, where $\frac1t=\sum_{i=1}^m \frac{1}{t_i}$ and $\sigma=\prod_{i=1}^m \sigma_i^{\frac{t}{t_i}}$. Hence, \eqref{eq:Apmu} applied to the exponents $\vec{r}$ and weights $\vec{u} \in A_{\vec{r}, \B_{\rho}}$ in \eqref{rru} yields 
\begin{align}\label{eq:Lsu}
T \text{ is bounded from $L^{r_1}(\Sigma, u_1) \times \cdots \times L^{r_m}(\Sigma, u_m)$ to $L^r(\Sigma, u)$}. 
\end{align}
By \eqref{eq:Ap-2}, we see that  
\begin{align}\label{eq:Lqv}
T \text{ is compact from $L^{s_1}(\Sigma, v_1) \times \cdots \times L^{s_m}(\Sigma, v_m)$ to $L^s(\Sigma, v)$}.  
\end{align} 
Additionally, it follows from Lemma \ref{lem:ApAp} that 
\begin{align}\label{uvAi}
u, v \in A_{\infty, \B_{\rho}}. 
\end{align}
Therefore, invoking \eqref{eq:Lpsq}, \eqref{eq:Lsu}--\eqref{uvAi}, and Theorem \ref{thm:WMIP-4}, we deduce that $T$ is compact from $L^{p_1}(\Sigma, w_1) \times \cdots \times L^{p_m}(\Sigma, w_m)$ to $L^p(\Sigma, w)$. The proof is complete. 
\qed

%%%%%%%%%%%%%%%%%%%%% SUBSECTION SUBSECTION SUBSECTION %%%%%%%%%%%%%%%%%%
%%%%%%%%%%%%%%%%%%%%% SUBSECTION SUBSECTION SUBSECTION %%%%%%%%%%%%%%%%%%
\subsection{Proof of Theorem \ref{thm:TTb}}
To show Theorem \ref{thm:TTb}, we establish an extrapolation from multilinear operators to the corresponding commutators but with Banach exponents. We will use Theorem \ref{thm:extraAp} to get the full range of exponents. 

%%%%%%%%%%%%%%%%%%%%%% THEOREM THEOREM THEOREM %%%%%%%%%%%%%%%%%%%%%%
\begin{theorem}\label{thm:TTbb}
Let $(\Sigma, \rho, \mu)$ be a space of homogeneous type. Let $\frac1r=\sum_{i=1}^m \frac{1}{r_i} \le 1$ with $1 \le r_1, \ldots, r_m<\infty$. Assume that $T$ is an $m$-linear or $m$-linearizable operator such that for all $\vec{v}=(v_1, \ldots, v_m) \in A_{\vec{r}, \B_{\rho}}$, 
\begin{align}\label{ttb-5}
\|T(\vec{f})\|_{L^r(\Sigma, v)}
\lesssim \prod_{i=1}^m  \|f_i\|_{L^{r_i}(\Sigma, v_i)}, 
\end{align}
where $v=\prod_{i=1}^m v_i^{\frac{r}{r_i}}$. Then, for all $\vec{w}=(w_1, \ldots, w_m) \in A_{\vec{r}, \B_{\rho}}$, for all $\b = (b_1, \ldots, b_m) \in \BMO_{\B_{\rho}}^m$, and for each multi-index $\alpha \in \N^m$, 
\begin{align}\label{ttb-6}
\|[T, \b]_{\alpha}(\vec{f})\|_{L^r(\Sigma, w)}
\lesssim \prod_{i=1}^m \|b_i\|_{\BMO_{\B_{\rho}}}^{\alpha_i} \|f_i\|_{L^{r_i}(\Sigma, w_i)}, 
\end{align}
where $w=\prod_{i=1}^m w_i^{\frac{r}{r_i}}$. 
\end{theorem}
%%%%%%%%%%%%%%%%%%%%%% THEOREM THEOREM THEOREM %%%%%%%%%%%%%%%%%%%%%%

We need the following John-Nirenberg inequality on spaces of homogeneous type.
 
%%%%%%%%%%%%%%%%%%%%%%%% LEMMA LEMMA LEMMA %%%%%%%%%%%%%%%%%%%%%%%%
\begin{lemma}\label{lem:JN} 
Let $(\Sigma, \rho, \mu)$ be a spaces of homogeneous type. There exist constants $c_0, c_1 \in (1, \infty)$ such that for any $f \in \BMO_{\B_{\rho}}$, 
\begin{align}\label{JN-1}
\mu(\{x \in B: |f(x)-f_B| > \lambda\})
\le c_0 \, e^{-\frac{c_1 \lambda}{\|f\|_{\BMO_{\B_{\rho}}}}} \mu(B),  
\end{align}
for all $B \in \B_{\rho}$ and $\lambda>0$. In particular, 
\[
\frac{c_1}{c_0+1} \|f\|_{\osc_{\exp L, \B_{\rho}}} \le \|f\|_{\BMO_{\B_{\rho}}} \le \|f\|_{\osc_{\exp L, \B_{\rho}}}.
\]  
\end{lemma}
%%%%%%%%%%%%%%%%%%%%%%%% LEMMA LEMMA LEMMA %%%%%%%%%%%%%%%%%%%%%%%%

%%%%%%%%%%%%%%%%%%%%%%%% PROOF PROOF PROOF %%%%%%%%%%%%%%%%%%%%%%%%%
\begin{proof}
The inequality \eqref{JN-1} is a consequence of \cite[Theorem 1.4]{MP}. Then by \cite[Proposition 1.1.4 ]{Gra-1}, we have for any $B \in \B$ and $\gamma \ge \frac{c_0+1}{c_1}$, 
\begin{align*}
&\fint_B \Big(e^{\frac{|f-f_B|}{\gamma \|f\|_{\BMO_{\B_{\rho}}}}} -1 \Big) d\mu 
\\ 
&=\frac{1}{\mu(B)} \int_0^{\infty} e^\lambda \mu(\{x \in B: |f(x)-f_B|>\gamma \lambda \|f\|_{\BMO_{\B_{\rho}}}\}) \, dt
\\
&\le  c_0 \int_0^{\infty} e^{\lambda} e^{-c_1 \gamma \lambda} \, d\lambda
=\frac{c_0}{c_1\gamma-1}
\le 1, 
\end{align*}
which means that $\|f\|_{\osc_{\exp L, \B_{\rho}}} \le \frac{c_0+1}{c_1} \|f\|_{\BMO_{\B_{\rho}}}$. 
\end{proof}
%%%%%%%%%%%%%%%%%%%%%%%% END END END PROOF %%%%%%%%%%%%%%%%%%%%%%%%%

With Lemmas \ref{lem:RH} and \ref{lem:JN} in hand, mimicking the proof of {\cite[Theorem 4.13]{BMMST}}, we can use the Cauchy integral trick to prove Theorem \ref{thm:TTbb}. We mention that Lemmas \ref{lem:RH} and \ref{lem:JN} are used to show that given $\vec{w}=(w_1, \ldots, w_m) \in A_{\vec{p}, \B_{\rho}}$, 
\begin{align*}
\vec{v} = (v_1, \ldots, v_m) 
:=(w_1 e^{-\mathrm{Re}(z_1) p_1 b_1}, \ldots, w_m e^{-\mathrm{Re}(z_m) p_m b_m}) 
\in A_{\vec{p}, \B_{\rho}}, 
\end{align*}
whenever $|z_i| \le \delta_i$ for some appropriate $\delta_i>0$, $i=1, \ldots, m$. Details are left to the reader. 

Finally, let us present the proof of Theorem \ref{thm:TTb}. Fix $\b = (b_1, \ldots, b_m) \in \BMO_{\B_{\rho}}^m$ and $\alpha \in \N^m$. Set 
\begin{align*}
r_1=\cdots=r_m=m+1 \quad\text{ and }\quad 
\frac1r := \sum_{i=1}^m \frac{1}{s_i} = \frac{m}{m+1}<1.
\end{align*}
Observe that \eqref{ttb-1} coincides with \eqref{ttb-3}. Then Theorem \ref{thm:extraAp} applied to the exponents $\vec{r}=(r_1, \ldots, r_m)$ gives that for all $\vec{v}=(v_1, \ldots, v_m) \in A_{\vec{r}, \B_{\rho}}$, 
\begin{align}\label{ttb-7}
\|T(\vec{f})\|_{L^r(\Sigma, v)} 
\lesssim \prod_{i=1}^m \|f_i\|_{L^{r_i}(\Sigma, v_i)}, 
\end{align}
where $v=\prod_{i=1}^m v_i^{\frac{r}{r_i}}$. Note that \eqref{ttb-7} agrees with \eqref{ttb-5}. Hence, it follows from Theorem \ref{thm:TTbb} that for all $\vec{v}=(v_1, \ldots, v_m) \in A_{\vec{r}, \B_{\rho}}$, 
\begin{align}\label{ttb-8}
\|[T, \b]_{\alpha}(\vec{f})\|_{L^r(\Sigma, v)}
\lesssim \prod_{i=1}^m \|b_i\|_{\BMO_{\B_{\rho}}}^{\alpha_i} \|f_i\|_{L^{r_i}(\Sigma, v_i)}, 
\end{align}
where $v=\prod_{i=1}^m v_i^{\frac{r}{r_i}}$. Now we see that \eqref{ttb-8} verifies \eqref{ttb-3} for the exponents $\vec{r}$ and the $(m+1)$-tuples $\big([T, \b]_{\alpha}(\vec{f}), f_1, \ldots, f_m \big)$ instead of $\vec{q}$ and $(f, f_1, \ldots, f_m)$. Consequently, by Theorem \ref{thm:extraAp}, we conclude that for all $\vec{p}=(p_1, \ldots, p_m)$ with $1<p_1, \ldots, p_m<\infty$ and for all $\vec{w}=(w_1, \ldots, w_m) \in A_{\vec{p}, \B_{\rho}}$, 
\begin{align*}
\|[T, \b]_{\alpha}(\vec{f})\|_{L^p(\Sigma, w)}
\lesssim \prod_{i=1}^m \|b_i\|_{\BMO_{\B_{\rho}}}^{\alpha_i} \|f_i\|_{L^{p_i}(\Sigma, w_i)}, 
\end{align*}
where $\frac1p=\sum_{i=1}^m \frac{1}{p_i}$ and $w=\prod_{i=1}^m w_i^{\frac{p}{p_i}}$. This shows \eqref{ttb-2} and completes the proof. \qed

\vspace{0.4cm}

\noindent{\bf Acknowledgements } 
The first author acknowledges financial support from Spanish Ministry of Science and Innovation through the Ram\'{o}n y Cajal 2021 (RYC2021-032600-I), through the ``Severo Ochoa Programme for Centres of Excellence in R\&D'' (CEX2019-000904-S), and through PID2019-107914GB-I00, and from the Spanish National Research Council through the ``Ayuda extraordinaria a Centros de Excelencia Severo Ochoa'' (20205CEX001). The second author was partially supported by CONICET and SECYT-UNC. The third author was partially supported by FONCyT PICT 2018-02501 and PICT 2019-00018 and by Junta de Andaluc\'{i}a UMA18FEDERJA002 and FQM 354. The fourth author was partly supported by the National Key R\&D Program of China (No. 2020YFA0712900) and NNSF of China (No. 12271041). Finally, the authors would like to thank the referee for careful reading and valuable comments, which lead to the improvement of this paper. 

\medskip 

\noindent{\bf Data Availability} Our manuscript has no associated data.

\medskip 
\noindent{\bf\Large Declarations}
\medskip 

\noindent{\bf Conflict of interest} The authors state that there is no conflict of interest.

%%%%%%%%%%%%%%%%%%% BIBLIGRAPHY BIBLIGRAPHY BIBLIGRAPHY %%%%%%%%%%%%%%%%%%
%%%%%%%%%%%%%%%%%%% BIBLIGRAPHY BIBLIGRAPHY BIBLIGRAPHY %%%%%%%%%%%%%%%%%%

\end{document}